\documentclass[reqno]{amsart}

\usepackage{amssymb,amscd}

\usepackage{multirow}

\usepackage[utf8]{inputenc}
\usepackage[table]{xcolor}

\usepackage{hyperref}
\hypersetup{colorlinks=false}

\theoremstyle{plain}

\newtheorem{thm}{Theorem}[section]
\newtheorem{lem}[thm]{Lemma}
\newtheorem{cor}[thm]{Corollary}
\newtheorem{prop}[thm]{Proposition}

\theoremstyle{definition}

\newtheorem{defn}[thm]{Definition}

\newtheorem{ex}[thm]{Example}

\theoremstyle{remark}

\newtheorem{rem}{Remark}
\newtheorem{claim}{Claim}

\newtheorem*{ack}{Acknowledgment}

\newcommand{\Z}{\mathbb{Z}}
\newcommand{\R}{\mathbb{R}}

\newcommand{\N}{\mathbb{N}}

\renewcommand{\S}{\mathbb{S}}

\renewcommand{\AA}{\mathcal{A}}
\newcommand{\BB}{\mathcal{B}}

\newcommand{\DD}{\mathcal{D}}
\newcommand{\EE}{\mathcal{E}}
\newcommand{\FF}{\mathcal{F}}
\newcommand{\GG}{\mathcal{G}}
\newcommand{\HH}{\mathcal{H}}

\newcommand{\NN}{\mathcal{N}}
\newcommand{\OO}{\mathcal{O}}
\newcommand{\PP}{\mathcal{P}}

\newcommand{\RR}{\mathcal{R}}
\renewcommand{\SS}{\mathcal{S}}
\newcommand{\UU}{\mathcal{U}}

\newcommand{\supp}{\operatorname{supp}}

\newcommand{\id}{\operatorname{id}}

\newcommand{\sign}{\operatorname{sign}}

\newcommand{\Aut}{\operatorname{Aut}}
\newcommand{\Mor}{\operatorname{Mor}}
\newcommand{\depth}{\operatorname{depth}}

\newcommand{\diam}{\operatorname{diam}}
\newcommand{\vol}{\operatorname{vol}}

\newcommand{\Cl}{\operatorname{Cl}}
\newcommand{\Hess}{\operatorname{Hess}}

\newcommand{\Crit}{\operatorname{Crit}}
\newcommand{\can}{\operatorname{can}}
\newcommand{\spec}{\operatorname{spec}}
\newcommand{\Tr}{\operatorname{Tr}}

\newcommand{\Cinf}{C^\infty}
\newcommand{\D}{\Delta}
\newcommand{\sm}{\smallsetminus}

\newcommand{\ol}{\overline}

\newcommand{\bDelta}{\boldsymbol{\D}}

\newcommand{\bd}{\mathbf{d}}
\newcommand{\bw}{\mathbf{w}}

\newcommand{\bD}{\mathbf{D}}

\newcommand{\bH}{\mathbf{H}}

\newcommand{\fF}{\mathfrak{F}}
\newcommand{\fG}{\mathfrak{G}}
\newcommand{\fH}{\mathfrak{H}}
\newcommand{\fN}{\mathfrak{N}}

\newdimen\theight
\def\TeXref#1{%
             \leavevmode\vadjust{\setbox0=\hbox{{\tt
                     \quad\quad  {\small \textrm #1}}}%
             \theight=\ht0
             \advance\theight by \lineskip
             \kern -\theight \vbox to
             \theight{\rightline{\rlap{\box0}}%
             \vss}%
             }}%

\title{Witten's perturbation on strata}

\author[J.A. \'Alvarez L\'opez]{Jes\'us A. \'Alvarez L\'opez}
\address{Departamento de Xeometr\'{\i}a e Topolox\'{\i}a\\
         Facultade de Matem\'aticas\\
         Universidade de Santiago de Compostela\\
         15782 Santiago de Compostela\\ Spain}
\email{jesus.alvarez@usc.es}
\thanks{The first author is partially supported by MICINN, Grants MTM2008-02640 and MTM2011-25656, and by MEC, Grant PR2009-0409}

\author[M. Calaza]{Manuel Calaza}
\address{Research Lab 10 and Rheumatology Unit\\
Hospital Clinico Universitario de Santiago\\
15706 Santiago de Compostela\\ Spain}
\email{manuel.calaza@usc.es}

\date{}



\begin{document}

\maketitle

\begin{abstract}
  The main result is a version of Morse inequalities for the minimum and maximum ideal boundary conditions of the de~Rham complex on strata of compact Thom-Mather stratifications, equipped with adapted metrics. An adaptation of the analytic method of Witten is used in the proof, as well as certain operator related with the Dunkl harmonic oscillator.
\end{abstract}

\tableofcontents

\section{Introduction} \label{s:intro}

Let $d$ be the de~Rham derivative acting on the space $\Omega_0(M)$ of compactly supported differential forms on a Riemannian manifold $M$. Its closed extensions to complexes in the Hilbert space of square integrable differential forms are called ideal boundary conditions (i.b.c.). There is a minimum i.b.c., $d_{\text{\rm min}}=\ol{d}$, and a maximum i.b.c., $d_{\text{\rm max}}=\delta^*$, where $\delta$ is de~Rham coderivative acting on $\Omega_0(M)$. Each i.b.c.\ defines a Laplacian in a standard way; in particular, the Laplacian defined by $d_{\text{\rm min/max}}$ is denoted by $\D_{\text{\rm min/max}}$. It is well known that $d_{\text{\rm min}}=d_{\text{\rm max}}$ if $M$ is complete, but suppose that $M$ may not be complete. The i.b.c.\ $d_{\text{\rm min/max}}$ defines the min/max-cohomology $H_{\text{\rm min/max}}(M)$, min/max-Betti numbers $\beta_{\text{\rm min/max}}^r$, and min/max-Euler characteristic $\chi_{\text{\rm min/max}}$ (if the min/max-Betti numbers are finite); these are quasi-isometric invariants of $M$. In particular, $H_{\text{\rm max}}(M)$ is the $L^2$ cohomology of $M$ \cite{Cheeger1980}.  If $M$ is orientable, then $\D_{\text{\rm max}}$ corresponds to $\D_{\text{\rm min}}$ by the Hodge star operator. These concepts can indeed be defined for arbitrary elliptic complexes  \cite{BruningLesch1992}.

From now on, assume that $M$ is a stratum of a compact Thom-Mather stratification $A$ \cite{Thom1969,Mather1970,Mather1973,Verona1984}. Roughly speaking, on a neighborhood $O$ of each $x\in\ol{M}$, there is a chart of $A$ with values in a product $\R^m\times c(L)$, where:
  \begin{itemize}
    
    \item $L$ is a compact Thom-Mather stratification of lower depth, and $c(L)=L\times[0,\infty)/L\times\{0\}$ (the cone with link $L$);
    
    \item $x$ corresponds to $(0,*)$, where $*$ is the vertex of $c(L)$; and
    
    \item $M\cap O$ corresponds to $\R^m\times M'$ for some stratum $M'$ of $c(L)$. 
    
  \end{itemize}
We have, either $M'=N\times\R_+$ for some stratum $N$ of $L$, or $M'=\{*\}$; thus $x\in M$ just when $M'=\{*\}$. The radial function $\rho:c(L)\to[0,\infty)$ is induced by the second factor projection $L\times[0,\infty)\to[0,\infty)$. This radial function is required to be preserved by the changes of the above type of charts. If $\rho_0$ denotes the standard norm of $\R^m$, it is also said that $\sqrt{\rho_0^2+\rho^2}$ is the radial function of $\R^m\times c(L)$.

Equip $M$ with a Riemannian metric $g$, which is adapted in the following sense defined by induction on the depth of $M$ \cite{Cheeger1980,Cheeger1983}: there is a chart as above around each $x\in\ol{M}\sm M$ so that $g|_{M\cap O}$ is quasi-isometric to a model metric of the form $g_0+\rho^2\tilde g+(d\rho)^2$ on $\R^m\times N\times\R_+$, where $g_0$ is the Euclidean metric on $\R^m$ and $\tilde g$ an adapted metric on $N$; this is well defined since $\depth N<\depth M$. Note that $g$ may not be complete. The first main result of the paper is the following.

\begin{thm}\label{t:spectrum of Delta_min/max}
  The following properties hold on any stratum of a compact Thom-Mather stratification with an adapted metric:
    \begin{itemize}
  
      \item[(i)] $\D_{\text{\rm min/max}}$ has a discrete spectrum, $0\le\lambda_{\text{\rm min/max},0}\le\lambda_{\text{\rm min/max},1}\le\cdots$, where each eigenvalue is repeated according to its multiplicity.
    
      \item[(ii)] $\liminf_k\lambda_{\text{\rm min/max},k}\,k^{-\theta}>0$ for some $\theta>0$.
    
    \end{itemize}
\end{thm}

Theorem~\ref{t:spectrum of Delta_min/max}-(i) is essentially due to J.~Cheeger \cite{Cheeger1980,Cheeger1983}. Theorem~\ref{t:spectrum of Delta_min/max}-(ii) is a weak version on strata of the Weyl's asymptotic formula (see e.g.\ \cite[Theorem~8.16]{Roe1998}); to the authors' knowledge, it is a completely new contribution. Other developments of elliptic theory on strata were made in \cite{BruningLesch1993,
Lesch1997,HunsickerMazzeo2005,Schulze2009,DebordLescureNistor2009,AlbinLeichtnamMazzeoPiazza2012,AlbinLeichtnamMazzeoPiazza:Hodge}. In particular, another proof of Theorem~\ref{t:spectrum of Delta_min/max}-(i) was given in \cite{AlbinLeichtnamMazzeoPiazza2012,AlbinLeichtnamMazzeoPiazza:Hodge}, as a first step in the study of the signature operator on strata. 

In this paper, the term ``relative(ly)'' (or simply ``rel-'') usually means that some condition is required in the intersection of $M$ with small neighborhoods of the points in $\ol M$. Sometimes, this idea can be simplified because the stratification is compact, but non-compact stratifications will be also considered in the proofs.

A smooth function $f$ on $M$ is called rel-admissible when the functions $|df|$ and $|\Hess f|$ are bounded. In this case, $f$ may not have any continuous extension to $\ol{M}$, but it has a continuous extension to the (componentwise) metric completion $\widehat{M}$ of $M$, which is another Thom-Mather stratification. Then it makes sense to say that $x\in\widehat{M}$ is a rel-critical point of $f$ when there is a sequence $(y_k)$ in $M$ such that $\lim_ky_k=x$ in $\widehat{M}$ and $\lim_k|df(y_k)|=0$. The set of rel-critical points of $f$ is denoted by $\Crit_{\text{\rm rel}}(f)$. It will be said that $f$ is a rel-Morse function on $M$ if it is rel-admissible, and there exists a local model of $\widehat{M}$ centered at every $x\in\Crit_{\text{\rm rel}}(f)$ of the form $\R^{m_+}\times\R^{m_-}\times c(L_+)\times c(L_-)$ so that:
  \begin{itemize}
  
    \item $M$ corresponds to the stratum $\R^{m_+}\times\R^{m_-}\times M_+\times M_-$, where $M_\pm$ is a stratum of $c(L_\pm)$; and
    
    \item $f$ corresponds to a constant plus the model function $\frac{1}{2}(\rho_+^2-\rho_-^2)$ on $\R^{m_+}\times\R^{m_-}\times M_+\times M_-$, where $\rho_\pm$ is the radial function on $\R^{m_\pm}\times c(L_\pm)$.
    
  \end{itemize}
This local model makes sense because the product of two Thom-Mather stratifications admits a Thom-Mather structure; in particular, the product of two cones becomes a cone. There is no canonical choice of a product Thom-Mather structure, but all of them have the same adapted metrics. The above local condition is used instead of requiring that $\Hess f$ is ``rel-non-degenerate'' at the rel-critical points because a ``rel-Morse lemma'' is missing.

Suppose that $f$ is a rel-Morse function on $M$. For each $r\in\Z$ and $x\in\Crit_{\text{\rm rel}}(f)$, the above local data is used to define a natural number $\nu_{x,\text{\rm min/max}}^r$ (Definition~\ref{d: nu_min/max^r}); we omit the precise definition here because it is rather involved. Let $\nu_{\text{\rm min/max}}^r=\sum_x\nu_{x,\text{\rm min/max}}^r$ with $x$ running in $\Crit_{\text{\rm rel}}(f)$. Our second main result is the following.

\begin{thm}\label{t:Morse inequalities}
  For any rel-Morse function on a stratum of a compact Thom-Mather stratification, equipped with an adapted metric, we have the inequalities
    \begin{align*}
      \beta_{\text{\rm min/max}}^0&\le\nu_{\text{\rm min/max}}^0\;,\\
      \beta_{\text{\rm min/max}}^1-\beta_{\text{\rm min/max}}^0&\le\nu_{\text{\rm min/max}}^1-\nu_{\text{\rm min/max}}^0\;,\\
      \beta_{\text{\rm min/max}}^2-\beta_{\text{\rm min/max}}^1+\beta_{\text{\rm min/max}}^0&\le\nu_{\text{\rm min/max}}^2-\nu_{\text{\rm min/max}}^1+\nu_{\text{\rm min/max}}^0\;,
    \end{align*}
  etc., and the equality
    \[
      \chi_{\text{\rm min/max}}=\sum_r(-1)^r\,\nu_{\text{\rm min/max}}^r\;.
    \]
\end{thm}

We also show the existence of rel-Morse functions (Proposition~\ref{p:existence of rel-Morse functions}). For instance, for any smooth action of a compact Lie group on a closed manifold, any invariant Morse-Bott function whose critical manifolds are orbits induces a rel-Morse function on the regular stratum of the orbit space.

We adapt the well known analytic method of E.~Witten \cite{Witten1982} to prove Theorem~\ref{t:Morse inequalities}; specially, as it is described in \cite[Chapters~9 and~14]{Roe1998}. Since the difference between $\D$ and its Witten's perturbation is bounded, Witten's method is also used to prove Theorem~\ref{t:spectrum of Delta_min/max}, which becomes a step in the proof of Theorem~\ref{t:Morse inequalities}. In our setting, Witten's method reduces the proof to a rel-local analysis around the rel-critical points. The rel-local analysis is made for a cone whose link is a compact stratification of lower depth, where we consider a model rel-Morse function and a model metric. We use induction on the depth in this way. For the cone, the Witten's complex turns out to be a direct sum of simple elliptic complexes. The Laplacians of these simple complexes can be studied using the Dunkl harmonic oscillator \cite{AlvCalaza2014}, obtaining the needed spectral information of their maximum/minimum i.b.c. Following Witten's method, this rel-local analysis gives the ``cohomological contribution'' from the rel-critical points. Another step of the method shows the ``null cohomological contribution'' away from the rel-critical points. In this part, some arguments of \cite[Chapter~14]{Roe1998} cannot be used by the lack of a Sobolev embedding theorem in this setting. Then a new argument is applied using Theorem~\ref{t:spectrum of Delta_min/max}-(ii). For the reader's convenience, an overall idea of the strategy of the proofs is given in Section~\ref{s: overall idea}.

The needed proofs about stratifications and Hilbert complexes are confined in Appendices~\ref{a: proofs stratifications} and~\ref{a: proofs Hilbert}, so that the readers are quickly guided to the main steps of the proofs of Theorems~\ref{t:spectrum of Delta_min/max} and~\ref{t:Morse inequalities}.

Goresky-MacPherson proved a version of the Morse inequalities for complex analytic varieties with Whitney stratifications \cite[Chapter~6, Section~6.12]{GoreskyMacPherson1988:stratifiedMorse}. They involve intersection homology with lower middle perversity, which is isomorphic to the $L^2$ cohomology of the regular stratum with any adapted metric \cite{CheegerGoreskyMacPherson1982}. Our version of Morse functions, critical points and numbers $\nu_{x,\text{\rm min/max}}^r$ are different; thus Theorem~\ref{t:Morse inequalities} can be considered as a complement to their theory. U.~Ludwig gave an analytic interpretation of the Morse theory of Goresky-MacPherson for conformally conic manifolds \cite{Ludwig2010,Ludwig2011b,Ludwig2011a,Ludwig2013}; in particular, her Morse functions are quite different from ours: contrary to our conditions, the differential is bounded away from zero close to the frontier of the regular stratum, and the Hessian may be unbounded. In \cite{Ludwig2014}, she also studied in a different way the Witten's deformation for radial Morse functions on stratified pseudo-manifolds, which overlaps our work. The radial Morse functions are a particular case of our rel-Morse functions, where only one of the factors $c(L_\pm)$ is considered in the local charts around the rel-critical points. She also assumes that the stratified pseudo-manifolds satisfy the so called Witt condition or have isolated singularities. Then she establishes a spectral gap theorem that gives the Morse inequalities. In the case of isolated singularities, she also continues with the adaptation of the rest of Witten's program on analytic Morse theory, comparing the complex of eigenforms corresponding to small eigenvalues with a version of the Morse-Thom-Smale complex. Thus her results go further than Theorem~\ref{t:Morse inequalities}, but under additional restrictive hypotheses.

After finishing the paper, R.~Mazzeo pointed out to us that the remarks of Section~\ref{s: Sobolev} were already known by him.

It would be interesting to extend our results in the following ways: for ``rel-Morse-Bott functions'', whose rel-critical point set consists of Thom-Mather substratifications, and for more general adapted metrics \cite{Nagase1983,Nagase1986,BrasseletHectorSaralegi1992}, obtaining Morse inequalities for intersection homology with arbitrary perversity. With this generality, it would be also interesting to develop the rest of Witten's program on analytic Morse theory.

\begin{ack}
  We thank F.~Alcalde for pointing out a mistake in a previous version for orbit spaces \cite{Calaza2000}, Y.A.~Kordyukov and M.~Saralegui for helpful conversations on topics of this paper, R.~Sjamaar for indirectly helping us (via M.~Saralegui), and R.~Israel and another anonymous MathOverflow user for answering questions concerning parts of this work. We also thank the referee for helpful remarks correcting and improving the paper.
\end{ack}

\section{An overall idea of the proofs of the main theorems}\label{s: overall idea}

With the notation of Section~\ref{s:intro}, consider the Witten's complex $d_s=e^{-sf}\,d\,e^{sf}$ ($s>0$) defined by any rel-Morse function $f$ on $M$, and the corresponding Witten's Laplacian $\D_s$. We get the i.b.c.\ $d_{s,\text{\rm min/max}}=e^{-sf}\,d_{\text{\rm min/max}}\,e^{sf}$, with Laplacian $\D_{s,\text{\rm min/max}}$. Since $d_{s,\text{\rm min/max}}$ is conjugated to $d_{\text{\rm min/max}}$, it also defines the min/max-Betti numbers $\beta_{\text{\rm min/max}}^r$.

\subsection{The rel-local analysis around rel-critical points}\label{ss:rel-local analysis}

 (Sections~\ref{s:2 simple types of elliptic complexes} and~\ref{s:Witten cone}--\ref{s:splitting}.) Via stratification charts, this rel-local analysis is made on $\R^{m_+}\times\R^{m_-}\times M_+\times M_-$, with the model functions $\frac{1}{2}(\rho_+^2-\rho_-^2)$. Up to quasi-isometries, we can also consider a model metric. By the version of the K\"unneth formula for Hilbert complexes \cite{BruningLesch1992}, this study can be reduced to the case of the functions $\pm\frac{1}{2}\rho^2$ on $N\times\R_+$, where $\rho$ is the radial function of $c(L)$. By induction on the depth, it is assumed that Theorem~\ref{t:spectrum of Delta_min/max} holds for $(N,\tilde g)$. Then the discrete spectral decomposition for $(N,\tilde g)$ and the factor $d\rho$ are used to split the Witten's complex $d_s$ on $N\times\R_+$ into a direct sum of simple elliptic complexes of two types, with length one and two. The Laplacians of these simple elliptic complexes are given by a version of the Dunkl harmonic oscillator on $\R_+$ \cite{AlvCalaza2014}, whose spectrum is well known (Section~\ref{s:P}). Using this knowledge, we get a description of the maximum/minimum i.b.c.\ of these simple elliptic complexes, and of the spectra of the corresponding Laplacians (Sections~\ref{s:2 simple types of elliptic complexes},~\ref{s:types} and~\ref{s:splitting}). Combining this information in the direct sum, we obtain a description of $d_{s,\text{\rm min/max}}$ and the spectrum of $\D_{s,\text{\rm min/max}}$ on $N\times\R_+$ (Proposition~\ref{p:splitting} and Corollary~\ref{c:d^pm_s,min/max are discrete}); in particular, we have a rel-local version of Theorem~\ref{t:spectrum of Delta_min/max} for $d_{s,\text{\rm min/max}}$.

\subsection{Specific arguments of the proof of Theorem~\ref{t:spectrum of Delta_min/max}} 

(Sections~\ref{s: globalization} and~\ref{s:proof of Theorem spectrum of D_min/max}.) In the rel-local analysis of Section~\ref{ss:rel-local analysis}, we modify the model function around the vertex so that it vanishes on some rel-neighborhood of the vertex. In this way, the new Witten's complex corresponds to $d$ via stratification charts. Since the difference between the above two rel-local Witten's Laplacians is bounded, it follows from the min-max principle that the new Witten's complex also satisfies a version of Theorem~\ref{t:spectrum of Delta_min/max}. Then a simple globalization result (Proposition~\ref{p:globalization of d_min/max}) gives the spectral discreteness for $d_{\text{\rm min/max}}$ on $M$. A much more involved result (Proposition~\ref{p:globalization of d_min/max-2}) also  globalizes the weak Weyl's asymptotic formula for $d_{\text{\rm min/max}}$ on $M$.

\subsection{Specific arguments of the proof of Theorem~\ref{t:Morse inequalities}} 

(Sections~\ref{s:functional calculus}--\ref{s: proof of Thm Morse inequalities}.) Using Theorem~\ref{t:spectrum of Delta_min/max}-(ii), it follows that $\phi(\D_{s,\text{\rm min/max}})$ is of trace class for any rapidly decreasing function $\phi$ on $\R$. As usual, this operator can be given by a Schwartz kernel $K_s$, and the trace of $\phi(\D_{s,\text{\rm min/max}})$ is given by the integral of the pointwise trace of $K_s$ on the diagonal. But it is unknown if $K_s$ is uniformly bounded because a version of the Sobolev embedding theorem is missing---the usual way to prove it fails because the version of Sobolev spaces defined by $\D_{s,\text{\rm min/max}}$ depend on the choice of the adapted metric (Section~\ref{s: Sobolev}). By this problem, some parts of the proof are different from the usual arguments, depending more on Theorem~\ref{t:spectrum of Delta_min/max}-(ii). 

Consider the waive operator $\exp(itD_{s,\text{\rm min/max}})$, where $D_{s,\text{\rm min/max}}=d_{s,\text{\rm min/max}}+d_{s,\text{\rm min/max}}^*$. Another ingredient needed in the proof is that $\exp(itD_{s,\text{\rm min/max}})$ propagates supports towards/from the rel-critical points with finite speed (Proposition~\ref{p:finite propagation speed towards/from the critical points}). With the ideas explained in Section~\ref{t:spectrum of Delta_min/max}, the proof of this property can be reduced to the case of the simple complexes, where it follows adapting standard arguments (Proposition~\ref{p:finite propagation speed, simple}).

Suppose that moreover $\phi(0)=1$. For each degree $r$, let $\mu_{s,\text{\rm min/max}}^r$ be the trace of $\phi(\D_{s,\text{\rm min/max}})$ on $r$-forms. Using $\mu_{s,\text{\rm min/max}}^r$ instead of each $\nu_{\text{\rm min/max}}^r$ in the Morse inequalities, we get the so called analytic inequalities (Proposition~\ref{p:analytic inequalities}), whose proof is formally the same as in the case of closed manifolds. Thus the Morse inequalities follow by showing that $\mu_{s,\text{\rm min/max}}^r\to\nu_{\text{\rm min/max}}^r$ as $s\to\infty$ for all degree $r$. This limit can be expressed as sum of two terms: the limit of the trace of $\phi(\D_{s,\text{\rm min/max}})$ on $r$-forms supported on some small rel-neighborhood of $\Crit_{\text{\rm rel}}(f)$ (the contribution from $\Crit_{\text{\rm rel}}(f)$), and the limit of the trace of $\phi(\D_{s,\text{\rm min/max}})$ on $r$-forms supported on the complement of this rel-neighborhood (the contribution away from $\Crit_{\text{\rm rel}}(f)$). 

We prove that the contribution away from $\Crit_{\text{\rm rel}}(f)$ is null in the following way. Using the expression of $\D_s=\D+s\,\boldsymbol{\Hess}f+s^2\,|df|^2$, where $\boldsymbol{\Hess}f$ is an endomeorphism defined by the Hessian of $f$, we get some $C>0$ so that $\D_{s,\text{\rm min/max}}\ge\D_{\text{\rm min/max}}+Cs^2$ away from $\Crit_{\text{\rm rel}}(f)$ for $s$ large enough. Let $h$ be a cut-off function equal to $1$ near $\Crit_{\text{\rm rel}}(f)$ and vanishing away from $\Crit_{\text{\rm rel}}(f)$. Then $T_{s,\text{\rm min/max}}=\D_{s,\text{\rm min/max}}+hCs^2$ is a self-adjoint operator with a discrete spectrum satisfying $T_{s,\text{\rm min/max}}\ge\D_{\text{\rm min/max}}+Cs^2$ for $s$ large enough. By the min-max principle, this means that the eigenvalues $\lambda_{s,\text{\rm min/max},k}$ of $T_{s,\text{\rm min/max}}$ satisfy $\lambda_{s,\text{\rm min/max},k}\ge\lambda_{\text{\rm min/max},k}+Cs^2$. On the other hand, like in the case of closed manifolds, if moreover $\phi$ is an even Schwartz function whose Fourier transform is supported near $0$, then the finite propagation speed of the waive operator gives $\phi(\D_{s,\text{\rm min/max}})=\phi(T_{s,\text{\rm min/max}})$ on $r$-forms supported in the complement of a slightly smaller rel-neighborhood of $\Crit_{\text{\rm rel}}(f)$. Furthermore we can assume that $\phi\ge0$, and that $\phi$ is monotone on $[0,\infty)$. Combining all of the above properties, we get that the contribution away from $\Crit_{\text{\rm rel}}(f)$ is
	\[
		\le\lim_s\sum_k\phi(\lambda_{s,\text{\rm min/max},k})
		\le\lim_s\sum_k\phi(\lambda_{\text{\rm min/max},k}+Cs^2)=0\;,
	\] 
for each degree $r$. This argument differs from the usual one: in the case of closed manifolds, the Sobolev embedding theorem is used to prove that, away from the critical points, $K_s\to0$ uniformly, but that kind of theorem is not available here.

Like in the case of closed manifolds, it is shown that the contribution from $\Crit_{\text{\rm rel}}(f)$ is $\nu_{\text{\rm min/max}}^r$: using the finite propagation speed of the waive operator, we can pass to the case of cones with model functions and model metrics, and then the spectral analysis indicated in Section~\ref{ss:rel-local analysis} gives the desired limit.

\section{Preliminaries on Thom-Mather stratifications}\label{s:preliminaries on stratifications}

This section mainly recalls the needed concepts, notation and results about Thom-Mather stratifications and adapted metrics on their strata. Some new, concepts, remarks and results are also given, specially concerning products and metric completion of strata. The proofs of the non-elementary new results are given in Appendix~\ref{a: proofs stratifications} (Lemmas~\ref{l:a product of two cones is a cone} and~\ref{l:uniqueness}, and Proposition~\ref{p:widehat M}).

\subsection{Thom-Mather stratifications}\label{ss:stratifications}

The concepts recalled here were introduced by R.~Thom \cite{Thom1969} and J.~Mather \cite{Mather1970}. We mainly follow \cite{Verona1984}.

\subsubsection{Thom-Mather stratifications and their morphisms}\label{sss:stratifications}

Let $A$ be a Hausdorff, locally compact and second countable topological space. Let $X\subset A$ be a locally closed subset. Two subsets $Y,Z\subset A$ are said to be {\em equal near\/} $X$ (or {\em $Y=Z$ near\/} $X$) if $Y\cap U=Z\cap U$ for some neighborhood $U$ of $X$ in $A$. It is also said that two maps, $f:Y\to B$ and $g:Z\to B$, are {\em equal near\/} $X$ (or {\em $f=g$ near\/} $X$) when there is some neighborhood $U$ of $X$ in $A$ such that $Y\cap U=Z\cap U$, and the restrictions of $f$ and $g$ to $Y\cap U$ are equal. 

Consider triples $(T,\pi,\rho)$, where $T$ is an open neighborhood of $X$ in $A$, $\pi:T\to X$ is a continuous retraction, and $\rho:T\to[0,\infty)$ is a continuous function such that $\rho^{-1}(0)=X$. Two such triples, $(T,\pi,\rho)$ and $(T',\pi',\rho')$, are said to be {\em equal near\/} $X$ when $T=T'$, $\pi=\pi'$ and $\rho=\rho'$ near $X$. This defines an equivalence relation whose equivalence classes are called {\em tubes\/} of $X$ in $A$. The notation $[T,\pi,\rho]$ is used for the tube represented by $(T,\pi,\rho)$. If $X$ is open in $A$, then $[X,\id_X,0]$ is its unique tube (the {\em trivial tube\/}).

\begin{defn}[{See \cite[1.2.1]{Verona1984}}]\label{d:stratification}
 A {\em Thom-Mather stratification\/}\footnote{The term {\em Thom-Mather stratified space\/} is also used. It is called {\em abstract prestratification\/} in \cite{Mather1970} and {\em abstract stratification\/} in \cite{Verona1984}.} is a triple $(A,\SS,\tau)$, where:
    \begin{itemize}

      \item[(i)] $A$ is a Hausdorff, locally compact and second countable space,
    
      \item[(ii)] $\SS$ is a partition of $A$ into locally closed subspaces with the additional structure of smooth ($C^\infty$) manifolds, called {\em strata\/}, and 
    
      \item[(iii)] $\tau$ is the assignment of a tube $\tau_X$ of each $X\in\SS$ in $A$,
    
    \end{itemize}
  such that the following conditions are satisfied with some choice of $(T_X,\pi_X,\rho_X)\in\tau_X$ for each $X\in\SS$:
    \begin{itemize}
  
      \item[(iv)] For all $X,Y\in\SS$, if $X\cap\ol{Y}\ne\emptyset$, then $X\subset\ol{Y}$. The notation $X\le Y$ is used in this case, and this defines a partial order relation on $\SS$. As usual, $X<Y$ means that $X\le Y$ but $X\ne Y$.
    
      \item[(v)] If $Y\ne X$ in $\SS$ and $T_X\cap Y\ne\emptyset$, then $X<Y$ and $(\pi_X,\rho_X):T_X\cap Y\to X\times\R_+$ is a smooth submersion; in particular, $\dim X<\dim Y$.
    
      \item[(vi)] If $X<Y$ in $\SS$, then $\pi_Y(T_X\cap T_Y)\subset T_X$, and $\pi_X\,\pi_Y=\pi_X$ and $\rho_X\,\pi_Y=\rho_X$ on $T_X\cap T_Y$. 
    
    \end{itemize}
  It may be also said that $(\SS,\tau)$ is a {\em Thom-Mather stratification\/} of $A$.
\end{defn}

\begin{rem}\label{r:stratification}
    \begin{itemize} 
    
      \item[(i)] $A$ is paracompact and normal.
      
      \item[(ii)] By the normality of $A$, we can also assume that, if $X,Y\in\SS$ and $T_X\cap T_Y\ne\emptyset$, then $X\le Y$ or $Y\le X$. 
      
      \item[(iii)] The frontier of a stratum $X$ equals the union of the strata $Y<X$. 
      
      \item[(iv)] The connected components of each stratum may have different dimensions. 
      
      \item[(v)] The connected components of the strata, with the restrictions of the tubes, define an induced Thom-Mather stratification $A_{\text{\rm con}}\equiv(A,\SS_{\text{\rm con}},\tau_{\text{\rm con}})$.
      
    \end{itemize}
\end{rem}

\begin{rem}\label{r:variants of stratification}
  The following are some variants of Definition~\ref{d:stratification} and related notions:
    \begin{itemize}
  
      \item[(i)] A {\em weak Thom-Mather stratification\/} is defined by removing the condition $\rho_X\,\pi_Y=\rho_X$ from Definition~\ref{d:stratification}-(vi).
    
      \item[(ii)] A {\em stratification\/} is a pair $(A,\SS)$ satisfying Definition~\ref{d:stratification}-(i),(ii),(iv); it is also said that $\SS$ is a {\em stratification\/} of $A$. Definition~\ref{d:stratification}-(iv) is called the {\em frontier condition\/}. If moreover $\tau$ satisfies the other conditions of  Definition~\ref{d:stratification}, then it is called {\em Thom-Mather structure\/} on $(A,\SS)$.
      
      \item[(iii)] If $A$ is a subspace of a smooth manifold $M$, then a stratification $\SS$ of $A$ is usually required to consist of regular submanifolds of $M$; the term {\em stratified subspace\/} of $M$ is used in this case.  In \cite{GibsonWirthmullerPlessisLooijenga1976}, a weaker version of this notion is defined by requiring local finiteness of $\SS$ instead of the frontier condition.
    
      \item[(iv)] A {\em Whitney stratification\/} of a subspace (or {\em Whitney stratified subspace\/}) of a smooth manifold $M$ is a stratified subspace of $M$ satisfying the condition~(B) of H.~Whitney \cite{Whitney1965a,Whitney1965b}\footnote{Certain condition~(A) was also introduced by H.~Whitney in \cite{Whitney1965a,Whitney1965b}, but J.~Mather  \cite{Mather1970} has observed that it follows from condition~(B).}.

    \end{itemize}
\end{rem}

\begin{ex}\label{ex:stratification}
  \begin{itemize}
  
    \item[(i)] Any smooth manifold is a Thom-Mather stratification with one stratum and the trivial tube.
    
    \item[(ii)] Any smooth manifold with boundary is a stratification with two strata, the interior and the boundary. It can be equipped with a Thom-Mather structure by using a collar of the boundary. 
    
    \item[(iii)] Any subanalytic subset of $\R^m$ has primary and secondary stratifications; the secondary one satisfies condition~(B) \cite{Lojasiewicz1971,Mather1973,Hironaka1973a,Hironaka1973b,Hironaka1977}.
    
    \item[(iv)]  J.~Mather \cite{Mather1970} has proved that any Whitney stratified subspace of a smooth manifold admits a Thom-Mather structure (see also \cite[Proposition~2.6 and Corollary~2.7]{GibsonWirthmullerPlessisLooijenga1976}).
        
    \end{itemize}
\end{ex}
 
For a stratification $A\equiv(A,\SS)$, the {\em depth\/} of any $X\in\SS$, denoted by $\depth X$, is the supremum of the naturals $n$ such that there exist strata $X_0<X_1<\dots<X_n=X$. Notice that $\depth X\le\dim X$. Moreover $\depth X=0$ if and only if $X$ is closed in $A$. The {\em depth\/} and {\em dimension\/} of $A$ are the supremum of the depths and dimensions of its strata, respectively. The dimension of $A$ equals its topological dimension, which may be infinite. The depth of $A$ is zero if and only if all strata are open and closed. 

Let $A\equiv(A,\SS,\tau)$ be a Thom-Mather stratification. Let $B\subset A$ be a locally closed subset. Suppose that, for all $X\in\SS$, $X\cap B$ is a smooth submanifold of $X$, and $B\cap\pi_X^{-1}(X\cap B)$, equipped with the restrictions of $\pi_X$ and $\rho_X$, defines a tube $\tau_{X\cap B}$ of $X\cap B$ in $B$. Then let $\SS|_B=\{\,X\cap B\mid X\in\SS\,\}$, and let $\tau|_B$ be defined by the assignment of $\tau_{X\cap B}$ to each $X\cap B\in\SS|_B$. If $(B,\SS|_B,\tau|_B)$ satisfies the conditions of a stratification, it is said that the stratification $A$ (or $(\SS,\tau)$) can be {\em restricted\/} to $B$, and $B\equiv(B,\SS|_B,\tau|_B)$ (respectively, $(\SS|_B,\tau|_B)$) is called a {\em restriction\/} of $A$ (respectively, $(\SS,\tau)$); it may be also said that $B$ is a {\em Thom-Mather substratification\/} of $A$. For instance, $A$ can be restricted to any open subset and to any locally closed union of strata. A restriction of a restriction of $A$ is a restriction of $A$. 

For a stratum $X$ of $A$, we can consider the restriction of $A$ to $\ol{X}$. In this way, to study $X$, we can assume that $X$ is dense in $A$ and $\dim X=\dim A$ if desirable.

A locally closed subset $B\subset A$ is said to be {\em saturated\/} if the stratification $A$ can be restricted to $B$ and, for every $X\in\SS$, there is a representative $(T_X,\pi_X,\rho_X)$ of $\tau_X$ such that $\pi_X^{-1}(X\cap B)=T_X\cap B$.

Let $A'\equiv(A',\SS',\tau')$ be another Thom-Mather stratification. A continuous map $f:A\to A'$ is called a {\em morphism\/} if, for any $X\in\SS$, there is some $X'\in\SS'$ such that $f(X)\subset X'$, the restriction $f:X\to X'$ is smooth, and there are $(T_X,\pi_X,\rho_X)\in\tau_X$ and $(T'_{X'},\pi'_{X'},\rho'_{X'})\in\tau'_{X'}$ such that $f(T_X)\subset T'_{X'}$, $f\pi_X=\pi'_{X'}f$ and $f\rho_X=\rho'_{X'}$. Notice that the continuity of a morphism follows from the other conditions. Morphisms between stratifications form a category with the operation of composition; in particular, we have the corresponding concepts of {\em isomorphism\/} and {\em automorphism\/}. The set of morphisms $A\to A'$ is denoted by $\Mor(A,A')$, and the group of automorphisms of $A$ is denoted by $\Aut(A)$. The other variants of the concept ``stratification'' given in Remark~\ref{r:variants of stratification} also have obvious corresponding versions of morphisms, isomorphisms and automorphisms; in particular, we get the concept of {\em weak morphism\/} between weak Thom-Mather stratifications. A (weak) morphism is called {\em submersive\/} when it restricts to smooth submersions between the strata.

\begin{ex}\label{ex:orbit}
  Let $G$ be a compact Lie group $G$ acting smoothly on a closed manifold $M$. Consider the orbit type stratifications of $M$ and $G\backslash M$ \cite{Bredon1972}. It is well known that $G\backslash M$ admits a Thom-Mather structure \cite[Introduction]{Verona1984}, which can be seen as follows. $G\backslash M$ is locally isomorphic to a semi-algebraic subset of an Euclidean space whose primary and secondary stratifications are equal \cite{Bierstone1975}. By using an invariant smooth partition of unity of $M$, like in the Whitney's embedding theorem, it follows that $G\backslash M$ is isomorphic to a Whitney stratified subspace of some Euclidean space, and therefore it admits a Thom-Mather structure. This can also be seen by observing that the stratification of $M$ satisfies condition~(B), and the proof of \cite[Proposition~2.6]{GibsonWirthmullerPlessisLooijenga1976} can be adapted to produce an invariant\footnote{$G$ acts by automorphisms.} Thom-Mather structure on $M$, which induces a Thom-Mather structure on $G\backslash M$.
\end{ex}

The following two lemmas are easy to prove.

\begin{lem}\label{l:stratification is local}
  Let $A$ be a Hausdorff, locally compact and second countable space, $\{U_i\}$ an open covering of $A$, and $(\SS_i,\tau_i)$ a Thom-Mather stratification of each $U_i$.
    \begin{itemize}
  
      \item[(i)] If $(\SS_i,\tau_i)$ and $(\SS_j,\tau_j)$ have the same restrictions to $U_{ij}:=U_i\cap U_j$ for all $i$ and $j$, then there is a unique Thom-Mather stratification $(\SS,\tau)$ on $A$ whose restriction to each $U_i$ is $(\SS_i,\tau_i)$.
    
      \item[(ii)] If $((\SS_i|_{U_{ij}})_{\text{\rm con}},(\tau_i|_{U_{ij}})_{\text{\rm con}})=((\SS_j|_{U_{ij}})_{\text{\rm con}},(\tau_j|_{U_{ij}})_{\text{\rm con}})$ for all $i$ and $j$, then there is a unique Thom-Mather stratification $(\SS,\tau)$ on $A$ with connected strata such that $((\SS|_{U_i})_{\text{\rm con}},(\tau|_{U_i})_{\text{\rm con}})=(\SS_{i,\text{\rm con}},\tau_{i,\text{\rm con}})$.
    
    \end{itemize}
\end{lem}

\begin{lem}\label{l:morphism is local}
  Let $(A',\SS',\tau')$ be another Thom-Mather stratification.
    \begin{itemize}
  
      \item[(i)] With the notation of Lemma~\ref{l:stratification is local}-\upn{(}i\upn{)}, let $f_i:(U_i,\SS_i,\tau_i)\to(A',\SS',\tau')$ be a morphism for each $i$. If $f_i|_{U_{ij}}=f_j|_{U_{ij}}$ for all $i$ and $j$, then the combination of the maps $f_i$ is a morphism $f:(A,\SS,\tau)\to(A',\SS',\tau')$.
    
      \item[(ii)] With the notation of Lemma~\ref{l:stratification is local}-\upn{(}ii\upn{)}, let $f_i:(U_i,\SS_{i,\text{\rm con}},\tau_{i,\text{\rm con}})\to(A',\SS',\tau')$ be a morphism for each $i$. If $f_i|_{U_{ij}}=f_j|_{U_{ij}}$ for all $i$ and $j$, then the combination of the maps $f_i$ is a morphism $f:(A,\SS,\tau)\to(A',\SS',\tau')$.
    
    \end{itemize}
\end{lem}

\begin{rem}\label{r:bigsqcup_iA_i}
  As a particular case of Lemma~\ref{l:stratification is local}, given a countable family of Thom-Mather stratifications, $\{A_i\equiv(A_i,\SS_i,\tau_i)\}$, there is a unique Thom-Mather stratification $(\SS,\tau)$ on the topological sum $\bigsqcup_iA_i$ whose restriction to each $A_i$ is $(\SS_i,\tau_i)$; this $(\SS,\tau)$ will be called the {\em sum\/} of the Thom-Mather stratifications $(\SS_i,\tau_i)$.
\end{rem}

\subsubsection{Products}\label{sss:products}

The product of two weak Thom-Mather stratifications, $A$ and $A'$, is a weak Thom-Mather stratification $A\times A'\equiv(A\times A',\SS'',\tau'')$ with $\SS''=\{\,X\times X'\mid X\in\SS,\ X'\in\SS'\,\}$ and $\tau''_{X\times X'}=[T''_{X\times X'},\pi''_{X\times X'},\rho''_{X\times X'}]$, where $T''_{X\times X'}=T_X\times T'_{X'}$, $\pi''_{X\times X'}=\pi_X\times\pi'_{X'}$ and $\rho''_{X\times X'}(x,x')=\rho_X(x)+\rho'_{X'}(x')$. 

If $A$ and $A'$ are Thom-Mather stratifications and the depth of at least one of them is zero, then $A\times A'$ is a Thom-Mather stratification, but this is not true when the depths of $A$ and $A'$ are positive \cite[Section~1.2.9, pp.~5--6]{Verona1984}. This can be seen in the following simple example.

\begin{ex}
	Let $A=A'=[0,\infty)$, with the strata $X=\{0\}<Y=(0,\infty)$, taking $T_X=[0,\infty)$, $T_Y=Y$, $\pi_X(x)=0$, $\pi_Y(y)=y$, $\rho_X(x)=x$ and $\rho_Y(y)=0$. Then the second equality of Definition~\ref{d:stratification}-(vi) fails for the strata $X\times X<X\times Y$ of $A\times A'$:
	\[
		\rho''_{X\times X}\,\pi''_{X\times Y}(x,x')=\rho''_{X\times X}(0,x')=x'\ne x+x'=\rho''_{X\times X}(x,x')
	\]
for all $(x,x')\in(0,\infty)^2$, which is an open dense subset of $T''_{X\times X}\cap T''_{X\times Y}=T''_{X\times Y}=[0,\infty)\times(0,\infty)$, contradicting the second equality of Definition~\ref{d:stratification}-(vi).  
\end{ex}

Thus another choice of $\rho''_{X\times X'}$ is needed to get the second equality of Definition~\ref{d:stratification}-(vi). For instance, $\rho''_{X\times X'}=\max\{\rho_X,\rho'_{X'}\}$ satisfies that condition, but it is not smooth on the intersection of the strata with $T''_{X\times X'}$. To solve this problem, pick up a function $h:[0,\infty)^2\to[0,\infty)$ that is continuous, homogeneous of degree one, smooth on $\R_+^2$, with $h^{-1}(0)=\{(0,0)\}$, and such that, for some $C>1$, we have $h(r,s)=\max\{r,s\}$ if $C\min\{r,s\}<\max\{r,s\}$. Then $A\times A'$ becomes a Thom-Mather stratification by setting $\rho''_{X\times X'}(x,x')=h(\rho_X(x),\rho'_{X'}(x'))$; it will be called {\em a product\/} of $A$ and $A'$.

\subsubsection{Cones}\label{sss:cones}

Recall that the {\em cone\/} with {\em link\/} a non-empty topological space $L$ is the quotient space $c(L)=L\times[0,\infty)/L\times\{0\}$. The class $*=L\times\{0\}$ is called the {\em vertex\/} or {\em summit\/} of $c(L)$. The element of $c(L)$ represented by each $(x,\rho)\in L\times[0,\infty)$ will be denoted by $[x,\rho]$. The function on $c(L)$ induced by the second factor projection $L\times[0,\infty)\to[0,\infty)$ will be called its {\em radial function\/}, and will be usually denoted by $\rho$. Notice that $c(L)$ is locally compact if and only if $L$ is compact. It is also declared that $c(\emptyset)=\{*\}$. 

Now, suppose that $L$ is a compact Thom-Mather stratification. Then $c(L)$ has a canonical Thom-Mather stratification so that $\{*\}$ is a stratum, its restriction to $c(L)\sm\{*\}=L\times\R_+$ is the product Thom-Mather stratification, and the tube of $\{*\}$ is $[c(L),\pi,\rho]$, where $\rho$ is the radial function and $\pi$ is the unique map $c(L)\to\{*\}$. If $L\ne\emptyset$, then $\depth c(L)=\depth L+1$ and $\dim c(L)=\dim L+1$. For any $\epsilon>0$, let $c_\epsilon(L)=\rho^{-1}([0,\epsilon))$.

Let $L'$ be another compact Thom-Mather stratification, and let $*'$ denote the vertex of $c(L')$. If $L\ne\emptyset$, the {\em cone\/} of any morphism $f:L\to L'$ is the morphism $c(f):c(L)\to c(L')$ induced by $f\times\id:L\times[0,\infty)\to L'\times[0,\infty)$. If $L=\emptyset$, $c(f)$ is defined by $*\mapsto*'$. Reciprocally, it is easy to check that, for any morphism $h:c(L)\to c(L')$, there is some morphism $f:L\to L'$ such that $h=c(f)$ near $*$; in particular, $h(*)=*'$. Let $c(\Aut(L))=\{\,c(f)\mid f\in \Aut(L)\,\}\subset\Aut(c(L))$.

\begin{ex}\label{ex:c(S^m-1)=R^m}
  For each integer $m\ge1$, there is a canonical homeomorphism $\can:c(\S^{m-1})\to\R^m$ defined by $\can([x,\rho])=\rho x$. Of course, this is not an isomorphism of Thom-Mather stratifications, but it restricts to a diffeomorphism of the stratum $\S^{m-1}\times\R_+$ of $c(\S^{m-1})$ to $\R^m\sm\{0\}$. Via $\can:c(\S^{m-1})\to\R^m$, the radial function of $c(\S^{m-1})$ corresponds to the function $\rho_0(x)=|x|$ on $\R^m$, which will be also called the {\em radial function\/} on $\R^m$ for the scope of this paper. If $\rho_1$ is the radial function on $c(L)$ for some compact Thom-Mather stratification $L$, then the function $\rho=\sqrt{\rho_0^2+\rho_1^2}$ will be called the {\em radial function\/} on $\R^m\times c(L)$.
\end{ex}

\begin{lem}\label{l:a product of two cones is a cone}
	A product of two cones is isomorphic to a cone.
\end{lem}

\subsubsection{Conic bundles}\label{sss:conic bundles}

  Let $X$ be a smooth manifold, $L$ a compact Thom-Mather stratification, and $\pi:T\to X$ a fiber bundle whose typical fiber is $c(L)$ and whose structural group can be reduced to $c(\Aut(L))$. Thus there is a family of local trivializations of $\pi$, $\{(U_i,\phi_i)\}$, such that the corresponding transition functions define a cocycle with values in $c(\Aut(L))$; i.e., for all $i$ and $j$, there is a map $h_{ij}:U_{ij}:=U_i\cap U_j\to c(\Aut(L))$ such that $\phi_j\phi_i^{-1}(x,y)=(x,h_{ij}(x)(y))$ for every $x\in U_{ij}$ and $y\in c(L)$. Thus we get another cocycle consisting of maps $g_{ij}:U_{ij}\to\Aut(L)$ so that $h_{ij}(x)=c(g_{ij}(x))$ for all $x\in U_{ij}$. Consider the Thom-Mather stratification on each open subset $\pi^{-1}(U_i)\subset T$ that corresponds by $\phi_i$ to the product Thom-Mather stratification on $U_i\times c(L)$. For each connected open $V\subset U_{ij}$ and every stratum $N_0$ of $L$, there is an stratum $N_1$ of $L$ such that $g_{ij}(x)(N_0)=N_1$ for all $x\in V$, and suppose also that, in this case, the map $V\times N_0\to N_1$, $(x,y)\mapsto g_{ij}(x)(y)$, is smooth. Then each mapping $(x,y)\mapsto(x,g_{ij}(x)(y))$ defines an automorphism of $U_{ij}\times L$. This means that the induced Thom-Mather stratifications on $\pi^{-1}(U_i)$ and $\pi^{-1}(U_j)$ have the same restriction to $\pi^{-1}(U_{ij})$. By Lemma~\ref{l:stratification is local}-(i), we get a unique Thom-Mather stratification on $T$ with the above restriction to each $\pi^{-1}(U_i)$. Moreover there is a canonical section of $\pi$, called the {\em vertex\/} (or {\em summit\/}) {\em section\/}, which is well defined by $x\mapsto *_x=\phi_i^{-1}(x,*)$ if $x\in U_i$, where $*$ is the vertex of $c(L)$; each $*_x$ can be called the {\em vertex\/} of $\pi^{-1}(x)$. The set $\{\,*_x\mid x\in X\,\}$ is a stratum of $T$, called the {\em vertex\/} (or {\em summit\/}) {\em stratum\/}, which is diffeomorphic to $X$.
  
  If $\pi:T\to X$ is equipped with a maximal family $\Phi$ of trivializations satisfying the above conditions, it will be called a {\em conic bundle\/}, and the corresponding Thom-Mather stratification on $T$ is called its {\em conic bundle Thom-Mather stratification\/}. It will be also said that $\Phi$ is the {\em conic bundle structure\/} of $\pi$.
  
  Let $\rho:c(L)\to[0,\infty)$ be the radial function. Its lift to each $U_i\times c(L)$ is also denoted by $\rho$. The functions $\phi_i^*\rho$ on the sets $\pi^{-1}(U_i)$ can be combined to define a function $\rho:T\to[0,\infty)$. The tubular neighborhood of $X$ in $T$ is $[T,\pi,\rho]$, and $(T,\pi,\rho)$ is called its {\em canonical representative\/}. 
 
  Let $\pi':T'\to X'$ be another conic bundle, whose structure is given by a family $\Phi'$ of trivializations as above. Let $F:T\to T'$ be a fiber bundle morphism over a map $f:X\to X'$. Then we can choose $\{(U_i,\phi_i)\}$ as above and a family $\{(U'_i,\phi'_i)\}\subset\Phi'$ such that $f(U_i)\subset U'_i$ for all $i$, and therefore $F(\pi^{-1}(U_i))\subset{\pi'}^{-1}(U'_i)$. Let $h'_{ij}=c(g'_{ij}):U'_{ij}:=U'_i\cap U'_j\to c(\Aut(L'))$ be the maps defined by the transition maps $\phi'_j\,{\phi'_i}^{-1}$ as above. Suppose that there are maps $\kappa_i:U_i\to\Mor(L,L')$ such that $\kappa_j(x)\, g_{ij}(x)=g'_{ij}(f(x))\,\kappa_i(x)$ for all $x\in U_{ij}$. For each connected open $V\subset U_i$ and every stratum $N$ of $L$, there is an stratum $N'$ of $L'$ such that $\kappa_i(x)(N)\subset N'$ for all $x\in V$, and assume also that, in this case, the map $V\times N\to N'$, $(x,y)\mapsto\kappa_i(x)(y)$, is smooth. Then $F$ is called a {\em morphism of conic bundles\/}. In this case, each mapping $(x,y)\mapsto(f(x),\kappa_i(x)(y))$ defines a morphism $U_i\times c(L)\to U'_i\times c(L')$. So each restriction $F:\pi^{-1}(U_i)\to{\pi'}^{-1}(U'_i)$ is a morphism of Thom-Mather stratifications, and therefore $F:T\to T'$ is a morphism of Thom-Mather stratifications by Lemma~\ref{l:morphism is local}-(i). According to Section~\ref{sss:cones}, any morphism of Thom-Mather stratifications between conic bundles, preserving the vertex stratum, equals a conic bundle morphism near the vertex stratum.
  
The case of conic bundles is specially important because, as pointed out in \cite[Chapitre~A, Remarque~3]{BrasseletHectorSaralegi1991}, the proof of \cite[Theorem~2.6, pp.~16--17]{Verona1984} can be easily adapted to get the following. 

\begin{prop}\label{p:pi is a conic bundle}
  Let $A\equiv(A,\SS,\tau)$ be a Thom-Mather stratification with connected strata. Then, for any $X\in\SS$, there is some $(T,\pi,\rho)\in\tau_X$ such that $\pi:T\to X$ admits a structure $\Phi$ of conic bundle so that the corresponding conic bundle Thom-Mather stratification is $(\SS|_T,\tau|_T)$.
\end{prop}

\begin{rem}\label{r:pi_X is a conic bundle}
    \begin{itemize}
      
      \item[(i)] The notation $T_X$, $\pi_X$, $\rho_X$, $L_X$ and $\Phi_X$ will be used when a reference to the stratum $X$ is desired.
      
       \item[(ii)] We can choose $\rho$ so that $(T,\pi,\rho)$ is the canonical representative of the tube around $X$ in $T$ with its conic bundle Thom-Mather stratification.
            
    \end{itemize}
\end{rem}

\begin{defn}[{See \cite[1.2]{BrasseletHectorSaralegi1992}}]\label{d:chart}
  A {\em chart\/} or {\em distinguished neighborhood\/} of $A$ is a pair $(O,\xi)$, where $O$ is open in $A$ and, for some $X\in\SS$ and $\epsilon>0$, with the notation and conditions of Proposition~\ref{p:pi is a conic bundle}, $\xi$ is an isomorphism $O\to B\times c_\epsilon(L)$ defined by some $(U,\phi)\in\Phi$ and some chart $(U,\zeta)$ of $X$ with $\zeta(U)=B$, where $B$ is an open subset of $\R^m$ for $m=\dim X$. It is said that $(O,\xi)$ is {\em centered\/} at $x\in X$ if $B$ is an open ball centered at $0$ and $\xi(x)=(0,*)$, where $*$ is the vertex of $c(L)$. A collection of charts that cover $A$ is called an {\em atlas\/} of $A$.
\end{defn}

\begin{rem}\label{r:chart}
    Definition~\ref{d:chart} also includes the case where some factor of the product $\R^m\times c(L)$ is missing by taking $m=0$ or $L=\emptyset$.
\end{rem}

\begin{rem}\label{r:by using charts}
  By using charts and induction on the depth, we get the following:
    \begin{itemize}
    
      \item[(i)] In any Thom-Mather stratification, there is at most one dense stratum, which is open.
      
      \item[(ii)] Any stratum with compact closure has a finite number of connected components.
      
    \end{itemize}
\end{rem}

\subsubsection{Uniqueness of Thom-Mather stratifications}\label{sss:uniqueness}

\begin{lem}\label{l:uniqueness}
  Let $A$ be a Hausdorff, locally compact and second countable space, let $(A',\SS',\tau')$ be a Thom-Mather stratification with connected strata, and let $f:A\to A'$ be a continuous map. Then there is at most one Thom-Mather stratification $(\SS,\tau)$ on $A$ with connected strata so that $f:(A,\SS,\tau)\to(A',\SS',\tau')$ is a morphism that restricts to local diffeomorphisms between corresponding strata.
\end{lem}

\subsubsection{Relatively local properties on strata}\label{sss:rel-local properties}

The following kind of terminology will be used for a stratum $X$ of a Thom-Mather stratification $A$. Let $\PP$ be a property that may hold on open subsets $U\subset X$; for the sake of simplicity, let us say that ``$U$ is $\PP$'' when $\PP$ holds on $U$. It is is said that $X$ is {\em relatively locally\/} (or {\em rel-locally\/}) $\PP$ {\em at\/} some $x\in\ol{X}$ if there is a base $\UU$ of open neighborhoods of $x$ in $A$ such that $U\cap X$ is $\PP$ for all $U\in\UU$. If $X$ is rel-locally $\PP$ at all points of $\ol{X}$, then $X$ is said to be {\em relatively locally\/} (or {\em rel-locally\/}) $\PP$. Similarly, $\PP$ is said to be a {\em relatively local\/} (or {\em rel-local\/}) property when $X$ is $\PP$ if and only if it is rel-locally $\PP$. For instance, on $X$, we will consider functions that are rel-locally bounded or rel-locally bounded away from zero, rel-locally finite open coverings, and rel-local connectedness at points of $\ol{X}$. Any locally finite covering of $\ol{X}$ by open subsets of $A$ restricts to a rel-locally finite open covering of $X$; thus there exist rel-locally finite open coverings of $X$ by the paracompactness of $A$ (Remark~\ref{r:stratification}-(i)). Observe that $\ol{X}$ is compact if and only if any rel-locally finite open covering of $X$ is finite.

\subsection{Adapted metrics on strata}\label{ss:adapted metrics}

The definition of adapted metrics was given for the regular stratum of any Thom-Mather stratification that is a pseudomanifold \cite{Cheeger1980,Cheeger1983,Nagase1983,Nagase1986}. But its definition has an obvious version for any stratum of a Thom-Mather stratification. In this paper, we will consider only the simplest type of adapted metrics, whose definition is recalled.

\subsubsection{Adapted metrics on strata and local quasi-isometries between Thom-Mather stratifications}

Let $A$ be a Thom-Mather stratification. The adapted metrics on its strata are combinations of the adapted metrics on their connected components, using $A_{\text{\rm con}}$ (Remark~\ref{r:stratification}-(v)). Thus we can assume that the strata of $A$ are connected to define adapted metrics. This definition is given by induction on the depth. 

\begin{defn}[{See \cite[2.1]{BrasseletHectorSaralegi1992}}]\label{d:adapted metric}
  Let $M$ be a stratum of $A$. If $\depth M=0$, then any Riemannian metric on $M$ is called {\em adapted\/}. If $\depth M>0$ and adapted metrics are defined for strata of lower depth, then an {\em adapted metric\/} on $M$ is a Riemannian metric $g$ such that, for any point $x\in\ol{M}\sm M$, there is some chart $(O,\xi)$ of $A$ centered at $x$, with $\xi(O)=B\times c_\epsilon(L)$ and $\xi(O\cap M)=B\times N\times(0,\epsilon)$ for some stratum $N$ of $L$, so that $g$ is quasi-isometric to $\xi^*(g_0+\rho^2\tilde g+(d\rho)^2)$ on $O\cap M$, where $g_0$ is the standard Riemannian metric on $\R^m$ ($m$ is the dimension of the stratum that contains $x$), $\rho$ is the standard coordinate of $\R_+$, and $\tilde g$ is some adapted metric on $N$.
\end{defn}

\begin{rem}\label{r:adapted metric}
  By taking charts and using induction on the depth, we get the following:
    \begin{itemize}
    
      \item[(i)] Any pair of adapted metrics on $M$, $g$ and $g'$, are rel-locally quasi-isometric; in particular, if $\ol{M}$ is compact, then $g$ and $g'$ are quasi-isometric.
      
      \item[(ii)] Any point in $\ol{M}$ has a countable base $\{\,O_m\mid m\in\N\,\}$ of open neighborhoods such that, with respect to any adapted metric, $\vol(M\cap O_m)\to0$ and $\max\{\,\diam P\mid P\in\pi_0(M\cap O_m)\,\}\to0$ as $m\to\infty$. Thus, if $\ol{M}$ is compact, then $\vol M<\infty$ and $\diam P<\infty$ for all $P\in\pi_0(M)$.
      
      \item[(iii)] Any morphism of Thom-Mather stratifications restricts to rel-locally uniformly continuous maps between corresponding strata with respect to arbitrary adapted metrics.
      
      \item[(iv)] If $g$ and $g'$ are adapted metrics on strata $M$ and $M'$ of Thom-Mather stratifications $A$ and $A'$, respectively, then $g+g'$ is an adapted metric on the stratum $M\times M'$ of any product Thom-Mather stratification on $A\times A'$.
          
    \end{itemize}
\end{rem}

In \cite[Appendix]{BrasseletHectorSaralegi1992}, it was proved that there exist adapted metrics on the regular stratum of any Thom-Mather stratification that is a pseudomanifold. It can be easily checked that the same argument proves the existence of adapted metrics on any stratum $M$ of every Thom-Mather stratification $A$. 

\begin{ex}\label{ex:adapted metrics}
The proof given in \cite[Appendix]{BrasseletHectorSaralegi1992} also shows the following:
  \begin{itemize}
  
    \item[(i)] With the notation of Definition~\ref{d:adapted metric}, the metric $g=g_0+\rho^2\tilde g+(d\rho)^2$ is adapted on the stratum $M=\R^m\times N\times\R_+$ of $c(L)$; it will be called a {\em model adapted metric\/}.
    
    \item[(ii)] Let $\{(O_a,\xi_a)\}$ be a locally finite atlas of $\ol{M}$, let $\{\lambda_a\}$ be a smooth partition of unity of $M$ subordinated to the open covering $\{M\cap O_a\}$, and let $g_a$ be an adapted metric on each $M\cap O_a$. Then the metric $\sum_a\lambda_ag_a$ is adapted on $M$.
    
  \end{itemize}
\end{ex}

\begin{ex}\label{ex:g_0}
  For an integer $m\ge1$, let $\tilde g_0$ be the restriction to $\S^{m-1}$ of the standard metric $g_0$ of $\R^m$. Then, via  $\can:c(\S^{m-1})\to\R^m$ (Example~\ref{ex:c(S^m-1)=R^m}), the model adapted metric $g_1=\rho^2\tilde g_0+(d\rho)^2$ on the stratum $\S^{m-1}\times\R_+$ of $c(\S^{m-1})$ corresponds to $g_0$ on $\R^m\sm\{0\}$.
\end{ex}

\begin{ex}\label{ex:adapted metrics on orbit spaces}
  With the notation of Example~\ref{ex:orbit}, for any invariant Riemannian metric $g$ on $M$, consider the Riemannian metric $\bar g$ on the strata of $G\backslash M$ so that the canonical projection of the strata of $M$ to the strata of $G\backslash M$ are Riemannian submersions. The proof of  \cite[Proposition~2.6]{GibsonWirthmullerPlessisLooijenga1976} can be easily adapted to produce an invariant Thom-Mather structure on $M$ so that the restriction of $g$ to any stratum is adapted. Hence $\bar g$ is adapted for the induced Thom-Mather structure of $G\backslash M$. 
\end{ex}

A weak isomorphism between Thom-Mather stratifications is called a {\em local quasi-isometry\/} if it restricts to rel-local quasi-isometries between their strata with respect to adapted metrics; this is independent of the choice of adapted metrics by Remark~\ref{r:adapted metric}-(i). In particular, a local quasi-isometry between compact Thom-Mather stratifications restricts to quasi-isometries between their strata; thus a local quasi-isometry between compact Thom-Mather stratifications will be called a {\em quasi-isometry\/}. The condition of being locally quasi-isometric defines an equivalence relation on the family of Thom-Mather stratifications on any Hausdorff, locally compact and second countable space; each equivalence class will be called a {\em local quasi-isometry type\/} of Thom-Mather stratifications. By Remark~\ref{r:adapted metric}-(iv), the product of Thom-Mather stratifications is unique up to local quasi-isometries. 

\begin{defn}\label{d:rho-rel-neighborhood}
  Let $d$ be the distance function defined by an adapted metric on a connected stratum $M$ of a Thom-Mather stratification $A$. For each $x\in\ol{M}$ and $\rho>0$, the {\em relative ball\/} (or {\em rel-ball\/}) of {\em radius\/} $\rho$ and {\em center\/} $x$ is the set consisting of the points $y\in M$ such that there is a sequence $(z_k)$ in $M$ with $\lim_kz_k=x$ in $\ol{M}$ and $\limsup_kd(y,z_k)<\rho$. The term {\em $\rho$-rel-neighborhood\/} of $x$ will be also used.
\end{defn}

\begin{ex}
  \begin{itemize}
  
    \item[(i)] The rel-balls centered at points of $M$ are the usual balls.
    
    \item[(ii)] In the case of a model adapted metric on a stratum of $c(L)$ of the form $M=N\times\R_+$, the $\rho$-rel-neighborhood of the vertex $*$ is $N\times(0,\rho)$.
  
  \end{itemize}
\end{ex}

\subsubsection{Relatively local completion}

Let $M$ be a stratum of a Thom-Mather stratification $A$, and fix an adapted metric $g$ on $M$.

\begin{defn}\label{d:widehat M}
  Assume first that $M$ is connected, and consider the distance function $d$ on $M$ induced by $g$. The {\em relatively local completion\/} (or {\em rel-local completion\/}) is the subspace $\widehat{M}$ of the metric completion of $M$ whose points can be represented by Cauchy sequences in $M$ that converge in $A$; the limits in $\ol{M}$ of those sequences define a canonical continuous map $\lim:\widehat{M}\to\ol{M}$. The canonical dense injection of $M$ into its metric completion restricts to a canonical dense injection $\iota:M\to\widehat{M}$ satisfying $\lim\,\iota=\id_M$. The notation $\lim_M$ and $\iota_M$ may be also used. 

  If $M$ is not connected, then $\widehat{M}$ is defined as the disjoint union of the rel-local completions of its connected components.
\end{defn}

\begin{rem}\label{r:widehat M}
    \begin{itemize}
    
      \item[(i)] By Remark~\ref{r:adapted metric}-(i), $\widehat{M}$ is independent of the choice of the adapted metric.
      
      \item[(ii)] For any open $O\subset A$, $\widehat{M\cap O}$ can be canonically identified to the open subspace $\lim^{-1}(\ol{M}\cap O)\subset\widehat{M}$.
      
    \end{itemize}
\end{rem}

\begin{ex}\label{ex:rel-local completion of the strata of cones}
  Let $L$ be a compact Thom-Mather stratification and $M$ a stratum of $c(L)$. With the notation of Section~\ref{sss:cones}, if $M=\{*\}$, then $\widehat{M}=M$, obviously. Now, suppose that $M=N\times\R_+$ for some stratum $N$ of $L$. Consider the model adapted metric $g=\rho^2\tilde g+(d\rho)^2$ for some adapted metric $\tilde g$ on $N$, and the corresponding rel-local completion $\widehat{M}$. By Remark~\ref{r:by using charts}-(ii), $\pi_0(N)$ is finite. For each $P\in\pi_0(N)$, let $\widehat{P}$ denote the rel-local completion of $P$ with respect to $L_{\text{\rm con}}$, which is independent of the choice of $\tilde g$. Then it is easy to check that
    \[
      \begin{CD}
        M\equiv\bigsqcup_PP\times\R_+ 
        @>{\bigsqcup_P\iota_P\times\id}>> \bigsqcup_P\widehat{P}\times\R_+
        \hookrightarrow\bigsqcup_Pc(\widehat{P})
      \end{CD}
    \]
  extends to a homeomorphism $\widehat{M}\to\bigsqcup_{P\in\pi_0(N)}c(\widehat{P})$.
\end{ex}

\begin{rem}\label{r:rel-local completion of the strata of a cone}
  The following properties follow easily by using charts, induction on the depth of the strata, Example~\ref{ex:rel-local completion of the strata of cones} and Remark~\ref{r:adapted metric}-(ii):
    \begin{itemize}
    
      \item[(i)] $\lim:\widehat{M}\to\ol{M}$ is surjective with finite fibers.
      
      \item[(ii)] $M$ is rel-locally connected with respect to $\widehat{M}$.
      
      \item[(iii)] If $\ol{M}$ is compact, then $\widehat{M}$ is compact, and therefore its connected components are the metric completions of the connected components of $M$, as indicated in Section~\ref{s:intro}.
    
    \end{itemize}

\end{rem}

\begin{prop}\label{p:widehat M}
    \begin{itemize}
    
      \item[(i)] $\widehat{M}$ has a unique Thom-Mather stratification with connected strata such that $\lim:\widehat{M}\to\ol{M}$ is a morphism that restricts to local diffeomorphisms between corresponding strata. In particular, the connected components of $M$ can be considered as strata of $\widehat{M}$ via $\iota_M$.
      
      \item[(ii)] The restriction of $g$ to the connected components of $M$ are adapted metrics with respect to $\widehat{M}$.
      
      \item[(iii)] Let $M'$ be a connected stratum of another Thom-Mather stratification $A'$ equipped with an adapted metric. Then, for any morphism $f:A\to A'$ with $f(M)\subset M'$, the restriction $f:M\to M'$ extends to a morphism $\hat f:\widehat{M}\to\widehat{M'}$. Moreover $\hat f$ is an isomorphism if $f$ is an isomorphism.
      
    \end{itemize}
\end{prop}

\section{Relatively Morse functions}\label{s:rel-Morse}

The concept of rel-Morse function on strata is introduced and studied in this section; specially, their numerical contribution $\nu_{\text{\rm min/max}}^r$ for each degree $r$ is described using the rel-critical points. As a first step, we also introduce rel-admissible functions. A construction of rel-admissible functions from rel-local data is given (Lemma~\ref{l:|d lambda_a| is rel-locally bounded} and Proposition~\ref{p:existence of g so that f is rel-admissibility is rel-local}), and the existence of rel-Morse functions is shown (Proposition~\ref{p:existence of rel-Morse functions}); their proofs are given in Appendix~\ref{a: proofs stratifications}.

Let $M$ be a stratum of a Thom-Mather stratification $A$, and fix an adapted metric $g$ on $M$. Identify $M$ and its image by the canonical dense open embedding $\iota:M\to\widehat{M}$. Let $f\in\Cinf(M)$. Recall that the Hessian of $f$, with respect to $g$, is the smooth symmetric section of $TM^*\otimes TM^*$ defined by $\Hess f=\nabla df$.

\begin{defn}\label{d:rel-admissible, rel-critical, rel-non-degenerate}
    \begin{itemize}
    
      \item[(i)] It is said that $f$ is {\em relatively admissible\/} (or {\em rel-admissible\/}) with respect to $g$ if $f$, $|df|$ and $|\Hess f|$ are rel-locally bounded.
      
      \item[(ii)] A point $x\in\widehat{M}$ is called {\em relatively critical\/} (or {\em rel-critical\/}) if $\liminf|df(y)|=0$ as $y\to x$ in $\widehat M$ with $y\in M$. The set of rel-critical points of $f$ is denoted by $\Crit_{\text{\rm rel}}(f)$.
      
    \end{itemize}
\end{defn}

\begin{rem}\label{r:rel-admissible, rel-critical, rel-non-degenerate}
    \begin{itemize}
    
      \item[(i)] The rel-local boundedness of $|df|$ is invariant by rel-local quasi-isometries by Remark~\ref{r:adapted metric}-(i), and therefore it is independent of $g$. Similarly, the definition of rel-critical point is also independent of $g$ by Remarks~\ref{r:adapted metric}-(i) and~\ref{r:widehat M}-(i). But the rel-local boundedness of $|\Hess f|$ depends on the choice of $g$. However it follows from Lemma~\ref{l:|d lambda_a| is rel-locally bounded} and Proposition~\ref{p:existence of g so that f is rel-admissibility is rel-local} below that the existence of $g$ so that $f$ is rel-admissible with respect to $g$ is a rel-local property.
      
      \item[(ii)] If $\depth M=0$, then any smooth function is admissible, and its rel-critical points are its critical points.
      
      \item[(iii)] A rel-admissible function on $M$ may not have any continuous extension to $\ol{M}$, but it has a continuous extension to $\widehat{M}$ by the rel-local boundedness of $|df|$. Thus it becomes natural to define its rel-critical points in $\widehat{M}$.
      
      \item[(iv)] In Definition~\ref{d:rel-admissible, rel-critical, rel-non-degenerate}-(ii), if $f$ is rel-admissible, the condition $\liminf|df(y)|=0$ is equivalent to $\lim|df(y)|=0$ by the rel-local boundedness of $|\nabla df|$.
          
    \end{itemize}
\end{rem}

\begin{ex}\label{ex:rel-admissible function}
  With the notation of Example~\ref{ex:adapted metrics}-(i), for any $h\in\Cinf(\R_+)$ with $h'\in\Cinf_0(\R_+)$, the function $h(\rho)$ is rel-admissible on the stratum $\R^m\times N\times\R_+$ of $\R^m\times c(L)$ with respect to any model adapted metric. 
\end{ex}

\begin{ex}\label{ex:rel-admissible functions on orbit spaces}
   With the notation of Examples~\ref{ex:orbit} and~\ref{ex:adapted metrics on orbit spaces}, for any $G$-invariant smooth function $f$ on $M$, let $\bar f$ denote the induced function on $G\backslash M$, whose restriction to each stratum is smooth, and $df$ is the pull-back of $d\bar f$ on corresponding strata of $M$ and $G\backslash M$. Fix any invariant metric on $M$ and consider the induced adapted metric on the strata of $G\backslash M$. The restriction of $\Hess f$ to horizontal tangent vectors on the strata of $M$ corresponds via the canonical projection to $\Hess\bar f$ on the strata of $G\backslash M$ by \cite[Lemma~1]{ONeill1966}. It easily follows that $\bar f$ is rel-admissible on the strata of $G\backslash M$.
\end{ex}

\begin{lem}\label{l:|d lambda_a| is rel-locally bounded}
  For any locally finite covering $\{\,O_a\mid a\in\AA\,\}$ of $\ol{M}$ by open subsets of $A$, there is a smooth partition of unity $\{\lambda_a\}$ on $M$ subordinated to $\{M\cap O_a\}$ such that, for any adapted metric on $M$, each function $|d\lambda_a|$ is rel-locally bounded.
\end{lem}

\begin{prop}\label{p:existence of g so that f is rel-admissibility is rel-local}
  Let $\{\,O_a\mid a\in\AA\,\}$ be a locally finite covering of $\ol{M}$ by open subsets of $A$, let $\{\lambda_a\}$ be a partition of unity  on $M$ subordinated to the open covering $\{M\cap O_a\}$ satisfying the conditions of Lemma~\ref{l:|d lambda_a| is rel-locally bounded}, and let $f\in\Cinf(M)$ such that each $f|_{M\cap O_a}$ is rel-admissible with respect to some adapted metric $g_a$ on $M\cap O_a$. Then $f$ is rel-admissible with respect to the adapted metric $g=\sum_a\lambda_ag_a$ on $M$.
\end{prop}

We would like to define relatively Morse functions on $M$ as rel-admissible functions whose rel-critical points are ``rel-non-degenerate'' in an obvious sense. However an appropriate version of the Morse lemma \cite[Lemma~2.2]{Milnor1963:Morse} is missing, and thus the ``rel-local models'' around the rel-critical points are used to define them. 

\begin{defn}\label{d:rel-Morse function}
  It is said that $f\in\Cinf(M)$ is a {\em relatively Morse function\/} \upn{(}or {\em rel-Morse function\/}\upn{)} if it is rel-admissible with respect to some adapted metric and, for every $x\in\Crit_{\text{\rm rel}}(f)$, there exists a chart $(O,\xi)$ of $\widehat{M}$ centered at $x$, with  $\xi(O)=B\times c_\epsilon(L)$, such that, for some $m_\pm\in\N$ and compact Thom-Mather stratifications $L_\pm$, there exists a pointed diffeomorphism $\theta_0:(\R^m,0)\to(\R^{m_+}\times\R^{m_-},(0,0))$, and a local quasi-isometry $\theta_1:c(L)\to c(L_+)\times c(L_-)$ so that $f|_{M\cap O}$ corresponds to a constant plus $\frac{1}{2}(\rho_+^2-\rho_-^2)$ via $(\theta_0\times\theta_1)\,\xi$, where $\rho_\pm$ is the radial function on $\R^{m_\pm}\times c(L_\pm)$ (Example~\ref{ex:c(S^m-1)=R^m}).
\end{defn}

\begin{ex}\label{ex:rel-Morse functions on orbit spaces}
  With the notation of Examples~\ref{ex:orbit},~\ref{ex:adapted metrics on orbit spaces} and~\ref{ex:rel-admissible functions on orbit spaces}, the invariant Morse-Bott functions on $M$ whose critical submanifolds are orbits form a dense subset of the space of invariant smooth functions \cite[Lemma~4.8]{Wasserman1969}. They induce rel-Morse functions on every orbit type stratum of $G\backslash M$.
\end{ex}

\begin{defn}\label{d: nu_min/max^r}
	Let $f$ be a rel-Morse function on $M$. For each $x\in\Crit_{\text{\rm rel}}(f)$, with the notation of Definition~\ref{d:rel-Morse function}, let $M_\pm$ be the strata of $c(L_\pm)$ so that $(\theta_0\times\theta_1)\,\xi$ defines an open embedding of $M\cap O$ into $\R^{m_+}\times\R^{m_-}\times M_+\times M_-$, where either $M_\pm$ is the vertex stratum $\{*_\pm\}$ of $c(L_\pm)$, or $M_\pm=N_\pm\times\R_+$ for some stratum $N_\pm$ of $L_\pm$. Let $n_\pm=\dim M_\pm$. For every $r\in\Z$, define $\nu_{x,\text{\rm min/max}}^r=\nu_{x,\text{\rm min/max}}^r(f)$ in the following way:
		\begin{itemize}
		
			\item If $M_+=N_+\times\R_+$ and $M_-=N_-\times\R_+$, let
  				\[
    					\nu_{x,\text{\rm min/max}}^r
					=\sum_{r_+,r_-}\beta_{\text{\rm min/max}}^{r_+}(N_+)\,
					\beta_{\text{\rm min/max}}^{r_-}(N_-)\;,
  				\]
			where $(r_+,r_-)$ runs in the subset of $\Z^2$ determined by the conditions:
    				\begin{align}
      					r&=m_-+r_++r_-+1\;,\label{r=m_-+r_++r_-+1}\\
      					r_+&\le
        						\begin{cases}
          						\frac{n_+}{2}-1 & \text{if $n_+$ is even}\\
          						\frac{n_+-3}{2} & \text{if $n_+$ is odd, in the minimum i.b.c.\ case}\\
          						\frac{n_+-1}{2} & \text{if $n_+$ is odd, in the maximum i.b.c.\ case}\;,
        						\end{cases}
      					\label{r_+ le ...}\\
      					r_-&\ge
        						\begin{cases}
          						\frac{n_-}{2} & \text{if $n_-$ is even}\\
          						\frac{n_+-1}{2} & \text{if $n_-$ is odd, in the minimum i.b.c.\ case}\\
          						\frac{n_-+1}{2} & \text{if $n_-$ is odd, in the maximum i.b.c.\ case}\;,
        						\end{cases}
      					\label{r_- ge ...}
    				\end{align}
	
			\item If $M_+=\{*_+\}$ and $M_-=N_-\times\R_+$,  let $\nu_{x,\text{\rm min/max}}^r=\sum_{r_-}\beta_{\text{\rm min/max}}^{r_-}(N_-)$, where $r_-$ runs in the set of integers satisfying $r=m_-+r_-+1$ and~\eqref{r_- ge ...}.
			
			\item If $M_+=N_+\times\R_+$ and $M_-=\{*_-\}$, let $\nu_{x,\text{\rm min/max}}^r=\sum_{r_+}\beta_{\text{\rm min/max}}^{r_+}(N_+)$, where $r_+$ runs in the set of integers satisfying $r=m_-+r_+$ and~\eqref{r_+ le ...}. 
			
			\item If $M_+=\{*_+\}$ and $M_-=\{*_-\}$, let $\nu_{x,\text{\rm min/max}}^r=\delta_{r,m_-}$. 
			
		\end{itemize}
	Finally, let $\nu_{\text{\rm min/max}}^r=\sum_x\nu_{x,\text{\rm min/max}}^r$, where $x$ runs in the $\Crit_{\text{\rm rel}}(f)$.
\end{defn}

\begin{rem}\label{r:rel-Morse}
    \begin{itemize}
    
      \item[(i)] The rel-critical points of rel-Morse functions are isolated.
      
      \item[(ii)] The function $\frac{1}{2}(\rho_+^2-\rho_-^2)$ on $\R^{m_+}\times\R^{m_-}\times M_+\times M_-$ is rel-Morse, and will be called a {\em model rel-Morse function\/}.
    
    \end{itemize}
  
\end{rem}

The existence, and indeed certain abundance, of rel-Morse functions is guaranteed by the following result.

\begin{prop}\label{p:existence of rel-Morse functions}
  Let $\FF\subset C^\infty(M)$ denote the subset of functions with continuous extensions to $\ol{M}$ that restrict to rel-Morse functions on all strata $\le M$. Then $\FF$ is dense in $C^\infty(M)$ with the weak $C^\infty$ topology.
\end{prop}

\begin{rem}
  A ``(weak/strong) rel-$\Cinf$ topology'' can be easily defined on the set of rel-admissible functions on $M$. Then a much better density result of the rel-Morse functions should be true in this topology. Its proof would take us too far from the main goals of the paper.
\end{rem}

\section{Preliminaries on Hilbert complexes}\label{s:preliminaries Hilbert}

Here, we recall from \cite{BruningLesch1992} some basic definitions and needed results about Hilbert and elliptic complexes. Some elementary remarks are also made.

\subsection{Hilbert complexes}\label{ss:Hilbert}

For each $r\in\N$, let $\fH_r$ be a separable (real or complex) Hilbert space such that, for some $N\in\N$, we have $\fH_r=0$ for all $r>N$. They give rise to the graded Hilbert space $\fH=\bigoplus_r\fH_r$, where the terms $\fH_r$ are mutually orthogonal. For each degree $r$, let $\bd_r$ be a densely defined closed operator of $\fH_r$ to $\fH_{r+1}$. Let $\DD_r=\DD(\bd_r)$ (its domain) and $\RR_r=\bd_r(\DD_r)$ for each $r$, and let $\DD=\bigoplus_r\DD_r$ and $\bd=\bigoplus_r\bd_r$. Assume that $\RR_r\subset\DD_{r+1}$ and $\bd_{r+1}\bd_r=0$ for all $r$. Then the complex
  \[
    \begin{CD}
      0 @>>> \DD_0 @>{\bd_0}>> \DD_1 @>{\bd_1}>> \cdots @>{\bd_{N-1}}>> \DD_N @>>> 0
    \end{CD}
  \]
is called a {\em Hilbert complex\/}; its notation is abbreviated as $(\DD,\bd)$, or simply as $\bd$. Assuming that $\DD_0\ne0$, the maximum $N\in\N$ such that $\DD_N\ne0$ will be called the {\em length\/} of $(\DD,\bd)$. We may also consider Hilbert complexes with spaces of negative degree or with homogeneous operators of degree $-1$ without any essential change.

For the adjoint operator $\bd_r^*$ of each $\bd_r$, let $\DD^*_r=\DD(\bd_r^*)\subset\fH_{r+1}$ and $\RR^*_r=\bd_r^*(\DD_r^*)\subset\fH_r$, and set $\DD^*=\bigoplus_r\DD^*_r$ and $\bd^*=\bigoplus_r\bd_r^*$. Then we get a Hilbert complex
  \[
    \begin{CD}
      0 @<<< \DD^*_{-1} @<{\bd^*_0}<< \DD^*_0 @<{\bd^*_1}<< \cdots @<{\bd^*_{N-1}}<< \DD^*_{N-1} @<<< 0\;,
    \end{CD}
  \]
denoted by $(\DD^*,\bd^*)$ (or simply $\bd^*$), which is called {\em dual\/} or {\em adjoint\/} of $(\DD,\bd)$.

If $(\DD',\bd')$ is another Hilbert complex in the graded Hilbert space $\fH'=\bigoplus_r\fH'_r$, a homomorphism of complexes, $\zeta=\bigoplus_r\zeta_r:(\DD,\bd)\to(\DD',\bd')$, is called a {\em map of Hilbert complexes\/} if it is the restriction of a bounded map $\zeta:\fH\to \fH'$. If moreover $\zeta$ is an isomorphism of complexes and $\zeta^{-1}$ is a Hilbert complex map, then $\zeta$ is called an {\em isomorphism of Hilbert complexes\/}. If $\zeta:(\DD,\bd)\to(\widetilde{\DD}',\bd')$ is an isomorphism, where $\widetilde{\DD}'_r=\DD'_{r+r_0}$ for all $r$ and some fixed $r_0\ne0$, then it will be said that $\zeta:(\DD,\bd)\to(\DD',\bd')$ is an {\em isomorphism up to a shift of degree\/}.

Let
  \begin{alignat*}{2}
    \fH_{\text{\rm ev}}&=\bigoplus_r\mathfrak{\fH}_{2r}\;,&\quad\fH_{\text{\rm odd}}&=\bigoplus_r\fH_{2r+1}\;,\\
    \DD_{\text{\rm ev}}&=\bigoplus_r\DD_{2r}\;,&\quad\DD^*_{\text{\rm odd}}&=\bigoplus_r\DD^*_{2r-1}\;,\\
    \bd_{\text{\rm ev}}&=\bigoplus_r\bd_{2r}\;,&\quad\bd^*_{\text{\rm odd}}&=\bigoplus_r\bd^*_{2r-1}\;.
  \end{alignat*}
Note that $\DD^*_{\text{\rm odd}}\subset \fH_{\text{\rm ev}}$. The operator $\bD_{\text{\rm ev}}=\bd_{\text{\rm ev}}+\bd^*_{\text{\rm odd}}$, with domain $\DD_{\text{\rm ev}}\cap\DD_{\text{\rm odd}}^*$, is a densely defined closed operator of $\fH_{\text{\rm ev}}$ to $\fH_{\text{\rm odd}}$, whose adjoint is $\bD_{\text{\rm odd}}=\bd_{\text{\rm odd}}+\bd^*_{\text{\rm ev}}$. Thus
  \[
    \bD=
      \begin{pmatrix}
        0 & \bD_{\text{\rm ev}} \\
        \bD_{\text{\rm odd}} & 0
      \end{pmatrix}
    =\bd+\bd^*
  \]
is a self-adjoint operator in $\fH=\fH_{\text{\rm ev}}\oplus\fH_{\text{\rm odd}}$ with $\DD(\bD)=\DD\cap\DD^*$, and
  \[
    \bDelta=\bD^2=\bD_{\text{\rm odd}}\bD_{\text{\rm ev}}\oplus\bD_{\text{\rm ev}}\bD_{\text{\rm odd}}
    =\bd^*\bd+\bd\bd^*
  \]
is a self-adjoint non-negative operator, which can be called the {\em Laplacian\/} of $(\DD,\bd)$. Observe that $(\DD,\bd)$ and $(\DD^*,\bd^*)$ define the same Laplacian. The Hilbert complex $(\DD,\bd)$ can be reconstructed from $\bD_{\text{\rm ev}}$ \cite[Lemma~2.3]{BruningLesch1992}. The restriction of $\bDelta$ to each space $\DD_r$ will be denoted by $\bDelta_r$. Notice that $\ker\bDelta_r=\ker\bd_r\cap\ker\bd^*_{r-1}$ for all $r$. Moreover we have a weak Hodge decomposition \cite[Lemma~2.1]{BruningLesch1992}
  \[
    \fH_r=\ker\bDelta_r\oplus\ol{\RR_{r-1}}\oplus\ol{\RR^*_r}\;.
  \]

If $T$ is a self-adjoint operator in a Hilbert space, then $\DD^\infty(T)=\bigcap_{k\ge1}\DD(T^k)$ is a core\footnote{Recall that a {\em core\/} of a closed densely defined operator $T$ between Hilbert spaces is any subspace of its domain $\DD(T)$ which is dense with the graph norm.} for $T$, which is called its {\em smooth core\/}. In the case of the Laplacian $\bDelta$ of a Hilbert complex $(\DD,\bd)$ in a graded Hilbert space $\fH$, the smooth core $\DD^\infty(\bDelta)$, also denoted by $\DD^\infty(\bd)$ or $\DD^\infty$, is a subcomplex of $(\DD,\bd)$, and $(\DD^\infty,\bd)\hookrightarrow(\DD,\bd)$ induces an isomorphism in homology \cite[Theorem~2.12]{BruningLesch1992}. It will be also said that $\DD^\infty$ (respectively, $\DD^\infty_r$) is the {\em smooth core\/} of $\bd$ (respectively, $\bd_r$); notice that it is a core of $\bd$ (respectively, $\bd_r$). Let $\RR^\infty_r=\bd_r(\DD^\infty_r)$ and $\RR^{*\infty}_r=\bd^*_r(\DD^\infty_r)$, which are dense subspaces of $\RR_r$ and $\RR^*_r$.

The following properties are equivalent \cite[Theorem~2.4]{BruningLesch1992}:
  \begin{itemize}
  
    \item The homology of $(\DD,\bd)$ is of finite dimension and $\RR$ is closed in $\fH$.
    
    \item The homology of $(\DD,\bd)$ is of finite dimension.
    
    \item $\bD_{\text{\rm ev}}$ is a Fredholm operator.
    
    \item $0\not\in\spec_{\text{\rm ess}}(\bDelta)$ (the essential spectrum of $\bDelta$).
  
  \end{itemize}
In this case, $(\DD,\bd)$ is called a {\em Fredholm complex\/} and satisfies the following:
  \begin{itemize}
  
    \item $\RR$ and $\RR^*$ are closed in $\fH$ \cite[Corollary~2.5]{BruningLesch1992}, obtaining the stronger Hodge decompositions
      \[
       \fH_r=\ker\bDelta_r\oplus\RR_{r-1}\oplus\RR^*_r\;,\quad\DD^\infty=\ker\bDelta_r\oplus\RR^\infty_{r-1}\oplus\RR^{*\infty}_r\;.
      \]
    
    \item $\bd_r:\RR^{*\infty}_r\to\RR^\infty_r$ and $\bd^*_r:\RR^\infty_r\to\RR^{*\infty}_r$ are isomorphisms.
    
    \item $\ker\bDelta_r$ is isomorphic to the homology of degree $r$ of $(\DD,\bd)$.
  
  \end{itemize}

It is said that $(\DD,\bd)$ is {\em discrete\/} when $\bDelta$ has a discrete spectrum ($\spec_{\text{\rm ess}}(\bDelta)=\emptyset$). The following properties hold when $(\DD,\bd)$ is discrete:
  \begin{itemize}
    
    \item For each $\lambda\in\spec(\bDelta|_{\RR^\infty_r})$, we get isomorphisms
      \[
        \bd_r:E_\lambda(\bDelta|_{\RR^{*\infty}_r})\to E_\lambda(\bDelta|_{\RR^\infty_r})\;,\quad
        \bd^*_r:E_\lambda(\bDelta|_{\RR^\infty_r})\to E_\lambda(\bDelta|_{\RR^{*\infty}_r})
      \]
    between the corresponding eigenspaces. Thus $\spec(\bDelta|_{\RR^\infty_r})=\spec(\bDelta|_{\RR^{*\infty}_r})$.
    
    \item We have
      \[
        \spec(\bd_r|_{\RR^{*\infty}_r}\oplus\bd^*_r|_{\RR^\infty_r})=\{\,\pm\sqrt{\lambda}\mid\lambda\in\spec(\bDelta|_{\RR^\infty_r})\,\}\;,
      \]
    and, for each $\lambda\in\spec(\bDelta|_{\RR^\infty_r})$, $E_{\pm\sqrt{\lambda}}(\bd_r|_{\RR^\infty_r}\oplus \bd^*_r|_{\RR^{*\infty}_r})$ consists of the elements of the form $u\pm v$ with $u\in E_\lambda(\bDelta|_{\RR^\infty_r})$ and $v\in E_\lambda(\bDelta|_{\RR^{*\infty}_r})$ satisfying $\bd^*u=\sqrt{\lambda}\,v$ and $\bd v=\sqrt{\lambda}\,u$. Moreover the mapping $u+v\mapsto u-v$, for $u$ and $v$ as above, defines an isomorphism
      \[
        E_{\sqrt{\lambda}}(\bd_r|_{\RR^{*\infty}_r}\oplus\bd^*_r|_{\RR^\infty_r})\to E_{-\sqrt{\lambda}}(\bd_r|_{\RR^{*\infty}_r}\oplus\bd^*_r|_{\RR^\infty_r})\;.
      \]
    
    \item Any Hilbert complex $(\DD',\bd')$ isomorphic to $(\DD,\bd)$ is also discrete, and, if $\spec(\bDelta_r)$ and $\spec(\bDelta'_r)$ consist of the eigenvalues $0\le\lambda_0\le\lambda_1\le\cdots$ and $0\le\lambda'_0\le\lambda'_1\le\cdots$, respectively, then there is some $C\ge1$ such that $C^{-1}\lambda_k\le\lambda'_k\le C\lambda_k$ for all $k\in\N$ \cite[Lemma~2.17]{BruningLesch1992}.
  
  \end{itemize}

Consider Hilbert complexes, $(\DD',\bd')$ and $(\DD'',\bd'')$, in respective graded Hilbert spaces, $\fH'$ and $\fH''$. The Hilbert space tensor product\footnote{Recall that this is the Hilbert space completion of the algebraic tensor product $\fH'\otimes \fH''$ with respect to the scalar product defined by $\langle u'\otimes u'',v'\otimes v''\rangle=\langle u',v'\rangle'\,\langle u'',v''\rangle''$, where $\langle\ ,\ \rangle'$ and $\langle\ ,\ \rangle''$ are the scalar products of $\fH'$ and $\fH''$, respectively.}, $\fH=\fH'\widehat{\otimes}\fH''$, has a canonical grading ($\fH_r=\bigoplus_{p+q=r}\fH'_p\widehat{\otimes}\fH''_q$), and
  \[
    \widetilde{\DD}=(\DD'\otimes \fH'')\cap(\fH'\otimes\DD'')\subset \fH
  \]
is a dense graded subspace. Let $\tilde\bd=\bd'\otimes1+\bw\otimes\bd''$ with domain $\widetilde{\DD}$, where $\bw$ denotes the degree involution on $\fH'$, and let $\bd=\ol{\tilde\bd}$, whose domain is denoted by $\DD$. Then $(\DD,\bd)$ is a Hilbert complex in $\fH$ called the {\em tensor product\/} of $(\DD',\bd')$ and $(\DD'',\bd'')$. If $\bDelta'$, $\bDelta''$ and $\bDelta$ denote the Laplacians of $(\DD',\bd')$, $(\DD'',\bd'')$ and $(\DD,\bd)$, respectively, then $\bDelta=\bDelta'\otimes1+1\otimes\bDelta''$ on $\widetilde{\DD}$. The following result is elementary.

\begin{lem}\label{l:tensor product}
 If $(\DD',\bd')$ and $(\DD'',\bd'')$ are discrete, then $(\DD,\bd)$ is discrete. More precisely, given complete orthonormal systems of $\fH'$ and $\fH''$ consisting of eigenvectors $e'_k$ and $e''_k$ \upn{(}$k\in\N$\upn{)} of $\bDelta'$ and $\bDelta''$, with corresponding eigenvalues $\lambda'_k$ and $\lambda''_k$, respectively, we get a complete orthonormal system of $\fH$ consisting of the eigenvectors $e'_k\otimes e''_\ell\in\widetilde{\DD}$ of $\bDelta$ with corresponding eigenvalues $\lambda'_k+\lambda''_\ell$.
\end{lem}

Let $(\EE,d)$ be a densely defined complex in a graded separable Hilbert space $\fH$ ($\EE$ is a dense graded linear subspace of $\fH$). Consider the family of Hilbert complexes $(\DD,\bd)$ in $\fH$ extending $(\EE,d)$ ($(\EE,d)$ is a subcomplex of $(\DD,\bd)$) equipped with the order relation defined by ``being a subcomplex''. We will be interested in its minimum/maximum elements. Notice that, if $(\EE,d)$ has some Hilbert complex extension, then  $\ol{d}$ is the minimum Hilbert complex extension of $(\EE,d)$. Another complex of the form $(\EE,\delta)$, with $\delta_r:\EE_{r+1}\to\EE_r$ for each degree $r$, will be called a {\em formal adjoint\/} of $(\EE,d)$ if $\langle du,v\rangle=\langle u,\delta v\rangle$ for all $u,v\in\EE$; there is at most one formal adjoint by the density of $\EE$ in $\fH$. In this case, if $(\EE,\delta)$ has some Hilbert complex extension, then the adjoint of the minimum Hilbert complex extension of $(\EE,\delta)$ is the maximum Hilbert complex extension of $(\EE,d)$. 

Now, consider a countable family of densely defined complexes $(\EE^a,d^a)$ in separable graded Hilbert spaces $\fH^a$ ($a\in\N)$, let $(\DD^a,\bd^a)$ be a Hilbert complex extension of each $(\EE^a,d^a)$ in $\fH^a$, and let $\bDelta^a$ denote the corresponding Laplacian. Suppose that the Hilbert complexes $(\DD^a,\bd^a)$ are of uniformly finite length (there is some $N\in\N$ such that $\DD^a_r=0$ for all $r\ge N$ and all $a$). Let $(\EE,d)$ be the complex defined by $\EE=\bigoplus_a\EE^a$ and $d=\bigoplus_ad^a$. The Hilbert space direct sum\footnote{Recall that this is the Hilbert space completion of the algebraic direct sum, $\bigoplus_a\fH^a$, with respect to the scalar product $\langle(u^a),(v^a)\rangle=\sum_a\langle u^a,v^a\rangle_a$, where each $\langle\ ,\ \rangle_a$ is the scalar product of $\fH^a$. We have $\fH=\bigoplus_a\fH^a$ if the number of terms $\fH^a$ is finite.}, $\fH=\widehat{\bigoplus}_a\fH^a$, has an induced grading ($\fH_r=\widehat{\bigoplus}_a\fH^a_r$). Let $\bd=\widehat{\bigoplus}_a\bd^a$ (the graph of $\bd$ is the Hilbert space direct sum of the graphs of the maps $\bd^a$). The domain $\DD$ of $\bd$ consists of the points $(u^a)\in \fH$ such that $u^a\in\DD^a$ for all $a$ and $(\bd^au^a)\in \fH$. Moreover $\bd$ is defined by $(u^a)\mapsto(\bd^au^a)$. Clearly, $(\DD,\bd)$ is a Hilbert complex extension of $(\EE,d)$ in $\fH$ with
  \begin{align}
    \DD^\infty(\bd)&=\left\{\left.(u^a)\in\widehat{\bigoplus_a}\;\DD^\infty(\bd^a)\,\right|
    ((1+\bDelta^a)^ku^a)\in\widehat{\bigoplus_a}\;\DD^\infty(\bd^a)\ \forall k\in\N\right\}\;,
    \label{DD^infty(bd)}\\
    \bd^*&=\widehat{\bigoplus_a}\;{\bd^a}^*\;.\label{bd^*}
  \end{align}

The following lemma will be proved in Appendix~\ref{a: proofs Hilbert}.

\begin{lem}\label{l:minimal/maximal Hilbert complex extension}
    \begin{itemize}
    
      \item[(i)] If each $(\DD^a,\bd^a)$ is a minimum Hilbert complex extension of $(\EE^a,\bd^a)$ in $\fH^a$, then $(\DD,\bd)$ is a minimum Hilbert complex extension of $(\EE,\bd)$ in $\fH$.
      
      \item[(ii)] If each $(\EE^a,d^a)$ has a formal adjoint $(\EE^a,\delta^a)$ with some Hilbert complex extension, and each $(\DD^a,\bd^a)$ is a maximum Hilbert complex extension of $(\EE^a,\bd^a)$ in $\fH^a$, then $(\DD,\bd)$ is a maximum Hilbert complex extension of $(\EE,\bd)$ in $\fH$.
    
    \end{itemize}
\end{lem}

\subsection{Elliptic complexes}\label{ss:elliptic complexes}

Let $M$ be a possibly non-complete Riemannian manifold, and let $E=\bigoplus_rE_r$ be a graded Riemannian or Hermitean vector bundle over $M$, with $E_r=0$ if $r<0$ or $r>N$ for some $N\in\N$. The space of smooth sections of each $E_r$ will be denoted by $\Cinf(E_r)$, its subspace of compactly supported smooth sections will be denoted by $\Cinf_0(E_r)$, and the Hilbert space of square integrable sections of $E_r$ will be denoted by $L^2(E_r)$; then $\Cinf(E)=\bigoplus_r\Cinf(E_r)$, $\Cinf_0(E)=\bigoplus_r\Cinf_0(E_r)$ and $L^2(E)=\bigoplus_rL^2(E_r)$. For each $r$, let $d_r:\Cinf(E_r)\to\Cinf(E_{r+1})$ be a first order differential operator, and set $d=\bigoplus_rd_r$. Suppose that $(\Cinf(E),d)$ is an elliptic complex\footnote{Recall that this means that it is a complex and the sequence of principal symbols of the operators $d_r$ is exact in the fiber over each non-zero cotangent vector}; however, ellipticity is not needed for several elementary properties stated in this section. The simpler notation $(E,d)$ (or even $d$) will be preferred. Elliptic complexes with non-zero terms of negative degrees or homogeneous differential operators of degree $-1$ may be also considered without any essential change.

Consider the formal adjoint\ $\delta_r={}^td_r:\Cinf(E_{r+1})\to\Cinf(E_r)$ for each $r$, and set $\delta=\bigoplus_r\delta_r$. Then $(E,\delta)$ is another elliptic complex that will be called the {\em formal adjoint\/} of $(E,d)$, and its subcomplex $(C_0^\infty(E),\delta)$ is formal adjoint of $(C_0^\infty(E),d)$ in $L^2(E)$ in the sense of Section~\ref{ss:Hilbert}. Let $D=d+\delta$ and $\D=D^2=d\delta+\delta d$ on $\Cinf(E)$. This $\D$ can be called the {\em Laplacian\/} defined by $(E,d)$, and its components are $\D_r=d_{r-1}\delta_{r-1}+\delta_rd_r$.

Any Hilbert complex extension of $(\Cinf_0(E),d)$ in $L^2(E)$ is called an {\em ideal boundary condition\/} (shortly, {\em i.b.c.\/})\ of $(E,d)$. There always exist a minimum and maximum i.b.c., $d_{\text{\rm min}}=\ol{d}$ and $d_{\text{\rm max}}=\delta_{\text{\rm min}}^*$ \cite[Lemma~3.1]{BruningLesch1992}. The complex $d_{\text{\rm min/max}}$ defines the operator $D_{\text{\rm min/max}}=d_{\text{\rm min/max}}+\delta_{\text{\rm max/min}}$ and the Laplacian $\D_{\text{\rm min/max}}=D_{\text{\rm min/max}}^2$, 
which extend $D$ and $\D$ on $\Cinf_0(E)$. The homogeneous components of $\D_{\text{\rm min/max}}$ are
  \begin{equation}\label{D_min/max}
    \D_{\text{\rm min/max},r}=\delta_{\text{\rm max/min},r}\,d_{\text{\rm min/max},r}+d_{\text{\rm min/max},r-1}\,\delta_{\text{\rm max/min},r-1}\;.
  \end{equation}
The notation $d_{r,\text{\rm min/max}}$ and $\delta_{r,\text{\rm max/min}}$ also makes sense for $d_{\text{\rm min/max},r}$ and $\delta_{\text{\rm max/min},r}$ by considering $d_r$ and $\delta_r$ as differential complexes of length one (ellipticity is not needed here); similarly, any first order differential operator can be considered as a differential complex of length one and denote its minimum/maximum i.b.c.\ with the subindex ``min/max'', regardless of ellipticity.
  
For any i.b.c.\ $(\DD,\bd)$ of $(E,d)$, the map of complexes, $(\DD\cap\Cinf(E),d)\hookrightarrow(\DD,\bd)$, induces an isomorphism in homology \cite[Theorem~3.5]{BruningLesch1992}. We have $\DD^\infty\subset\DD\cap\Cinf(E)$ by elliptic regularity.
  
Let $(E',d')$ be another elliptic complex over another Riemannian manifold $M'$. Consider a vector bundle isomorphism $\zeta:E\to E'$ over a quasi-isometric diffeomorphism $\xi:M\to M'$ such that the restrictions of $\zeta$ to the fibers are quasi-isometries. It induces a map $\zeta:\Cinf(E)\to\Cinf(E')$ defined by $(\zeta u)(x')=\zeta(u(\xi^{-1}(x'))$ for $u\in\Cinf(E)$ and $x'\in M'$. If moreover $\zeta:(\Cinf(E'),d')\to(\Cinf(E),d)$ is a homomorphism of complexes, then it will be called a {\em quasi-isometric isomorphism\/} of elliptic complexes, and the simpler notation $\zeta:(E',d')\to(E,d)$ will be preferred. In this case, $\zeta$ induces a quasi-isometric isomorphism $\zeta:L^2(E')\to L^2(E)$, which restricts to the isomorphism $\zeta:(\Cinf_0(E'),d')\to(\Cinf_0(E),d)$. Moreover, for any i.b.c.\ $(\DD',\bd')$ of $(E',d')$, there is a unique i.b.c.\ $(\DD,\bd)$ of $(E,d)$ so that $\zeta:L^2(E')\to L^2(E)$ restricts to a Hilbert complex isomorphism $\zeta:(\DD',\bd')\to(\DD,\bd)$. In particular, $\zeta$ induces Hilbert complex isomorphisms between the corresponding minimum/maximum i.b.c. If $\xi$ is isometric and the restrictions to the fibers of $\zeta$ are isometries, then $\zeta:(E',d')\to(E,d)$ is called an {\em isometric isomorphism\/} of elliptic complexes.

Now, let $(E',d')$ and $(E'',d'')$ be elliptic complexes on Riemannian manifolds $M'$ and $M''$, respectively, and consider the exterior tensor product $E=E'\boxtimes E''$ on $M=M'\times M''$ with its canonical grading ($E_r=\bigoplus_{p+q=r}E'_p\boxtimes E''_q$). With the weak $\Cinf$ topology, $\Cinf(E')\otimes\Cinf(E'')$ can be canonically realized as a dense subspace of $\Cinf(E)$. Then $d=d'\otimes1+\bw\otimes d''$ on $\Cinf(E')\otimes\Cinf(E'')$ has a unique continuous extension to $\Cinf(E)$, also denoted by $d$. It turns out that $(E,d)$ is an elliptic complex. Moreover the minimum/maximum i.b.c.\ of $(E,d)$ is the tensor product, in the sense of Section~\ref{ss:Hilbert}, of the minimum/maximum i.b.c.\ of $(E',d')$ and $(E'',d'')$ \cite[Lemma~3.6]{BruningLesch1992}. 

\begin{ex}\label{ex:de Rham}
  A particular case of elliptic complex on $M$ is its de~Rham complex $(\Omega(M),d)$. In this case, $\delta$ is the de~Rham coderivative, the subcomplex of compactly supported differential forms is denoted by $\Omega_0(M)$, and the Hilbert space of $L^2$ differential forms is denoted by $L^2\Omega(M)$. Let $H_{\text{\rm min/max}}(M)$ denote the cohomology of the minimum/maximum i.b.c., $d_{\text{\rm min/max}}$, of $(\Omega_0(M),d)$, which is a quasi-isometric invariant of $M$. $H_{\text{\rm max}}(M)$ is the $L^2$-cohomology $H_{(2)}(M)$ \cite{Cheeger1980}; (a generalization to arbitrary elliptic complexes is given in \cite[Theorem~3.5]{BruningLesch1992}). The dimensions $\beta_{\text{\rm min/max}}^r(M)=\dim H_{\text{\rm min/max}}^r(M)$ can be called {\em min/max-Betti numbers\/}; if they are finite, then $\chi_{\text{\rm min/max}}(M)=\sum_r(-1)^r\,\beta_{\text{\rm min/max}}^r(M)$ is defined and can be called {\em min/max-Euler characteristic\/}; the simpler notation $\beta_{\text{\rm min/max}}^r$ and $\chi_{\text{\rm min/max}}$ may be used. If $M$ is orientable, then $\D_{\text{\rm max}}$ corresponds to $\D_{\text{\rm min}}$ by the Hodge star operator. It is known that $d_{\text{\rm min/max}}$ satisfies the following properties for special classes of Riemannian manifolds:
  \begin{itemize}
   
    \item If $M$ is complete, then $d_{\text{\rm min}}=d_{\text{\rm max}}$ (a particular case of \cite[Lemma~3.8]{BruningLesch1992}).
     
    \item If $M$ is the interior of a compact manifold with boundary, then $d_{\text{\rm min/max}}$ is given by the relative/absolute boundary conditions \cite[Theorem~4.1]{BruningLesch1992}. 
     
    \item Suppose that $M=\widetilde{M}\sm\Sigma$, where $\widetilde{M}$ is a closed Riemannian manifold of dimension $>2$ and $\Sigma$ is a closed finite union of submanifolds with codimension $\ge2$. Then $d_{\text{\rm min}}=d_{\text{\rm max}}$ \cite[Theorem~4.4]{BruningLesch1992}. 
     
    \item Let $A$ be a compact Thom-Mather stratification that is a pseudomanifold. If $M$ is the regular stratum of $A$ equipped with an adapted metric, then $H_{(2)}(M)$ is isomorphic to the intersection homology of $A$ with lower middle perversity \cite{CheegerGoreskyMacPherson1982}.
   
  \end{itemize}
 Given another Riemannian manifold $M'$, for any quasi-isometric (respectively, isometric) diffeomorphism $\xi:M\to M'$, the induced isomorphism $\xi^*$ between the corresponding de~Rham complexes is quasi-isometric (respectively, isometric). 
\end{ex}

\section{Sobolev spaces defined by an i.b.c.}
\label{s:Sobolev sp. defined by an i.b.c.}

Here, we study Sobolev spaces associated to i.b.c.\ of an elliptic complex, specially its minimum/maximum i.b.c. These results will be used in the globalization results of Section~\ref{s: globalization}. Their proofs are given in Appendix~\ref{a: proofs Hilbert}.

Let $T$ be a self-adjoint operator in a Hilbert space $\fH$. For each $m\in\N$, the {\em Sobolev space\/} of {\em order\/} $k$ associated to $T$ is the Hilbert space completion $W^m=W^m(T)$ of $\DD^\infty=\DD^\infty(T)$ with respect to the scalar product $\langle\ ,\ \rangle_m$ defined by $\langle u,v\rangle_m=\langle u,(1+T)^mv\rangle$. The notation $\|\ \|_m$ and $\Cl_m$ (or $\|\ \|_{W^m}$ and $\Cl_{W^m}$) will be used for the norm and closure in $W^m$. There are continuous inclusions $W^{m+1}\hookrightarrow W^m$, and we have $\DD^\infty=\bigcap_mW^m$. Moreover $T$ defines a bounded operator $W^{m+2}\to W^m$.

Now, let $(\DD,\bd)$ be an i.b.c.\ of an elliptic complex $(E,d)$ on a Riemannian manifold $M$. Its adjoint $(\DD^*,\bd^*)$ is an i.b.c.\ of the elliptic complex $(E,\delta)$, where $\delta={}^td$. We get the operators $D=d+\delta$ and $\bD=\bd+\bd^*$, and the Laplacians $\D=D^2$ and $\bDelta=\bD^2$. Then $W^m=W^m(\bDelta)$ can be called the {\em Sobolev space\/} of {\em order\/} $m$ associated to $(\DD,\bd)$, and may be also denoted by $W^m(\bd)$; the notation $W^m(\bd_r)$ will be also used when we consider its subspace of homogeneous elements of degree $r$. Since $(\DD,\bd)$ and $(\DD^*,\bd^*)$ define the same Laplacian, we get $W^m(\bd)=W^m(\bd^*)$ for all $m$. For $u\in\DD^\infty_r$, we have
  \[
    \|u\|_1^2=\|u\|^2+\|Du\|^2=\|u\|^2+\|d_ru\|^2+\|\delta_{r-1}u\|^2\;.
  \]
So
  \begin{gather}
    W^1=\DD(\bD)=\DD\cap\DD^*\;,\label{W^1}\\
    \|u\|_1^2=\|u\|^2+\|\bD u\|^2=\|u\|^2+\|\bd_ru\|^2+\|\bd_{r-1}^*u\|^2\label{|u|_1^2}
  \end{gather}
for $u\in W^1(\bd_r)$. The following generalizes the Rellich lemma on compact manifolds.

\begin{lem}\label{l:W^m}
  The following properties are equivalent:
    \begin{itemize}
    
      \item[(i)] $(\DD,\bd)$ is discrete.
      
      \item[(ii)] $W^1\hookrightarrow W^0=L^2(E)$ is compact.
      
      \item[(iii)] $W^{m+1}\hookrightarrow W^m$ is compact for all $m$.
    
    \end{itemize}
\end{lem}

For any fixed $f\in\Cinf(M)$, let $f$ also denote the operator of multiplication by $f$ on $\Cinf(E)$ (or on $L^2(E)$ if $f$ is bounded). Observe that $[d,f]$ is of order zero because $d$ is of first order; moreover $[d,f]^*=-[\delta,f]$.

\begin{lem}\label{l:f DD(d_min/max) subset DD(d_min/max)}
  If $f$ and $|[d,f]|$ are bounded, then:
    \begin{itemize}
      
      \item[(i)] $f\,\DD(d_{\text{\rm min/max}})\subset\DD(d_{\text{\rm min/max}})$ and $[d_{\text{\rm min/max}},f]=[d,f]$; and
      
      \item[(ii)] $f\,W^1(d_{\text{\rm min/max}})\subset W^1(d_{\text{\rm min/max}})$.
      
    \end{itemize}
\end{lem}

Let $(E',d')$ be another elliptic complex on a Riemannian manifold $M'$. The scalar product of $L^2(E')$ will be denoted by $\langle\ ,\ \rangle'$, and let $\delta'={}^td'$. Let $U$ and $U'$ be open subsets of $M$ and $M'$, respectively, so that $U\supset\supp f$, and let $\zeta:(E|_U,d)\to(E'|_{U'},d')$ be a quasi-isometric isomorphism of elliptic complexes whose underlying quasi-isometric diffeomorphism is $\xi:U\to U'$. For each $u\in L^2(E)$, identify $fu$ to $fu|_U$, and identify $\zeta(fu)\in L^2(E'|_{U'})$ with its extension by zero to the whole of $M'$; in this way, we get a subspace $\zeta(f\,\DD(d_{\text{\rm min/max}}))\subset L^2(E')$.

\begin{lem}\label{l:zeta(f DD(d min/max)) subset DD(d'_min/max)}
  If $f$ and $|[d,f]|$ are bounded, then the following properties hold:
    \begin{itemize}
      \item[(i)] We have $\zeta(f\,\DD(d_{\text{\rm min/max}}))\subset\DD(d'_{\text{\rm min/max}})$, and $d'_{\text{\rm min/max}}\zeta=\zeta d_{\text{\rm min/max}}$ on $f\,\DD(d_{\text{\rm min/max}})$
      
      \item[(ii)] If moreover $\zeta$ is isometric, then $\zeta(f\,W^1(d_{\text{\rm min/max}}))\subset W^1(d'_{\text{\rm min/max}})$.
      
    \end{itemize}
\end{lem}

\section{A perturbation of the harmonic oscillator}\label{s:P}

The main analytic tool used in this paper is a perturbation of the harmonic oscillator introduced and studied in \cite{AlvCalaza2014}, which is recalled in this section.

Let $\rho$ denote the canonical coordinate of $\R_+$. For each $a\in\R$, the operator of multiplication by the function $\rho^a$ on $C^\infty(\R_+)$ will be also denoted by $\rho^a$. We have
  \begin{equation}\label{[d/d rho,rho^a right]}
    \left[d/d\rho,\rho^a\right]=a\rho^{a-1}\;,\quad
    \left[d^2/d\rho^2,\rho^a\right]=2a\rho^{a-1}\,d/d\rho+a(a-1)\rho^{a-2}\;.
  \end{equation}
  
Recall that the harmonic oscillator on $C^\infty(\R_+)$ is the operator $H=-\frac{d^2}{d\rho^2}+s^2\rho^2$, depending on a parameter $s>0$. For $c_1,c_2\in\R$, consider its perturbation
  \begin{equation}\label{P}
    P=H-2c_1\rho^{-1}\,\frac{d}{d\rho}+c_2\rho^{-2}\;.
  \end{equation}
By~\eqref{[d/d rho,rho^a right]}, we get an operator of the same type if $\rho^{-1}$ and $\frac{d}{d\rho}$ are interchanged.

Let $\SS_{\text{\rm ev/odd}}$ denote the space of even/odd functions in the Schwartz space $\SS=\SS(\R)$. The restrictions of those functions to $\R_+$ form the space denoted by $\SS_{\text{\rm ev/odd},+}$. The scalar product of $L^2(\R_+,\rho^{2c_1}\,d\rho)$ will be denoted by $\langle\ ,\ \rangle_{c_1}$, and the corresponding norm by $\|\ \|_{c_1}$. For each $\sigma>-1/2$, let $p_k$ denote the sequence of orthogonal polynomials associated with the measure $e^{-sx^2}|x|^{2\sigma}\,dx$ on $\R$ \cite{Szego1975}, called generalized Hermite polynomials. The corresponding generalized Hermite functions are $\phi_k=p_ke^{-sx^2/2}$.

\begin{prop}[{\cite[Theorem~1.4]{AlvCalaza2014}}]\label{p:P}
  If there is some $a\in\R$ such that
    \begin{gather}
      a^2+(2c_1-1)a-c_2=0\;,\label{a}\\
      \sigma:=a+c_1>-1/2\;,\label{sigma}
    \end{gather}
  then:
    \begin{itemize}
  
      \item[(i)] $P$, with domain $\rho^a\,\SS_{\text{\rm ev},+}$, is essentially self-adjoint in $L^2(\R_+,\rho^{2c_1}\,d\rho)$; 
    
      \item[(ii)] the spectrum of its self-adjoint extension, denoted by $\PP$, consists of the eigenvalues $(4k+1+2\sigma)s$ {\rm(}$k\in\N${\rm)} with multiplicity one and normalized eigenfunctions $\chi_{s,a,\sigma,k}:=\sqrt{2}\,\rho^a\phi_{2k,+}$ {\rm(}or simply $\chi_k${\rm)}; and
    
      \item[(iii)] $\DD^\infty(\PP)=\rho^a\SS_{\text{\rm ev},+}$.
    
    \end{itemize}
\end{prop}

  By Proposition~\ref{p:P}-(iii), we have $h\,\DD^\infty(\PP)\subset\DD^\infty(\PP)$ for all $h\in\Cinf(\R_+)$ such that $h'\in\Cinf_0(\R_+)$. More precisely, we have the following.

\begin{lem}\label{l: |P^k(h phi)|_c_1}
	Let $h\in\Cinf(\R_+)$ with $h'\in\Cinf_0(\R_+)$. Then, for all $k\in\N$, there is some $C_k=C_k(c_1,h)>0$ such that,  for all $\phi\in\DD^\infty(\PP)$,
		\[
		  \|(1+P)^k(h\phi)\|_{c_1}\le C_k\,\|(1+P)^k\phi\|_{c_1}\;.
		\]
\end{lem}

\begin{proof}
  With the notation of \cite{AlvCalaza2014}, recall that the Dunkl operator $T_\sigma$ ($\sigma>-1/2$) on $\Cinf(\R)$ is the perturbation of $\frac{d}{dx}$ defined by $T_\sigma=\frac{d}{dx}$ on even functions and $T_\sigma=\frac{d}{dx}+2\sigma\frac{1}{x}$ on odd functions, where $x$ denotes the canonical coordinate of $\R$. This gives rise to the Dunkl harmonic oscillator, and Dunkl annihilation and creation operators are the perturbations $L=-T_\sigma^2+s^2x^2$, $B=sx+T_\sigma$ and $B'=sx-T_\sigma$ ($s>0$), which are perturbations of the harmonic oscillator, $H=-\frac{d^2}{dx^2}+s^2x^2$, and annihilation and creation operators, $A=sx+\frac{d}{dx}$ and $A'=sx-\frac{d}{dx}$. The well known relations satisfied by $H$, $A$ and $A'$ can be generalized for $L$, $B$ and $B'$; for instance,
    \begin{equation}\label{L=BB'-(1+2 Sigma)s}
      L=BB'-(1+2\Sigma)s=B'B+(1+2\Sigma)s=\frac{1}{2}(BB'+B'B)\;,
    \end{equation}
  where $\Sigma$ is the operator of multiplication by $\sigma$ on even functions, and by $-\sigma$ on odd functions \cite[Eq.~(4)]{AlvCalaza2014}. These operators preserve $\SS$. Considering these operators in $L^2(\R,|x|^{2\sigma}\,dx)$ with domain $\SS$, we have that $B'$ is symmetric of $B$, and $L$ is essentially self-adjoint in $L^2(\R,|x|^{2\sigma}\,dx)$. Moreover the smooth core of the closure of $L$ is $\SS$, and its spectrum can be described like in the case of $H$, $A$ and $A'$.
  
  Now, let $f$ be any smooth even function on $\R$ such that $f'$ is compactly supported. Note that the function $h$ of the statement extends to function on $\R$ satisfying this conditions. According to \cite[Theorem~1.4 and Section~5]{AlvCalaza2014}, it is enough to prove the following.
    
  \begin{claim}\label{cl: |L^k(f phi)|_sigma}
    For all $k\in\N$, there is some $C_k=C_k(\sigma,f)>0$ such that, for all $\phi\in\SS$,
      \[
        \|(1+L)^k(f\phi)\|_\sigma\le C_k\,\|(1+L)^k\phi\|_\sigma\;.
      \]
  \end{claim}
  
  For any non-commutative polynomial of two variables, $p=p(X,Y)$, let $p'=p'(X,Y)$ be the polynomial obtained from $p$ by reversing the order of the variables; for example, if $p=XY^2$, then $p'=X^2Y$. It is said that $p$ is symmetric if $p=p'$; in this case, $p(B,B')$ is a symmetric operator. From~\eqref{L=BB'-(1+2 Sigma)s} and by induction on $k$, we easily get the following.
  
  \begin{claim}\label{cl: |p(B,B')phi|_sigma}
    For any non-commutative polynomial $p=p(X,Y)$ of degree $k\in\N$, there is some $C_p>0$ such that $\|p(B,B')\phi\|_\sigma\le C\,\langle(1+L)^k\phi,\phi\rangle_\sigma$ for all $\phi\in\SS$.
  \end{claim}
  
  Observe that $[B,f]=f'$ and $[B,f]=-f'$. Then Claim~\ref{cl: |L^k(f phi)|_sigma} easily follows from~\eqref{L=BB'-(1+2 Sigma)s} for $k=1$ (the case $k=0$ is trivial). On the other hand, by \cite[Lemma~4.5]{AlvCalaza2014}, for $k>1$, we have $(1+L)^k=\sum_aq'_a(B',B)q_a(B,B')$ for some finite family of homogeneous non-commutative polynomials $q_a$ of degree $\le k$. Hence
    \[
      (1+L)^k(f\phi)=f(1+L)^k\phi+\sum_{l=1}^kf^{(l)}\,p_l(B,B')\phi\;,
    \]
  where each $p_l$ is a non-commutative polynomial of degree $\le k-l$, which depend on the non-commutative polynomials $q_a$. Therefore, by Claim~\ref{cl: |p(B,B')phi|_sigma},
    \begin{align*}
      \|(1+L)^k(f\phi)\|_\sigma&\le\|f(1+L)^k\phi\|_\sigma+\sum_{l=1}^k\|f^{(l)}\,p_l(B,B')\phi\|_\sigma\\
      &\le(\max|f|)\,\|(1+L)^k\phi\|_\sigma+\sum_{l=1}^k(\max|f^{(l)}|)\,\|p_l(B,B')\phi\|_\sigma\\
      &\le(\max|f|)\,\|(1+L)^k\phi\|_\sigma+\sum_{l=1}^k(\max|f^{(l)}|)\,c_{p_l}\|(1+L)^l\phi\|_\sigma\;,
    \end{align*}
  which gives Claim~\ref{cl: |L^k(f phi)|_sigma}.
\end{proof}

The existence of $a\in\R$ satisfying~\eqref{a} is characterized by the condition
  \begin{equation} \label{(2c_1-1)^2+4c_2 ge0}
    (2c_1-1)^2+4c_2\ge0\;.
  \end{equation}
Observe that~\eqref{(2c_1-1)^2+4c_2 ge0} is satisfied if $c_2\ge\min\{0,2c_1\}$. 

\begin{lem}\label{l:chi_s,a,sigma,0}
  If $h$ is a bounded measurable function on $\R_+$ with $h(\rho)\to1$ as $\rho\to0$, then $\langle h\chi_{s,a,\sigma,0},\chi_{s,a,\sigma,0}\rangle_{c_1}\to1$ as $s\to\infty$.
\end{lem}

\begin{proof}
  Given any $\epsilon>0$, take some $\rho_0>0$ such that $|h(\rho)-1|\le\epsilon/2$ for $\rho\le\rho_0$. For $s$ large enough, we have
    \[
      \int_{\rho_0}^\infty e^{-s\rho^2}\,\rho^{2\sigma}\,d\rho\le\frac{\epsilon}{4p_0^2\,\max|h-1|}
    \]
  Hence, for $s$ large enough,
    \begin{multline*}
      |\langle(1-h)\chi_{s,a,\sigma,0},\chi_{s,a,\sigma,0}\rangle_{c_1}|
      \le2p_0^2\int_0^\infty|1-h(\rho)|\,e^{-s\rho^2}\,\rho^{2\sigma}\,d\rho\\
      \le p_0^2\epsilon\int_0^{\rho_0}e^{-s\rho^2}\,\rho^{2\sigma}\,d\rho+2p_0^2\,(\max|1-h|)\int_{\rho_0}^\infty e^{-s\rho^2}\,\rho^{2\sigma}\,d\rho\\
      <p_0^2\epsilon\int_0^\infty e^{-s\rho^2}\,\rho^{2\sigma}\,d\rho+\frac{\epsilon}{2}=\frac{\epsilon}{2}\,\|\chi_{s,a,\sigma,0}\|_{c_1}^2+\frac{\epsilon}{2}=\epsilon\;.\qed
    \end{multline*}
\renewcommand{\qed}{}
\end{proof}

\section{Two simple types of elliptic complexes}\label{s:2 simple types of elliptic complexes}

Here, we study two simple elliptic complexes. They will show up in a direct sum splitting of the local model of Witten's perturbation (Section~\ref{s:splitting}). We could describe the spectra of the Laplacians associated to their minimum/maximum i.b.c., but this will be done with the local model of the Witten's perturbation (Section~\ref{s:types}).

\subsection{Some more results on general elliptic complexes}\label{ss:lemmas}

Consider the notation of the beginning of Section~\ref{ss:elliptic complexes}.

\begin{lem}\label{l:general}
  Let $\GG\subset C^\infty(E)\cap L^2(E)$ be a graded linear subspace containing $C^\infty_0(E)$, preserved by $d$ and $\delta$, and such that $\langle du,v\rangle=\langle u,\delta v\rangle$ for all $u,v\in\GG$. Let $d_\GG$, $\delta_\GG$ and $\D_\GG$ denote the restrictions of $d$, $\delta$ and $\D$ to $\GG$. Assume that $\D_{\GG}$ is essentially self-adjoint in $L^2(E)$, and $\GG$ is the smooth core of $\ol{\D_{\GG}}$. Then the following properties hold:
    \begin{itemize}
    
      \item[(i)] If $\GG_r\subset\DD(d_{\text{\rm min},r})$ and $\GG_{r-1}\subset\DD(d_{\text{\rm min},r-1})$ for some degree $r$, then $\GG_r$ is the smooth core of $d_{\text{\rm min},r}$.
      
      \item[(ii)] If $\GG_r\subset\DD(\delta_{\text{\rm min},r-1})$ and $\GG_{r+1}\subset\DD(\delta_{\text{\rm min},r})$ for some degree $r$, then $\GG_r$ is the smooth core of $d_{\text{\rm max},r}$.

    \end{itemize}
\end{lem}

\begin{proof}
  For each degree $r$, the restrictions $d_r:\GG_r\to\GG_{r+1}$, $\delta_r:\GG_{r+1}\to\GG_r$ and $\D_r:\GG_r\to\GG_r$ will be denoted by $d_{\GG,r}$, $\delta_{\GG,r}$ and $\D_{\GG,r}$, respectively. Suppose that $\GG_r\subset\DD(d_{\text{\rm min},r})$ and $\GG_{r-1}\subset\DD(d_{\text{\rm min},r-1})$, and therefore $d_{\GG,r}\subset d_{\text{\rm min},r}$ and $d_{\GG,r-1}\subset d_{\text{\rm min},r-1}$. Since $C^\infty_0(E)\subset\GG$ and $\langle du,v\rangle=\langle u,\delta v\rangle$ for all $u,v\in\GG$, it follows that $\GG_{r+1}\subset\DD(\delta_{\text{\rm max},r})$ and $\GG_r\subset\DD(\delta_{\text{\rm max},r-1})$, and therefore $\delta_{\GG,r}\subset\delta_{\text{\rm max},r}$ and $\delta_{\GG,r-1}\subset\delta_{\text{\rm max},r-1}$. By~\eqref{D_min/max}, we get $\D_{\GG,r}\subset\D_{\text{\rm min},r}$. So $\ol{\D_{\GG,r}}\subset\D_{\text{\rm min},r}$, and therefore $\ol{\D_{\GG,r}}=\D_{\text{\rm min},r}$ because these operators are self-adjoint in $L^2(E_r)$. Then $\GG_r$ is the smooth core of $d_{\text{\rm min},r}$, completing the proof of~(i).
  
  Now, assume that $\GG_r\subset\DD(\delta_{\text{\rm min},r-1})$ and $\GG_{r+1}\subset\DD(\delta_{\text{\rm min},r})$, and therefore $\delta_{\GG,r-1}\subset\delta_{\text{\rm min},r-1}$ and $\delta_{\GG,r}\subset\delta_{\text{\rm min},r}$. As above, it follows that $d_{\GG,r-1}\subset d_{\text{\rm max},r-1}$ and $d_{\GG,r}\subset d_{\text{\rm max},r}$. By~\eqref{D_min/max}, we get $\D_{\GG,r}\subset\D_{\text{\rm max},r}$. So $\ol{\D_{\GG,r}}\subset\D_{\text{\rm max},r}$, obtaining $\ol{\D_{\GG,r}}=\D_{\text{\rm max},r}$ as before.  Thus $\GG_r$ is the smooth core of $d_{\text{\rm max},r}$, completing the proof of~(ii).
\end{proof}

Now, suppose that there is an orthogonal decomposition $E_{r+1}=E_{r+1,1}\oplus E_{r+1,2}$ for some degree $r+1$. Thus
  \begin{align*}
    C^\infty(E_{r+1})&\equiv C^\infty(E_{r+1,1})\oplus C^\infty(E_{r+1,2})\;,\\ 
    C^\infty_0(E_{r+1})&\equiv C^\infty_0(E_{r+1,1})\oplus C^\infty_0(E_{r+1,2})\;,\\
    L^2(E_{r+1})&\equiv L^2(E_{r+1,1})\oplus L^2(E_{r+1,2})\;,
  \end{align*}
and we can write
  \[
    d_r=
      \begin{pmatrix}
        d_{r,1}\\
        d_{r,2}
      \end{pmatrix}\;,\quad
    \delta_r=
      \begin{pmatrix}
        \delta_{r,1} & \delta_{r,2}
      \end{pmatrix}\;.
  \]
For $i\in\{1,2\}$, let $\D_{r,i}=\delta_{r,i}d_{r,i}+d_{r-1}\delta_{r-1}$ on $C^\infty(E_r)$.
  
\begin{lem}\label{l:3}
  We have:
    \[
      \DD(d_{\text{\rm max},r})=\DD(d_{r,1,\text{\rm max}})\cap\DD(d_{r,2,\text{\rm max}})\;,\quad
      d_{\text{\rm max},r}=
        \begin{pmatrix}
          d_{r,1,\text{\rm max}}|_{\DD(d_{\text{\rm max},r})}\\
          d_{r,2,\text{\rm max}}|_{\DD(d_{\text{\rm max},r})}
        \end{pmatrix}\;.
    \]
\end{lem}

\begin{proof}
  Let $u\in L^2(E_r)$. We have $u\in\DD(d_{\text{\rm max},r})$ if and only if there is some $w\in L^2(E_{r+1})$ such that $\langle u,\delta_rv\rangle=\langle w,v\rangle$ for all $v\in C^\infty_0(E_{r+1})$, and moreover $d_{\text{\rm max},r}u=w$ in this case. Writing $w=w_1\oplus w_2$ and $v=v_1\oplus v_2$, this condition on $u$ means that $\langle u,\delta_{r,i}v_i\rangle=\langle w_i,v_i\rangle$ for all $v_i\in C^\infty_0(E^i_{r+1})$ and $i\in\{1,2\}$. In turn, this is equivalent to $u\in\DD(d_{r,1,\text{\rm max}})\cap\DD(d_{r,2,\text{\rm max}})$ with $d_{r,i,\text{\rm max}}u=w_i$.
\end{proof}

\begin{cor}\label{c:3}
  If $a\,\D_r=b\,\D_{r,i}+c$ for some $a,b,c\in\R$ with $a,b\ne0$, $d_{\text{\rm min},r}$ and $d_{r,i,\text{\rm min}}$ have the same smooth core, and $d_{r,i,\text{\rm min}}=d_{r,i,\text{\rm max}}$ for some $i\in\{0,1\}$, then $d_{\text{\rm min},r}=d_{\text{\rm max},r}$.
\end{cor}

\begin{proof}
  By Lemma~\ref{l:3} and since $d_{r,i,\text{\rm min}}=d_{r,i,\text{\rm max}}$, we get $\DD(d_{\text{\rm max},r})\subset\DD(d_{r,i,\text{\rm min}})$. Because $a\,\D_r=b\,\D_{r,i}+c$ for some $a,b,c\in\R$ with $a,b\ne0$, it follows that
    \begin{multline*}
      \{\,u\in\DD(d_{\text{\rm max},r})\cap C^\infty(E_r)\mid\D_r^ku\in L^2(E_r)\ \forall k\in\N\,\}\\
      \subset\{\,u\in\DD(d_{r,i,\text{\rm min}})\cap C^\infty(E_r)\mid \D_{r,i}^ku\in L^2(E_r)\ \forall k\in\N\,\}\;.
    \end{multline*}
  This means that the smooth core of $d_{\text{\rm max},r}$ is contained in the smooth core of $d_{r,i,\text{\rm min}}$, which equals the smooth core of $d_{\text{\rm min},r}$. Then $d_{\text{\rm max},r}=d_{\text{\rm min},r}$.
\end{proof}

\subsection{An elliptic complex of length one}\label{ss:complex 1}

Consider the standard metric on $\R_+$. Let $E$ be the graded Riemannian/Hermitian vector bundle over $\R_+$ whose non-zero terms are $E_0$ and $E_1$, which are real/complex trivial line bundles equipped with the standard Riemannian/Hemitian metrics. Thus
  \[
    C^\infty(E_0)\equiv C^\infty(\R_+)\equiv C^\infty(E_1)\;,\quad L^2(E_0)\equiv L^2(\R_+,d\rho)\equiv L^2(E_1)\;,
  \]
where real/complex valued functions are considered in $C^\infty(\R_+)$ and $L^2(\R_+,d\rho)$. For any fixed $s>0$ and $\kappa\in\R$, let 
\begin{center}
  \begin{picture}(129,26) 
    \put(0,10){$C^\infty(E_0)$}
    \put(94,10){$C^\infty(E_1)$}
    \put(62,19){\Small$d$}
    \put(62,0){\Small$\delta$}
    \put(41,14){\vector(1,0){47}}
    \put(88,11){\vector(-1,0){47}}
  \end{picture}
\end{center}
be the differential operators defined by 
  \[
    d=\frac{d}{d\rho}-\kappa\rho^{-1}\pm s\rho\;,\quad
    \delta=-\frac{d}{d\rho}-\kappa\rho^{-1}\pm s\rho\;.
  \]
It is easy to check that $(E,d)$ is an elliptic complex, whose formal adjoint is $(E,\delta)$. Using~\eqref{[d/d rho,rho^a right]}, we easily get that the homogeneous components of the corresponding Laplacian $\D$ are:
  \begin{gather}
    \D_0=\delta d=H+\kappa(\kappa-1)\rho^{-2}\mp s(1+2\kappa)\;,\label{delta d}\\
    \D_1=d\delta=H+\kappa(\kappa+1)\rho^{-2}\pm s(1-2\kappa)\;,\label{d delta}
  \end{gather}
where $H$ is the harmonic oscillator on $C^\infty(\R_+)$ defined with the constant $s$. Then $\D_0$ and $\D_1$ are of the form of $P$ in~\eqref{P} plus a constant. In these cases, Table~\ref{table: Delta_0, Delta_1} contains the possibilities for $a$ given by~\eqref{a}, the corresponding values of $\sigma$, the condition~\eqref{sigma} expressed in terms of $\kappa$, and the smooth cores of the corresponding self-adjoint operators with discrete spectra in $L^2(\R_+,d\rho)$, given by Proposition~\ref{p:P}.

\begin{table}[h]
\renewcommand{\arraystretch}{1.3}
\begin{tabular}{c|c|c|l|c|}
\cline{2-5}
& $a$ & $\sigma$ & Condition & Smooth core \\
\hline
\multicolumn{1}{|c|}{\multirow{2}{*}{$\D_0$}} & $\kappa$ & $\kappa$ & $\kappa>-1/2$ & $\rho^\kappa\,\SS_{\text{\rm ev},+}$\\
\cline{2-5}
\multicolumn{1}{|c|}{} & $1-\kappa$ & $1-\kappa$ & $\kappa<3/2$ & $\rho^{1-\kappa}\,\SS_{\text{\rm ev},+}$ \\
\hline
\multicolumn{1}{|c|}{\multirow{2}{*}{$\D_1$}} & $1+\kappa$ & $1+\kappa$ & $\kappa>-3/2$ & $\rho^{1+\kappa}\,\SS_{\text{\rm ev},+}$\\
\cline{2-5}
\multicolumn{1}{|c|}{} & $-\kappa$ & $-\kappa$ & $\kappa<1/2$ & $\rho^{-\kappa}\,\SS_{\text{\rm ev},+}$ \\
\hline
\end{tabular}
\caption{Self-adjoint operators defined by $\D_0$ and $\D_1$}
\label{table: Delta_0, Delta_1}
\end{table}

    
    
    
  
    
    
    

Let $\EE_i\subset C^\infty(E)\cap L^2(E)$ ($i\in\{1,2\}$) be the dense graded linear subspace described in Table~\ref{table: EE_1, EE_2}. Observe that, by restricting $d$ and $\delta$, we get complexes $(\EE_1,d)$ and $(\EE_1,\delta)$ when $\kappa>-1/2$, and complexes $(\EE_2,d)$ and $(\EE_2,\delta)$ when $\kappa<1/2$. Thus $\D$ preserves $\EE_1$ when $\kappa>-1/2$, and preserves $\EE_2$ when $\kappa<1/2$.

\begin{table}[h]
\renewcommand{\arraystretch}{1.3}
\begin{tabular}{c|l|c|c|}
\cline{2-4}
& Conditions on $\kappa$ & $\EE_i^0$ & $\EE_i^1$ \\
\hline
\multicolumn{1}{|c|}{$\EE_1$} & $\kappa>-1/2$ & $\rho^\kappa\,\SS_{\text{\rm ev},+}$ & $\rho^{1+\kappa}\,\SS_{\text{\rm ev},+}$ \\
\hline
\multicolumn{1}{|c|}{$\EE_2$} & $\kappa<1/2$ & $\rho^{1-\kappa}\,\SS_{\text{\rm ev},+}$ & $\rho^{-\kappa}\,\SS_{\text{\rm ev},+}$ \\
\hline
\end{tabular}
\caption{$\EE_1$ and $\EE_2$}
\label{table: EE_1, EE_2}
\end{table}


\begin{prop}\label{p:2}
    \begin{itemize}
    
      \item[(i)] If $|\kappa|<1/2$, then $\EE_1$ and $\EE_2$ are the smooth cores of $d_{\text{\rm max}}$ and $d_{\text{\rm min}}$, respectively.
      
      \item[(ii)] If $|\kappa|\ge1/2$, then $(E,d)$ has a unique i.b.c., whose smooth core is $\EE_1$ when $\kappa\ge1/2$, and $\EE_2$ when $\kappa\le-1/2$.
    
    \end{itemize}
\end{prop}


The following lemma will be used in the proof of Proposition~\ref{p:2}.

\begin{lem}\label{l:theta}
  Suppose that either $\theta>1/2$, or $\theta=1/2=\kappa$ \upn{(}respectively, $\theta=1/2=-\kappa$\upn{)}. Then, for each $\xi\in\rho^\theta\,\SS_{\text{\rm ev},+}$, considered as subspace of $C^\infty(E_0)$ \upn{(}respectively, $C^\infty(E_1)$\upn{)}, there is a sequence $(\xi_n)$ in $C^\infty_0(E_0)$ \upn{(}respectively, $C^\infty_0(E_1)$\upn{)}, independent of $\kappa$, such that $\lim_n\xi_n=\xi$ in $L^2(E_0)$ \upn{(}respectively, $L^2(E_1)$\upn{)} and $\lim_nd\xi_n=d\xi$ in $L^2(E_1)$ \upn{(}respectively, $\lim_n\delta\xi_n=\delta\xi$ in $L^2(E_0)$\upn{)}. In particular, $\rho^\theta\,\SS_{\text{\rm ev},+}$ is contained in $\DD(d_{\text{\rm min}})$ \upn{(}respectively, $\DD(\delta_{\text{\rm min}})$\upn{)}.
\end{lem}

\begin{proof}
  The proof is made for $\DD(d_{\text{\rm min}})$; the case of $\DD(\delta_{\text{\rm min}})$ is analogous. 
  
  Let $0<a<b$ and $f\in C^\infty(\R_+)$ such that $0\le f\le1$, $f(\rho)=1$ for $\rho\le a$, and $f(\rho)=0$ for $\rho\ge b$. For each $n\in\N$, let $g_n,h_n\in C^\infty(\R_+)$ be defined by $g_n(\rho)=f(n\rho)$ and $h_n(\rho)=f(\rho/n)$. It is clear that
    \begin{equation}\label{(1-g_n)h_n}
      \chi_{[\frac{b}{n},na]}\le(1-g_n)h_n\le\chi_{[\frac{a}{n},nb]}\;,
    \end{equation}
  where $\chi_S$ denotes the characteristic function of each subset $S\subset\R_+$. 
  
  Let $\phi\in\SS_{\text{\rm ev},+}$. From~\eqref{(1-g_n)h_n}, we get $(1-g_n)h_n\rho^\theta\phi\in C^\infty_0(E_0)$ and $(1-g_n)h_n\rho^\theta\phi\to\rho^\theta\phi$ in $L^2(E_0)$ as $n\to\infty$. Note that the conditions on $\theta$ and $\kappa$ guarantee that $d(\rho^\theta\phi)\in L^2(E_1)$. Observe that
    \[
      d((1-g_n)h_n\rho^\theta\phi)=-g_n'h_n\rho^\theta\phi+(1-g_n)h_n'\rho^\theta\phi+(1-g_n)h_n\,d(\rho^\theta\phi)\;.
    \]
  In the right hand side of this equality, the last term converges to $d(\rho^\theta\phi)$ in $L^2(E_1)$ as $n\to\infty$ by~\eqref{(1-g_n)h_n}. Moreover
    \begin{multline*}
      \|(1-g_n)h_n'\rho^\theta\phi\|^2=\int_0^\infty(1-g_n)^2{h_n'}^2(\rho)\rho^{2\theta}\phi^2(\rho)\,d\rho\\
      \le(\max\rho^{2\theta}\phi^2)\,n^{-2}\int_0^\infty{f'}^2(\rho/n)\,d\rho=(\max\rho^{2\theta}\phi^2)\,n^{-1}\int_0^\infty{f'}^2(x)\,dx\\
      =(\max\rho^{2\theta}\phi^2)\,n^{-1}\,\|f'\|^2\;,
    \end{multline*}
  which converges to zero as $n\to\infty$, and
    \begin{multline*}
      \|g_n'h_n\rho^\theta\phi\|^2=\int_0^\infty{g_n'}^2(\rho)h_n^2(\rho)\rho^{2\theta}\phi^2(\rho)\,d\rho
      \le(\max\phi^2)\,n^2\int_0^\infty{f'}^2(n\rho)\rho^{2\theta}\,d\rho\\
      =(\max\phi^2)\,n^{1-2\theta}\int_0^\infty{f'}^2(x)x^{2\theta}\,dx=(\max\phi^2)\,n^{1-2\theta}\,\|f'\rho^\theta\|^2\;,
    \end{multline*}
  which converges to zero as $n\to\infty$ if $\theta>1/2$. 
  
  In the case $\theta=1/2$, it is enough to prove that $f$ can be chosen so that $\|f'\rho^{1/2}\|$ is as small as desired. For $m>1$ and $0<\epsilon<1$, observe that there is some $f$ as above such that:
    \begin{itemize}
    
      \item the support of $f'$ is contained in $[e^{-\epsilon},e^m]$,
      
      \item $-\frac{1}{m\rho}\le f'\le0$, and
      
      \item $f'(\rho)=-\frac{1}{m\rho}$ if $1\le\rho\le e^{m-\epsilon}$.
    
    \end{itemize}
  Then
    \[
      \|f'\rho^{1/2}\|^2=\int_{e^{-\epsilon}}^{e^m}{f'}^2(\rho)\rho\,d\rho\le\frac{1}{m^2}\int_{e^{-\epsilon}}^{e^m}\frac{d\rho}{\rho}=\frac{m+\epsilon}{m^2}\;,
    \]
  which converges to zero as $m\to\infty$.
\end{proof}

\begin{proof}[Proof of Proposition~\ref{p:2}]
  Suppose that $|\kappa|<1/2$. Since $1\pm\kappa>1/2$, by Lemma~\ref{l:theta}, $\EE_2^0\subset\DD(d_{\text{\rm min}})$ and $\EE_1^1\subset\DD(\delta_{\text{\rm min}})$. The other conditions of Lemma~\ref{l:general} are satisfied by $d$ with $\GG=\EE_2$, and by $\delta$ with $\GG=\EE_1$ by the discussion previous to Proposition~\ref{p:2}. So $\EE_2$ is the smooth core of $d_{\text{\rm min}}$ and $\EE_1$ is the smooth core of $d_{\text{\rm max}}$ by Lemma~\ref{l:general}. 
    
  Now, assume that $\kappa\ge1/2$ (respectively, $\kappa\le-1/2$), yielding also $1+\kappa>1/2$ (respectively, $1-\kappa>1/2$). Then, by Lemma~\ref{l:theta}, $\EE_1^0\subset\DD(d_{\text{\rm min}})$ and $\EE_1^1\subset\DD(\delta_{\text{\rm min}})$ (respectively, $\EE_2^0\subset\DD(d_{\text{\rm min}})$ and $\EE_2^1\subset\DD(\delta_{\text{\rm min}})$). By the discussion previous to Proposition~\ref{p:2}, the other conditions of Lemma~\ref{l:general} are satisfied by $d$ and $\delta$ with $\GG=\EE_1$ (respectively, $\GG=\EE_2$). So, by Lemma~\ref{l:general}, $\EE_1$ (respectively, $\EE_2$) is the smooth core of $d_{\text{\rm min}}$ and $d_{\text{\rm max}}$.
\end{proof}

\subsection{An elliptic complex of length two}\label{ss:complex 2}

Consider again the standard metric on $\R_+$. Let $F$ be the graded Riemannian/Hermitian vector bundle over $\R_+$ whose non-zero terms are $F_0$, $F_1$ and $F_2$, which are trivial real/complex vector bundles of ranks $1$, $2$ and $1$, respectively, equipped with the standard Riemannian/Hermitian metrics. Thus
  \begin{gather*}
    C^\infty(F_0)\equiv C^\infty(\R_+)\equiv C^\infty(F_2)\;,\quad C^\infty(F_1)\equiv C^\infty(\R_+)\oplus C^\infty(\R_+)\;,\\
    L^2(F_0)\equiv L^2(\R_+,d\rho)\equiv L^2(F_2)\;,\quad L^2(F_1)\equiv L^2(\R_+,d\rho)\oplus L^2(\R_+,d\rho)\;,
  \end{gather*}
where real/complex valued functions are considered in $C^\infty(\R_+)$ and $L^2(\R_+,d\rho)$. Fix $s,c>0$ and $\kappa\in\R$, and let
\begin{center}
  \begin{picture}(286,46) 
    \put(0,21){$C^\infty(F_0)$}
    \put(124,21){$C^\infty(F_1)$}
    \put(248,21){$C^\infty(F_2)\;,$}
    \put(56,35){\Small$d_0\equiv
      \begin{pmatrix}
        d_{0,1}\\
        d_{0,2}
      \end{pmatrix}$}
    \put(48,8){\Small$\delta_0\equiv
      \begin{pmatrix}
        \delta_{0,1} &
        \delta_{0,2}
      \end{pmatrix}$}
    \put(171,35){\Small$d_1\equiv
      \begin{pmatrix}
        d_{1,1} &
        d_{1,2}
      \end{pmatrix}$}
    \put(183,8){\Small$\delta_1\equiv
      \begin{pmatrix}
        \delta_{1,1}\\
        \delta_{1,2}
      \end{pmatrix}$}
    \put(41,25){\vector(1,0){77}}
    \put(118,22){\vector(-1,0){77}}
    \put(166,25){\vector(1,0){77}}
    \put(243,22){\vector(-1,0){77}}
  \end{picture}
\end{center}
be the differential operators with $\delta_0={}^td_0$ and $\delta_1={}^td_1$ (thus $\delta_{i,j}={}^td_{i,j}$), and
  \begin{align*}
    d_{0,1}&=\frac{c}{\sqrt{1+c^2}}\left(\frac{d}{d\rho}+\kappa\rho^{-1}\pm s\rho\right)\;,\\
    d_{0,2}&=\frac{1}{\sqrt{1+c^2}}\left(\frac{d}{d\rho}-(\kappa+1)\rho^{-1}\pm s\rho\right)\;,\\
    d_{1,1}&=\frac{1}{\sqrt{1+c^2}}\left(\frac{d}{d\rho}-\kappa\rho^{-1}\pm s\rho\right)\;,\\
    d_{1,2}&=\frac{c}{\sqrt{1+c^2}}\left(-\frac{d}{d\rho}-(\kappa+1)\rho^{-1}\mp s\rho\right)\;.
  \end{align*}
A direct computation shows that $d_0$ and $d_1$ define an elliptic complex $(F,d)$ of length two. By~\eqref{delta d},~\eqref{d delta} and~\eqref{[d/d rho,rho^a right]}, the homogeneous components of the corresponding Laplacian $\D$ are given by
  	\begin{align*}
    		\D_0&=\delta_{0,1}d_{0,1}+\delta_{0,2}d_{0,2}
    		=H+\kappa(\kappa+1)\rho^{-2}\mp s\left(2+\frac{1-c^2}{1+c^2}(1+2\kappa)\right)\;,\\
    		\D_2&=d_{1,1}\delta_{1,1}+d_{1,2}\delta_{1,2}=H+\kappa(\kappa+1)\rho^{-2}\pm s\left(2+\frac{1-c^2}{1+c^2}(1+2\kappa)\right)\;,\\
    		\D_1&=
      			\begin{pmatrix}
        				d_{0,1}\delta_{0,1}+\delta_{1,1}d_{1,1} & d_{0,1}\delta_{0,2}+\delta_{1,1}d_{1,2}\\
        				d_{0,2}\delta_{0,1}+\delta_{1,2}d_{1,1} & d_{0,2}\delta_{0,2}+\delta_{1,2}d_{1,2}
      			\end{pmatrix}\\
      		&=
      			\begin{pmatrix}
        				\D_{1,1} & 0\\
        				0 & \D_{1,2}
      			\end{pmatrix}
			=\D_{1,1}\oplus\D_{1,2}\;,\\
      		\D_{1,1}&=H+\kappa(\kappa-1)\rho^{-2}\mp s\frac{1-c^2}{1+c^2}(1+2\kappa)\;,\\
    		\D_{1,2}&=H+(\kappa+1)(\kappa+2)\rho^{-2}\mp s\frac{1-c^2}{1+c^2}(1+2\kappa)\;.
  	\end{align*}
Thus $\D_0$, $\D_2$, $\D_{1,1}$ and $\D_{1,2}$ are of the form of $P$ in~\eqref{P} plus a constant. In these cases, Table~\ref{table: Delta_0, Delta_1, Delta_1,1, Delta_1,2} contains the possibilities for $a$ given by~\eqref{a}, the corresponding values of $\sigma$, the condition~\eqref{sigma} expressed in terms of $\kappa$, and the smooth cores of the corresponding self-adjoint operators with discrete spectra in $L^2(\R_+,d\rho)$, given by Proposition~\ref{p:P}.

\begin{table}[h]
\renewcommand{\arraystretch}{1.3}
\begin{tabular}{c|c|c|l|c|}
\cline{2-5}
& $a$ & $\sigma$ & Condition & Smooth core \\
\hline
\multicolumn{1}{|c|}{\multirow{2}{*}{$\D_0$ and $\D_2$}} & $1+\kappa$ & $1+\kappa$ & $\kappa>-3/2$ & $\rho^{1+\kappa}\,\SS_{\text{\rm ev},+}$\\
\cline{2-5}
\multicolumn{1}{|c|}{} & $-\kappa$ & $-\kappa$ & $\kappa<3/2$ & $\rho^{-\kappa}\,\SS_{\text{\rm ev},+}$ \\
\hline
\multicolumn{1}{|c|}{\multirow{2}{*}{$\D_{1,1}$}} & $\kappa$ & $\kappa$ & $\kappa>-1/2$ & $\rho^\kappa\,\SS_{\text{\rm ev},+}$\\
\cline{2-5}
\multicolumn{1}{|c|}{} & $1-\kappa$ & $1-\kappa$ & $\kappa<3/2$ & $\rho^{1-\kappa}\,\SS_{\text{\rm ev},+}$ \\
\hline
\multicolumn{1}{|c|}{\multirow{2}{*}{$\D_{1,2}$}} & $2+\kappa$ & $2+\kappa$ & $\kappa>-5/2$ & $\rho^{2+\kappa}\,\SS_{\text{\rm ev},+}$\\
\cline{2-5}
\multicolumn{1}{|c|}{} & $-1-\kappa$ & $-1-\kappa$ & $\kappa<-1/2$ & $\rho^{-1-\kappa}\,\SS_{\text{\rm ev},+}$ \\
\hline
\end{tabular}
\caption{Self-adjoint operators defined by $\D_0$, $\D_2$, $\D_{1,1}$ and $\D_{1,2}$}
\label{table: Delta_0, Delta_1, Delta_1,1, Delta_1,2}
\end{table}

Let $\FF_i\subset C^\infty(F)\cap L^2(F)$ ($i\in\{1,2\}$) be the dense graded linear subspace described in Table~\ref{table: EE_1, EE_2}. By restricting $d$ and $\delta$, we get complexes $(\FF_1,d)$ and $(\FF_1,\delta)$ when $\kappa>-1/2$, and complexes $(\FF_2,d)$ and $(\FF_2,\delta)$ when $\kappa<-1/2$. Thus $\D$ preserves $\FF_1$ when $\kappa>-1/2$, and preserves $\FF_2$ when $\kappa<-1/2$. 

\begin{table}[h]
\renewcommand{\arraystretch}{1.3}
\begin{tabular}{c|l|c|c|c|}
\cline{2-5}
& Condition & $\FF_i^0$ & $\FF_i^1$ & $\FF_i^2$\\
\hline
\multicolumn{1}{|c|}{$\FF_1$} & $\kappa>-1/2$ & $\rho^{1+\kappa}\,\SS_{\text{\rm ev},+}$ & $\rho^\kappa\,\SS_{\text{\rm ev},+}\oplus\rho^{2+\kappa}\,\SS_{\text{\rm ev},+}$ & $\rho^{1+\kappa}\,\SS_{\text{\rm ev},+}$ \\
\hline
\multicolumn{1}{|c|}{$\FF_2$} & $\kappa<-1/2$ & $\rho^{-\kappa}\,\SS_{\text{\rm ev},+}$ & $\rho^{1-\kappa}\,\SS_{\text{\rm ev},+}\oplus\rho^{-1-\kappa}\,\SS_{\text{\rm ev},+}$ & $\rho^{1+\kappa}\,\SS_{\text{\rm ev},+}$ \\
\hline
\end{tabular}
\caption{$\FF_1$ and $\FF_2$}
\label{table: FF_1, FF_2}
\end{table}

\begin{prop}\label{p:3}
  Suppose that $\kappa\ne-1/2$. Then $(F,d)$ has a unique i.b.c., whose smooth core is $\FF_1$ if $\kappa>-1/2$, and $\FF_2$ if $\kappa<-1/2$.
\end{prop}

\begin{proof}
  We prove only the case with $\kappa>-1/2$; the other case is analogous. 
  
  By Lemma~\ref{l:theta} (using the independence of $(\xi_n)$ on $\kappa$ in its statement), we get $\FF_1^0\subset\DD(d_{0,\text{\rm min}})$ and $\FF_1^2\subset\DD(\delta_{1,\text{\rm min}})$. Then, by the discussion previous to this proposition, the other conditions of Lemma~\ref{l:general} are satisfied by the complexes defined by $d$ and $\delta$ with $\GG=\FF_1$, obtaining that $\FF_1^0$ and $\FF_1^2$ are the smooth cores of $d_{0,\text{\rm min}}$ and $\delta_{1,\text{\rm min}}$, respectively. By Proposition~\ref{p:2} and since $1+\kappa,2+\kappa>1/2$, we get $d_{0,2,\text{\rm min}}=d_{0,2,\text{\rm max}}$ with smooth core $\FF_1^0$, and $\delta_{1,2,\text{\rm min}}=\delta_{1,2,\text{\rm max}}$ with smooth core $\FF_1^2$. So, according to the discussion previous to this proposition, the conditions of Corollary~\ref{c:3} are satisfied with $d$ and $\delta$, obtaining $d_{0,\text{\rm min}}=d_{0,\text{\rm max}}$ and $\delta_{1,\text{\rm min}}=\delta_{1,\text{\rm max}}$, which also gives $d_{1,\text{\rm min}}=d_{1,\text{\rm max}}$.
\end{proof}


\subsection{Finite propagation speed of the wave equation}
\label{s:finite propagation speed, simple}

For the Hermitian bundle versions of $E$ and $F$, consider the wave equation
  \begin{equation}\label{wave, simple}
    \frac{du_t}{dt}-iDu_t=0
  \end{equation}
on any open subset of $\R_+$, where $D=d+\delta$ and $u_t$ is in $\Cinf(E)$ or $\Cinf(F)$, depending smoothly on $t\in\R$.

\begin{prop}\label{p:finite propagation speed, simple}
  For $0<a<b$, let $u_t\in\DD^\infty(d_{\text{\rm min/max}})$, depending smoothly on $t\in\R$. The following properties hold:
    \begin{itemize}
    
      \item[(i)] If $u_t$ satisfies~\eqref{wave, simple} on $(0,b)$ and $\supp u_0\subset[a,\infty)$, then $\supp u_t\subset[a-|t|,\infty)$ for $0<|t|\le a$.
      
      \item[(ii)] If $u_t$ satisfies~\eqref{wave, simple} on $[a,\infty)$ and $\supp u_0\subset(0,a]$, then $\supp u_t\subset(0,a+|t|]$ for $0<|t|\le b-a$.
    
    \end{itemize}
\end{prop}

\begin{proof}
  We prove Proposition~\ref{p:finite propagation speed, simple} only for $E$; the proof is clearly analogous for $F$, but with more cases because $F$ is of length two. Let $u_{t,0}\in\Cinf(E_0)\equiv\Cinf(\R_+)$ and $u_{t,1}\in\Cinf(E_1)\equiv\Cinf(\R_+)$ be the homogeneous components of $u_t$. From the description of the smooth core of $d_{\text{\rm min/max}}$ in Proposition~\ref{p:2}, it follows that
    \begin{equation}\label{lim_rho to0 (u_t,0 u_t,1)(rho)=0}
      \lim_{\rho\downarrow0}(u_{t,0}\,u_{t,1})(\rho)=0\;.
    \end{equation}
  We have
    \begin{align*}
      \frac{d}{dt}\int_0^{a-t}|u_t(\rho)|^2\,d\rho&=\int_0^{a-t}((iDu_t,u_t)+(u_t,iDu_t))(\rho)\,d\rho-|u_t(a-t)|^2\\
      &=i\int_0^{a-t}((Du_t,u_t)-(u_t,Du_t))(\rho)\,d\rho-|u_t(a-t)|^2\;.
    \end{align*}
  But, since $d$ and $\delta$ are respectively equal to $d/d\rho$ and $-d/d\rho$ up to the sum of multiplication operators by the same real valued functions,
    \begin{multline*}
      (Du_t,u_t)-(u_t,Du_t)=\frac{du_{t,0}}{d\rho}\cdot\overline{u_{t,1}}-\frac{du_{t,1}}{d\rho}\cdot\overline{u_{t,0}}-u_{t,1}\cdot\frac{d\overline{u_{t,0}}}{d\rho}+u_{t,0}\cdot\frac{d\overline{u_{t,1}}}{d\rho}\\
      =2\,\Im\left(\frac{d u_{t,0}}{d\rho}\cdot\overline{u_{t,1}}+u_{t,0}\cdot\frac{d\overline{u_{t,1}}}{d\rho}\right)
      =2\,\Im\frac{d}{d\rho}(u_{t,0}\,\overline{u_{t,1}})\;,
    \end{multline*}
  giving
    \begin{multline*}
      \left|\int_0^{a-t}((Du_t,u_t)-(u_t,Du_t))(\rho)\,d\rho\right|\le2\left|(u_{t,0}\,\ol{u_{t,1}})(a-t)-\lim_{\rho\downarrow0}(u_{t,0}\,\ol{u_{t,1}})(\rho)\right|\\
      =2\,|(u_{t,0}\,\ol{u_{t,1}})(a-t)|\le|u_{t,0}(a-t)|^2+|u_{t,1}(a-t)|^2=|u_t(a-t)|^2
    \end{multline*}
  by~\eqref{lim_rho to0 (u_t,0 u_t,1)(rho)=0}. So
    \[
      \frac{d}{dt}\int_0^{a-t}|u_t(\rho)|^2\,d\rho\le0\;,
    \]
  giving
    \[
      \int_0^{a-t}|u_t(\rho)|^2\,d\rho\le\int_0^a|u_0(\rho)|^2\,d\rho=0\;,
    \]
  and~(i) follows.
  
  Property~(ii) can be proved with the same kind of arguments, but using that $\lim_{\rho\to\infty}u(\rho)=0$ for all $u\in\DD^\infty(d_{\text{\rm min/max}})$ instead of~\eqref{lim_rho to0 (u_t,0 u_t,1)(rho)=0}.
\end{proof}

\section{Preliminaries on Witten's perturbation of the de~Rham complex}\label{s:Witten}

Let $M\equiv(M,g)$ be a Riemannian $n$-manifold. For all $x\in M$ and $\alpha\in T_xM^*$, let
  \[
    \alpha\lrcorner=(-1)^{nr+n+1}\,\star\,\alpha\!\wedge\,\star\quad\text{on}\quad\bigwedge^rT_xM^*\;,
  \]
involving the Hodge star operator $\star$ on $\bigwedge T_xM^*$ defined by any choice of orientation of $T_xM$. Writing $\alpha=g(X,\cdot)$ for $X\in T_xM$, we have $\alpha\lrcorner=-\iota_X$, where $\iota_X$ denotes the inner product by $X$. Moreover let
  \[
    R_\alpha=\alpha\!\wedge\text{}-\alpha\lrcorner\;,\quad L_\alpha=\alpha\!\wedge\text{}+\alpha\lrcorner
  \]
on $\bigwedge T_xM^*$. Recall that there is an isomorphism between the underlying linear spaces of the exterior and Clifford algebras of $T_xM^*$,
  \[
    \bigwedge T_xM^*\to\Cl(T_xM^*)\;,\quad e_{i_1}\wedge\dots\wedge e_{i_r}\mapsto e_{i_1}\bullet\dots\bullet e_{i_r}\;,
  \]
where $(e_1,\dots,e_n)$ is an orthonormal frame of $T_xM^*$ and ``$\bullet$'' denotes Clifford multiplication. By this linear isomorphism, $L_\alpha$ and $R_\alpha\bw$ correspond to left and right Clifford multiplication by $\alpha$. So $L_\alpha$ and $R_\beta$ anticommute for any $\alpha,\beta\in T_xM^*$. Any symmetric bilinear form $H\in T_xM^*\otimes T_xM^*$ induces an endomorphism $\bH$ of $\bigwedge T_xM^*$ defined by
  \begin{equation}\label{bH}
    \bH=\sum_{i,j=1}^nH(e_i,e_j)\,L_{e_i}\,R_{e_j}\;,
  \end{equation}
by using an orthonormal frame $(e_1,\dots,e_n)$ of $T_xM^*$. Observe that $|\bH|=|H|$.

On the graded algebra of differential forms, $\Omega(M)$, let $d$ and $\delta$ be the de~Rham derivative and coderivative, let $D=d+\delta$ (the de~Rham operator), and let $\D=D^2=d\delta+\delta d$ (the Laplacian on differential forms). For any $f\in\Cinf(M)$, E.~Witten \cite{Witten1982} has introduced the following perturbations of the above operators, depending on a parameter $s\ge0$:
  \begin{gather}
    d_s=e^{-sf}\,d\,e^{sf}=d+s\,df\!\wedge\;,\label{d_s}\\
    \delta_s=e^{sf}\,\delta\,e^{-sf}=\delta-s\,df\lrcorner\;,\label{delta_s}\\
    D_s=d_s+\delta_s=D+sR\;,\notag\\
    \D_s=D_s^2=d_s\delta_s+\delta_sd_s=\D+s(RD+DR)+s^2R^2\;,\label{Delta_s}
  \end{gather}
where $R=R_{df}$. Notice that $\delta_s={}^td_s$; thus $D_s$ and $\D_s$ are formally self-adjoint. 

Let $\boldsymbol{\Hess}f$ be the endomorphism of $\bigwedge TM^*$ induced by $\Hess f$ according to~\eqref{bH}. Then $RD+DR=\boldsymbol{\Hess}f$ and $R^2=|df|^2$ \cite[Lemma~9.17]{Roe1998}. So~\eqref{Delta_s} becomes
  \begin{equation}\label{Delta_s with Hess f and df}
    \D_s=\D+s\,\boldsymbol{\Hess}f+s^2\,|df|^2\;.
  \end{equation}
  
The Witten's perturbed operators also make sense with complex valued differential forms, and the above equalities hold as well.

\begin{ex}\label{ex:d^pm_0,s}
  Let $d^\pm_{0,s}$, $\delta^\pm_{0,s}$, $D^\pm_{0,s}$, $\D^\pm_{0,s}$ denote the Witten's perturbed operators on $\Omega(\R^m)$ defined by the model Morse function $\pm\frac{1}{2}\,\rho_0^2$ and the standard metric $g_0$. According to \cite[Proposition~9.18 and the proof of Lemma~14.11]{Roe1998}, $\D^\pm_{0,s}$, with domain $\Omega_0(\R^m)$, is essentially self-adjoint in $L^2\Omega(\R^m,g_0)$, and its self-adjoint extension has a discrete spectrum of the following form:
    \begin{itemize}
    
      \item $0$ is an eigenvalue of multiplicity one, and the corresponding eigenforms are of degree zero in the case of  $\D^+_{0,s}$, and of degree $m$ in the case of $\D^-_{0,s}$.
      
      \item Let $e_s^\pm$ be a $0$-eigenform of $\D^\pm_{0,s}$ with norm one, and let $h$ be a bounded measurable function on $\R^m$ such that $h(x)\to1$ as $x\to0$. Then $\langle he_s^\pm,e_s^\pm\rangle\to1$ as $s\to\infty$.
      
      \item All non-zero eigenvalues, as functions of $s$, are in $O(s)$ as $s\to\infty$.
      
    \end{itemize}
  Therefore $(\bigwedge T{\R^m}^*,d^\pm_{0,s})$ has a unique i.b.c., which is discrete.
\end{ex}

\section{Witten's perturbation on a cone}\label{s:Witten cone}

For our version of Morse functions, the local analysis of the Witten's perturbed Laplacian will be reduced to the case of the functions $\pm\frac{1}{2}\rho^2$ on a stratum of a cone with a model adapted metric, where $\rho$ denotes the radial function. This kind of local analysis begins in this section.

\subsection{Laplacian on a cone}\label{ss:Laplacian cone}

Let $L$ be a non-empty compact Thom-Mather stratification, let $\rho$ be the radial function on $c(L)$, let $N$ be a stratum of $L$ of dimension $\tilde n$, let $M=N\times\R_+$ be the corresponding stratum of $c(L)$ with dimension $n=\tilde n+1$, and let $\pi:M\to N$ denote the first factor projection. From $\bigwedge TM^*=\bigwedge TN^*\boxtimes\bigwedge T\R_+^*$, we get a canonical identity
  \begin{equation}\label{bigwedge^rTM^*}
    \bigwedge^rTM^*\equiv\pi^*\bigwedge^rTN^*\oplus d\rho\wedge\pi^*\bigwedge^{r-1}TN^*
    \equiv\pi^*\bigwedge^rTN^*\oplus\pi^*\bigwedge^{r-1}TN^*
  \end{equation}
for each degree $r$, obtaining
  \begin{align}
    \Omega^r(M)&\equiv\Cinf(\R_+,\Omega^r(N))\oplus d\rho\wedge\Cinf(\R_+,\Omega^{r-1}(N))
    \label{Omega^r(M) with d rho}\\
    &\equiv\Cinf(\R_+,\Omega^r(N))\oplus\Cinf(\R_+,\Omega^{r-1}(N))\;.
    \label{Omega^r(M)}
  \end{align}
Here, smooth functions $\R_+\to\Omega(N)$ are defined by considering $\Omega(N)$ as Fr\'echet space with the weak $C^\infty$ topology. Let $d$ and $\tilde d$ denote the exterior derivatives on $\Omega(M)$ and $\Omega(N)$, respectively. The following lemma is elementary. 

\begin{lem}\label{l:d}
  According to~\eqref{Omega^r(M)},
    \[
      d\equiv
        \begin{pmatrix}
          \tilde d & 0 \\
          \frac{d}{d\rho} & -\tilde d
        \end{pmatrix}\;.
    \]
\end{lem}

Fix an adapted metric $\tilde g$ on $N$, and let $g=\rho^2\tilde g+(d\rho)^2$ be the corresponding adapted metric on $M$. The induced metrics on $\bigwedge TM^*$ and $\bigwedge TN^*$ are also denoted by $g$ and $\tilde g$, respectively. According to~\eqref{bigwedge^rTM^*}, on $\bigwedge^rTM^*$,
  \begin{equation}\label{g}
    g=\rho^{-2r}\,\tilde g\oplus\rho^{-2(r-1)}\,\tilde g\;.
  \end{equation}

Given an orientation on an open subset $W\subset N$, and denoting by $\tilde\omega$ the corresponding $\tilde g$-volume form on $W$, consider the orientation on $W\times\R_+\subset M$ so that the corresponding $g$-volume form is
  \begin{equation}\label{omega}
    \omega=\rho^{n-1}\,d\rho\wedge\tilde\omega\;.
  \end{equation}
The corresponding star operators on $\bigwedge T(W\times\R_+)^*$ and $\bigwedge TW^*$ will be denoted by $\star$ and $\tilde\star$, respectively.

\begin{lem}\label{star}
  According to~\eqref{bigwedge^rTM^*}, on $\bigwedge^rT(W\times\R_+)^*$,
    \[
      \star\equiv
        \begin{pmatrix}
          0 & \rho^{n-2r+1}\tilde\star \\
          (-1)^r\rho^{n-2r-1}\tilde\star & 0
        \end{pmatrix}\;.
    \]
\end{lem}

\begin{proof}
  Let $\alpha,\alpha'\in\pi^*\bigwedge TN^*$, at the same point $(x,\rho)\in W\times\R_+$. If $\alpha$ and $\alpha'$ are of degree $r$, then
    \begin{multline*}
      \alpha'\wedge\rho^{n-2r-1}\,d\rho\wedge\tilde\star\alpha=(-1)^r\rho^{n-2r-1}\,d\rho\wedge\alpha'\wedge\tilde\star\alpha\\
      =(-1)^r\rho^{n-2r-1}\tilde g(\alpha',\alpha)\,d\rho\wedge\tilde\omega=(-1)^rg(\alpha',\alpha)\,\omega
    \end{multline*}
  by~\eqref{g} and~\eqref{omega}, giving $\star\alpha=(-1)^r\rho^{n-2r-1}d\rho\wedge\tilde\star\alpha$. Similarly, if $\alpha$ and $\alpha'$ are of degree $r-1$, then
    \[
      d\rho\wedge\alpha'\wedge\rho^{n-2r+1}\tilde\star\alpha=\rho^{n-2r+1}\tilde g(\alpha',\alpha)\,d\rho\wedge\tilde\omega=g(d\rho\wedge\alpha',d\rho\wedge\alpha)\,\omega\;,
    \]
   obtaining $\star(d\rho\wedge\alpha)=\rho^{n-2r+1}\tilde\star\alpha$.
\end{proof}

Let $L^2\Omega^r(M,g)$ and $L^2\Omega^r(N,\tilde g)$ be simply denoted by $L^2\Omega^r(M)$ and $L^2\Omega^r(N)$. From~\eqref{g} and~\eqref{omega}, it follows that~\eqref{Omega^r(M)} induces a unitary isomorphism
  \begin{multline}\label{L^2Omega^r(M)}
    L^2\Omega^r(M)\cong(L^2(\R_+,\rho^{n-2r-1}\,d\rho)\,\widehat{\otimes}\,L^2\Omega^r(N))\\
    \text{}\oplus(L^2(\R_+,\rho^{n-2r+1}\,d\rho)\,\widehat{\otimes}\,L^2\Omega^{r-1}(N))\;,
  \end{multline}
which will be considered as an identity.

Let $\delta$ and $\tilde\delta$ denote the exterior coderivatives on $\Omega(M)$ and $\Omega(N)$, respectively. 

\begin{lem}\label{l:delta}
  According to~\eqref{Omega^r(M)}, on $\Omega^r(M)$,
    \[
      \delta\equiv
        \begin{pmatrix}
          \rho^{-2}\,\tilde\delta & -\frac{d}{d\rho}-(n-2r+1)\rho^{-1} \\
          0 & -\rho^{-2}\,\tilde\delta
        \end{pmatrix}\;.
    \]
\end{lem}

\begin{proof}
  For an oriented open subset $W\subset N$, consider the orientation on $W\times\R_+$ defined as above, and let $\star$ and $\tilde\star$ denote the corresponding star operators on $\bigwedge T(W\times\R_+)^*$ and $\bigwedge TW^*$. By Lemmas~\ref{l:d} and~\ref{star}, on $\Omega^r(W\times\R_+)$,
    \begin{align*}
      \delta&=(-1)^{nr+n+1}\star d\star\\
      &\equiv(-1)^{nr+n+1}
        \begin{pmatrix}
          0 & \rho^{-n+2r-1}\tilde\star \\
          (-1)^{n-r+1}\rho^{-n+2r-3}\tilde\star & 0
        \end{pmatrix}
        \begin{pmatrix}
          \tilde d & 0 \\
          \frac{d}{d\rho} & -\tilde d
        \end{pmatrix}\\
      &\phantom{\equiv\text{}}\text{}\times
        \begin{pmatrix}
          0 & \rho^{n-2r+1}\tilde\star \\
          (-1)^r\rho^{n-2r-1}\tilde\star & 0
        \end{pmatrix}\\
      &=(-1)^{nr+n+1}
        \begin{pmatrix}
          -(-1)^r\rho^{-2}\tilde\star\tilde d\tilde\star & \rho^{-n+2r-1}\,\frac{d}{d\rho}\,\rho^{n-2r+1}\tilde\star^2\\
          0 & (-1)^{n-r+1}\rho^{-2}\tilde\star\tilde d\tilde\star
        \end{pmatrix}\\
      &=
        \begin{pmatrix}
          \rho^{-2}\tilde\delta & -\rho^{-n+2r-1}\,\frac{d}{d\rho}\,\rho^{n-2r+1}\\
          0 & -\rho^{-2}\tilde\delta
        \end{pmatrix}\;,
    \end{align*}
  which equals the matrix of the statement by~\eqref{[d/d rho,rho^a right]}.
\end{proof}

Let $\D$ and $\widetilde{\D}$ denote the Laplacians on $\Omega(M)$ and $\Omega(N)$, respectively.

\begin{cor}\label{c:D}
  According to~\eqref{Omega^r(M)},
    \[
      \D\equiv
        \begin{pmatrix}
          P & -2\rho^{-1}\,\tilde d \\
          -2\rho^{-3}\,\tilde\delta & Q
        \end{pmatrix}
      \]
  on $\Omega^r(M)$, where
  \begin{align*}
    P&=\rho^{-2}\,\widetilde{\D}-\frac{d^2}{d\rho^2}-(n-2r-1)\rho^{-1}\,\frac{d}{d\rho}\;,\\
    Q&=\rho^{-2}\,\widetilde{\D}-\frac{d^2}{d\rho^2}-(n-2r+1)\,\frac{d}{d\rho}\,\rho^{-1}\;.
  \end{align*}
\end{cor}

\begin{proof}
  By Lemmas~\ref{l:d} and~\ref{l:delta}, and~\eqref{[d/d rho,rho^a right]},
    \begin{align*}
      \delta d&\equiv
          \begin{pmatrix}
            \rho^{-2}\,\tilde\delta & -\frac{d}{d\rho}-(n-2r-1)\rho^{-1} \\
            0 & -\rho^{-2}\,\tilde\delta
          \end{pmatrix}
          \begin{pmatrix}
            \tilde d & 0 \\
            \frac{d}{d\rho} & -\tilde d
          \end{pmatrix}\\
        &=
          \begin{pmatrix}
            \rho^{-2}\,\tilde\delta\tilde d-\frac{d^2}{d\rho^2}-(n-2r-1)\rho^{-1}\frac{d}{d\rho} & (\frac{d}{d\rho}+(n-2r-1)\rho^{-1})\tilde d \\
            -\rho^{-2}\tilde\delta\frac{d}{d\rho} & \rho^{-2}\,\tilde\delta\tilde d
          \end{pmatrix}\;,\\
        d\delta&\equiv
          \begin{pmatrix}
            \tilde d & 0 \\
            \frac{d}{d\rho} & -\tilde d
          \end{pmatrix}
          \begin{pmatrix}
            \rho^{-2}\,\tilde\delta & -\frac{d}{d\rho}-(n-2r+1)\rho^{-1} \\
            0 & -\rho^{-2}\,\tilde\delta
          \end{pmatrix}\\
        &=
          \begin{pmatrix}
            \rho^{-2}\,\tilde d\tilde\delta & -\tilde d(\frac{d}{d\rho}+(n-2r+1)\rho^{-1}) \\
            \rho^{-2}\frac{d}{d\rho}\tilde\delta-2\rho^{-3}\tilde\delta & -\frac{d^2}{d\rho^2}-(n-2r+1)\frac{d}{d\rho}\rho^{-1}+\rho^{-2}\,\tilde d\tilde\delta
          \end{pmatrix}\;.
    \end{align*}
  The sum of these matrices is the matrix of the statement.
\end{proof}

\subsection{Witten's perturbation on a cone}\label{ss:Witten cone}

Let $d^\pm_s$, $\delta^\pm_s$, $D^\pm_s$ and $\D^\pm_s$ ($s\ge0$) denote the Witten's perturbations of $d$, $\delta$, $D$ and $\D$ induced by the function $f=\pm\frac{1}{2}\rho^2$ on $M$. In this case, $df=\pm\rho\,d\rho$. According to~\eqref{Omega^r(M)},
  \[
    \rho\,d\rho\wedge\equiv
      \begin{pmatrix}
        0 & 0\\
        \rho & 0
      \end{pmatrix}\;,\quad
    -\rho\,d\rho\lrcorner\equiv
      \begin{pmatrix}
        0 & \rho\\
        0 & 0
      \end{pmatrix}\;.
    \]
  So the following is a consequence of Lemmas~\ref{l:d} and~\ref{l:delta},~\eqref{d_s} and~\eqref{delta_s}.

\begin{cor}\label{c:d_s^pm, delta_s^pm}
  According to~\eqref{Omega^r(M)}, on $\Omega^r(M)$,
    \begin{align*}
      d^\pm_s&\equiv
        \begin{pmatrix}
          \tilde d & 0 \\
          \frac{d}{d\rho}\pm s\rho & -\tilde d
        \end{pmatrix}\;,\\
      \delta^\pm_s&\equiv
        \begin{pmatrix}
          \rho^{-2}\,\tilde\delta & -\frac{d}{d\rho}-(n-2r+1)\rho^{-1}\pm s\rho\\
          0 & -\rho^{-2}\,\tilde\delta
        \end{pmatrix}\;.
    \end{align*}
\end{cor}

With the notation of Section~\ref{s:Witten},
  \[
    R=\pm\rho(d\rho\!\wedge-\,d\rho\lrcorner)\equiv\pm
       \begin{pmatrix}
          0 & \rho\\
          \rho & 0
        \end{pmatrix}\;,
  \]
and therefore
  \begin{equation}\label{R^2=rho^2}
    R^2\equiv
       \begin{pmatrix}
          \rho^2 & 0\\
          0 & \rho^2
        \end{pmatrix}
      \equiv\rho^2\;.
  \end{equation}

\begin{lem}\label{l:RD+DR}
  $RD+DR=\pm(2r-n)$ on $\Omega^r(M)$.
\end{lem}

\begin{proof}
  By Lemmas~\ref{l:d} and~\ref{l:delta}, and according to~\eqref{Omega^r(M)},
    \begin{align*}
      RD&\equiv\pm
        \begin{pmatrix}
          0 & \rho\\
          \rho & 0
        \end{pmatrix}
        \begin{pmatrix}
          \tilde d+\rho^{-2}\tilde\delta & -\frac{d}{d\rho}-(n-2r+1)\rho^{-1}\\
          \frac{d}{d\rho} & -\tilde d-\rho^{-2}\tilde\delta
        \end{pmatrix}\\
      &=\pm
        \begin{pmatrix}
          \rho\,\frac{d}{d\rho} & -\rho\tilde d-\rho^{-1}\tilde\delta\\
          \rho\,\tilde d+\rho^{-1}\tilde\delta & -\rho\,\frac{d}{d\rho}-n+2r-1
        \end{pmatrix}\;,\\
      DR&\equiv\pm
        \begin{pmatrix}
          \tilde d+\rho^{-2}\,\tilde\delta & -\frac{d}{d\rho}-(n-2r-1)\rho^{-1}\\
          \frac{d}{d\rho} & -\tilde d-\rho^{-2}\,\tilde\delta
        \end{pmatrix}
        \begin{pmatrix}
          0 & \rho\\
          \rho & 0
        \end{pmatrix}\\
      &=\pm
        \begin{pmatrix}
          -\frac{d}{d\rho}\,\rho-n+2r+1 & \rho\,\tilde d+\rho^{-1}\,\tilde\delta\\
          -\rho\,\tilde d-\rho^{-1}\tilde\delta & \frac{d}{d\rho}\,\rho
        \end{pmatrix}\;.
    \end{align*}
  So the result follows using~\eqref{[d/d rho,rho^a right]}.
\end{proof}

The following is a consequence of~\eqref{Delta_s with Hess f and df}, Corollary~\ref{c:D} and Lemma~\ref{l:RD+DR}.

\begin{cor}\label{c:D_s^pm}
  According to~\eqref{Omega^r(M)},
    \[
      \D^\pm_s\equiv
        \begin{pmatrix}
          P^\pm_s & -2\rho^{-1}\tilde d \\
          -2\rho^{-3}\tilde\delta & Q^\pm_s
        \end{pmatrix}
    \]
  on $\Omega^r(M)$, where
    \begin{align*}
      P^\pm_s&=\rho^{-2}\widetilde{\D}+H-(n-2r-1)\rho^{-1}\,\frac{d}{d\rho}\mp s(n-2r)\;,\\
      Q^\pm_s&=\rho^{-2}\widetilde{\D}+H-(n-2r+1)\rho^{-1}\,\frac{d}{d\rho}+(n-2r+1)\rho^{-2}\mp s(n-2r)\;.
    \end{align*}
\end{cor}

\section{Domains of the Witten's Laplacian on a cone}\label{s:types}

Theorem~\ref{t:spectrum of Delta_min/max} is proved by induction on the dimension. Thus, with the notation of Section~\ref{s:Witten cone}, suppose that $\tilde d_{\text{\rm min/max}}$ satisfies the statement of Theorem~\ref{t:spectrum of Delta_min/max}. Let
  \[
    \widetilde{\HH}_{\text{\rm min/max}}=\ker\widetilde{D}_{\text{\rm min/max}}=\ker\widetilde{\D}_{\text{\rm min/max}}\;,
  \]
which is a graded subspace of $\Omega(N)$. For each degree $r$, let
  \[
    \widetilde{\RR}_{\text{\rm min/max},r-1},\widetilde{\RR}_{\text{\rm min/max},r}^*\subset L^2\Omega^r(N)
  \]
be the images of $\tilde d_{\text{\rm min/max},r-1}$ and $\tilde\delta_{\text{\rm min/max},r}$, respectively, whose intersections with $\DD^\infty(\widetilde{\D})$ are denoted by $\widetilde{\RR}_{\text{\rm min/max},r-1}^\infty$ and $\widetilde{\RR}_{\text{\rm min/max},r}^{*\infty}$. According to Section~\ref{ss:Hilbert}, $\widetilde{\D}$ preserves $\widetilde{\RR}_{\text{\rm min/max},r-1}^\infty$ and $\widetilde{\RR}_{\text{\rm min/max},r-1}^{*\infty}$, and its restrictions to these spaces have the same eigenvalues. For any eigenvalue $\tilde\lambda$ of the restriction of $\widetilde{\D}$ to $\widetilde{\RR}_{\text{\rm min/max},r-1}^\infty$, let
  \begin{align*}
    \widetilde{\RR}_{\text{\rm min/max},r-1,\tilde\lambda}
    &=E_{\tilde\lambda}(\widetilde{\D}_{\text{\rm min/max}})\cap\widetilde{\RR}_{\text{\rm min/max},r-1}^\infty\;,\\
    \widetilde{\RR}_{\text{\rm min/max},r-1,\tilde\lambda}^*
    &=E_{\tilde\lambda}(\widetilde{\D}_{\text{\rm min/max}})\cap\widetilde{\RR}_{\text{\rm min/max},r-1}^{*\infty}\;.
  \end{align*}
Moreover
  \begin{equation}\label{L^2 Omega^r(N)}
    L^2\Omega^r(N)=\widetilde{\HH}_{\text{\rm min/max}}^r\oplus
    \widehat{\bigoplus_{\tilde\lambda}}\left(\widetilde{\RR}_{\text{\rm min/max},r-1,\tilde\lambda}
    \oplus\widetilde{\RR}^*_{\text{\rm min/max},r,\tilde\lambda}\right)\;,
  \end{equation}
where $\tilde\lambda$ runs in the spectrum of $\widetilde{\D}_{\text{\rm min/max}}$ on $\widetilde{\RR}_{\text{\rm min/max},r-1}^\infty$ and $\widetilde{\RR}^*_{\text{\rm min/max},r}$.

Now, consider the Witten's perturbed Laplacian $\D^\pm_s$. In the following, suppose that $s>0$.

\subsection{Domains of first type}\label{ss:1st type}

For some degree $r$, let $0\neq\gamma\in \widetilde{\HH}_{\text{\rm min/max}}^r$. By Corollary~\ref{c:D_s^pm},
  \[
    \D^\pm_s\equiv H-(n-2r-1)\rho^{-1}\,\frac{d}{d\rho}\mp s(n-2r)
  \]
 on $C^\infty(\R_+)\equiv C^\infty(\R_+)\,\gamma\subset\Omega^r(M)$. This operator is of the type of $P$ in~\eqref{P} with $c_2=0$. Thus~\eqref{(2c_1-1)^2+4c_2 ge0} is satisfied. In this case, Table~\ref{table: self-adjoint op 1st type} contains the possibilities for $a$ given by~\eqref{a}, the corresponding values of $2\sigma$, the condition~\eqref{sigma} expressed in terms of $r$, and the smooth cores of the corresponding self-adjoint operators with discrete spectra in $L^2(\R_+,\rho^{n-2r-1}\,d\rho)$, given by Proposition~\ref{p:P}. The corresponding eigenvalues are also indicated in Table~\ref{table: self-adjoint op 1st type}, referring to the expressions
  \begin{gather}
    (4k+(1\mp1)(n-2r))s\;,\label{eigenvalues, mu=0, without d rho, u=1, a=0}\\
    (4k+4-(1\pm1)(n-2r))s\;.\label{eigenvalues, mu=0, without d rho, u=1, a=-n+2r+2}
  \end{gather}
They are of multiplicity one, with corresponding normalized eigenfunctions $\chi_k$. More precisely, for $\D^+_s$ and  $\D^-_s$,~\eqref{eigenvalues, mu=0, without d rho, u=1, a=0} becomes $4ks$ and $(4k+2(n-2r))s$, respectively, and~\eqref{eigenvalues, mu=0, without d rho, u=1, a=-n+2r+2} becomes $(4k+4-2(n-2r))s$ and $(4k+4)s$, respectively. Table~\ref{table: eigenvalues op 1st type} indicates the signs of these eigenvalues. In all tables of Section~\ref{s:types}, grey color indicates the cases where there exist some negative eigenvalue or a too restrictive condition (the cases that will be disregarded).

\begin{table}[h]
\renewcommand{\arraystretch}{1.3}
\begin{tabular}{|c|c|l|c|c|}
\hline
$a$ & $2\sigma$ & Condition & Smooth core & Eigenvalues \\
\hline
$0$ & $n-2r-1$ & $r\le\frac{n-1}{2}$ & $\SS_{\text{\rm ev},+}$ & given by~\eqref{eigenvalues, mu=0, without d rho, u=1, a=0} \\
\hline
$-n+2r+2$ & $-n+2r+3$ & $r\ge\frac{n-3}{2}$ & $\rho^{-n+2r+2}\,\SS_{\text{\rm ev},+}$ & given by~\eqref{eigenvalues, mu=0, without d rho, u=1, a=-n+2r+2} \\
\hline
\end{tabular}
\caption{Self-adjoint operators first type}
\label{table: self-adjoint op 1st type}
\end{table}

\begin{table}[h]
\renewcommand{\arraystretch}{1.3}
\begin{tabular}{cl|l|c|c|c|c|}
\cline{3-7}
&& \multirow{2}{*}{Conditions} & \multirow{2}{*}{Smooth core} & \multicolumn{3}{c|}{Sign of the  eigenvalues} \\
\cline{5-7}
& & & & $<0$ & $0$ & $>0$ \\
\hline
\multicolumn{2}{|c|}{$\D^+_s$} & \multirow{2}{*}{$r\le\frac{n-1}{2}$} & \multirow{2}{*}{$\SS_{\text{\rm ev},+}$} && $k=0$ & $k\ge1$ \\
\cline{1-2}\cline{5-7}
\multicolumn{2}{|c|}{$\D^-_s$} &&&&& $\forall k$ \\
\hline
\multicolumn{1}{|c|}{\multirow{3}{*}{$\D^+_s$}} & \multicolumn{1}{|l|}{$r\ge\frac{n-1}{2}$} & \multirow{4}{*}{$r\ge\frac{n-3}{2}$} & \multirow{4}{*}{$\rho^{-n+2r+2}\,\SS_{\text{\rm ev},+}$} &&& $\forall k$ \\
\cline{2-2}\cline{5-7}
\multicolumn{1}{|c|}{} & \multicolumn{1}{|l|}{$r=\frac{n}{2}-1$} &&&& $k=0$ & $k\ge1$ \\
\cline{2-2}\cline{5-7}
\multicolumn{1}{|c|}{} & \multicolumn{1}{|l|}{\color{lightgray} $r=\frac{n-3}{2}$} &&& \color{lightgray} $k=0$ && $k\ge1$ \\
\cline{1-2}\cline{5-7}
\multicolumn{2}{|c|}{$\D^-_s$} &&&&& $\forall k$ \\
\hline
\end{tabular}
\caption{Sign of the eigenvalues for operators of first type}
\label{table: eigenvalues op 1st type}
\end{table}

When $\frac{n-3}{2}\le r\le\frac{n-1}{2}$, we have got two essentially self-adjoint operators, with $a=0$ and $a=-n+2r+2$. These two operators are equal just when $r=\frac{n}{2}-1$.

All of the above operators defined by $\D^\pm_s$, as well as their domains, will be said to be of {\em first type\/}.

\subsection{Domains of second type}\label{ss:2nd type}

With the notation of Section~\ref{ss:1st type}, 
  \[
    \D^\pm_s\equiv H-(n-2r-1)\rho^{-1}\,\frac{d}{d\rho}+(n-2r-1)\rho^{-2}\mp s(n-2r-2)
  \]
on $C^\infty(\R_+)\equiv C^\infty(\R_+)\,d\rho\wedge\gamma\subset\Omega^{r+1}(M)$ by Corollary~\ref{c:D_s^pm}. This is an operator of the type of $P$ in~\eqref{P} with $c_2=2c_1$. Thus~\eqref{(2c_1-1)^2+4c_2 ge0} is also satisfied. In this case, Table~\ref{table: self-adjoint op 2nd type} contains the possibilities for $a$ given by~\eqref{a}, the corresponding values of $2\sigma$, the condition~\eqref{sigma} expressed in terms of $r$, and the smooth cores of the corresponding self-adjoint operators with discrete spectra in $L^2(\R_+,\rho^{n-2r-1}\,d\rho)$, given by Proposition~\ref{p:P}. The corresponding eigenvalues are also indicated in Table~\ref{table: self-adjoint op 1st type}, referring to the expressions
  \begin{gather}
    (4k+4+(1\mp1)(n-2r-2))s\;,\label{eigenvalues, mu=0, with d rho, u=1, a=1}\\
    (4k-(1\pm1)(n-2r-2))s\;.\label{eigenvalues, mu=0, with d rho, u=1, a=-n+2r+1}
  \end{gather}
They are of multiplicity one, with corresponding normalized eigenfunctions $\chi_k$. More precisely, for $\D^+_s$ and  $\D^-_s$,~\eqref{eigenvalues, mu=0, with d rho, u=1, a=1} becomes $(4k+4)s$ and $(4k+2(n-2r))s$, respectively, and~\eqref{eigenvalues, mu=0, with d rho, u=1, a=-n+2r+1} becomes $(4k+4-2(n-2r))s$ and $4ks$, respectively. Table~\ref{table: eigenvalues op 2nd type} indicates the signs of these eigenvalues.

\begin{table}[h]
\renewcommand{\arraystretch}{1.3}
\begin{tabular}{|c|c|l|c|c|}
\hline
$a$ & $2\sigma$ & Condition & Smooth core & Eigenvalues \\
\hline
$1$ & $n-2r+1$ & $r\le\frac{n+1}{2}$ & $\rho\,\SS_{\text{\rm ev},+}$ & given by~\eqref{eigenvalues, mu=0, with d rho, u=1, a=1} \\
\hline
$-n+2r+1$ & $-n+2r+1$ & $r\ge\frac{n-1}{2}$ & $\rho^{-n+2r+1}\,\SS_{\text{\rm ev},+}$ & given by~\eqref{eigenvalues, mu=0, with d rho, u=1, a=-n+2r+1} \\
\hline
\end{tabular}
\caption{Self-adjoint operators second type}
\label{table: self-adjoint op 2nd type}
\end{table}

\begin{table}[h]
\renewcommand{\arraystretch}{1.3}
\begin{tabular}{cl|l|c|c|c|c|}
\cline{3-7}
&& \multirow{2}{*}{Condition} & \multirow{2}{*}{Smooth core} & \multicolumn{3}{c|}{Sign of the  eigenvalues} \\
\cline{5-7}
& & & & $<0$ & $0$ & $>0$ \\
\hline
\multicolumn{2}{|c|}{$\D^+_s$} & \multirow{4}{*}{$r\le\frac{n+1}{2}$} & \multirow{4}{*}{$\rho\,\SS_{\text{\rm ev},+}$} &&& $\forall k$ \\
\cline{1-2}\cline{5-7}
\multicolumn{1}{|c|}{\multirow{3}{*}{$\D^-_s$}} & \multicolumn{1}{|l|}{\color{lightgray} $r=\frac{n+1}{2}$} &&& \color{lightgray} $k=0$ && $k\ge1$  \\
\cline{2-2}\cline{5-7}
\multicolumn{1}{|c|}{} & \multicolumn{1}{|l|}{$r=\frac{n}{2}$} &&&& $k=0$ & $k\ge1$ \\
\cline{2-2}\cline{5-7}
\multicolumn{1}{|c|}{} & \multicolumn{1}{|l|}{$r\le\frac{n-1}{2}$} &&&&& $\forall k$ \\
\hline
\multicolumn{2}{|c|}{$\D^+_s$} & \multirow{2}{*}{$r\ge\frac{n-1}{2}$} & \multirow{2}{*}{$\rho^{-n+2r+1}\,\SS_{\text{\rm ev},+}$} &&& $\forall k$ \\
\cline{1-2}\cline{5-7}
\multicolumn{2}{|c|}{$\D^-_s$} &&&& $k=0$ & $k\ge1$ \\
\hline
\end{tabular}
\caption{Sign of the eigenvalues for operators of second type}
\label{table: eigenvalues op 2nd type}
\end{table}

For $\frac{n-1}{2}\le r\le\frac{n+1}{2}$, we have obtained two essentially self-adjoint operators, with $a=1$ and $a=-n+2r+1$. These operators are equal just when $r=\frac{n}{2}$.

All of the above operators defined by $\D^\pm_s$, as well as their domains, will be said to be of {\em second type\/}.

\subsection{Domains of third type}\label{ss:3rd type}

Let $\mu=\sqrt{\tilde\lambda}$ for an eigenvalue $\tilde\lambda$ of the restriction of $\widetilde{\D}_{\text{\rm min/max}}$ to $\widetilde{\RR}_{\text{\rm min/max},r-1}^\infty$. According to Section~\ref{ss:Hilbert}, there are non-zero differential forms,
  \[
    \alpha\in\widetilde{\RR}_{\text{\rm min/max},r-1,\tilde\lambda}\subset\Omega^r(N)\;,\quad
    \beta\in\widetilde{\RR}_{\text{\rm min/max},r-1,\tilde\lambda}^*\subset\Omega^{r-1}(N)\;,
  \]
such that $\tilde d\beta=\mu\alpha$ and $\tilde\delta\alpha=\mu\beta$. By Corollary~\ref{c:D_s^pm},
  \[
    \D^\pm_s\equiv H-(n-2r+1)\rho^{-1}\,\frac{d}{d\rho}+\mu^2\rho^{-2}\mp(n-2r+2)s
  \]
on $C^\infty(\R_+)\equiv C^\infty(\R_+)\,\beta\subset\Omega^{r-1}(M)$. This operator is of the type of $P$ in~\eqref{P} with $c_2=\mu^2>0$. Thus~\eqref{(2c_1-1)^2+4c_2 ge0} is satisfied, and~\eqref{a} becomes
  \begin{equation}\label{3rd type, a}
    a=\frac{-n+2r\pm\sqrt{(n-2r)^2+4\mu^2}}{2}\;.
  \end{equation}
These two possibilities for $a$ have different sign because $\mu>0$.

For the choice of positive square root in~\eqref{3rd type, a}, we get
  \begin{equation}\label{3rd type, sigma, +sqrt}
    \sigma=\frac{1+\sqrt{(n-2r)^2+4\mu^2}}{2}>\frac{1}{2}
  \end{equation}
according to~\eqref{sigma}. Then Proposition~\ref{p:P} asserts that $\D^\pm_s$, with domain $\rho^a\,\SS_{\text{\rm ev},+}$, is essentially self-adjoint in $L^2(\R_+,\rho^{n-2r+1}\,d\rho)$; the spectrum of its closure consists of the eigenvalues
  \begin{equation}\label{eigenvalues, 3rd type, a with +}
    \left(4k+2+\sqrt{(n-2r)^2+4\mu^2}\mp(n-2r+2)\right)s\;,
  \end{equation}
with multiplicity one and corresponding normalized eigenfunctions $\chi_k$; and the smooth core of its closure is $\rho^a\,\SS_{\text{\rm ev},+}$. Notice that~\eqref{eigenvalues, 3rd type, a with +} is $>0$ for all $k$.

For the choice of negative square root in~\eqref{3rd type, a}, we get
  \begin{equation}\label{sigma, 3rd type, -sqrt}
    \sigma=\frac{1-\sqrt{(n-2r)^2+4\mu^2}}{2}
  \end{equation}
according to~\eqref{sigma}. Then $\sigma>-1/2$ if and only if
  \begin{equation}\label{mu<1 and |n-2r|<2 sqrt 1-mu^2}
    \mu<1\quad\text{and}\quad|n-2r|<2\sqrt{1-\mu^2}\;,
  \end{equation}
which is equivalent to $\frac{\sqrt{3}}{2}\le\mu<1$ and $r=\frac{n}{2}$, or $\mu<\frac{\sqrt{3}}{2}$ and $\frac{n-1}{2}\le r\le\frac{n+1}{2}$. In this case, Proposition~\ref{p:P} asserts that $\D^\pm_s$, with domain $\rho^a\,\SS_{\text{\rm ev},+}$, is essentially self-adjoint in $L^2(\R_+,\rho^{n-2r+1}\,d\rho)$; the spectrum of its closure consists of the eigenvalues
  \begin{equation}\label{eigenvalues, 3rd type, a with -}
    \left(4k+2-\sqrt{(n-2r)^2+4\mu^2}\mp(n-2r+2)\right)s\;,
  \end{equation}
with multiplicity one and corresponding normalized eigenfunctions $\rho^a\,\phi_{2k,+}$; and the smooth core of its closure is $\rho^a\,\SS_{\text{\rm ev},+}$. For $\D^+_s$,~\eqref{eigenvalues, 3rd type, a with -} is $<0$ for $k=0$. For $\D^-_s$,~\eqref{eigenvalues, 3rd type, a with -} is $>0$ for all $k$.

Table~\ref{table: eigenvalues op 3rd type} summarizes the information about the sign of the eigenvalues for all choices of $a$ and the sign of the model function.

\begin{table}[h]
\renewcommand{\arraystretch}{1.3}
\begin{tabular}{c|c|c|l|}
\cline{2-4}
& $a$ & Condition & Sign of the  eigenvalues \\
\hline
\multicolumn{1}{|c|}{$\D^\pm_s$} & \eqref{3rd type, a} with $+\sqrt{}$ & No restriction & $>0\ \forall k$ \\
\hline
\multicolumn{1}{|c|}{\color{lightgray} $\D^+_s$} & \multirow{2}{*}{\color{lightgray} \eqref{3rd type, a} with $-\sqrt{}$} & \multirow{2}{*}{\color{lightgray} \eqref{mu<1 and |n-2r|<2 sqrt 1-mu^2} (strong)} & \color{lightgray} $<0$ for $k=0$ \\
\cline{1-1}\cline{4-4}
\multicolumn{1}{|c|}{\color{lightgray} $\D^-_s$} &&& $>0\ \forall k$ \\
\hline
\end{tabular}
\caption{Sign of the eigenvalues for operators of third type}
\label{table: eigenvalues op 3rd type}
\end{table} 

When~\eqref{mu<1 and |n-2r|<2 sqrt 1-mu^2} is satisfied, we have got two different essentially self-adjoint operators defined by the two different choices of $a$ in~\eqref{3rd type, a}.

All of the above operators defined by $\D^\pm_s$, as well as their domains, will be said to be of {\em third type\/}.

\subsection{Domains of fourth type}
\label{ss:4th type}

Let $\mu$, $\alpha$ and $\beta$ be like in Section~\ref{ss:3rd type}. By Corollary~\ref{c:D_s^pm},
  \[
    \D^\pm_s\equiv H-(n-2r-1)\rho^{-1}\,\frac{d}{d\rho}+(\mu^2+n-2r-1)\rho^{-2}\mp(n-2r-2)s
  \]
  on $C^\infty(\R_+)\equiv C^\infty(\R_+)\,d\rho\wedge\alpha\subset\Omega^{r+1}(M)$. This is another operator of the type of $P$ in~\eqref{P}, which satisfies~\eqref{(2c_1-1)^2+4c_2 ge0} because
  \[
    (1-(n-2r-1))^2+4(\mu^2+n-2r-1)=(n-2r)^2+4\mu^2>0\;.
  \]
Moreover~\eqref{a} becomes
  \begin{equation}\label{4th type, a}
      a=\frac{-n+2r+2\pm\sqrt{(n-2r)^2+4\mu^2}}{2}\;.
  \end{equation}
These two possibilities for $a$ are different because $\mu>0$.

With the choice of positive square root in~\eqref{4th type, a} and according to~\eqref{sigma}, $\sigma$ is also given by~\eqref{3rd type, sigma, +sqrt}, which is $>1/2$.  Then Proposition~\ref{p:P} asserts that $\D^\pm_s$, with domain $\rho^a\,\SS_{\text{\rm ev},+}$, is essentially self-adjoint in $L^2(\R_+,\rho^{n-2r-1}\,d\rho)$; the spectrum of its closure consists of the eigenvalues
  \begin{equation}\label{eigenvalues, 4th type, a with +}
    \left(4k+2+\sqrt{(n-2r)^2+4\mu^2}\mp(n-2r-2)\right)s\;,
  \end{equation}
with multiplicity one and corresponding normalized eigenfunctions $\chi_k$; and the smooth core of its closure is $\rho^a\,\SS_{\text{\rm ev},+}$. Observe that~\eqref{eigenvalues, 4th type, a with +} is $>0$ for all $k$.

With the choice of negative square root in~\eqref{4th type, a} and according to~\eqref{sigma}, $\sigma$ is also given by~\eqref{sigma, 3rd type, -sqrt}, which is $>-1/2$ if and only if~\eqref{mu<1 and |n-2r|<2 sqrt 1-mu^2} is satisfied. In this case, Proposition~\ref{p:P} asserts that $\D^\pm_s$, with domain $\rho^a\,\SS_{\text{\rm ev},+}$, is essentially self-adjoint in $L^2(\R_+,\rho^{n-2r-1}\,d\rho)$; the spectrum of its closure consists of the eigenvalues
  \begin{equation}\label{eigenvalues, 4th type, a with -}
    \left(4k+2-\sqrt{(n-2r)^2+4\mu^2}\mp(n-2r-2)\right)s\;,
  \end{equation}
with multiplicity one and corresponding normalized eigenfunctions $\chi_k$; and the smooth core of its closure is $\rho^a\,\SS_{\text{\rm ev},+}$. For $\D_s^+$,~\eqref{eigenvalues, 4th type, a with -} is $>0$ for all $k$. For $\D_s^-$,~\eqref{eigenvalues, 4th type, a with -} is $<0$ for $k=0$.

Table~\ref{table: eigenvalues op 4th type} summarizes the information about the sign of the eigenvalues for all choices of $a$ and the sign of the model function.

\begin{table}[h]
\renewcommand{\arraystretch}{1.3}
\begin{tabular}{c|c|c|l|}
\cline{2-4}
& $a$ & Condition & Sign of the  eigenvalues \\
\hline
\multicolumn{1}{|c|}{$\D^\pm_s$} & \eqref{4th type, a} with $+\sqrt{}$ & No restriction & $>0\ \forall k$ \\
\hline
\multicolumn{1}{|c|}{\color{lightgray} $\D^+_s$} & \multirow{2}{*}{\color{lightgray} \eqref{4th type, a} with $-\sqrt{}$} & \multirow{2}{*}{\color{lightgray} \eqref{mu<1 and |n-2r|<2 sqrt 1-mu^2} (strong)} & $>0\ \forall k$ \\
\cline{1-1}\cline{4-4}
\multicolumn{1}{|c|}{\color{lightgray} $\D^-_s$} &&& \color{lightgray} $<0$ for $k=0$ \\
\hline
\end{tabular}
\caption{Sign of the eigenvalues for operators of fourth type}
\label{table: eigenvalues op 4th type}
\end{table} 

When~\eqref{mu<1 and |n-2r|<2 sqrt 1-mu^2} is satisfied, we have got two different essentially self-adjoint operators defined by the two different choices of $a$ in~\eqref{4th type, a}.

All of the above operators defined by $\D^\pm_s$, as well as their domains, will be said to be of {\em fourth type\/}.

\subsection{Domains of fifth type}\label{ss:5th type}

Let $\mu$, $\alpha$ and $\beta$ be like in Sections~\ref{ss:3rd type} and~\ref{ss:4th type}. By Corollary~\ref{c:D_s^pm},
  \[
    \D^\pm_s\equiv
      \begin{pmatrix}
        P^\pm_{\mu,s} & -2\rho^{-1}\mu \\
        -2\rho^{-3}\mu & Q^\pm_{\mu,s}
      \end{pmatrix}
  \]
 on
  \[
    C^\infty(\R_+)\oplus C^\infty(\R_+)\equiv C^\infty(\R_+)\,\alpha+C^\infty(\R_+)\,d\rho\wedge\beta\subset\Omega^r(M)\;,
  \]
where
  \begin{align*}
    P^\pm_{\mu,s}&=H-(n-2r-1)\rho^{-1}\,\frac{d}{d\rho}+\mu^2\rho^{-2}\mp(n-2r)s\;,\\
    Q^\pm_{\mu,s}&=H-(n-2r+1)\rho^{-1}\,\frac{d}{d\rho}+(\mu^2+n-2r+1)\rho^{-2}\mp(n-2r)s\;.
  \end{align*}
We will conjugate this matrix expression of $\D^\pm_s$ by some non-singular matrix $\Theta$, whose entries are functions of $\rho$, to get a diagonal matrix whose diagonal entries are operators of the type of $P$ in~\eqref{P}. This matrix will be of the form $\Theta=BC$ with
  \[
    B=
      \begin{pmatrix}
        1 & 0\\
        0 & \rho^{-1}
      \end{pmatrix}\;,\quad
    C=
      \begin{pmatrix}
        c_{11} & c_{12}\\
        c_{21} & c_{22}
      \end{pmatrix}\;,
  \]
where $c_{ij}$ are constants to be determined. Let $P^\pm_{\mu,s}$ and $Q^\pm_{\mu,s}$ be simply denoted by $P$ and $Q$. A key observation here is that, by~\eqref{[d/d rho,rho^a right]},
  $$
    Q-\rho^{-1}\,P\,\rho=2(n-2r)\rho^{-2}\;,
  $$
obtaining
  \begin{multline*}
    B^{-1}\,\D^\pm_s\,B=
    \begin{pmatrix}
        1 & 0\\
        0 & \rho
      \end{pmatrix}
      \begin{pmatrix}
        P & -2\mu\,\rho^{-1}\\
        -2\mu\,\rho^{-3} & Q
      \end{pmatrix}
      \begin{pmatrix}
        1 & 0\\
        0 & \rho^{-1}
      \end{pmatrix}\\
      =
      \begin{pmatrix}
        P & -2\mu\rho^{-2}\\
        -2\mu\rho^{-2} & \rho\,Q\,\rho^{-1}
      \end{pmatrix}
      =
      \begin{pmatrix}
        P & -2\mu\rho^{-2}\\
        -2\mu\rho^{-2} & P+2(n-2r)\rho^{-2}
      \end{pmatrix}\;.
  \end{multline*}
On the other hand, $C$ must be non-singular and
  \[
    C^{-1}=\frac{1}{\det C}
      \begin{pmatrix}
        c_{22} & -c_{12} \\
        -c_{21} & c_{11}
      \end{pmatrix}\;.
  \]
Therefore $\Theta^{-1}\D^\pm_s\Theta=(X_{ij})$ with
  \begin{align*}
    X_{11}&=P+\frac{2}{\det C}\,\left(\mu\,(-c_{22}c_{21}+c_{12}c_{11})-(n-2r)c_{12}c_{21}\right)\rho^{-2}\;,\\
    X_{12}&=\frac{2}{\det C}\,\left(\mu\,(-c_{22}^2+c_{12}^2)-(n-2r)c_{12}c_{22}\right)\rho^{-2}\;,\\
    X_{21}&=\frac{2}{\det C}\,\left(\mu\,(c_{21}^2-c_{11}^2)+(n-2r)c_{11}c_{21}\right)\rho^{-2}\;,\\
    X_{22}&=P+\frac{2}{\det C}\,\left(\mu(c_{21}c_{22}-c_{11}c_{12})+(n-2r)c_{11}c_{22}\right)\rho^{-2}\;.
  \end{align*}
We want $(X_{ij})$ to be diagonal, so we require
  \[
    \mu(c_{12}^2-c_{22}^2)-(n-2r)c_{12}c_{22}=\mu(c_{11}^2-c_{21}^2)-(n-2r)c_{11}c_{21}=0\;.
  \]
Both of these equations are of the form
  \begin{equation}\label{mu(y 2-x 2)-(bar n-2r)xy=0}
    \mu(x^2-y^2)-(n-2r)xy=0\;,
  \end{equation}
with $x=c_{12}$ and $y=c_{22}$ in the first equation, and $x=c_{11}$ and $y=c_{21}$ in the second one. There is some $c\in\R\sm\{0\}$ such that
  \begin{equation}\label{x 2-y 2-frac bar n-2r mu xy=(x+cy)(x-frac y c)}
    x^2-y^2-\frac{n-2r}{\mu}\,xy=(x+cy)\left(x-\frac{y}{c}\right)\;;
  \end{equation}
in fact, we need $c-\frac{1}{c}=-\frac{n-2r}{\mu}$, giving
  \begin{equation}\label{mu c 3+(bar n-2r) c-mu=0}
    \mu c^2+(n-2r)c-\mu=0\;,
  \end{equation}
whose solutions are
  \begin{equation}\label{c_pm}
    c_\pm=\frac{-n+2r\pm\sqrt{(n-2r)^2+4\mu^2}}{2\mu}\;.
  \end{equation}
Observe that $c_+c_-=-1$. Let $c=c_+>0$, and therefore $-1/c=c_-$. By~\eqref{x 2-y 2-frac bar n-2r mu xy=(x+cy)(x-frac y c)}, the solutions of~\eqref{mu(y 2-x 2)-(bar n-2r)xy=0} are given by $x+cy=0$ and  $cx-y=0$. Then we can take
  \[
    C=
      \begin{pmatrix}
        1 & -c \\
        c & 1
      \end{pmatrix}\;,
  \]
with $\det C=1+c^2>0$. So, for
  \[
    \Theta=
      \begin{pmatrix}
        1 & 0 \\
        0 & \rho^{-1}
      \end{pmatrix}
      \begin{pmatrix}
        1 & -c \\
        c & 1
      \end{pmatrix}
    =
      \begin{pmatrix}
        1 & -c \\
        c\rho^{-1} & \rho^{-1}
      \end{pmatrix}\;,
  \]
we get $X_{12}=X_{21}=0$, and
  \begin{align*}
    X_{11}&=P+\frac{2(-2\mu c+(n-2r)c^2)}{1+c^2}\,\rho^{-2}\;,\\
    X_{22}&=P+\frac{2(2\mu c+n-2r)}{1+c^2}\,\rho^{-2}\;.
  \end{align*}
The notation $X=X_{11}$ and $Y=X_{22}$ will be used; thus $\Theta^{-1}\D^\pm_s\Theta=X\oplus Y$. The above expressions of $X$ and $Y$ can be simplified as follows. We have
  \[
    1+c^2=\frac{2\mu-(n-2r)c}{\mu}
  \]
by~\eqref{mu c 3+(bar n-2r) c-mu=0}, obtaining
  $$
    \frac{2(-2\mu c+(n-2r)c^2)}{1+c^2}=-2\mu c\;,\quad
    \frac{2(2\mu c+n-2r)}{1+c^2}=\frac{2\mu(2\mu c+n-2r)}{2\mu-(n-2r)c}\;.
  $$
Moreover
  \[
    (2\mu c+n-2r)^2=(n-2r)^2+4\mu^2>0
  \]
by~\eqref{c_pm}, and
  \begin{multline*}
    (2\mu-(n-2r)c)(2\mu c+n-2r)\\
      \begin{aligned}
        &=4\mu^2c+2\mu(n-2r)-(n-2r)2\mu c^2-(n-2r)^2c\\
        &=4\mu^2c+2\mu(n-2r)-(n-2r)2\mu(1-\frac{n-2r}{\mu}\,c)-(n-2r)^2c\\
        &=c(4\mu^2+(n-2r)^2)
      \end{aligned}
  \end{multline*}
 by~\eqref{mu c 3+(bar n-2r) c-mu=0}. Therefore
  \begin{multline*}
     \frac{2(2\mu c+n-2r)}{1+c^2}=\frac{2\mu(2\mu c+n-2r)^2}{(2\mu-(n-2r)c)(2\mu c+n-2r)}\\
     =\frac{2\mu((n-2r)^2+4\mu^2)}{c(4\mu^2+(n-2r)^2)}=\frac{2\mu}{c}\;.
  \end{multline*}
It follows that $X=P-2\mu c\rho^{-2}$ and $Y=P+\frac{2\mu}{c}\rho^{-2}$, obtaining
  \begin{align*}
    X&=H-(n-2r-1)\rho^{-1}\,\frac{d}{d\rho}+(\mu^2-2\mu c)\rho^{-2}\mp(n-2r)s\;,\\
    Y&=H-(n-2r-1)\rho^{-1}\,\frac{d}{d\rho}+(\mu^2+\frac{2\mu}{c})\rho^{-2}\mp(n-2r)s\;.
  \end{align*}
These operators are of the type of $P$ in~\eqref{P}, and satisfy~\eqref{(2c_1-1)^2+4c_2 ge0} because, by~\eqref{c_pm},
  \begin{gather*}
    (1-(n-2r-1))^2+4(\mu^2-2\mu c)
    =(2-\sqrt{(n-2r)^2+4\mu^2})^2\ge0\;,\\
    (1-(n-2r-1))^2+4(\mu^2+\frac{2\mu}{c})
    =(2+\sqrt{(n-2r)^2+4\mu^2})^2>0\;.
  \end{gather*}
So, for $X$ and $Y$, the constants~\eqref{a} and~\eqref{sigma} become
  \begin{align}
    a&=\frac{2-n+2r\pm(2-\sqrt{(n-2r)^2+4\mu^2})}{2}\;,\label{a for X}\\
    b&=\frac{2-n+2r\pm(2+\sqrt{(n-2r)^2+4\mu^2})}{2}\;,\label{b for Y}\\
    \sigma&=\frac{1\pm(2-\sqrt{(n-2r)^2+4\mu^2})}{2}\;,\label{sigma for X}\\
    \tau&=\frac{1\pm(2+\sqrt{(n-2r)^2+4\mu^2})}{2}\;.\label{tau for Y}
  \end{align}
Suppose that $\sigma,\tau>-1/2$. By Proposition~\ref{p:P}, $X$ and $Y$, with respective domains $\rho^a\,\SS_{\text{\rm ev},+}$ and $\rho^b\,\SS_{\text{\rm ev},+}$, are essentially self-adjoint in $L^2(\R_+,\rho^{n-2r-1}\,d\rho)$; the spectra of their closures consist of the eigenvalues
  \begin{gather}
    (4k+2a+(1\mp1)(n-2r))s\;,\label{eigenvalues of X}\\
    (4k+2+2b+(1\mp1)(n-2r))s\;,\label{eigenvalues of Y}
  \end{gather}
with multiplicity one and corresponding normalized eigenfunctions $\chi_{s,a,\sigma,k}$ and $\chi_{s,b,\tau,k}$, respectively, and the smooth cores of their closures are $\rho^a\,\SS_{\text{\rm ev},+}$ and $\rho^b\,\SS_{\text{\rm ev},+}$.

Since $\frac{1}{\sqrt{1+c^2}}C$ is an orthogonal matrix, it defines a unitary isomorphism
  \begin{multline*}
    L^2(\R_+,\rho^{n-2r-1}\,d\rho)\oplus L^2(\R_+,\rho^{n-2r-1}\,d\rho)\\
    \to L^2(\R_+,\rho^{n-2r-1}\,d\rho)\oplus L^2(\R_+,\rho^{n-2r-1}\,d\rho)\;,
  \end{multline*}
and we already know that
\begin{multline*}
    B=1\oplus\rho^{-1}:L^2(\R_+,\rho^{n-2r-1}\,d\rho)\oplus L^2(\R_+,\rho^{n-2r-1}\,d\rho)\\
    \to L^2(\R_+,\rho^{n-2r-1}\,d\rho)\oplus L^2(\R_+,\rho^{n-2r+1}\,d\rho)
  \end{multline*}
is a unitary isomorphism too. So $\frac{1}{\sqrt{1+c^2}}\Theta$ is a unitary isomorphism
  \begin{multline*}
    L^2(\R_+,\rho^{n-2r-1}\,d\rho)\oplus L^2(\R_+,\rho^{n-2r-1}\,d\rho)\\
    \to L^2(\R_+,\rho^{n-2r-1}\,d\rho)\oplus L^2(\R_+,\rho^{n-2r+1}\,d\rho)\;.
  \end{multline*}
Therefore, when $\sigma,\tau>-1/2$, the operator $\D^\pm_s$, with domain
  \begin{equation}\label{Theta(...)}
    \Theta(\rho^a\,\SS_{\text{\rm ev},+}\oplus\rho^b\,\SS_{\text{\rm ev},+})
    =\{(\rho^a\phi-c\rho^b\psi,c\rho^{a-1}\phi+\rho^{b-1}\psi)\ |\ \phi,\psi\in\SS_{\text{\rm ev},+}\}\;,
  \end{equation}
is essentially self-adjoint in
  \begin{multline}\label{L 2(R +,rho bar n-2r-1 d rho) oplus L 2(R +,rho bar n-2r+1 d rho)}
    L^2(\R_+,\rho^{n-2r-1}\,d\rho)\oplus L^2(\R_+,\rho^{n-2r+1}\,d\rho)\\
    \equiv L^2(\R_+,\rho^{n-2r-1}\,d\rho)\,\alpha+L^2(\R_+,\rho^{n-2r+1}\,d\rho)\,d\rho\wedge\beta\;,
  \end{multline}
which is a Hilbert subspace of $L^2\Omega^r(M,g)$; the spectrum of its closure consists of the eigenvalues~\eqref{eigenvalues of X} and~\eqref{eigenvalues of Y}, with multiplicity one and corresponding normalized eigenvectors $\frac{1}{\sqrt{1+c^2}}\,\Theta(\chi_{s,a,\sigma,k},0)$ and $\frac{1}{\sqrt{1+c^2}}\,\Theta(0,\chi_{s,b,\tau,k})$; and the smooth core of its closure is~\eqref{Theta(...)}.

The condition $\tau>-1/2$ only holds with the choice
  \[
    \tau=\frac{3+\sqrt{(n-2r)^2+4\mu^2}}{2}
  \]
in~\eqref{tau for Y}, which corresponds to the choice
  \begin{equation}\label{5th type, b with +sqrt}
    b=\frac{4-n+2r+\sqrt{(n-2r)^2+4\mu^2}}{2}
  \end{equation}
in~\eqref{b for Y}. With this choice, the eigenvalues~\eqref{eigenvalues of Y} become
  \begin{equation}\label{eigenvalues of Y, b with + sqrt}
    \left(4k+6\mp(n-2r)+\sqrt{(n-2r)^2+4\mu^2}\right)s\;,
  \end{equation}
which are $>0$ for all $k$.

Consider the choice
  \begin{equation}\label{5th type, a with +sqrt}
    a=\frac{-n+2r+\sqrt{(n-2r)^2+4\mu^2}}{2}
  \end{equation}
in~\eqref{a for X}, and, correspondingly,
  \[
    \sigma=\frac{-1+\sqrt{(n-2r)^2+4\mu^2}}{2}>-\frac{1}{2}
  \]
in~\eqref{sigma for X}. Then the eigenvalues~\eqref{eigenvalues of X} become
  \begin{equation}\label{eigenvalues of X, a with + sqrt}
    \left(4k\mp(n-2r)+\sqrt{(n-2r)^2+4\mu^2}\right)s\;,
  \end{equation}
which are $>0$ for all $k$.

Now, consider the choice
  \begin{equation}\label{5th type, a with -sqrt}
    a=\frac{4-n+2r-\sqrt{(n-2r)^2+4\mu^2}}{2}
  \end{equation}
in~\eqref{a for X}, and therefore
  \[
    \sigma=\frac{3-\sqrt{(n-2r)^2+4\mu^2}}{2}
  \]
in~\eqref{sigma for X}. In this case, the condition $\sigma>-1/2$ means that
  \begin{equation}\label{mu<2, |n-2r|<2 sqrt 4-mu^2}
    \mu<2\quad\text{and}\quad|n-2r|<2\sqrt{4-\mu^2}\;.
  \end{equation}
The eigenvalues~\eqref{eigenvalues of X} become
  \begin{equation}\label{eigenvalues of X with the 1st choice of a}
    \left(4k+4\mp(n-2r)-\sqrt{(n-2r)^2+4\mu^2}\right)s\;.
  \end{equation}

For $\D_s^+$,~\eqref{eigenvalues of X with the 1st choice of a} is:
  \begin{itemize}

    \item $\ge0$ for all $k$ if and only if $n-2r\le2-\mu^2/2$, and

    \item $=0$ just when $k=0$ and $n-2r=2-\mu^2/2$.

  \end{itemize}

For $\D_s^-$,~\eqref{eigenvalues of X with the 1st choice of a} is:
  \begin{itemize}

    \item $\ge0$ for all $k$ if and only if $n-2r\ge\mu^2/2-2$, and

    \item $=0$ just when $k=0$ and $n-2r=\mu^2/2-2$.

  \end{itemize}
  
Table~\ref{table: eigenvalues op 5th type} summarizes the information about the sign of the eigenvalues for all choices of $a$ and $b$.

\begin{table}[h]
\renewcommand{\arraystretch}{1.3}
\begin{tabular}{|c|c|c|c|}
\hline
$a$ & $b$ & Condition & Sign of the  eigenvalues \\
\hline
\eqref{5th type, a with +sqrt} & \multirow{2}{*}{\eqref{5th type, b with +sqrt}} & No restriction & $>0\ \forall k$ \\
\cline{1-1}\cline{3-4}
\color{lightgray} \eqref{5th type, a with -sqrt} && \color{lightgray} \eqref{mu<2, |n-2r|<2 sqrt 4-mu^2} (strong) & \color{lightgray} $\exists$ eigenvalues $<0$ easily \\
\hline
& \color{lightgray} The other choice & \color{lightgray} Impossible &  \\
\hline
\end{tabular}
\caption{Sign of the eigenvalues for operators of fifth type}
\label{table: eigenvalues op 5th type}
\end{table}

All of the above operators defined by $\D^\pm_s$, as well as their domains, will be said to be of {\em fifth type\/}.

\section{Splitting of the Witten's complex on a cone}\label{s:splitting}

\subsection{Subcomplexes defined by domains of first and second types}
\label{ss:subcomplexes 1,2}

Consider the notation of Sections~\ref{ss:1st type} and~\ref{ss:2nd type}. The following result follows from Corollary~\ref{c:d_s^pm, delta_s^pm}.

\begin{lem}\label{l:complexes 1,2}
  For $s\ge0$, $d^\pm_s$ and $\delta^\pm_s$ define maps
    \begin{center}
      \begin{picture}(302,32) 
        \put(0,13){$0$}
        \put(64,13){$C^\infty(\R_+)\,\gamma$}
        \put(168,13){$C^\infty(\R_+)\,d\rho\wedge\gamma$}
        \put(294,13){$0\;,$}
        \put(21,22){\Small$d^\pm_{s,r-1}$}
        \put(21,3){\Small$\delta^\pm_{s,r-1}$}
        \put(130,22){\Small$d^\pm_{s,r}$}
        \put(130,3){\Small$\delta^\pm_{s,r}$}
        \put(251,22){\Small$d^\pm_{s,r+1}$}
        \put(251,3){\Small$\delta^\pm_{s,r+1}$}
        \put(10,17){\vector(1,0){47}}
        \put(57,14){\vector(-1,0){47}}
        \put(114,17){\vector(1,0){47}}
        \put(161,14){\vector(-1,0){47}}
        \put(240,17){\vector(1,0){47}}
        \put(287,14){\vector(-1,0){47}}
      \end{picture}
    \end{center}
   which are given by
     \[
       d^\pm_{s,r}=\frac{d}{d\rho}\pm s\rho\;,\quad
       \delta^\pm_{s,r}=-\frac{d}{d\rho}-(n-2r-1)\rho^{-1}\pm s\rho\;,
     \]
  using the canonical identities
    \[
      C^\infty(\R_+)\,\gamma\equiv C^\infty(\R_+)\,d\rho\wedge\gamma\equiv C^\infty(\R_+)\;.
    \]
\end{lem}

According to Sections~\ref{ss:1st type} and~\ref{ss:2nd type}, $\gamma$ can be used to define the following domains of first and second types:
  \begin{alignat*}{2}
    \EE_{\gamma,1}^r&=\SS_{\text{\rm ev},+}\,\gamma&\quad\text{for}\quad r&\le\frac{n-1}{2}\;,\\
    \EE_{\gamma,2}^r&=\rho^{-n+2r+2}\,\SS_{\text{\rm ev},+}\,\gamma&\quad\text{for}\quad r&\ge\frac{n-3}{2}\;,\\
    \EE_{\gamma,1}^{r+1}&=\rho\,\SS_{\text{\rm ev},+}\,d\rho\wedge\gamma&\quad\text{for}\quad r&\le\frac{n+1}{2}\;,\\
    \EE_{\gamma,2}^{r+1}&=\rho^{-n+2r+1}\,\SS_{\text{\rm ev},+}\,d\rho\wedge\gamma&\quad\text{for}\quad r&\ge\frac{n-1}{2}\;.
  \end{alignat*}
The following is a direct consequence of Lemma~\ref{l:complexes 1,2}.

\begin{lem}\label{l:EE_gamma}
  For any $s\ge0$, $d^\pm_s$ and $\delta^\pm_s$ define maps
    \begin{center}
      \begin{picture}(225,32) 
        \put(0,13){$0$}
        \put(64,13){$\EE_{\gamma,i}^r$}
        \put(138,13){$\EE_{\gamma,i}^{r+1}$}
        \put(217,13){$0\;,$}
        \put(21,22){\Small$d^\pm_{s,r-1}$}
        \put(21,3){\Small$\delta^\pm_{s,r-1}$}
        \put(100,22){\Small$d^\pm_{s,r}$}
        \put(100,3){\Small$\delta^\pm_{s,r}$}
        \put(174,22){\Small$d^\pm_{s,r+1}$}
        \put(174,3){\Small$\delta^\pm_{s,r+1}$}
        \put(10,17){\vector(1,0){47}}
        \put(57,14){\vector(-1,0){47}}
        \put(84,17){\vector(1,0){47}}
        \put(131,14){\vector(-1,0){47}}
        \put(163,17){\vector(1,0){47}}
        \put(210,14){\vector(-1,0){47}}
     \end{picture}
    \end{center}
  where $i=1$ if $r\le\frac{n-1}{2}$, and $i=2$ if $r\ge\frac{n-1}{2}$.
\end{lem}

\begin{rem}\label{r:EE_gamma,i if n is odd}
   If $n$ is odd, by Lemma~\ref{l:complexes 1,2} and~\eqref{L^2Omega^r(M)}, and since $\SS_{\text{\rm ev},+}\subset L^2(\R_+,\rho^{2\sigma}\,d\rho)$ if and only if $\sigma>-1/2$, we get
  \begin{alignat*}{2}
    d^\pm_s(\EE_{\gamma,2}^r)&\not\subset L^2\Omega^{r+1}(M)
    &\quad\text{for}\quad r&=\frac{n-3}{2}\;,\\
    \delta^\pm_s(\EE_{\gamma,1}^{r+1})&\not\subset L^2\Omega^r(M)
    &\quad\text{for}\quad r&=\frac{n+1}{2}\;.
  \end{alignat*}
This is compatible with $\D_s^+\not\ge0$ on $\EE_{\gamma,2}^r$ when $r=\frac{n-3}{2}$ (Section~\ref{ss:1st type}), and $\D_s^-\not\ge0$ on $\EE_{\gamma,1}^{r+1}$ when $r=\frac{n+1}{2}$ (Section~\ref{ss:2nd type}).
\end{rem}

\begin{rem}\label{r:EE_gamma,i if n is even}
If $n$ is even, notice that
  \begin{alignat*}{2}
        \EE_{\gamma,1}^r&=\EE_{\gamma,2}^r=\SS_{\text{\rm ev},+}\,\gamma&\quad\text{for}\quad r&=\frac{n}{2}-1\;,\\
        \EE_{\gamma,1}^{r+1}&=\EE_{\gamma,2}^{r+1}=\rho\,\SS_{\text{\rm ev},+}\,d\rho\wedge\gamma&\quad\text{for}\quad r&=\frac{n}{2}\;.
  \end{alignat*}
\end{rem}

The domains of first and second type are summarized in Tables~\ref{table: domains 1st & 2nd type, n even} and~\ref{table: domains 1st & 2nd type, n odd}, omitting the differential form part. Grey ground color is used for the repeated terms, and grey color for the terms that are not mapped to $L^2\Omega(M)$ by $d^\pm_s$ or $\delta^\pm_s$.

\begin{table}[h]
\renewcommand{\arraystretch}{1.3}
\begin{tabular}{cc|c|c|c|c|}
\cline{3-6}
&& $\EE_{\gamma,1}^r$ & $\EE_{\gamma,2}^r$ & $\EE_{\gamma,1}^{r+1}$ & $\EE_{\gamma,2}^{r+1}$ \\
\hline
\multicolumn{1}{|c}{\multirow{8}{*}{$r$}} & \multicolumn{1}{|c|}{$n$} && $\rho^{n+2}\,\SS_{\text{\rm ev},+}$ && $\rho^{n+1}\,\SS_{\text{\rm ev},+}$ \\
\multicolumn{1}{|c}{} & \multicolumn{1}{|c|}{$\vdots$} &&\vdots&& \vdots \\
\multicolumn{1}{|c}{} & \multicolumn{1}{|c|}{$\frac{n}{2}+1$} && $\rho^4\,\SS_{\text{\rm ev},+}$ && $\rho^3\,\SS_{\text{\rm ev},+}$ \\
\cline{2-6}
\multicolumn{1}{|c}{} & \multicolumn{1}{|c|}{$\frac{n}{2}$} && $\rho^2\,\SS_{\text{\rm ev},+}$ & \cellcolor{lightgray} $\rho\,\SS_{\text{\rm ev},+}$ & \cellcolor{lightgray} $\rho\,\SS_{\text{\rm ev},+}$ \\
\cline{2-6}
\multicolumn{1}{|c}{} & \multicolumn{1}{|c|}{$\frac{n}{2}-1$} & \cellcolor{lightgray} $\SS_{\text{\rm ev},+}$ & \cellcolor{lightgray} $\SS_{\text{\rm ev},+}$ & $\rho\,\SS_{\text{\rm ev},+}$ & \\
\cline{2-6}
\multicolumn{1}{|c}{} & \multicolumn{1}{|c|}{$\frac{n}{2}-2$} & $\SS_{\text{\rm ev},+}$ && $\rho\,\SS_{\text{\rm ev},+}$ & \\
\multicolumn{1}{|c}{} & \multicolumn{1}{|c|}{\vdots} & \vdots && \vdots & \\
\multicolumn{1}{|c}{} & \multicolumn{1}{|c|}{$0$} & $\SS_{\text{\rm ev},+}$ && $\rho\,\SS_{\text{\rm ev},+}$ & \\
\hline
\end{tabular}
\caption{Domains of first and second type when $n$ is even}
\label{table: domains 1st & 2nd type, n even}
\end{table} 

\begin{table}[h]
\renewcommand{\arraystretch}{1.3}
\begin{tabular}{cc|c|c|c|c|}
\cline{3-6}
&& $\EE_{\gamma,1}^r$ & $\EE_{\gamma,2}^r$ & $\EE_{\gamma,1}^{r+1}$ & $\EE_{\gamma,2}^{r+1}$ \\
\hline
\multicolumn{1}{|c}{\multirow{9}{*}{$r$}} & \multicolumn{1}{|c|}{$n$} && $\rho^{n+2}\,\SS_{\text{\rm ev},+}$ && $\rho^{n+1}\,\SS_{\text{\rm ev},+}$ \\
\multicolumn{1}{|c}{} & \multicolumn{1}{|c|}{$\vdots$} &&\vdots&& \vdots \\
\multicolumn{1}{|c}{} & \multicolumn{1}{|c|}{$\frac{n+3}{2}$} && $\rho^5\,\SS_{\text{\rm ev},+}$ && $\rho^4\,\SS_{\text{\rm ev},+}$ \\
\cline{2-6}
\multicolumn{1}{|c}{} & \multicolumn{1}{|c|}{$\frac{n+1}{2}$} && $\rho^3\,\SS_{\text{\rm ev},+}$ & \color{lightgray} $\rho\,\SS_{\text{\rm ev},+}$ & $\rho^2\,\SS_{\text{\rm ev},+}$ \\
\cline{2-6}
\multicolumn{1}{|c}{} & \multicolumn{1}{|c|}{$\frac{n-1}{2}$} & $\SS_{\text{\rm ev},+}$ & $\rho\,\SS_{\text{\rm ev},+}$ & $\rho\,\SS_{\text{\rm ev},+}$ & $\SS_{\text{\rm ev},+}$ \\
\cline{2-6}
\multicolumn{1}{|c}{} & \multicolumn{1}{|c|}{$\frac{n-3}{2}$} & $\SS_{\text{\rm ev},+}$ & \color{lightgray} $\rho^{-1}\,\SS_{\text{\rm ev},+}$ & $\rho\,\SS_{\text{\rm ev},+}$ & \\
\cline{2-6}
\multicolumn{1}{|c}{} & \multicolumn{1}{|c|}{$\frac{n-5}{2}$} & $\SS_{\text{\rm ev},+}$ && $\rho\,\SS_{\text{\rm ev},+}$ & \\
\multicolumn{1}{|c}{} & \multicolumn{1}{|c|}{\vdots} & \vdots && \vdots & \\
\multicolumn{1}{|c}{} & \multicolumn{1}{|c|}{$0$} & $\SS_{\text{\rm ev},+}$ && $\rho\,\SS_{\text{\rm ev},+}$ & \\
\hline
\end{tabular}
\caption{Domains of first and second type when $n$ is odd}
\label{table: domains 1st & 2nd type, n odd}
\end{table}

By Lemma~\ref{l:EE_gamma}, $\EE_{\gamma,i}=\EE_{\gamma,i}^r\oplus\EE_{\gamma,i}^{r+1}$ is a subcomplex of length one of $\Omega(M)$ with $d^\pm_s$ and $\delta^\pm_s$, even for $s=0$, where $i=1$ for $r\le\frac{n-1}{2}$, and $i=2$ for $r\ge\frac{n-1}{2}$. Moreover let $\EE_{\gamma,0}$ denote the dense subcomplex of $\EE_{\gamma,i}$ defined by
  \begin{gather*}
    \EE_{\gamma,0}^r=C^\infty_0(\R_+)\,\gamma\equiv C^\infty_0(\R_+)\;,\\
    \EE_{\gamma,0}^{r+1}=C^\infty_0(\R_+)\,d\rho\wedge\gamma\equiv C^\infty_0(\R_+)\;.
  \end{gather*}
The closure of $\EE_{\gamma,i}$ (and $\EE_{\gamma,0}$) in $L^2\Omega(M)$ is denoted by $L^2\EE_\gamma$. We have
  \begin{gather*}
    L^2\EE_\gamma^r=L^2(\R_+,\rho^{n-2r-1}\,d\rho)\,\gamma\equiv L^2(\R_+,\rho^{n-2r-1}\,d\rho)\;,\\
    L^2\EE_\gamma^{r+1}=L^2(\R_+,\rho^{n-2r-1}\,d\rho)\,d\rho\wedge\gamma\equiv L^2(\R_+,\rho^{n-2r-1}\,d\rho)\;.
  \end{gather*}
  
Assume now that $s>0$. With the notation of Section~\ref{ss:complex 1}, consider the real version of the elliptic complex $(E,d)$ determined by the constants $s$ and
  \begin{equation}\label{kappa=frac n-2r-1 2}
    \kappa=\frac{n-2r-1}{2}\;,
  \end{equation}
and also its subcomplexes $\EE_i$, where $i=1$ if $\kappa>-1/2$ ($r\le\frac{n-1}{2}$), and $i=2$ if $\kappa<1/2$ ($r\ge\frac{n-1}{2}$).
  
\begin{prop}\label{p:EE_gamma,i}
  There is a unitary isomorphism $L^2\EE_\gamma\to L^2(E)$, which restricts to isomorphisms of complexes up to a shift of degree, $(\EE_{\gamma,0},d^\pm_s)\to(C^\infty_0(E),d)$ and $(\EE_{\gamma,i},d^\pm_s)\to(\EE_i,d)$, where $i=1$ if $r\le\frac{n-1}{2}$, and $i=2$ if $r\ge\frac{n-1}{2}$.
\end{prop}

\begin{proof}
  The unitary isomorphism
    \[
      \rho^\kappa:L^2(\R_+,\rho^{n-2r-1}\,d\rho)\to L^2(\R_+,d\rho)
    \]
  defines a unitary isomorphism $L^2\EE_\gamma\to L^2(E)$, which restricts to an isomorphism $\EE_{\gamma,0}\to C^\infty_0(E)$. Furthermore
    \begin{gather*}
      \rho^\kappa\EE_{\gamma,1}^r=\rho^\kappa\SS_{\text{\rm ev},+}\,\gamma\equiv\rho^\kappa\SS_{\text{\rm ev},+}\equiv\EE_1^0\;,\\
      \rho^\kappa\EE_{\gamma,1}^{r+1}=\rho^{1+\kappa}\,\SS_{\text{\rm ev},+}\,d\rho\wedge\gamma\equiv\rho^{1+\kappa}\SS_{\text{\rm ev},+}\equiv\EE_1^1
    \end{gather*}
  if $r\le\frac{n-1}{2}$, and
    \begin{gather*}
      \rho^\kappa\EE_{\gamma,2}^r=\rho^{\kappa-n+2r+2}\,\SS_{\text{\rm ev},+}\,\gamma\equiv\rho^{1-\kappa}\,\SS_{\text{\rm ev},+}\equiv\EE_2^0\;,\\
      \rho^\kappa\EE_{\gamma,2}^{r+1}=\rho^{\kappa-n+2r+1}\,\SS_{\text{\rm ev},+}\,\gamma\equiv\rho^{-\kappa}\,\SS_{\text{\rm ev},+}\equiv\EE_2^1
    \end{gather*}
  if $r\ge\frac{n-1}{2}$. By Lemma~\ref{l:complexes 1,2} and~\eqref{[d/d rho,rho^a right]}, we also have
    \[
      \rho^\kappa\,d^\pm_{s,r}\,\rho^{-\kappa}=\rho^\kappa\left(\frac{d}{d\rho}\pm s\rho\right)\rho^{-\kappa}=\frac{d}{d\rho}-\kappa\rho^{-1}\pm s\rho\;,
    \]
  which is the operator $d$ of Section~\ref{ss:complex 1}.
\end{proof}

\begin{cor}\label{c:EE_gamma,i}
    \begin{itemize}
    
      \item[(i)] If $r\ne\frac{n-1}{2}$, then $(\EE_{\gamma,0},d^\pm_s)$ has a unique Hilbert complex extension in $L^2\EE_\gamma$, whose smooth core is $\EE_{\gamma,i}$, where $i=1$ if $r<\frac{n-1}{2}$, and $i=2$ if $r>\frac{n-1}{2}$.
      
      \item[(ii)] If $r=\frac{n-1}{2}$, then $(\EE_{\gamma,0},d^\pm_s)$ has different minimum and maximum Hilbert complex extensions in $L^2\EE_\gamma$, whose smooth cores are $\EE_{\gamma,2}$ and $\EE_{\gamma,1}$, respectively.
    
    \end{itemize}
\end{cor}

\begin{proof}
  This follows from Propositions~\ref{p:2} and~\ref{p:EE_gamma,i}.
\end{proof}

For each degree $r$, we will choose one of the possible domains of first and second type defined by $\gamma$, denoted by $\EE_\gamma^r$ and $\EE_\gamma^{r+1}$, so that $\EE_\gamma=\EE_{\gamma}^r\oplus\EE_{\gamma}^{r+1}$ is a subcomplex of $(\Omega(M),d^\pm_s)$ according to Lemma~\ref{l:EE_gamma}.

If $n$ is even, there is only one choice of domains of first and second types by Remark~\ref{r:EE_gamma,i if n is even}. Thus $\EE_{\gamma}^r$ and $\EE_{\gamma}^{r+1}$ have only one possible definition in this case.

If $n$ is odd, there are two possible choices of domains of first and second types just for the following values of $r$: 
  \begin{alignat*}{2}
    &\left.
      \begin{aligned}
        \EE_{\gamma,1}^r&=\SS_{\text{\rm ev},+}\,\gamma\\
        \EE_{\gamma,2}^r&=\rho^{-1}\,\SS_{\text{\rm ev},+}\,\gamma
      \end{aligned}
    \right\}
    &\quad\text{for}\quad r&=\frac{n-3}{2}\;,\\
    &\left.
      \begin{aligned}
        \EE_{\gamma,1}^r&=\SS_{\text{\rm ev},+}\,\gamma\\
        \EE_{\gamma,2}^r&=\rho\,\SS_{\text{\rm ev},+}\,\gamma\\
        \EE_{\gamma,1}^{r+1}&=\rho\,\SS_{\text{\rm ev},+}\,d\rho\wedge\gamma\\
        \EE_{\gamma,2}^{r+1}&=\SS_{\text{\rm ev},+}\,d\rho\wedge\gamma
      \end{aligned}
    \right\}
    &\quad\text{for}\quad r&=\frac{n-1}{2}\;,\\
    &\left.
      \begin{aligned}
        \EE_{\gamma,1}^{r+1}&=\rho\,\SS_{\text{\rm ev},+}\,d\rho\wedge\gamma\\
        \EE_{\gamma,2}^{r+1}&=\rho^2\,\SS_{\text{\rm ev},+}\,d\rho\wedge\gamma
      \end{aligned}
    \right\}
    &\quad\text{for}\quad r&=\frac{n+1}{2}\;.
  \end{alignat*}
By Remark~\ref{r:EE_gamma,i if n is odd} and Corollary~\ref{c:EE_gamma,i}, we choose
  \begin{alignat*}{2}
    \EE_{\gamma}^r&=\EE_{\gamma,1}^r&\quad\text{for}\quad r&=\frac{n-3}{2}\;,\\
    \EE_{\gamma}^{r+1}&=\EE_{\gamma,2}^{r+1}&\quad\text{for}\quad r&=\frac{n+1}{2}\;.
  \end{alignat*}

In order to get the minimum and maximum i.b.c.\ of $(\bigwedge TM^*,d)$, according to Corrollary~\ref{c:EE_gamma,i}, we choose
  \[
  \left.
  \begin{alignedat}{2}
    &\left.
      \begin{aligned}
        \EE_{\gamma}^r&=\EE_{\gamma,2}^r\\
        \EE_{\gamma}^{r+1}&=\EE_{\gamma,2}^{r+1}
      \end{aligned}
    \right\}
    &\quad\text{if}\quad\gamma&\in\widetilde{\HH}_{\text{\rm min}}^r\\
    &\left.
      \begin{aligned}
        \EE_{\gamma}^r&=\EE_{\gamma,1}^r\\
        \EE_{\gamma}^{r+1}&=\EE_{\gamma,1}^{r+1}
      \end{aligned}
    \right\}
    &\quad\text{if}\quad\gamma&\in\widetilde{\HH}_{\text{\rm max}}^r
  \end{alignedat}
  \right\}\quad\text{for}\quad r=\frac{n-1}{2}\;.
  \]

According to Corollary~\ref{c:EE_gamma,i}, the above choices of $\EE_\gamma$ satisfy the following.

\begin{cor}\label{c:EE_gamma}
    \begin{itemize}
    
      \item[(i)] If $r\ne\frac{n-1}{2}$, then $(\EE_{\gamma,0},d^\pm_s)$ has a unique Hilbert complex extension in $L^2\EE_\gamma$, whose smooth core is $\EE_{\gamma}$.
      
      \item[(ii)] If $r=\frac{n-1}{2}$, then $(\EE_{\gamma,0},d^\pm_s)$ has different minimum and maximum Hilbert complex extensions in $L^2\EE_\gamma$. If $\gamma\in \widetilde{\HH}_{\text{\rm min/max}}$, then $\EE_\gamma$ is the smooth core of the minimum/maximum Hilbert complex extension of $(\EE_{\gamma,0},d^\pm_s)$.
    
    \end{itemize}
\end{cor}

Let $(\DD_\gamma,\mathbf{d}^\pm_{s,\gamma})$ denote the Hilbert complex extension of $(\EE_{\gamma,0},d^\pm_s)$ with smooth core $\EE_\gamma$, let $\mathbf{\D}^\pm_{s,\gamma}$ be the corresponding Laplacian, and let $\HH^\pm_{s,\gamma}=\HH_{s,\gamma}^{\pm,r}\oplus\HH_{s,\gamma}^{\pm,r+1}=\ker\mathbf{\D}^\pm_{s,\gamma}$. The following result follows from Sections~\ref{ss:1st type} and~\ref{ss:2nd type}, Lemma~\ref{l:chi_s,a,sigma,0} and the choices made to define $\EE_\gamma$.

\begin{prop}\label{l:mathbf D^pm_s,gamma}
    \begin{itemize}
    
      \item[(i)] $(\DD_\gamma,\mathbf{d}^\pm_{s,\gamma})$ is discrete.
      
      \item[(ii)] $\HH_{s,\gamma}^{+,r+1}=0$, $\dim\HH_{s,\gamma}^{+,r}=1$ if
        \[
          r\le
            \begin{cases}
              \frac{n}{2}-1 & \text{if $n$ is even}\\
              \frac{n-3}{2} & \text{if $n$ is odd and $\gamma\in\widetilde{\HH}_{\text{\rm min}}^r$}\\
              \frac{n-1}{2} & \text{if $n$ is odd and $\gamma\in\widetilde{\HH}_{\text{\rm max}}^r$}\;,
            \end{cases}
        \]
      and $\HH_{s,\gamma}^{+,r}=0$ otherwise.
      
      \item[(iii)] $\HH_{s,\gamma}^{-,r}=0$, $\dim\HH_{s,\gamma}^{-,r+1}=1$ if
        \[
          r\ge
            \begin{cases}
              \frac{n}{2} & \text{if $n$ is even}\\
              \frac{n-1}{2} & \text{if $n$ is odd and $\gamma\in\widetilde{\HH}_{\text{\rm min}}^r$}\\
              \frac{n+1}{2} & \text{if $n$ is odd and $\gamma\in\widetilde{\HH}_{\text{\rm max}}^r$}\;,
            \end{cases}
        \]
      and $\HH_{s,\gamma}^{-,r+1}=0$ otherwise.
      
      \item[(iv)] If $e_s^\pm\in\HH^\pm_{s,\gamma}$ with norm one for each $s$, and $h$ is a bounded measurable function on $\R_+$ with $h(\rho)\to1$ as $\rho\to0$, then $\langle he_s^\pm,e_s^\pm\rangle\to1$ as $s\to\infty$.
      
      \item[(v)] All non-zero eigenvalues of $\mathbf{\D}^\pm_{s,\gamma}$ are in $O(s)$ as $s\to\infty$.
      
    \end{itemize}
\end{prop}

\subsection{Subomplexes defined by domains of third, fourth and fifth types}
\label{ss:subcomplexes 3,4,5}

Consider the notation of Sections~\ref{ss:3rd type}--\ref{ss:5th type}. The following result follows from Corollary~\ref{c:d_s^pm, delta_s^pm}.

\begin{lem}\label{l:complexes 3,4,5}
  For $s\ge0$, $d^\pm_s$ and $\delta^\pm_s$ define maps
    \begin{center}
      \begin{picture}(302,67) 
        \put(0,48){$0$}
        \put(64,48){$C^\infty(\R_+)\,\beta$}
        \put(168,48){$C^\infty(\R_+)\,\alpha+C^\infty(\R_+)\,d\rho\wedge\beta$}
        \put(168,13){$C^\infty(\R_+)\,d\rho\wedge\alpha$}
        \put(296,13){$0\;,$}
        \put(21,57){\Small$d^\pm_{s,r-2}$}
        \put(21,38){\Small$\delta^\pm_{s,r-2}$}
        \put(125,57){\Small$d^\pm_{s,r-1}$}
        \put(125,38){\Small$\delta^\pm_{s,r-1}$}
        \put(130,22){\Small$d^\pm_{s,r}$}
        \put(130,3){\Small$\delta^\pm_{s,r}$}
        \put(251,22){\Small$d^\pm_{s,r+1}$}
        \put(251,3){\Small$\delta^\pm_{s,r+1}$}
        \put(10,52){\vector(1,0){47}}
        \put(57,49){\vector(-1,0){47}}
        \put(114,52){\vector(1,0){47}}
        \put(161,49){\vector(-1,0){47}}
        \put(114,17){\vector(1,0){47}}
        \put(161,14){\vector(-1,0){47}}
        \put(240,17){\vector(1,0){47}}
        \put(287,14){\vector(-1,0){47}}
      \end{picture}
    \end{center}
  which are given by
    \begin{align*}
      d^\pm_{s,r-1}&=
        \begin{pmatrix}
          \mu\\
          \frac{d}{d\rho}\pm s\rho
        \end{pmatrix}\;,\\
      \delta^\pm_{s,r-1}&=
        \begin{pmatrix}
          \mu\rho^{-2} & -\frac{d}{d\rho}-(n-2r+1)\rho^{-1}\pm s\rho
        \end{pmatrix}\;,\\
      d^\pm_{s,r}&=
        \begin{pmatrix}
          \frac{d}{d\rho}\pm s\rho & -\mu
        \end{pmatrix}\;,\\
      \delta^\pm_{s,r}&=
        \begin{pmatrix}
          -\frac{d}{d\rho}-(n-2r-1)\rho^{-1}\pm s\rho\\
          -\mu\rho^{-2}
        \end{pmatrix}\;,
    \end{align*}
  according to the canonical identities
    \begin{gather*}
      C^\infty(\R_+)\,\beta\equiv C^\infty(\R_+)\,d\rho\wedge\alpha\equiv C^\infty(\R_+)\;,\\
      C^\infty(\R_+)\,\alpha+C^\infty(\R_+)\,d\rho\wedge\beta\equiv C^\infty(\R_+)\oplus C^\infty(\R_+)\;.
    \end{gather*}
\end{lem}

Consider only the choices of $a$ given by the positive square roots in~\eqref{3rd type, a} and~\eqref{4th type, a} for domains of third and fourth types, and~\eqref{5th type, a with +sqrt} for domains of fifth type; the other choices of $a$ are rejected because they are very restrictive on $\mu$ and $r$, and give rise to some negative eigenvalues. If these values of $a$ are denoted by $a_3$, $a_4$ and $a_5$ according to the types of domains, then $a_5=a_3=a_4-1$, and therefore the notation $a_5=a_3=a$ and $a_4=a+1$ will be used. Recall also that we only have the choice~\eqref{5th type, b with +sqrt} for $b$, which equals $a+2$. So we only consider the following domains of third, fourth and fifth types defined by $\alpha$ and $\beta$:
  \begin{gather*}
    \FF_{\alpha,\beta}^{r-1}=\rho^a\,\SS_{\text{\rm ev},+}\,\beta\equiv\rho^a\,\SS_{\text{\rm ev},+}\;,\\
    \FF_{\alpha,\beta}^{r+1}=\rho^{a+1}\,\SS_{\text{\rm ev},+}\,d\rho\wedge\alpha\equiv\rho^{a+1}\,\SS_{\text{\rm ev},+}\;,\\
      \begin{aligned}
        \FF_{\alpha,\beta}^r&=\rho^a\,\left\{\,(\phi-c\rho^2\psi)\,\alpha+(c\rho^{-1}\phi+\rho\psi)\,d\rho\wedge\beta\mid\phi,\psi\in\SS_{\text{\rm ev},+}\,\right\}\\
        &\equiv\rho^a\,\left\{\,(\phi-c\rho^2\psi,c\rho^{-1}\phi+\rho\psi)\mid\phi,\psi\in\SS_{\text{\rm ev},+}\,\right\}\;.
    \end{aligned}
  \end{gather*}

\begin{lem}\label{l:FF_alpha,beta}
  For any $s\ge0$, $d^\pm_s$ and $\delta^\pm_s$ define maps
    \begin{center}
      \begin{picture}(308,32) 
        \put(0,13){$0$}
        \put(64,13){$\FF_{\alpha,\beta}^{r-1}$}
        \put(145,13){$\FF_{\alpha,\beta}^r$}
        \put(224,13){$\FF_{\alpha,\beta}^{r+1}$}
        \put(305,13){$0$}
        \put(21,22){\Small$d^\pm_{s,r-2}$}
        \put(21,3){\Small$\delta^\pm_{s,r-2}$}
        \put(102,22){\Small$d^\pm_{s,r-1}$}
        \put(102,3){\Small$\delta^\pm_{s,r-1}$}
        \put(186,22){\Small$d^\pm_{s,r}$}
        \put(186,3){\Small$\delta^\pm_{s,r}$}
        \put(262,22){\Small$d^\pm_{s,r+1}$}
        \put(262,3){\Small$\delta^\pm_{s,r+1}$}
        \put(10,17){\vector(1,0){47}}
        \put(57,14){\vector(-1,0){47}}
        \put(91,17){\vector(1,0){47}}
        \put(138,14){\vector(-1,0){47}}
        \put(170,17){\vector(1,0){47}}
        \put(217,14){\vector(-1,0){47}}
        \put(251,17){\vector(1,0){47}}
        \put(298,14){\vector(-1,0){47}}
      \end{picture}
    \end{center}
\end{lem}

\begin{proof}
  Lemma~\ref{l:complexes 3,4,5} gives $\delta^\pm_s(\FF_{\alpha,\beta}^{r-1})=d^\pm_s(\FF_{\alpha,\beta}^{r+1})=0$.
  
  Observe that
    \begin{equation}\label{a=c mu}
      a=c\mu\;,
    \end{equation}
  obtaining
    \begin{equation}\label{c(a+n-2r)=mu}
      c(a+n-2r)=\mu
    \end{equation}
  by~\eqref{mu c 3+(bar n-2r) c-mu=0}. By Lemma~\ref{l:complexes 3,4,5},~\eqref{a=c mu} and~\eqref{c(a+n-2r)=mu}, for $h\in\SS_{\text{\rm ev},+}$,
    \begin{align}
      d^\pm_s(\rho^ah\,\beta)&=\rho^a\left(\mu h\,\alpha+\left(\frac{d}{d\rho}+c\mu\rho^{-1}\pm s\rho\right)(h)\,d\rho\wedge\beta\right)\;,\label{tilde d b,s pm(rho a 3 r h gamma)}\\
      \delta^\pm_s(\rho^{a+1}h\,d\rho\wedge\alpha)&=\rho^a\left(\left(-\rho\,\frac{d}{d\rho}-\frac{\mu}{c}\pm s\rho^2\right)(h)\,\alpha-\mu\rho^{-1}h\,d\rho\wedge\beta\right)\;.\label{tilde delta b,s pm(rho a 4 r h d rho wedge xi)}
    \end{align}
  The inclusion $d^\pm_s(\FF_{\alpha,\beta}^{r-1})\subset\FF_{\alpha,\beta}^r$ follows from~\eqref{tilde d b,s pm(rho a 3 r h gamma)} if there are $\phi,\psi\in\SS_{\text{\rm ev},+}$ so that
    \begin{align}
      \phi-c\rho^2\psi&=\mu h\;,\label{phi-c r+1 rho 2 psi=mu h}\\
      c\rho^{-1}\phi+\rho\psi&=\left(\frac{d}{d\rho}+c\mu\rho^{-1}\pm s\rho\right)(h)\;.
      \label{c r+1 rho -1 phi+rho psi=(frac d d rho+c r+1 mu rho -1 pm s rho)(h)}
    \end{align}
  Subtract $c\rho^{-2}$ times~\eqref{phi-c r+1 rho 2 psi=mu h} from $\rho^{-1}$ times~\eqref{c r+1 rho -1 phi+rho psi=(frac d d rho+c r+1 mu rho -1 pm s rho)(h)} to get
    \[
      \psi=\frac{1}{1+c^2}\left(\rho^{-1}\frac{d}{d\rho}\pm s\right)(h)\;,
    \]
  which is well defined in $\SS_{\text{\rm ev},+}$. This $\psi$ and $\phi=\mu h+c\rho^2\psi$ satisfy~\eqref{phi-c r+1 rho 2 psi=mu h} and~\eqref{c r+1 rho -1 phi+rho psi=(frac d d rho+c r+1 mu rho -1 pm s rho)(h)}.

  The inclusion $\delta^\pm_s(\FF_{\alpha,\beta}^{r+1})\subset\FF_{\alpha,\beta}^r$ holds by~\eqref{tilde delta b,s pm(rho a 4 r h d rho wedge xi)} if there are $\phi,\psi\in\SS_{\text{\rm ev},+}$ so that
    \begin{align}
      \phi-c\rho^2\psi&=\left(-\rho\,\frac{d}{d\rho}-\frac{\mu}{c}\pm s\rho^2\right)(h)\;,
      \label{phi-c r-1 rho 2 psi=-rho frac d d rho-frac mu c r-1 pm s rho 2)(h)}\\
      c\rho^{-1}\phi+\rho\psi&=-\mu\rho^{-1}h\;.
      \label{c r-1 rho -1 phi+rho psi=-mu rho -1 h}
    \end{align}
  The sum of~\eqref{phi-c r-1 rho 2 psi=-rho frac d d rho-frac mu c r-1 pm s rho 2)(h)} and $c\rho$ times~\eqref{c r-1 rho -1 phi+rho psi=-mu rho -1 h} gives
    \[
      \phi=\frac{1}{1+c^2}\left(-\rho\,\frac{d}{d\rho}-\frac{1+c^2}{c}\,\mu\pm s\rho^2\right)(h)\;,
    \]
  which belongs to $\SS_{\text{\rm ev},+}$. The even extensions of $h$ and $\phi$ to $\R$, also denoted by $h$ and $\phi$, satisfy $c\phi(0)=-\mu h(0)$, and therefore $\mu h+c\phi\in\rho^2\,\SS_{\text{\rm ev}}$. It follows that $\psi=\rho^{-2}(\mu h+c\phi)\in\SS_{\text{\rm ev},+}$. These functions $\phi$ and $\psi$ satisfy~\eqref{phi-c r-1 rho 2 psi=-rho frac d d rho-frac mu c r-1 pm s rho 2)(h)} and~\eqref{c r-1 rho -1 phi+rho psi=-mu rho -1 h}.

  For arbitrary $\phi,\psi\in\SS_{\text{\rm ev},+}$, let
    \begin{equation}\label{zeta}
      \zeta=\rho^a\left(\left(\phi-c\rho^2\psi\right)\,\alpha+\left(c\rho^{-1}\phi+\rho\psi\right)\,d\rho\wedge\beta\right)\;.
    \end{equation}
  By Corollary~\ref{c:d_s^pm, delta_s^pm},~\eqref{a=c mu} and~\eqref{c(a+n-2r)=mu},
    \begin{align*}
      d^\pm_s(\zeta)
      &=\rho^{a+1}\biggl(\left(\rho^{-1}\,\frac{d}{d\rho}\pm s\right)(\phi)\notag\\
      &\phantom{=\rho^{a+1}\biggl(}\text{}+c\left(-\rho\,\frac{d}{d\rho}-\left(\frac{c^2+1}{c}\,\mu+2\right)\mp s\rho^2\right)(\psi)\biggr)\,d\rho\wedge\alpha\;,\\
      \delta^\pm_s(\zeta)&=\rho^a\biggl(c\left(-\rho^{-1}\,\frac{d}{d\rho}\pm s\right)(\phi)\notag\\
      &\phantom{=\rho^a\biggl(}\text{}+\left(-\rho\,\frac{d}{d\rho}-\left(\frac{c^2+1}{c}\,\mu+2\right)\pm s\rho^2\right)(\psi)\biggr)\,\beta\;,
    \end{align*}
  showing that $d^\pm_s(\FF_{\alpha,\beta}^r)\subset\FF_{\alpha,\beta}^{r+1}$ and $\delta^\pm_s(\FF_{\alpha,\beta}^r)\subset\FF_{\alpha,\beta}^{r-1}$.
\end{proof}

By Lemma~\ref{l:FF_alpha,beta}, $\FF_{\alpha,\beta}=\FF_{\alpha,\beta}^{r-1}\oplus\FF_{\alpha,\beta}^r\oplus\FF_{\alpha,\beta}^{r+1}$ is a subcomplex of length two of $\Omega(M)$ with $d^\pm_s$ and $\delta^\pm_s$. Let $\FF_{\alpha,\beta,0}$ denote the dense subcomplex of $\FF_{\alpha,\beta}$ defined by
  \begin{gather*}
    \FF_{\alpha,\beta,0}^{r-1}=C^\infty_0(\R_+)\,\beta\equiv C^\infty_0(\R_+)\;,\quad
    \FF_{\alpha,\beta,0}^{r+1}=C^\infty_0(\R_+)\,d\rho\wedge\alpha\equiv C^\infty_0(\R_+)\;,\\
    \FF_{\alpha,\beta,0}^r=C^\infty_0(\R_+)\,\alpha+C^\infty_0(\R_+)\,d\rho\wedge\beta
    \equiv C^\infty_0(\R_+)\oplus C^\infty_0(\R_+)\;.
  \end{gather*}
The closure of $\FF_{\alpha,\beta}$ (and $\FF_{\alpha,\beta,0}$) in $L^2\Omega(M)$ is denoted by $L^2\FF_{\alpha,\beta}$. We have
  \begin{gather*}
    L^2\FF_{\alpha,\beta}^{r-1}=L^2(\R_+,\rho^{n-2r+1}\,d\rho)\,\beta\equiv L^2(\R_+,\rho^{n-2r+1}\,d\rho)\;,\\
    L^2\FF_{\alpha,\beta}^{r+1}=L^2(\R_+,\rho^{n-2r-1}\,d\rho)\,d\rho\wedge\alpha\equiv L^2(\R_+,\rho^{n-2r-1}\,d\rho)\;,\\
    \begin{aligned}
      L^2\FF_{\alpha,\beta}^r&=L^2(\R_+,\rho^{n-2r-1}\,d\rho)\,\alpha+L^2(\R_+,\rho^{n-2r+1}\,d\rho)\,d\rho\wedge\beta\\
      &\equiv L^2(\R_+,\rho^{n-2r-1}\,d\rho)\oplus L^2(\R_+,\rho^{n-2r+1}\,d\rho)\;.
    \end{aligned}
  \end{gather*}

Assume now that $s>0$. With the notation of Section~\ref{ss:complex 2}, consider the real version of the elliptic complex $(F,d)$, as well as its subcomplex $\FF_1$, determined by the constants $s$, $c$ and
  \begin{equation}\label{kappa>-1/2}
    \kappa=\frac{-1+\sqrt{(n-2r)^2+4\mu^2}}{2}>-\frac{1}{2}\;.
  \end{equation}
By~\eqref{c_pm},
  \begin{equation}\label{kappa}
    \kappa=c\mu+\frac{n-2r-1}{2}=\frac{\mu}{c}-\frac{n-2r+1}{2}\;.
  \end{equation}
  
\begin{prop}\label{p:FF_alpha,beta}
  There is a unitary isomorphism $L^2\FF_{\alpha,\beta}\to L^2(F)$, which restricts to isomorphisms of complexes up to a shift of degree, $(\FF_{\alpha,\beta},d^\pm_s)\to(\FF_1,d)$ and $(\FF_{\alpha,\beta,0},d^\pm_s)\to(C^\infty_0(F),d)$. 
\end{prop}

\begin{proof}
  As an intermediate step, let
    \begin{gather*}
      \widehat{\FF}_{\alpha,\beta}^{r-1}=\rho\,\FF_{\alpha,\beta}^{r-1}=\rho^{a+1}\,\SS_{\text{\rm ev},+}\;,\quad
      \widehat{\FF}_{\alpha,\beta}^{r+1}=\FF_{\alpha,\beta}^{r+1}=\rho^{a+1}\,\SS_{\text{\rm ev},+}\;,\\
      \widehat{\FF}_{\alpha,\beta}^r=\Theta^{-1}(\FF_{\alpha,\beta}^r)=\rho^a\,\SS_{\text{\rm ev},+}\oplus\rho^{a+2}\,\SS_{\text{\rm ev},+}\;,\\
      \widehat{\FF}_{\alpha,\beta}=\widehat{\FF}_{\alpha,\beta}^{r-1}\oplus\widehat{\FF}_{\alpha,\beta}^r\oplus \widehat{\FF}_{\alpha,\beta}^{r+1}\;,\quad \widehat{\FF}_{\alpha,\beta,0}=\FF_{\alpha,\beta,0}\;,\\
      L^2\widehat{\FF}_{\alpha,\beta}^{r-1}=L^2\widehat{\FF}_{\alpha,\beta}^{r+1}=L^2\FF_{\alpha,\beta}^{r+1}=L^2(\R_+,\rho^{n-2r-1}\,d\rho)\;,\\
      L^2\widehat{\FF}_{\alpha,\beta}^r\equiv L^2(\R_+,\rho^{n-2r-1}\,d\rho)\oplus L^2(\R_+,\rho^{n-2r-1}\,d\rho)\;,\\
      L^2\widehat{\FF}_{\alpha,\beta}=L^2\widehat{\FF}_{\alpha,\beta}^{r-1}\oplus L^2\widehat{\FF}_{\alpha,\beta}^r\oplus L^2\widehat{\FF}_{\alpha,\beta}^{r+1}\;.
    \end{gather*}
  Moreover let $\Xi:L^2\FF_{\alpha,\beta}\to L^2\widehat{\FF}_{\alpha,\beta}$ be the unitary isomorphism defined by
    \[
      \rho:L^2\FF_{\alpha,\beta}^{r-1}\to L^2\widehat{\FF}_{\alpha,\beta}^{r-1}\;,\quad
     \sqrt{1+c^2}\,\Theta^{-1}:L^2\FF_{\alpha,\beta}^r\to L^2\widehat{\FF}_{\alpha,\beta}^r
    \]
  and the identity map $L^2\FF_{\alpha,\beta}^{r+1}\to L^2\widehat{\FF}_{\alpha,\beta}^{r+1}$. It restricts to isomorphisms $\FF_{\alpha,\beta}\to\widehat{\FF}_{\alpha,\beta}$ and $\FF_{\alpha,\beta,0}\to\widehat{\FF}_{\alpha,\beta,0}$. Thus, by Lemma~\ref{l:FF_alpha,beta}, $(\FF_{\alpha,\beta},d^\pm_s)$ induces via $\Xi$ a complex
    \begin{center}
      \begin{picture}(313,22) 
        \put(0,3){$0$}
        \put(64,3){$\widehat{\FF}_{\alpha,\beta}^{r-1}$}
        \put(145,3){$\widehat{\FF}_{\alpha,\beta}^r$}
        \put(224,3){$\widehat{\FF}_{\alpha,\beta}^{r+1}$}
        \put(305,3){$0\;.$}
        \put(21,12){\Small$\hat d^\pm_{s,r-2}$}
        \put(102,12){\Small$\hat d^\pm_{s,r-1}$}
        \put(186,12){\Small$\hat d^\pm_{s,r}$}
        \put(262,12){\Small$\hat d^\pm_{s,r+2}$}
        \put(10,7){\vector(1,0){47}}
        \put(91,7){\vector(1,0){47}}
        \put(170,7){\vector(1,0){47}}
        \put(251,7){\vector(1,0){47}}
       \end{picture}
    \end{center}
  By Lemma~\ref{l:complexes 3,4,5} and~\eqref{[d/d rho,rho^a right]},
    \begin{align}
      \hat d^\pm_{s,r-1}&=\frac{1}{\sqrt{1+c^2}}
        \begin{pmatrix}
          1 &  c\rho\\
          -c & \rho
        \end{pmatrix}
        \begin{pmatrix}
          \mu\\
          \frac{d}{d\rho}\pm s\rho
        \end{pmatrix}
      \rho^{-1}\notag\\
      &=\frac{1}{\sqrt{1+c^2}}
        \begin{pmatrix}
          c\,\frac{d}{d\rho}+(\mu-c)\rho^{-1}\pm cs\rho\\
          \frac{d}{d\rho}-(c\mu+1)\rho^{-1}\pm s\rho
        \end{pmatrix}\;,
      \label{hat d^pm_s,r-1}\\
      \hat d^\pm_{s,r}&=\frac{1}{\sqrt{1+c^2}}
        \begin{pmatrix}
          \frac{d}{d\rho}\pm s\rho & -\mu
        \end{pmatrix}
        \begin{pmatrix}
          1 & -c\\
          c\rho^{-1} & \rho^{-1}
        \end{pmatrix}\notag\\
      &=\frac{1}{\sqrt{1+c^2}}
        \begin{pmatrix}
          \frac{d}{d\rho}-c\mu\rho^{-1}\pm s\rho & -c\,\frac{d}{d\rho}-\mu\rho^{-1}\mp cs\rho
        \end{pmatrix}\;.
      \label{hat d^pm_s,r}
    \end{align}
  
  Now, the unitary isomorphism
    \[
      \rho^{\frac{n-2r-1}{2}}:L^2(\R_+,\rho^{n-2r-1}\,d\rho)\to L^2(\R_+,d\rho)
    \]
  induces a unitary isomorphism $L^2\widehat{\FF}_{\alpha,\beta}\to L^2(F)$, which restricts to isomorphisms $\widehat{\FF}_{\alpha,\beta}\to\FF_1$ and $\widehat{\FF}_{\alpha,\beta,0}\to C^\infty_0(F)$. Moreover, by~\eqref{hat d^pm_s,r-1},~\eqref{hat d^pm_s,r},~\eqref{[d/d rho,rho^a right]} and~\eqref{kappa},
    \begin{multline*}
      \rho^{\frac{n-2r-1}{2}}\,\hat d^\pm_{s,r-1}\,\rho^{-\frac{n-2r-1}{2}}\\
        \begin{aligned}
          &=\frac{1}{\sqrt{1+c^2}}\,\rho^{\frac{n-2r-1}{2}}
            \begin{pmatrix}
              c\,\frac{d}{d\rho}+(\mu-c)\rho^{-1}\pm cs\rho\\
              \frac{d}{d\rho}-(c\mu+1)\rho^{-1}\pm s\rho
            \end{pmatrix}
          \rho^{-\frac{n-2r-1}{2}}\\
          &=\frac{1}{\sqrt{1+c^2}}
            \begin{pmatrix}
              c\,\bigl(\frac{d}{d\rho}+\kappa\rho^{-1}\pm s\rho\bigr)\\
              \frac{d}{d\rho}-(\kappa+1)\rho^{-1}\pm s\rho
            \end{pmatrix}\;,
        \end{aligned}
    \end{multline*}
    \begin{multline*}
      \rho^{\frac{n-2r-1}{2}}\,\hat d^\pm_{s,r}\,\rho^{-\frac{n-2r-1}{2}}\\
        \begin{aligned}
          &=\frac{1}{\sqrt{1+c^2}}\,\rho^{\frac{n-2r-1}{2}}
            \begin{pmatrix}
              \frac{d}{d\rho}-c\mu\rho^{-1}\pm s\rho & -c\,\frac{d}{d\rho}-\mu\rho^{-1}\mp cs\rho
            \end{pmatrix}
          \rho^{-\frac{n-2r-1}{2}}\\
          &=\frac{1}{\sqrt{1+c^2}}
            \begin{pmatrix}
              \frac{d}{d\rho}-\kappa\rho^{-1}\pm s\rho & c\,\bigl(-\frac{d}{d\rho}-(\kappa+1)\rho^{-1}\mp s\rho\bigr)
            \end{pmatrix}\;,
        \end{aligned}
    \end{multline*}
  which are the operators $d_0$ and $d_1$ of Section~\ref{ss:complex 2}.
\end{proof}

\begin{cor}\label{c:FF_alpha,beta}
  $(\FF_{\alpha,\beta,0},d^\pm_s)$ has a unique Hilbert complex extension in $L^2\FF_{\alpha,\beta}$, whose smooth core is $\FF_{\alpha,\beta}$.
\end{cor}

\begin{proof}
  This follows from Propositions~\ref{p:3} and~\ref{p:FF_alpha,beta}.
\end{proof}

Let $(\DD_{\alpha,\beta},\mathbf{d}^\pm_{s,\alpha,\beta})$ denote the unique Hilbert complex extension of $(\FF_{\alpha,\beta,0},d^\pm_s)$, according to Corollary~\ref{c:FF_alpha,beta}, and let $\mathbf{\D}^\pm_{s,\alpha,\beta}$ denote the corresponding Laplacian.  The following result follows from Sections~\ref{ss:3rd type}--\ref{ss:5th type}.

\begin{prop}\label{p:mathbf D^pm_s,alpha,beta}
  \begin{itemize}
  
    \item[(i)] $(\DD_{\alpha,\beta},\mathbf{d}^\pm_{s,\alpha,\beta})$ is discrete.
    
    \item[(ii)] The eigenvalues of $\mathbf{\D}^\pm_{s,\alpha,\beta}$ are positive and in $O(s)$ as $s\to\infty$.
  
  \end{itemize} 
\end{prop}

\subsection{Splitting into subcomplexes}\label{ss:splitting}

Let $\BB_{\text{\rm min/max},0}$ denote an orthonormal frame of $\widetilde{\HH}_{\text{\rm min/max}}$ consisting of homogeneous differential forms. For each positive eigenvalue $\mu$ of $\widetilde{D}_{\text{\rm min/max}}$, let $\BB_{\text{\rm min/max},\mu}$ be an orthonormal frame of $E_\mu(\widetilde{D}_{\text{\rm min/max}})$ consisting of differential forms $\alpha+\beta$ like in Section~\ref{ss:subcomplexes 3,4,5}. Then let
  \[
    \mathbf{d}^\pm_{s,\text{\rm min/max}}=\bigoplus_\gamma\mathbf{d}^\pm_{s,\gamma}\oplus\widehat{\bigoplus_\mu}\bigoplus_{\alpha+\beta}\mathbf{d}^\pm_{s,\alpha,\beta}\;,
  \]
where $\gamma$ runs in $\BB_{\text{\rm min/max},0}$, $\mu$ runs in the positive spectrum of $\widetilde{D}_{\text{\rm min/max}}$, and $\alpha+\beta$ runs in $\BB_{\text{\rm min/max},\mu}$. Observe that the domain of $\mathbf{d}^\pm_{s,\text{\rm min/max}}$ is independent of $s$, and therefore it is denoted by $\DD_{\text{\rm min/max}}$. Let also
  \[
    \GG_{\text{\rm min/max}}=\bigoplus_\gamma\EE_{\gamma,0}\oplus\bigoplus_\mu\bigoplus_{\alpha+\beta}\FF_{\alpha,\beta,0}\;.
  \]

\begin{prop}\label{p:splitting}
  $d^\pm_{s,\text{\rm min/max}}=\mathbf{d}^\pm_{s,\text{\rm min/max}}$.
\end{prop}

\begin{proof}
  By Corollaries~\ref{c:EE_gamma} and~\ref{c:FF_alpha,beta}, Lemma~\ref{l:minimal/maximal Hilbert complex extension} and~\eqref{L^2 Omega^r(N)}, $(\DD_{\text{\rm min/max}},\mathbf{d}^\pm_{s,\text{\rm min/max}})$ is the minimum/maximum Hilbert complex extension of $(\GG_{\text{\rm min/max}},d^\pm_s)$. Then the result easily follows from the following assertions.
  
  \begin{claim}\label{cl:GG_min/max subset DD(d^pm_s,min/max}
    $\GG_{\text{\rm min/max}}\subset\DD(d^\pm_{s,\text{\rm min/max}})$.
  \end{claim}
  
  \begin{claim}\label{cl:Omega_0(M) subset DD_min/max}
    $\Omega_0(M)\subset\DD_{\text{\rm min/max}}$.
  \end{claim}

  Let $\hat d^\pm_{s,\text{\rm min/max}}$ denote the minimum/maximum Hilbert complex extension of $(\Omega_0(M),d^\pm_s)$ with respect to the product metric $\hat g=\tilde g+ (d\rho)^2$ on $M=N\times \R_+$. With the terminology of \cite[p.~110]{BruningLesch1992}, observe that $(\Omega(M),d^\pm_s)$ is the product complex of the de~Rham complex of $N$, $(\Omega(N),\tilde d)$, and the Witten's deformation of the de~Rham complex of $\R_+$, defined by the function $\pm\frac{1}{2}\rho^2$. Then, by \cite[Lemma~3.6 and~(2.38b)]{BruningLesch1992}, 
    \begin{align}
      \DD(\hat d^\pm_{s,\text{\rm min/max}})&\supset C_0^\infty(\R_+)\,\DD(\tilde d_{\text{\rm min/max}})+C_0^\infty(\R_+)\,d\rho\wedge\DD(\tilde d_{\text{\rm min/max}})\notag\\
      &\supset\GG_{\text{\rm min/max}}\;.\label{DD(hat d^pm_s,min/max)}
    \end{align}
  On the other hand, for $0<a<b<\infty$, let $L^2_{a,b}\Omega(M,g)$ and $L^2_{a,b}\Omega(M,\hat g)$ denote the Hilbert subspaces of $L^2\Omega(M,g)$ and $L^2\Omega(M,\hat g)$, respectively, consisting of $L^2$ differential forms supported in $N\times[a,b]$. Since $g$ and $\hat g$ are quasi-isometric on $N\times(a',b')$ for $0<a'<a$ and $b<b'<\infty$, it follows that
    \begin{equation}\label{L^2_a,bOmega(M,g)}
      \DD(d^\pm_{s,\text{\rm min/max}})\cap L^2_{a,b}\Omega(M,g)=\DD(\hat d^\pm_{s,\text{\rm min/max}})\cap L^2_{a,b}\Omega(M,\hat g)\;.
    \end{equation}
  Moreover
    \begin{equation}\label{GG_min/max}
      \GG_{\text{\rm min/max}}\subset\bigcup_{0<a<b<\infty}L^2_{a,b}\Omega(M,g)\;.
    \end{equation}
  Now Claim~\ref{cl:GG_min/max subset DD(d^pm_s,min/max} follows from~\eqref{DD(hat d^pm_s,min/max)}--\eqref{GG_min/max}.
  
  Finally, Claim~\ref{cl:Omega_0(M) subset DD_min/max} follows from
    \begin{equation}\label{Omega_0(M)}
      \Omega_0(M)\subset\bigoplus_\gamma\EE_{\gamma,0}\oplus\widehat{\bigoplus_\mu}\bigoplus_{\alpha+\beta}\FF_{\alpha,\beta,0}\;,
    \end{equation}
  where $\gamma$, $\mu$ and $\alpha+\beta$ vary as above. The inclusion~\eqref{Omega_0(M)} can be proved as follows. According to~\eqref{Omega^r(M) with d rho}, any $\xi\in\Omega_0(M)$ can be written as $\xi=\xi_0+d\rho\wedge\xi_1$ with $\xi_0,\xi_1\in C^\infty_0(\R_+,\Omega_0(N))$. Then, by~\eqref{L^2 Omega^r(N)}, we get functions $f_{k,\gamma},f_{k,\ell,\alpha,\beta}\in C^\infty_0(\R_+)$ ($k,\ell\in\{0,1\}$) defined by $f_{k,\gamma}(\rho)=\langle\xi_k(\rho),\gamma\rangle_{\tilde g}$, $f_{k,0,\alpha,\beta}(\rho)=\langle\xi_k(\rho),\beta\rangle_{\tilde g}$ and $f_{k,1,\alpha,\beta}(\rho)=\langle\xi_k(\rho),\alpha\rangle_{\tilde g}$, where $\langle\ ,\ \rangle_{\tilde g}$ denotes the scalar product of $L^2\Omega(N)$, and moreover
    \begin{multline*}
      \alpha=\sum_\gamma(f_{0,\gamma}\,\gamma+f_{1,\gamma}\,d\rho\wedge\gamma)\\
      \text{}+\sum_\mu\sum_{\alpha+\beta}(f_{0,0,\alpha,\beta}\,\beta+f_{1,0,\alpha,\beta}\,\alpha+f_{1,0,\alpha,\beta}\,d\rho\wedge\beta+f_{1,1,\alpha,\beta}\,d\rho\wedge\alpha)
    \end{multline*}
  in $L^2\Omega(M,g)$, where $\gamma$, $\mu$ and $\alpha+\beta$ vary as above. Thus $\xi$ belongs to the space in the right hand side of~\eqref{Omega_0(M)}.
\end{proof}

\begin{rem}\label{r:h(rho) DD^infty(d^pm_s,min/max)}
  From~\eqref{DD^infty(bd)}, Lemma~\ref{l: |P^k(h phi)|_c_1}, and Propositions~\ref{p:2},~\ref{p:3} and~\ref{p:splitting}, it follows that, with the notation of Example~\ref{ex:rel-admissible function}, $h(\rho)\,\DD^\infty(d^\pm_{s,\text{\rm min/max}})\subset\DD^\infty(d^\pm_{s,\text{\rm min/max}})$ for all $h\in\Cinf(\R_+)$ such that $h'\in\Cinf_0(\R_+)$.
\end{rem}

Let $\HH_{s,\text{\rm min/max}}^\pm=\bigoplus_r\HH_{s,\text{\rm min/max}}^{\pm,r}=\ker\D^\pm_{s,\text{\rm min/max}}$.

\begin{cor}\label{c:d^pm_s,min/max are discrete}
  \begin{itemize}
  
    \item[(i)] $d^\pm_{s,\text{\rm min/max}}$ is discrete.
    
    \item[(ii)] $\HH_{\text{\rm min}}^{+,r}\cong H_{\text{\rm min}}^r(N)$ if
      \[
        r\le
          \begin{cases}
            \frac{n}{2}-1 & \text{if $n$ is even}\\
            \frac{n-3}{2} & \text{if $n$ is odd}\;,
          \end{cases}
      \]
    and $\HH_{\text{\rm min}}^{+,r}=0$ otherwise.
    
    \item[(iii)] $\HH_{\text{\rm max}}^{+,r}\cong H_{\text{\rm max}}^r(N)$ if
      \[
        r\le
          \begin{cases}
            \frac{n}{2}-1 & \text{if $n$ is even}\\
            \frac{n-1}{2} & \text{if $n$ is odd}\;,
          \end{cases}
      \]
    and $\HH_{\text{\rm max}}^{+,r}=0$ otherwise.
    
    \item[(iv)] $\HH_{\text{\rm min}}^{-,r+1}\cong H_{\text{\rm min}}^r(N)$ if
      \[
        r\ge
          \begin{cases}
            \frac{n}{2} & \text{if $n$ is even}\\
            \frac{n-1}{2} & \text{if $n$ is odd}\;,
          \end{cases}
      \]
    and $\HH_{\text{\rm min}}^{+,r+1}=0$ otherwise.
    
    \item[(v)] $\HH_{\text{\rm max}}^{-,r+1}\cong H_{\text{\rm max}}^r(N)$ if
      \[
        r\ge
          \begin{cases}
            \frac{n}{2} & \text{if $n$ is even}\\
            \frac{n+1}{2} & \text{if $n$ is odd}\;,
          \end{cases}
      \]
    and $\HH_{\text{\rm max}}^{+,r+1}=0$ otherwise.
    
    \item[(vi)] If $e_s^\pm\in\HH_{s,\text{\rm min/max}}^\pm$ has norm one for each $s$, and $h$ is a bounded measurable function on $\R_+$ with $h(\rho)\to1$ as $\rho\to0$, then $\langle he_s^\pm,e_s^\pm\rangle\to1$ as $s\to\infty$.
    
    \item[(vii)] Let $0\le\lambda^\pm_{s,\text{\rm min/max},0}\le\lambda^\pm_{s,\text{\rm min/max},1}\le\cdots$ be the eigenvalues of $\D_{s,\text{\rm min/max}}$, repeated according to their multiplicities. Given $k\in\N$, if $\lambda^\pm_{s,\text{\rm min/max},k}>0$ for some $s$, then $\lambda^\pm_{s,\text{\rm min/max},k}\in O(s)$ as $s\to\infty$.
  
  \item[(viii)] There is some $\theta>0$ such that $\liminf_k\lambda^\pm_{s,\text{\rm min/max},k}k^{-\theta}>0$.
  
  \end{itemize}
\end{cor}

\begin{proof}
  For $\gamma$, $\mu$ and $\alpha+\beta$ as above, the spectra of $\D^\pm_s$ on $\EE_\gamma$ and $\FF_{\alpha,\beta}$ is discrete by Propositions~\ref{l:mathbf D^pm_s,gamma}-(i) and~\ref{p:mathbf D^pm_s,alpha,beta}-(i). Moreover the union of all of these spectra has no accumulation points according to Sections~\ref{ss:1st type}--\ref{ss:5th type} and since $\widetilde{\D}_{\text{\rm min/max}}$ has a discrete spectrum. Then~(i) follows by Proposition~\ref{p:splitting}.
  
  Now, properties~(ii)--(vii) follow directly from Propositions~\ref{l:mathbf D^pm_s,gamma},~\ref{p:mathbf D^pm_s,alpha,beta} and~\ref{p:splitting}.
  
  To prove~(viii), let $0\le\tilde\lambda_{\text{\rm min/max},0}\le\tilde\lambda_{\text{\rm min/max},1}\le\cdots$ denote the eigenvalues of $\widetilde{\D}_{\text{\rm min/max}}$, repeated according to their multiplicities, and let $\mu_{\text{\rm min/max},\ell}=\sqrt{\tilde\lambda_{\text{\rm min/max},\ell}}$ for each $\ell\in\N$. Since $N$ satisfies Theorem~\ref{t:spectrum of Delta_min/max}-(ii) with $\tilde g$, there is some $C_0,\tilde\theta>0$ such that
    \begin{equation}\label{tilde lambda_min/max,ell}
      \tilde\lambda_{\text{\rm min/max},\ell}\ge C_0^2\ell^{\tilde\theta}
    \end{equation}
  for all large enough $\ell$. Consider the counting function
    \[
      \fN^\pm_{s,\text{\rm min/max}}(\lambda)=\#\left\{\,k\in\N\mid\lambda^\pm_{s,\text{\rm min/max},k}<\lambda\,\right\}
    \]
  for $\lambda>0$. From~\eqref{eigenvalues, mu=0, without d rho, u=1, a=0}--\eqref{eigenvalues, mu=0, with d rho, u=1, a=-n+2r+1},~\eqref{eigenvalues, 3rd type, a with +},~\eqref{eigenvalues, 4th type, a with +},~\eqref{eigenvalues of Y, b with + sqrt},~\eqref{eigenvalues of X, a with + sqrt} and~\eqref{tilde lambda_min/max,ell}, and the choices made in Section~\ref{s:splitting}, it follows that there are some $C_1,C_2,C_3>0$ such that
    \begin{align*}
      \fN^\pm_{s,\text{\rm min/max}}(\lambda)
      &\le\#\left\{\,(k,\ell)\in\N^2\mid C_1k+C_2\,\mu_{\text{\rm min/max},\ell}\le\lambda\,\right\}\\
      &\le\#\{\,(k,\ell)\in\N^2\mid C_1k+C_2C_0\ell^{\tilde\theta/2}-C_3\le\lambda\,\}\\
      &\le\#\left\{\,(k,\ell)\in\N^2\;\Bigg|\;\ell\le\left(\frac{\lambda+C_3}{C_2C_0}-\frac{C_1k}{C_2C_0}\right)^{2/\tilde\theta}\,\right\}\\
      &\le\int_0^{\frac{\lambda+C_3}{C_1}}\left(\frac{\lambda+C_3}{C_2C_0}-\frac{C_1x}{C_2C_0}\right)^{2/\tilde\theta}\,dx\\
      &=\frac{\tilde\theta(\lambda+C_3)^{(2+\tilde\theta)/\tilde\theta}}{(2+\tilde\theta)(C_2C_0)^{2/\tilde\theta}C_1}\;.
    \end{align*}
  So $\fN^\pm_{s,\text{\rm min/max}}(\lambda)\le C\lambda^{(2+\tilde\theta)/\tilde\theta}$ for some $C>0$ and all large enough $\lambda$, giving~(viii) with $\theta=\frac{\tilde\theta}{2+\tilde\theta}$.
\end{proof}

\begin{ex}\label{ex:d^pm_s and d^pm_0,s}
  Consider the notation of Examples~\ref{ex:c(S^m-1)=R^m},~\ref{ex:g_0} and~\ref{ex:d^pm_0,s}. On the stratum $\S^{m-1}\times\R_+$ of $c(\S^{m-1})$, the model rel-Morse function $\pm\frac{1}{2}\,\rho^2$ and the metric $g_1$ define the Witten's perturbed operators $d^\pm_s$, $\delta^\pm_s$, $D^\pm_s$ and $\D^\pm_s$. Since $\rho_0$ and $g_0$ respectively correspond to $\rho$ and $g_1$ by $\can:\S^{m-1}\times\R_+\to\R^m\sm\{0\}$, it follows that $d^\pm_s$, $\delta^\pm_s$, $D^\pm_s$ and $\D^\pm_s$ respectively correspond to $d^\pm_{0,s}$, $\delta^\pm_{0,s}$, $D^\pm_{0,s}$, $\D^\pm_{0,s}$ by $\can^*:\Omega(\R^m\sm\{0\})\to\Omega(\S^{m-1}\times\R_+)$, and moreover
  \begin{equation}\label{can^*}
    \begin{CD}
      L^2\Omega(\R^m,g_0)\equiv L^2\Omega(\R^m\sm\{0\},g_0) @>{\can^*}>> L^2\Omega(\S^{m-1}\times\R_+,g_1)
    \end{CD}
  \end{equation}
is a unitary isomorphism. The extension by zero defines a canonical injection $\Omega_0(\R^m\sm\{0\})\to\Omega_0(\R^m)$, whose composite with $(\can^*)^{-1}$ is an injective homomorphism of complexes, $(\Omega_0(\S^{m-1}\times\R_+),d^\pm_s)\to(\Omega_0(\R^m),d^\pm_{0,s})$. Thus the unique i.b.c.\  of  $(\bigwedge{T\R^m}^*,d^\pm_{0,s})$ in $L^2\Omega(\R^m,g_0)$ corresponds to $d^\pm_{s,\text{\rm max}}$ via~\eqref{can^*}. 

If $m\ge2$, then $H^{\frac{m-1}{2}}(\S^{m-1})=0$ for odd $m$. So $(\bigwedge T(\S^{m-1}\times\R_+)^*,d^\pm_s)$ has a unique i.b.c.\ by Corollaries~\ref{c:EE_gamma} and~\ref{c:FF_alpha,beta}, and Proposition~\ref{p:splitting}. 

If $m=1$, then $\Omega(\S^0)=\Omega^0(\S^0)\equiv\R^2$, and therefore, according to~\eqref{Omega^r(M) with d rho},~\eqref{Omega^r(M)} and Corollary~\ref{c:d_s^pm, delta_s^pm},
  \begin{gather*}
    \Omega^0(\S^0\times\R_+)\equiv\Cinf(\R_+,\R^2)\;,\\
    \Omega^1(\S^0\times\R_+)\equiv d\rho\wedge\Cinf(\R_+,\R^2)\equiv\Cinf(\R_+,\R^2)\;,\\
    d^\pm_s\equiv\frac{d}{d\rho}\pm s\rho\;,\quad\delta^\pm_s\equiv-\frac{d}{d\rho}\pm s\rho\;,
  \end{gather*}
giving $d^\pm_{s,\text{\rm min}}\ne d^\pm_{s,\text{\rm max}}$ by Proposition~\ref{p:2}-(i).
\end{ex}

\section{Local model of the Witten's perturbation}\label{s:local model of the Witten's perturbation}

The local model of our version of Morse functions around their critical points will be as follows. Let $m_\pm\in\N$, let $L_\pm$ be a compact Thom-Mather stratification, and let $M_\pm$ be a stratum in $c(L_\pm)$. Thus, either $M_\pm=N_\pm\times\R_+$ for some stratum $N_\pm$ of $L_\pm$, or $M_\pm$ is the vertex stratum $\{*_\pm\}$ of $c(L_\pm)$. On the stratum $M=\R^{m_+}\times\R^{m_-}\times M_+\times M_-$ of $\R^{m_+}\times\R^{m_-}\times c(L_+)\times c(L_-)$, for any choice of product Thom-Mather stratification on $c(L_+)\times c(L_-)$, consider an adapted metric given as product of standard metrics on the Euclidean spaces $\R^{m_\pm}$ and model adapted metrics on the strata $M_\pm$. Let $d_s$ be the Witten's perturbed differential map on $\Omega(M)$ induced by the model rel-Morse function $\frac{1}{2}(\rho_+^2-\rho_-^2)$ (Remark~\ref{r:rel-Morse}-(ii)). Let $\D_{s,\text{\rm min/max}}$ be the Laplacian defined by $d_{s,\text{\rm min/max}}$, and $\HH_{s,\text{\rm min/max}}=\bigoplus_r\HH_{s,\text{\rm min/max}}^r=\ker\D_{s,\text{\rm min/max}}$. The following result is a direct consequence of Example~\ref{ex:d^pm_0,s}, Corollary~\ref{c:d^pm_s,min/max are discrete} and Lemma~\ref{l:tensor product}.

\begin{cor}\label{c:R^m_+ times R^m_- times c(L_+) times c(L_-)}
  \begin{itemize}
  
    \item[(i)] $d_{s,\text{\rm min/max}}$ is discrete. 
    
    \item[(ii)]  If $M_+=N_+\times\R_+$ and $M_-=N_-\times\R_+$, then
      \[
        \HH_{s,\text{\rm min/max}}^r\cong\bigoplus_{r_+,r_-}H_{\text{\rm min/max}}^{r_+}(N_+)\otimes H_{\text{\rm min/max}}^{r_-}(N_-)\;,
      \]
    where $(r_+,r_-)$ runs in the subset of $\Z^2$ defined by~\eqref{r=m_-+r_++r_-+1}--\eqref{r_- ge ...}.
    
    \item[(iii)] If $M_+=\{*_+\}$ and $M_-=N_-\times\R_+$, then
      \[
        \HH_{s,\text{\rm min/max}}^r\cong\bigoplus_{r_-}H_{\text{\rm min/max}}^{r_-}(N_-)\;, 
      \]
    where $r_-$ runs in the subset of $\Z$ defined by $r=m_-+r_-+1$ and~\eqref{r_- ge ...}.
    
    \item[(iv)] If $M_+=N_+\times\R_+$ and $M_-=\{*_-\}$, then
      \[
        \HH_{s,\text{\rm min/max}}^r\cong\bigoplus_{r_+}H_{\text{\rm min/max}}^{r_+}(N_+)\;,
      \]
    where $r_+$ runs in the subset of $\Z$ defined by $r=m_-+r_+$ and~\eqref{r_+ le ...}.
    
    \item[(v)]  If $M_+=\{*_+\}$ and $M_-=\{*_-\}$, then $\dim\HH_{s,\text{\rm min/max}}^r=\delta_{r,m_-}$.
        
    \item[(vi)] If $e_s\in\HH_{s,\text{\rm min/max}}$ with norm one for each $s$, and $h$ is a bounded measurable function on $\R_+$ with $h(\rho)\to1$ as $\rho\to0$, then $\langle he_s,e_s\rangle\to1$ as $s\to\infty$.
    
    \item[(vii)] Let $0\le\lambda_{s,\text{\rm min/max},0}\le\lambda_{s,\text{\rm min/max},1}\le\cdots$ be the eigenvalues of $\D_{s,\text{\rm min/max}}$, repeated according to their multiplicities. Given $k\in\N$, if $\lambda_{s,\text{\rm min/max},k}>0$ for some $s$, then $\lambda_{s,\text{\rm min/max},k}\in O(s)$ as $s\to\infty$.
  
    \item[(viii)] There is some $\theta>0$ such that $\liminf_k\lambda_{s,\text{\rm min/max},k}\,k^{-\theta}>0$.
  
  \end{itemize}
\end{cor}

\section{Globalization of the weak Weyl's asymptotic formula}
\label{s: globalization}

Here, for the maximum/minimum i.b.c.\ of an elliptic complex, we show globalization results for its domain, the discreteness of the spectrum, and, mainly, the type of weak Weyl's asymptotic formula stated in Theorem~\ref{t:spectrum of Delta_min/max}-(ii). This will play a key role in the proof of Theorem~\ref{t:spectrum of Delta_min/max}.

Consider the notation of the Section~\ref{s:Sobolev sp. defined by an i.b.c.}. The following refinement of Lemma~\ref{l:W^m} is obtained with a deeper analysis.

\begin{lem}\label{l:liminf_k lambda_k k^-1/N >0}
  Suppose that $(\DD,\bd)$ is discrete, and let $0\le\lambda_1\le\lambda_2\le\cdots$ be the eigenvalues of $\bDelta$, repeated according to their multiplicities. Let $B^1$ be the standard unit ball of $W^1$, and $B_r$ the standard ball of radius $r>0$ in $L^2(E)$. Then the following properties are equivalent for $\theta>0$:
    \begin{itemize}
    
      \item[(i)] $\liminf_k\lambda_kk^{-\theta}>0$.
      
      \item[(ii)] There are some $C_0,C_1>0$ such that, for all $n\in\Z_+$, there is a linear subspace $Z_n\subset L^2(E)$ so that:
        \begin{itemize}
        
          \item[(a)] $Z_n$ is closed and of codimension $\le C_0\,n^{1/\theta}$ in $L^2(E)$;
          
          \item[(b)] $\bD(W^1\cap Z_n)\subset Z_n$; and 
          
          \item[(c)] $B^1\cap Z_n\subset B_{C_1/n}$.
          
        \end{itemize}
        
      \item[(iii)] There are some $C_0,\dots,C_4>0$ and $A\in\Z_+$ such that, for all $n\in\Z_+$, there is a linear map\footnote{For $A\in\Z_+$ and any topological vector space $L$, the notation $\bigoplus_AL$ is used for the direct sum of $A$ copies of $L$. Similarlarly, for any linear map between topological vector spaces, $T:L\to L'$, the notation $\bigoplus_AT:\bigoplus_AL\to\bigoplus_AL'$ is used for the direct sum of $A$ copies of $T$.} $R_n=(R_n^1,\dots,R_n^A):L^2(E)\to\bigoplus_AL^2(E)$ so that:
        \begin{itemize}
        
          \item[(a)]  $\dim\ker R_n\le C_0\,n^{1/\theta}$;
          
          \item[(b)] $\|R_nu\|\le C_1\,\|u\|$ for all $u\in L^2(E)$;
          
          \item[(c)] $\|R_nu\|\ge C_2\,\|u\|$ for all $u\in(\ker R_n)^\perp$;
          
          \item[(d)] $R_n^a(W^1)\subset W^1$ and $\|[\bD,R_n^a]u\|\le C_3\,\|u\|$ for all $u\in W^1$; and 
          
          \item[(e)] $B^1\cap R_n^a(L^2(E))\subset B_{C_4/n}$.
          
        \end{itemize}
    
    \end{itemize}
\end{lem}

\begin{proof}
  Let $(e_i)$ ($i\in\Z$) be a complete orthonormal system of $L^2(E)$ such that $e_{\pm k}$ is a $\pm\sqrt{\lambda_k}$-eigenvector of $\bD$ for each $k\in\N$. The mapping $u=\sum_iu_ie_i\mapsto(u_i)$ defines a unitary isomorphism $L^2(E)\cong\ell^2(\Z)$. Moreover $W^1$ consists of the elements $u\in L^2(E)$ with $\sum_k(1+\lambda_k)u_{\pm k}^2<\infty$. We have $\|u\|_1^2=\sum_k(1+\lambda_k)(u_k^2+u_{-k}^2)$ for $u\in W^1$. 
  
  Suppose that~(i) holds. Then there is some $C>0$ so that $1+\lambda_k\ge Ck^\theta$ for all $k$. For each $n\in\Z_+$, the linear subspace
    \[
      Z_n=\left\{\,u\in L^2(E)\mid u_{\pm k}=0\ \text{if}\ k\le(n/C)^{1/\theta}\,\right\}
    \]
  of $L^2(E)$ satisfies~(ii)-(a),(b) with $C_0=2/C^{1/\theta}$. Furthermore, for every $u\in B^1\cap Z_n$,
    \begin{multline*}
      \|u\|^2=\sum_{k>(n/C)^{1/\theta}}(u_k^2+u_{-k}^2)<\frac{C}{n}\sum_{k>(n/C)^{1/\theta}}k^\theta(u_k^2+u_{-k}^2)\\
      \le\frac{1}{n}\sum_{k>(n/C)^{1/\theta}}(1+\lambda_k)(u_k^2+u_{-k}^2)=\frac{\|u\|_1^2}{n}<\frac{1}{n}\;,
    \end{multline*}
  completing the proof of~(ii)-(c) with $C_1=1$.
    
  Now, assume that~(ii) is satisfied. By~(ii)-(a),
    \begin{equation}\label{L^2(E)=Z_n^perp oplus Z_n}
      L^2(E)=Z_n^\perp\oplus Z_n
    \end{equation}
   as topological vector space \cite[Chapter~I, 3.5]{Schaefer1971}. Furthermore, by~(ii)-(a) and the canonical linear isomorphism $W^1/(W^1\cap Z_n)\cong(W^1+Z_n)/Z_n$, we also get that $W^1\cap Z_n$ is a closed linear subspace of finite codimension in $W^1$. Hence
    \begin{equation}\label{W^1=Y_n oplus (W^1 cap Z_n)}
      W^1=Y_n\oplus(W^1\cap Z_n)
    \end{equation}
   as topological vector spaces for any linear complement $Y_n$ of $W^1\cap Z_n$ in $W^1$ \cite[Chapter~I, 3.5]{Schaefer1971}. 
   
   On the other hand, for each $u\in Z_n^\perp$, the linear mapping $v\mapsto\langle u,\bD v\rangle$ is bounded on $Y_n$ because $Y_n$ is of finite dimension, and $\langle u,\bD w\rangle=0$ for all $w\in W^1\cap Z_n$ by~(ii)-(b). So $v\mapsto\langle u,\bD v\rangle$ is bounded on $W^1$ by~\eqref{W^1=Y_n oplus (W^1 cap Z_n)}, obtaining that $u\in W^1$ by~\eqref{W^1} since $\bD$ is self-adjoint. Hence $Z_n^\perp\subset W^1$, and therefore we can take $Y_n=Z_n^\perp$ in~\eqref{W^1=Y_n oplus (W^1 cap Z_n)}, obtaining
    \begin{equation}\label{W^1=Z_n^perp oplus (W^1 cap Z_n)}
      W^1=Z_n^\perp\oplus(W^1\cap Z_n)
    \end{equation}
   as topological vector spaces. Note that $W^1\cap Z_n$ is dense in $Z_n$ by~\eqref{L^2(E)=Z_n^perp oplus Z_n} and~\eqref{W^1=Z_n^perp oplus (W^1 cap Z_n)}. So, since $\bD$ is self-adjoint, it follows from~(ii)-(b) and~\eqref{W^1=Z_n^perp oplus (W^1 cap Z_n)} that  $\bD$ preserves $Z_n^\perp$. 
   
   To get~(iii), take $A=1$ and $R_n$ equal to the orthogonal projection of $L^2(E)$ to $Z_n$. Then~(iii)-(a) follows from~(ii)-(a), and properties~(iii)-(b),(c) hold with $C_1=C_2=1$ because $R_n$ is an orthogonal projection. By~(ii)-(b) and since $\bD$ preserves $Z_n^\perp$, we get $R_n(W^1)\subset W^1$ and $DR_n=R_nD$ on $W^1$, showing~(iii)-(d). Property~(iii)-(e) is a consequence of~(ii)-(c).
   
   Finally, assume that~(iii) is true. The following general assertion will be used.
  
  \begin{claim}\label{cl:orthonormal set}
    Let $\fH$ be a (real or complex) Hilbert space, $\Pi$ an orthogonal projection of $\fH$ with finite rank $p$, and $0<C<1$. Then the cardinality of any orthonormal set contained in $U_C=\left\{\,u\in\fH\mid\|\Pi u\|>C\,\|u\|\,\right\}$ is $\le p/C^2$.
  \end{claim}

  Suppose $v_1,\dots,v_p$ is an orthonormal basis of $\Pi(\fH)$. Let $u_1,\dots,u_k$ be orthonormal vectors in $U_C$, and $\Pi'$ the orthogonal projection of $\fH$ to the linear subspace generated by them. We get Claim~\ref{cl:orthonormal set} because
    \[
      kC^2\le\sum_{j=1}^k\|\Pi u_j\|^2=\sum_{j=1}^k\sum_{i=1}^p|\langle v_i,u_j\rangle|^2=\sum_{i=1}^p\|\Pi'v_i\|^2\le p\;.
    \]
  
  Let $p_n=\lfloor C_0\,n^{1/\theta}\rfloor$.
  
  \begin{claim}\label{cl: |R_n e_i| ge C_2/sqrt 2}
    There is some $I\subset\Z$ with $\# I\le2p_n$ and $\|R_n e_i\|\ge C_2/\sqrt{2}$ for all $i\in\Z\sm I$.
  \end{claim}
  
  Let $\Pi_n$ and $\widetilde{\Pi}_n$ be the orthogonal projections of $L^2(E)$ to $\ker R_n$ and $(\ker R_n)^\perp$, respectively. By Claim~\ref{cl:orthonormal set}, the cardinality of the set $I=\{\,i\in\Z\mid\|\Pi_ne_i\|>1/\sqrt{2}\,\}$ is $\le2p_n$. For $i\in\Z\sm I$, we have
    \[
      \|R_ne_i\|=\|R_n\widetilde{\Pi}_ne_i\|\ge C_2\,\|\widetilde{\Pi}_ne_i\|\ge C_2/\sqrt{2}
    \]
  by~(iii)-(c), showing Claim~\ref{cl: |R_n e_i| ge C_2/sqrt 2}.
  
  From Claim~\ref{cl: |R_n e_i| ge C_2/sqrt 2}, it follows that there is some $i_n\in\Z$ such that
     \begin{gather}
       |i_n|\le p_n+1\;,\label{|i_n| le p_n+1}\\
       \|R_n e_{i_n}\|\ge C_2/\sqrt{2}\;.\label{|R_n e_i_n| ge C_2/sqrt 2}
     \end{gather}
   We have
     \begin{multline*}
       \|R_n^ae_{i_n}\|_1^2=\|R_n^ae_{i_n}\|^2+\|\bD R_n^ae_{i_n}\|^2\\
       \le\|R_n^ae_{i_n}\|^2+(\|R_n^a\bD e_{i_n}\|+\|[\bD,R_n^a]e_{i_n}\|)^2
       \le C_1^2+\left(C_1\sqrt{\lambda_{|i_n|}}+C_3\right)^2\;.
     \end{multline*}
   Hence
     \[
       u_{n,r}^a:=\frac{r}{\sqrt{C_1^2+\left(C_1\sqrt{\lambda_{|i_n|}}+C_3\right)^2}}\,R_n^ae_{i_n}
       \in B^1\cap R_n^a(L^2(E))
     \]
   for all $r\in[0,1)$, giving
     \begin{multline*}
       \frac{rC_2/\sqrt{2}}{\sqrt{C_1^2+\left(C_1\sqrt{\lambda_{|i_n|}}+C_3\right)^2}}
       \le\frac{r\,\|R_ne_{i_n}\|}{\sqrt{C_1^2+\left(C_1\sqrt{\lambda_{|i_n|}}+C_3\right)^2}}\\
       \le\frac{r\sum_a\|R_n^ae_{i_n}\|}{\sqrt{C_1^2+\left(C_1\sqrt{\lambda_{|i_n|}}+C_3\right)^2}}
       =\sum_a\|u_{n,r}^a\|<\frac{AC_4}{n}
     \end{multline*}
   for all $r\in[0,1)$ by~\eqref{|R_n e_i_n| ge C_2/sqrt 2} and~(iii)-(e). So there is some $C>0$, independent of $n$, such that
     \begin{equation}\label{lambda_|i_n| ge Cn^2}
       \lambda_{|i_n|}\ge\frac{1}{C_2^2}\left(\sqrt{\frac{C_2^2n^2}{2AC_4^2}-C_1^2}-C_3^2\right)^2\ge Cn^2
     \end{equation}
  for $n$ large enough. If $|i_{n-1}|\le k<|i_n|$ for $n$ large enough and $k\in\N$, then
    \[
      \lambda_k\ge\lambda_{|i_{n-1}|}\ge C(n-1)^2\ge Cn
      \ge C\left(\frac{|i_n|-1}{C_0}\right)^\theta\ge C(k/C_0)^\theta
    \]
  by~\eqref{lambda_|i_n| ge Cn^2} and~\eqref{|i_n| le p_n+1}. This shows~(i) because, since $|i_n|\to\infty$ as $n\to\infty$ by~\eqref{lambda_|i_n| ge Cn^2}, there is an increasing sequence $(n_\ell)$ in $\Z_+$ such that $[|i_{n_0-1}|,\infty)=\bigcup_\ell[|i_{n_\ell-1}|,|i_{n_\ell}|)$.
\end{proof}

\begin{prop}\label{p:globalization of d_min/max}
  Let $(E,d)$ be an elliptic complex on a Riemannian manifold $M$. Let $\{U_a\}$ be a finite open covering of $M$, and let $\{f_a\}$ be a smooth partition of unity on $M$ subordinated to $\{U_a\}$ such that each $|[d,f_a]|$ is bounded. Assume also that there is another family $\{\tilde f_a\}\subset\Cinf(M)$ such that $\tilde f_a$ and $|[d,\tilde f_a]|$ are bounded, $\tilde f_a=1$ on $\supp f_a$, and $\supp\tilde f_a\subset U_a$. For each $a$, let $(E^a,d^a)$ be an elliptic complex on a Riemannian manifold $M_a$, let $V_a\subset M_a$ be an open subset, and let $\zeta_a:(E|_{U_a},d)\to(E^a|_{V_a},d^a)$ be a quasi-isometric isomorphism of elliptic complexes over $\xi_a:U_a\to V_a$. Then the following properties hold:
    \begin{itemize}
    
      \item[(i)] $\DD(d_{\text{\rm min/max}})=\{\,u\in L^2(E)\mid\zeta_a(f_au)\in\DD(d^a_{\text{\rm min/max}})\ \forall a\,\}$.
      
      \item[(ii)] If $d^a_{\text{\rm min/max}}$ is discrete for all $a$, then $d_{\text{\rm min/max}}$ is discrete.
      
    \end{itemize}
\end{prop}

\begin{proof}
  The inclusion ``$\subset$'' of~(i) follows from Lemmas~\ref{l:f DD(d_min/max) subset DD(d_min/max)}-(i) and~\ref{l:zeta(f DD(d min/max)) subset DD(d'_min/max)}-(i). 

  Now, take any $u\in L^2(E)$ such that $\zeta_a(f_au)\in\DD(d^a_{\text{\rm min/max}})$ for all $a$. Let $g_a$ and $\tilde g_a$ be the smooth functions on each $M_a$, supported in $V_a$, that correspond to $f_a$ and $\tilde f_a$ via $\xi_a$. By Lemmas~\ref{l:f DD(d_min/max) subset DD(d_min/max)}-(i) and~\ref{l:zeta(f DD(d min/max)) subset DD(d'_min/max)}-(i),
    \[
      f_au=\zeta_a^{-1}\zeta_a(f_au)=\zeta_a^{-1}(\tilde g_a\,\zeta_a(f_au))\in\DD(d_{\text{\rm min/max}})\;.
    \]
  So $u=\sum_af_au\in\DD(d_{\text{\rm min/max}})$, completing the proof of~(i).

  To prove~(ii), we can make the following reduction. Since discreteness is invariant by quasi-isometric isomorphisms of elliptic complexes, like in the proof of Lemma~\ref{l:zeta(f DD(d min/max)) subset DD(d'_min/max)}-(i), after shrinking $\{U_a\}$ if necessary, we can assume that each $\zeta_a:(E|_{U_a},d)\to(E^a|_{V_a},d^a)$ is isometric. If every $d^a_{\text{\rm min/max}}$ is discrete, then each $W^1(d^a_{\text{\rm min/max}})\hookrightarrow L^2(E^a)$ is compact by Lemma~\ref{l:W^m}. So
    \[
      \Cl_1(g_a\,W^1(d^a_{\text{\rm min/max}}))\hookrightarrow\Cl_0(g_a\,L^2(E^a))
    \]
  is compact for all $a$ by Lemma~\ref{l:f DD(d_min/max) subset DD(d_min/max)}-(ii). Therefore
    \[
      \Cl_1(f_a\,W^1(d_{\text{\rm min/max}}))\hookrightarrow\Cl_0(f_a\,L^2(E))
    \]
  is compact by Lemma~\ref{l:zeta(f DD(d min/max)) subset DD(d'_min/max)}-(ii). Since $W^1(d_{\text{\rm min/max}})=\sum_af_aW^1(d_{\text{\rm min/max}})$ by Lemma~\ref{l:f DD(d_min/max) subset DD(d_min/max)}-(ii), it follows that $W^1(d_{\text{\rm min/max}})\hookrightarrow L^2(E)$ is compact. Hence $d_{\text{\rm min/max}}$ is discrete by Lemma~\ref{l:W^m}.
\end{proof}

\begin{prop}\label{p:globalization of d_min/max-2}
  With the notation of Proposition~\ref{p:globalization of d_min/max}, suppose that every $d^a_{\text{\rm min/max}}$ is discrete, and therefore $d_{\text{\rm min/max}}$ is also discrete. Let
    \[
      0\le\lambda^a_{\text{\rm min/max},0}\le\lambda^a_{\text{\rm min/max},1}\le\cdots\;,\quad
      0\le\lambda_{\text{\rm min/max},0}\le\lambda_{\text{\rm min/max},1}\le\cdots
    \]
  denote the eigenvalues, repeated according to their multiplicities, of the Laplacians $\D^a_{\text{\rm min/max}}$ and $\D_{\text{\rm min/max}}$ defined by $d^a_{\text{\rm min/max}}$ and $d_{\text{\rm min/max}}$, respectively. Suppose that, for all $a$, there is some\footnote{The notation $\theta_{a,\text{\rm min/max}}$ would be more correct, but, for the sake of simplicity, reference to the maximum/minimum i.b.c.\ is omitted here and in most of the notation of the proof.} $\theta_a>0$ such that $\liminf_k\lambda^a_{\text{\rm min/max},k}k^{-\theta_a}>0$. Then $\liminf_k\lambda_{\text{\rm min/max},k}\,k^{-\theta}>0$ with $\theta=\min_a\theta_a$.
\end{prop}

\begin{proof}
  According to Sections~\ref{ss:Hilbert} and~\ref{ss:elliptic complexes}, the condition $\liminf_k\lambda^a_{\text{\rm min/max},k}k^{-\theta_a}>0$ is invariant by quasi-isometric isomorphisms of elliptic complexes. Thus, like in the proof of Proposition~\ref{p:globalization of d_min/max}-(ii), we can assume that $\zeta_a:(E|_{U_a},d)\to(E^a|_{V_a},d^a)$ is isometric. Set $D^a_{\text{\rm min/max}}=d^a_{\text{\rm min/max}}+\delta^a_{\text{\rm max/min}}$ and $W^{1,a}=W^1(d^a_{\text{\rm min/max}})$. Let $B^{1,a}$ denote the standard unit ball in $W^{1,a}$, and $B^a_r$ the standard ball of radius $r>0$ in $L^2(E^a)$. By Lemma~\ref{l:liminf_k lambda_k k^-1/N >0},  we get the following.
  
    \begin{claim}\label{cl:0}
      There are some $C_{a,0},C_{a,1}>0$ for every $a$ such that, for all $n\in\Z_+$, there is a linear subspace $Z^a_n\subset L^2(E^a)$ so that:
        \begin{itemize}
    
          \item[(a)] $Z^a_n$ is closed and of codimension $\le C_{a,0}\,n^{1/\theta_a}$ in $L^2(E^a)$;
      
          \item[(b)] $D^a_{\text{\rm min/max}}(W^{1,a}\cap Z^a_n)\subset Z^a_n$; and
      
          \item[(c)] $B^{1,a}\cap Z^a_n\subset B^a_{C_{a,1}/n}$.
                
    \end{itemize}
    \end{claim}
    
  For each $a$, fix an open subset $O_a\subset M$ such that $\supp f_a\subset O_a$, $\ol{O_a}\subset U_a$ and the frontier of $O_a$ has zero Riemannian measure. Let $P_a=\xi_a(O_a)$,
    \[
      \PP^a=\left\{\,v\in L^2(E^a)\mid\text{$v$ is essentially supported in $\ol{P_a}$}\,\right\}\;,
    \]
  and $Z^{a\,\prime}_n=Z^a_n\cap \PP^a$. Each $\PP^a$ is a closed linear subspace of $L^2(E^a)$ satisfying
    \begin{equation}\label{D^a_min/max(W^1,a_min/max cap WW^a) subset WW^a}
      D^a_{\text{\rm min/max}}(W^{1,a}\cap\PP^a)\subset \PP^a\;.
    \end{equation}
  
    \begin{claim}\label{cl:1}
      \begin{itemize}
    
        \item[(a)] $Z^{a\,\prime}_n$ is closed and of codimension $\le C_{a,0}\,n^{1/\theta_a}$ in $\PP^a$;
      
        \item[(b)] $D^a_{\text{\rm min/max}}(W^{1,a}\cap Z^{a\,\prime}_n)\subset Z^{a\,\prime}_n$; and
      
        \item[(c)] $B^{1,a}\cap Z^{a\,\prime}_n\subset B^a_{C_{a,1}/n}\cap\PP^a$.
      
    \end{itemize}
    \end{claim}
    
  Claim~\ref{cl:1}-(a) follows from Claim~\ref{cl:0}-(a) and the canonical linear isomorphism $\PP^a/Z^{a\,\prime}_n\cong(\PP^a+Z^a_n)/Z^a_n$. Claim~\ref{cl:1}-(b) is a consequence of Claim~\ref{cl:0}-(b) and~\eqref{D^a_min/max(W^1,a_min/max cap WW^a) subset WW^a}, and Claim~\ref{cl:1}-(c) follows from Claim~\ref{cl:0}-(c).
    
    Now, consider the linear spaces
      \begin{align*}
      \OO^a&=\left\{\,u\in L^2(E)\mid\text{$u$ is essentially supported in $\ol{O_a}$}\,\right\}\;,\\
      Z^{a\,\prime\prime}_n&=\left\{\,u\in \OO^a\mid\exists v\in Z^{a\,\prime}_n\ \text{so that}\  \zeta_a(u|_{U_a})=v|_{V_a}\,\right\}\;.
    \end{align*}
  Each $\OO^a$ is a closed linear subspace of $L^2(E)$, and we have $L^2(E)=\sum_a\OO^a$. Set $D_{\text{\rm min/max}}=d_{\text{\rm min/max}}+\delta_{\text{\rm max/min}}$ and $W^m=W^m(d_{\text{\rm min/max}})$ ($m\in\Z_+$). Let $B^1$ be the standard unit ball in $W^1$, and $B_r$ the standard ball of radius $r>0$ in $L^2(E)$. Since $\zeta_a:(E|_{U_a},d)\to(E^a|_{V_a},d^a)$ is isometric for all $a$, Claim~\ref{cl:1} gives the following.
    
    \begin{claim}\label{cl:2}
      \begin{itemize}
    
        \item[(a)] $Z^{a\,\prime\prime}_n$ is closed and of codimension $\le C_{a,0}\,n^{1/\theta_a}$ in $\OO^a$;
      
        \item[(b)] $D_{\text{\rm min/max}}(W^1\cap Z^{a\,\prime\prime}_n)\subset Z^{a\,\prime\prime}_n$; and  
      
        \item[(c)] $B^1\cap Z^{a\,\prime\prime}_n\subset B_{C_{a,1}/n}\cap\OO^a$.
      
      \end{itemize}
    \end{claim}
  
  Let $Y^a_n$ be a linear complement of each $Z^{a\,\prime\prime}_n$ in $\OO^a$. By Claim~\ref{cl:2}-(a), we have
    \begin{equation}\label{Y^a_min/max,n in L^2}
      \OO^a=Y^a_n\oplus Z^{a\,\prime\prime}_n
    \end{equation}
  as topological vector spaces \cite[Chapter~I, 3.5]{Schaefer1971}. On the other hand, for any $m\in\Z_+$, 
  	\begin{equation}\label{W^m cap OO^a}
		W^m\cap\OO^a\supset\{\,u\in\Cinf_0(E)\mid\supp u\subset O_a\,\}\;,
	\end{equation}
in particular, $W^m\cap\OO^a$ is dense in $\OO^a$. So we can choose $Y^a_n\subset W^m$ by Claim~\ref{cl:2}-(a); in this case, we get
    \begin{equation}\label{Y^a_min/max,n in W^m}
      W^m\cap\OO^a=Y^a_n\oplus(W^m\cap Z^{a\,\prime\prime}_n)
    \end{equation}
  as topological vector spaces with respect to the $\|\ \|$-topology.
  
  \begin{claim}\label{cl:W^m_min/max cap Z^a prime prime_min/max,n is dense}
      $W^m\cap Z^{a\,\prime\prime}_n$ is $\|\ \|$-dense in $Z^{a\,\prime\prime}_n$ for all $m\in\Z_+$.
  \end{claim}
   
   Choosing $Y^a_n\subset W^m$, Claim~\ref{cl:W^m_min/max cap Z^a prime prime_min/max,n is dense} follows from~\eqref{Y^a_min/max,n in L^2},~\eqref{Y^a_min/max,n in W^m} and the density of $W^m\cap\OO^a$ in $\OO^a$.
  
  For the case $m=1$, observe that~\eqref{Y^a_min/max,n in W^m} is satisfied with
    \begin{equation}\label{Y^a_min/max,n}
      Y^a_n
      =\OO^a\cap(W^1\cap Z^{a\,\prime\prime}_n)^{\perp_1}\;,
    \end{equation}
  where $\perp_1$ denotes $\langle\ ,\ \rangle_1$-orthogonality in $W^1$. From now on, consider this choice for $Y^a_n$. Thus~\eqref{Y^a_min/max,n in W^m} also holds with respect to the $\|\ \|_1$-topology whenever $Y^a_n\subset W^m$; in particular, it this is true for $m=1$.
  
  \begin{claim}\label{cl:Y^a_min/max,n is mapped to W^1_min/max}
    $D_{\text{\rm min/max}}(Y^a_n)\subset W^1$.
  \end{claim}
  
  Since the Riemannian measure of the frontier of $O_a$ is zero, $\OO^{a\,\perp}$ consists of the sections $u\in L^2(E)$ whose essential support is contained in $M\sm O_a$. Hence the set
    \[
      (W^1\cap\OO^{a\,\perp})+Y^a_n+(W^1\cap Z^{a\,\prime\prime}_n)
    \]
  is dense in $L^2(E)$ by~\eqref{Y^a_min/max,n in W^m} for $m=1$. It follows that, given any $u\in Y^a_n$, to check that $D_{\text{\rm min/max}} u\in W^1$, its enough to check that the mapping
    \[
      v\mapsto\langle D_{\text{\rm min/max}}u,D_{\text{\rm min/max}} v\rangle
    \]
  is bounded on $W^1\cap\OO^{a\,\perp}$, $Y^a_n$ and $W^1\cap Z^{a\,\prime\prime}_n$. This mapping vanishes on $W^1\cap\OO^{a\,\perp}$ because 
    \[
      D_{\text{\rm min/max}}(W^1\cap\OO^a)\subset\OO^a\;,\quad
      D_{\text{\rm min/max}}(W^1\cap\OO^{a\,\perp})\subset\OO^{a\,\perp}\;.
    \]
  Moreover it is bounded on $Y^a_n$ because this space is of finite dimension. Finally, for $v\in W^1\cap Z^{a\,\prime\prime}_n$, we have
    \[
      \langle D_{\text{\rm min/max}}u,D_{\text{\rm min/max}}v\rangle=-\langle u,v\rangle
    \]
  because $u\perp_1v$. Thus the above mapping is bounded on $W^1\cap Z^{a\,\prime\prime}_n$, which completes the proof of Claim~\ref{cl:Y^a_min/max,n is mapped to W^1_min/max}.
  
  \begin{claim}\label{cl:W^2_min/max cap Z^a prime prime_min/max,n is | |_1-dense}
      $W^2\cap Z^{a\,\prime\prime}_n$ is $\|\ \|_1$-dense in $Z^{a\,\prime\prime}_n$.
   \end{claim}
 
  From Claim~\ref{cl:Y^a_min/max,n is mapped to W^1_min/max}, we get $Y^a_n\subset W^2$. Hence~\eqref{Y^a_min/max,n in W^m} holds for $m\in\{1,2\}$ with respect to the $\|\ \|_1$-topology, yielding Claim~\ref{cl:W^2_min/max cap Z^a prime prime_min/max,n is | |_1-dense} because $W^2\cap\OO^a$ is $\|\ \|_1$-dense in $W^1\cap\OO^a$ by~\eqref{W^m cap OO^a}.

  \begin{claim}\label{cl:Y^a_min/max,n is preserved}
    $D_{\text{\rm min/max}}(Y^a_n)\subset Y^a_n$.
  \end{claim}
  
  For $u\in Y^a_n$ and $v\in W^2\cap Z^{a\,\prime\prime}_n$, since $D_{\text{\rm min/max}}$ is self-adjoint, we have
    \begin{multline*}
      \langle D_{\text{\rm min/max}}u,v\rangle_1
      =\langle D_{\text{\rm min/max}}u,v\rangle
      +\langle\D_{\text{\rm min/max}}u,D_{\text{\rm min/max}}v\rangle\\
      =\langle u,D_{\text{\rm min/max}}v\rangle
      +\langle D_{\text{\rm min/max}}u,\D_{\text{\rm min/max}}v\rangle
      =\langle u,D_{\text{\rm min/max}}v\rangle_1=0
    \end{multline*}
  by Claims~\ref{cl:Y^a_min/max,n is mapped to W^1_min/max} and~\ref{cl:2}-(b), and~\eqref{Y^a_min/max,n}. Then Claim~\ref{cl:Y^a_min/max,n is preserved} follows by Claim~\ref{cl:W^2_min/max cap Z^a prime prime_min/max,n is | |_1-dense}.
  
  \begin{claim}\label{cl:Y^a_min/max,n perp}
    $Y^a_n=\OO^a\cap(Z^{a\,\prime\prime}_n)^\perp$.
  \end{claim}
  
  Let $u\in Y^a_n$ and $v\in W^1\cap Z^{a\,\prime\prime}_n$. By Claim~\ref{cl:Y^a_min/max,n is preserved}, $\D_{\text{\rm min/max}}$ is a self-adjoint operator on $Y^a_n$. Then $u=(1+\D_{\text{\rm min/max}})u_0$ for $u_0=(1+\D_{\text{\rm min/max}})^{-1}u\in Y^a_n$, obtaining
    \[
      \langle u,v\rangle=\langle(1+\D_{\text{\rm min/max}})u_0,v\rangle=\langle u_0,v\rangle_1=0
    \]
  by~\eqref{Y^a_min/max,n}. This shows Claim~\ref{cl:Y^a_min/max,n perp} by Claim~\ref{cl:W^m_min/max cap Z^a prime prime_min/max,n is dense} and~\eqref{Y^a_min/max,n in L^2}.
  
  Let $\Pi^a_n:\OO^a\to Z^{a\,\prime\prime}_n$ denote the orthogonal projection. The following claim follows from~\eqref{Y^a_min/max,n in W^m} for $m=1$, and Claims~\ref{cl:2}-(b),~\ref{cl:Y^a_min/max,n is preserved} and~\ref{cl:Y^a_min/max,n perp}.
  
  \begin{claim}\label{cl:Pi^a_min/max,n}
    $\Pi^a_n(W^1\cap\OO^a)\subset W^1\cap\OO^a$, and $[D_{\text{\rm min/max}},\Pi^a_n]=0$ on $W^1\cap\OO^a$.
  \end{claim}
 
  Consider each function $f_a$ as the corresponding bounded multiplication operator on $L^2(E)$. Assuming that $a$ runs in $\{1,\dots,A\}$ for some $A\in\Z_+$, we get the bounded operator $T=(f_1,\dots,f_A):L^2(E)\to\bigoplus_AL^2(E)$.   Also, let $\Sigma:\bigoplus_AL^2(E)\to L^2(E)$ be the bounded operator defined by $\Sigma(u_1,\dots,u_A)=\sum_au_a$. We have $\Sigma T=1$ because $\{f_a\}$ is a partition of unity.
  
  \begin{claim}\label{cl:the image of T is closed}
    The image of $T$ is closed.
  \end{claim} 
  
  Let $(u^i)$ be a sequence in $L^2(E)$ such that $(Tu^i)$ converges to some $v$ in $\bigoplus_AL^2(E)$. Then $u^i=\Sigma T u^i\to\Sigma v$ as $i\to\infty$, obtaining $Tu^i\to T\Sigma v$ as $i\to\infty$. Hence $v=T\Sigma v\in T(L^2(E))$, showing Claim~\ref{cl:the image of T is closed}.
  
  By Claim~\ref{cl:the image of T is closed} and the open mapping theorem (see e.g.\ \cite[Chapter~III, 12.1]{Conway1985} or \cite[Chapter~III, 2.1]{Schaefer1971}), we get that $T$ is a topological homomorphism\footnote{Recall that a bounded operator between topological vector spaces, $T:\fH\to\fG$, is called a topological homomorphism if the map $T:\fH\to T(\fH)$ is open, where $T(\fH)$ is equipped with the restriction of the topology of $\fG$.}. So $T:L^2(E)\to T(L^2(E))$ is a quasi-isometric isomorphism; its inverse is $\Sigma:T(L^2(E))\to L^2(E)$. Since $\Pi_n:=\bigoplus_a\Pi^a_n$ is an orthogonal projection of $\bigoplus_AL^2(E)$, it follows that $R_n:=\Pi_nT$ satisfies Lemma~\ref{l:liminf_k lambda_k k^-1/N >0}-(iii)-(b),(c). Moreover, by Claim~\ref{cl:2}-(a),
    \[
      \dim\ker R_n\le\dim\ker\Pi_n
      =\sum_a\dim\ker\Pi^a_n\le\sum_aC_{0,a}\,n^{1/\theta_a}\le C_0\,n^{1/\theta}
    \]
  with $C_0=\sum_aC_{0,a}$ and $\theta=\min_a\theta_a$, showing that $R_n$ satisfies Lemma~\ref{l:liminf_k lambda_k k^-1/N >0}-(iii)-(a).
  
  We have $R_n=(R_n^1,\dots,R_n^A)$ with $R_n^a=\Pi^a_n\,f_a$. Since each function $|[d,f_a]|$ is uniformly bounded, it follows that $f_a\,W^1\subset W^1$ and $[D_{\text{\rm min/max}},f_a]:W^1\to L^2(E)$ extends to a bounded operator on $L^2(E)$. So each $R_n^a$ satisfies Lemma~\ref{l:liminf_k lambda_k k^-1/N >0}-(iii)-(d) by Claim~\ref{cl:Pi^a_min/max,n}.
  
  Finally, $R_n^a$ satisfies Lemma~\ref{l:liminf_k lambda_k k^-1/N >0}-(iii)-(e) by Claim~\ref{cl:2}-(c). Now, the result follows from Lemma~\ref{l:liminf_k lambda_k k^-1/N >0}.
\end{proof}

\section{Proof of Theorem~\ref{t:spectrum of Delta_min/max}}
\label{s:proof of Theorem spectrum of D_min/max}

Consider the notation of Theorem~\ref{t:spectrum of Delta_min/max}: $M$ is a stratum with compact closure of a Thom-Mather stratification $A$, and $g$ is an adapted metric on $M$. Let $\{(O_a,\xi_a)\}$ be a finite covering of $\ol{M}$ by charts of $A$. For each $a$, we have $\xi_a(O_a)=B_a\times c_{\epsilon_a}(L_a)$, where $B_a$ is an open subset of $\R^{m_a}$ for some $m_a\in\N$, $L_a$ is a compact Thom-Mather stratification, and $\epsilon_a>0$. Then each $\xi_a$ defines an open embedding of $M\cap O_a$ into $\R^{m_a}\times M_a$ for some stratum $M_a$ of $L_a$. We have, either $M_a=N_a\times\R_+$ for some stratum $N_a$ of $L_a$, or $M_a=\{*_a\}$, where $*_a$ is the vertex of $c(L_a)$. If $M_a=N_a\times\R_+$, then $\xi_a(M\cap O_a)=B_a\times N_a\times(0,\epsilon_a)$. If $M_a=\{*_a\}$, then $\xi_a(M\cap O_a)=B_a\times\{*_a\}\equiv B_a$. Thus every $\xi_a(M\cap O_a)$ is, either open in $\R^{m_a}$, or open in $\R^{m_a}\times N_a\times\R_+$. By shrinking $\{(O_a,\xi_a)\}$ if necessary, we can assume that each diffeomorphism $\xi_a:M\cap O_a\to\xi_a(M\cap O_a)$ is quasi-isometric with respect to a model adapted metric on $\R^{m_a}\times M_a$.

By Lemma~\ref{l:|d lambda_a| is rel-locally bounded}, there is a smooth partition of unity $\{\lambda_a\}$ of $M$ subordinated to the open covering $\{M\cap O_a\}$ such that each function $|d\lambda_a|$ is bounded. Also, using Example~\ref{ex:rel-admissible function}, it is easy to construct another family $\{\tilde\lambda_a\}\subset\Cinf(M)$ such that $\tilde\lambda_a$ and $|d\tilde\lambda_a|$ are bounded, $\tilde\lambda_a=1$ on $\supp\lambda_a$, and $\supp\tilde\lambda_a\subset M\cap O_a$. The existence of such families $\{\lambda_a\}$ and $\{\tilde\lambda_a\}$ is required to apply Propositions~\ref{p:globalization of d_min/max} and~\ref{p:globalization of d_min/max-2}.

Let $d_{a,s}$ be the Witten's perturbation of $d_a$ induced by the function $f_a=\frac{1}{2}\rho_a^2$ on $\R^{m_a}\times M_a$, where $\rho_a$ is the radial function of $\R^{m_a}\times c(L_a)$. According to Corollary~\ref{c:R^m_+ times R^m_- times c(L_+) times c(L_-)}-(i),(viii), each $d_{a,s,\text{\rm min/max}}$ satisfies the properties stated in Theorem~\ref{t:spectrum of Delta_min/max}, and let $\D_{a,s,\text{\rm min/max}}$ denote the corresponding Laplacian. 

Using Example~\ref{ex:rel-admissible function} again, it is easy to see that there is some rel-admissible function $h_a$ on $\R^{m_a}\times M_a$ such that $h_a=0$ on $\xi(M\cap O_a)$ and $h_a=1$ on the complement of some rel-compact neighborhood of $\xi(M\cap O_a)$ in $\R^m\times M_a$. Let $\hat d_{a,s}$ and $\widehat{\D}_{a,s}$ be the Witten's perturbation of $d_a$ and $\D_a$ induced by the function $\hat f_a=h_af_a$. The functions $|d_a\hat f_a-d_af_a|$ and $|\Hess\hat f_a-\Hess f_a|$ are uniformly bounded, and therefore $\widehat{\D}_{a,s}-\D_{a,s}$ is a homomorphism with uniformly bounded norm by~\eqref{Delta_s with Hess f and df}. By the min-max principle (see e.g.\ \cite[Theorem~XIII.1]{ReedSimon1978}), we get that $\hat d_{a,s,\text{\rm min/max}}$ satisfies the properties stated in Theorem~\ref{t:spectrum of Delta_min/max}. Then Theorem~\ref{t:spectrum of Delta_min/max} follows by Propositions~\ref{p:globalization of d_min/max} and~\ref{p:globalization of d_min/max-2}.

\section{Functions of the perturbed Laplacian on strata}\label{s:functional calculus}

The first ingredient to prove Theorem~\ref{t:Morse inequalities} is the following properties of the functional calculus of the perturbed Laplacian on strata.

Let $M$ be a stratum of a compact Thom-Mather stratification equipped with an adapted metric, and let $d$ and $\D$ be the de~Rham derivative and Laplacian on $M$. Let $f$ be any rel-admissible function on $M$, and let $d_s$ and $\D_s$ be the corresponding Witten's perturbations of $d$ and $\D$. Since $f$ is rel-admissible, for each $s$, $\D_s-\D$ is a homomorphism with uniformly bounded norm by~\eqref{Delta_s with Hess f and df}. Hence $d_{s,\text{\rm min/max}}$ defines the same Sobolev spaces as $d_{\text{\rm min/max}}$. Moreover the properties stated in Theorem~\ref{t:spectrum of Delta_min/max} can be extended to the perturbation $d_{s,\text{\rm min/max}}$ by~\eqref{Delta_s with Hess f and df} and the min-max principle. 

For any rapidly decreasing function $\phi$ on $\R$, we easily get that $\phi(\D_{s,\text{\rm min/max}})$ is a Hilbert-Schmidt operator on $L^2\Omega(M)$ by the version of Theorem~\ref{t:spectrum of Delta_min/max}-(ii) for $d_{s,\text{\rm min/max}}$. In fact, $\phi(\D_{s,\text{\rm min/max}})$ is a trace class operator because $\phi$ can be given as the product of two rapidly decreasing functions, $|\phi|^{1/2}$ and $\sign(\phi)\,|\phi|^{1/2}$, where $\sign(\phi)(x)=\sign(\phi(x))\in\{\pm1\}$ if $\phi(x)\ne0$.

The extension of Theorem~\ref{t:spectrum of Delta_min/max}-(ii) to $d_{s,\text{\rm min/max}}$ also shows that $\phi(\D_{s,\text{\rm min/max}})$ is valued in $W^\infty(d_{\text{\rm min/max}})\subset\Omega(M)$. Like in the case of closed manifolds (see e.g.\ \cite[Chapters~5 and~8]{Roe1998}), it can be easily proved that $\phi(\D_{s,\text{\rm min/max}})$ can be given by a Schwartz kernel $K_s$, and $\Tr\phi(\D_{s,\text{\rm min/max}})$ equals the integral of the pointwise trace of $K_s$ on the diagonal. But we do not know whether $K_s$ is uniformly bounded by the lack of a ``rel-Sobolev embedding theorem''  (Section~\ref{s: Sobolev}).

\section{Finite propagation speed of the wave equation on strata}\label{s:wave}

Let $M$ be a stratum of a compact Thom-Mather stratification, $g$ an adapted metric on $M$, and $f$ a rel-Morse function on $M$. Let $d_s$, $\delta_s$, $D_s$ and $\D_s$ ($s\ge0$) be the corresponding Witten's perturbed operators on $\Omega(M)$, defined by $f$ and $g$. Complex coefficients are needed to consider the induced wave equation
  \begin{equation}\label{wave}
    \frac{d\alpha_t}{dt}-iD_s\alpha_t=0\;,
  \end{equation}
where $i=\sqrt{-1}$ and $\alpha_t\in\Omega(M)$ depends smoothly on $t\in\R$. We may also consider that~\eqref{wave} is satisfied only on some open subset of $M$.

If~\eqref{wave} holds on the whole of $M$, then, given $\alpha\in\DD^\infty(d_{s,\text{\rm min/max}})$, a usual energy estimate shows the uniqueness of the solution of~\eqref{wave} with the initial conditions $\alpha_0=\alpha$ (see e.g.\ \cite[Proposition~7.4]{Roe1998}). In this case the solution is given by
  \[
    \alpha_t=\exp(itD_{s,\text{\rm min/max}})\alpha\in\DD^\infty(d_{s,\text{\rm min/max}})\;.
  \]
  
Compactly supported smooth solutions of~\eqref{wave} propagate at finite speed (see e.g.\ \cite[Proposition~7.20]{Roe1998}). To prove Theorem~\ref{t:Morse inequalities}, we need a version of that result for strata, stating this finite propagation speed towards/from the rel-critical points of $f$ using forms in $\DD^\infty(d_{s,\text{\rm min/max}})$. For that purpose, we have shown first the corresponding result for the simple elliptic complexes of Sections~\ref{ss:complex 1} and~\ref{ss:complex 2}.
  
Take a rel-Morse chart around each $x\in\Crit_{\text{\rm rel}}(f)$, like in Definition~\ref{d:rel-Morse function}, with values in a stratum $M'_x=\R^{m_{x,+}}\times\R^{m_{x,-}}\times M_{x,+}\times M_{x,-}$ of a product $\R^{m_{x,+}}\times\R^{m_{x,-}}\times c(L_{x,+})\times c(L_{x,-})$, where either $M_{x,\pm}=N_{x,\pm}\times\R_+$, or $M_{x,\pm}$ is the vertex stratum $\{*_{x,\pm}\}$ of $c(L_{x,\pm})$. We can assume that the domains of these rel-Morse charts are disjoint one another by Remark~\ref{r:rel-Morse}-(i). Consider a model metric $g_x$ on each $M'_x$. For each $\rho>0$, let $B_{x,\pm,\rho}$ be the standard ball of radius $\rho$ in $\R^{m_{x,\pm}}$. If $M_{x,+}=N_{x,+}\times\R_+$ and $M_{x,-}=N_{x,-}\times\R_+$, let
  \[
    U_{x,\rho}=B_{x,+,\rho}\times B_{x,-,\rho}\times N_{x,+}\times(0,\rho)\times N_{x,-}\times(0,\rho)\subset M'_x\;.
  \]
If $M_{x,\pm}=\{*_{x,\pm}\}$, remove the factor $N_{x,\pm}\times(0,\rho)$ from the definition of $U_{x,\rho}$ (or change it by $\{*_{x,\pm}\}$). Let $d'_{x,s}$, $\delta'_{x,s}$, $D'_{x,s}$ and $\D'_{x,s}$ denote Witten's perturbed operators on $\Omega(M'_x)$ defined by $g_x$ and the model rel-Morse function (Section~\ref{s:local model of the Witten's perturbation}). The corresponding wave equation is
  \begin{equation}\label{wave, model}
    \frac{d\alpha_t}{dt}-iD'_{x,s}\alpha_t=0\;,
  \end{equation}
with $\alpha_t\in\Omega(M'_x)$ depending smoothly on $t\in\R$. By Propositions~\ref{p:EE_gamma,i},~\ref{p:FF_alpha,beta} and~\ref{p:splitting}, the following result clearly boils down to the case of Proposition~\ref{p:finite propagation speed, simple}.

\begin{prop}\label{p:finite propagation speed, model}
  For $0<a<b$, let $\alpha_t\in\DD^\infty(d'_{x,s,\text{\rm min/max}})$, depending smoothly on $t\in\R$. The following properties hold:
    \begin{itemize}
    
      \item[(i)] If $\alpha_t$ satisfies~\eqref{wave, model} on $U_{x,b}$ and $\supp\alpha_0\subset M'_x\sm U_{x,a}$, then $\supp\alpha_t\subset M'_x\sm U_{x,a-|t|}$ for $0<|t|\le a$.
      
      \item[(ii)] If $\alpha_t$ satisfies~\eqref{wave, model} on $M'_x\sm\ol{U_{x,a}}$ and $\supp\alpha_0\subset\ol{U_{x,a}}$, then $\supp\alpha_t\subset\ol{U_{x,a+|t|}}$ for $0<|t|\le b-a$.
    
    \end{itemize}
\end{prop}

There is some $\rho_0>0$ such that  each $\ol{U_{x,\rho_0}}$ is contained in the image of the rel-Morse chart centered at $x$. We will identify each $U_{x,\rho_0}$ with an open subset of $M$ via the rel-Morse chart. According to Example~\ref{ex:adapted metrics}, we can choose $g$ so that its restriction to each $U_{x,\rho_0}$ is identified to the restriction of $g_x$. 

\begin{prop}\label{p:finite propagation speed towards/from the critical points}
  Let $0<a<b<\rho_0$ and $\alpha\in L^2\Omega(M)$. The following properties hold for $\alpha_t=\exp(itD_{s,\text{\rm min/max}})\alpha$:
    \begin{itemize}
    
      \item[(i)] If $\supp\alpha\subset M\sm U_{x,a}$, then $\supp\alpha_t\subset M\sm U_{x,a-|t|}$ for $0<|t|\le a$.
      
      \item[(ii)] If $\supp\alpha\subset\ol{U_{x,a}}$, then $\supp\alpha_t\subset\ol{U_{x,a+|t|}}$ for $0<|t|\le b-a$.
    
    \end{itemize}
\end{prop}

\begin{proof}
  Since $\exp(itD_{s,\text{\rm min/max}})$ is bounded, we can assume that $\alpha\in\DD^\infty(d_{s,\text{\rm min/max}})$, and therefore $\alpha_t\in\DD^\infty(d_{s,\text{\rm min/max}})$ for all $t$. According to Remark~\ref{r:h(rho) DD^infty(d^pm_s,min/max)}, there is some $h\in C^\infty(M)$ such that $\supp h\subset U_{x,\rho_0}$, $h=1$ on $U_{x,b}$, and $h\,\DD^\infty(d_{s,\text{\rm min/max}})\subset\DD^\infty(d_{s,\text{\rm min/max}})$. Then $h\alpha_t$, considered as a differential form on $M'_x$, satisfies~\eqref{wave, model} on $U_{x,b}$ in the case of~(i), and on $M'_x\sm\ol{U_{x,a}}$ in the case of~(ii), and belongs to $\DD^\infty(d'_{s,\text{\rm min/max}})$. Thus the result follows from Proposition~\ref{p:finite propagation speed, model} because $h=1$ on $U_{x,b}$.
\end{proof}

\section{Proof of Theorem~\ref{t:Morse inequalities}}\label{s: proof of Thm Morse inequalities}

Consider the notation of Section~\ref{s:wave}.

\subsection{Analytic inequalities}\label{ss:analytic inequalities}

By~\eqref{d_s}, $e^{sf}:(\Omega_0(M),d_s)\to(\Omega_0(M),d)$ is an isomorphism of complexes, and, since $f$ is bounded, $e^{sf}:L^2\Omega(M)\to L^2\Omega(M)$ is a quasi-isometric isomorphism. So we get the isomorphism of Hilbert complexes
  \[
    e^{sf}:(\DD(d_{s,\text{\rm min/max}}),d_{s,\text{\rm min/max}})\to(\DD(d_{\text{\rm min/max}}),d_{\text{\rm min/max}})\;,
  \]
and therefore
  \begin{equation}\label{beta_min/max^r}
    \beta_{\text{\rm min/max}}^r=\dim H^r(\DD(d_{s,\text{\rm min/max}}),d_{s,\text{\rm min/max}})
  \end{equation}
for all $s\ge0$. In fact, since $|df|$ is bounded, it also follows from~\eqref{d_s} that
  \[
    \DD(d_{s,\text{\rm min/max}})=\DD(d_{\text{\rm min/max}})\;,\quad
    d_{s,\text{\rm min/max}}=d_{\text{\rm min/max}}+s\,df\wedge\;.
  \]
Thus
  \[
    e^{sf}\,\DD(d_{\text{\rm min/max}})=\DD(d_{\text{\rm min/max}})\;.
  \]

Let $\phi$ be a smooth rapidly decreasing function on $\R$ with $\phi(0)=1$. Then the operator $\phi(\D_{s,\text{\rm min/max}})$ is of trace class (Section~\ref{s:functional calculus}), and set
  \[
    \mu_{s,\text{\rm min/max}}^r=\Tr(\phi(\D_{s,\text{\rm min/max},r}))\;.
  \]
By~\eqref{beta_min/max^r}, the following result follows with the obvious adaptation of the proof of \cite[Proposition~14.3]{Roe1998}.

\begin{prop}\label{p:analytic inequalities}
  We have the inequalities
    \begin{align*}
      \beta_{\text{\rm min/max}}^0&\le\mu_{\text{\rm min/max}}^0\;,\\
      \beta_{\text{\rm min/max}}^1-\beta_{\text{\rm min/max}}^0&\le\mu_{s,\text{\rm min/max}}^1-\mu_{s,\text{\rm min/max}}^0\;,\\
      \beta_{\text{\rm min/max}}^2-\beta_{\text{\rm min/max}}^1+\beta_{\text{\rm min/max}}^0&\le\mu_{s,\text{\rm min/max}}^2-\mu_{s,\text{\rm min/max}}^1+\mu_{s,\text{\rm min/max}}^0\;,
    \end{align*}
  etc., and the equality
    \[
      \chi_{\text{\rm min/max}}=\sum_r(-1)^r\,\mu_{s,\text{\rm min/max}}^r\;.
    \]
\end{prop}

\subsection{Null contribution away from the rel-critical points}\label{ss:away from the critical points}

By~\eqref{Delta_s with Hess f and df} and because $|df|$ and $|\Hess f|$ are bounded on $M$, for all $s\ge0$,
  \begin{gather}
    \DD(\D_{s,\text{\rm min/max}})=\DD(\D_{\text{\rm min/max}})\;,
    \label{DD(Delta_s,min/max)=DD(Delta_min/max)}\\
    \D_{s,\text{\rm min/max}}=\D_{\text{\rm min/max}}+s\,\boldsymbol{\Hess}f+s^2\,|df|^2\;.
    \label{Delta_s,min/max}
  \end{gather}

For $\rho\le\rho_0$, let $U_\rho=\bigcup_xU_{x,\rho}$, with $x$ running in $\Crit_{\text{\rm rel}}(f)$. Fix some $\rho_1>0$ such that $4\rho_1<\rho_0$. Let $\fG$ and $\fH$ be the Hilbert subspaces of $L^2\Omega(M)$ consisting of forms essentially supported in $M\sm U_{\rho_1}$ and $M\sm U_{3\rho_1}$, respectively. It follows from~\eqref{DD(Delta_s,min/max)=DD(Delta_min/max)} and~\eqref{Delta_s,min/max} that there is some $C>0$ so that, if $s$ is large enough\footnote{Recall that, for symmetric operators $S$ and $T$ in a Hilbert space, with the same domain $\DD$, it is said that $S\le T$ if $\langle Su,u\rangle\le\langle Tu,u\rangle$ for all $u\in\DD$.},
  \begin{equation}\label{Delta_s,min/max ge Delta_min/max+Cs^2}
    \D_{s,\text{\rm min/max}}\ge\D_{\text{\rm min/max}}+Cs^2
    \quad\text{on}\quad\fG\cap\DD(\D_{\text{\rm min/max}})\;.
  \end{equation}

Let $h$ be a rel-admissible function on $M$ such that $h\le0$, $h\equiv1$ on $U_{\rho_1}$ and $h\equiv0$ on $M\sm U_{2\rho_1}$ (Example~\ref{ex:rel-admissible function}). Then $T_{s,\text{\rm min/max}}=\D_{s,\text{\rm min/max}}+hCs^2$, with domain $\DD(\D_{\text{\rm min/max}})$, is self-adjoint in $L^2\Omega(M)$ with a discrete spectrum, and moreover
  \begin{equation}\label{T_s,min/max ge Delta_min/max + Cs^2}
    T_{s,\text{\rm min/max}}\ge\D_{\text{\rm min/max}}+Cs^2
  \end{equation}
for $s$ is large enough by~\eqref{Delta_s,min/max ge Delta_min/max+Cs^2}.

Given any $\phi_0\in\SS_{\text{\rm ev}}$ with compactly supported Fourier transform\footnote{The Schwartz functions with compactly supported Fourier transform are characterized by the Paley-Wiener-Schwartz theorem (see e.g.\ \cite[Theorem~7.3.1]{Hormander1990-I}). They form a dense subalgebra of $\SS$, which is invariant by linear changes of variables.}, the function $\phi_1(y)=\int_{-\infty}^yx\phi_0(x)^2\,dx$ satisfies the same properties as $\phi_0$ and has a monotone restriction to $[0,\infty)$. Now, by using a linear change of variable with $\phi_1/\phi_1(0)$, we get some $\phi\in\SS_{\text{\rm ev}}$ such that $\phi\ge0$, $\phi(0)=1$, $\supp\hat\phi\subset[-\rho_1,\rho_1]$, and $\phi|_{[0,\infty)}$ is monotone. Let $\psi\in\SS$ such that $\phi(x)=\psi(x^2)$. Using Proposition~\ref{p:finite propagation speed towards/from the critical points}-(i), the argument of the first part of the proof of \cite[Lemma~14.6]{Roe1998} gives the following. 

\begin{lem}\label{l:psi(Delta_s,min/max)}
  $\psi(\D_{s,\text{\rm min/max}})=\psi(T_{s,\text{\rm min/max}})$ on $\fH$.
\end{lem}

Let $\Pi:L^2\Omega(M)\to\fH$ denote the orthogonal projection. According to Section~\ref{s:functional calculus}, $\psi(\D_{s,\text{\rm min/max}})$ is of trace class for all $s\ge0$. Then the self-adjoint operator $\Pi\,\psi(\D_{s,\text{\rm min/max}})\,\Pi$ is also of trace class (see e.g.\ \cite[Proposition~8.8]{Roe1998}).

\begin{lem}\label{l:Tr ... to0}
  $\Tr(\Pi\,\psi(\D_{s,\text{\rm min/max}})\,\Pi)\to0$ as $s\to\infty$.
\end{lem}

\begin{proof}
  The eigenvalues of $\D_{\text{\rm min/max}}$ and $T_{s,\text{\rm min/max}}$ are respectively denoted by
    \[
      0\le\lambda_{\text{\rm min/max},0}\le\lambda_{\text{\rm min/max},1}\le\cdots\;,\quad
      0\le\lambda_{s,\text{\rm min/max},0}\le\lambda_{s,\text{\rm min/max},1}\le\cdots\;,
    \]
  repeated according to their multiplicities. By~\eqref{T_s,min/max ge Delta_min/max + Cs^2} and the min-max principle,
    \[
      \lambda_{s,\text{\rm min/max},k}\ge\lambda_{\text{\rm min/max},k}+Cs^2
    \]
  for $s$ large enough. So
    \[
      \Tr(\psi(T_{s,\text{\rm min/max}}))=\sum_k\psi(\lambda_{s,\text{\rm min/max},k})\le\sum_k\psi(\lambda_{\text{\rm min/max},k}+Cs^2)
    \]
  for $s$ large enough, giving $\Tr(\psi(T_{s,\text{\rm min/max}}))\to0$ as $s\to\infty$ since $\psi$ is rapidly decreasing. Then the result follows because, by Lemma~\ref{l:psi(Delta_s,min/max)},
    \[
       \Tr(\Pi\,\psi(\D_{s,\text{\rm min/max}})\,\Pi)=\Tr(\Pi\,\psi(T_{s,\text{\rm min/max}})\,\Pi)\le\Tr(\psi(T_{s,\text{\rm min/max}}))\;.\qed
    \]
\renewcommand{\qed}{}
\end{proof}

\subsection{Contribution from the rel-critical points}\label{ss:from the rel-critical points}

The following is a direct consequence of Corollary~\ref{c:R^m_+ times R^m_- times c(L_+) times c(L_-)}.

\begin{cor}\label{c:lim_s to infty Tr(h(rho) phi(Delta_x,s,min/max,r))}
  If $h$ is a bounded measurable function on $\R_+$ such that $h(\rho)\to1$ as $\rho\to0$, then
    \[
      \lim_{s\to\infty}\Tr(h(\rho)\,\phi(\D'_{x,s,\text{\rm min/max},r}))=\nu_{x,\text{\rm min/max}}^r\;.
    \]
\end{cor}

For each $x\in\Crit_{\text{\rm rel}}(f)$, let $\widetilde{\fH}_x\subset L^2\Omega(M)$ be the Hilbert subspace of differential forms supported in $\ol{U_{x,3\rho_1}}$; it can be also considered as a Hilbert subspace of $L^2\Omega(M'_x)$ since $g$ and $g_x$ have identical restrictions to $U_{x,\rho_0}$. Moreover $\D_s$ and $\D'_{x,s}$ can be identified on differential forms supported in $U_{x,\rho_0}$. By using Proposition~\ref{p:finite propagation speed towards/from the critical points}-(ii), the argument of the first part of the proof of \cite[Lemma~14.6]{Roe1998} can be obviously adapted to show the following. 

\begin{lem}\label{l:phi(Delta_s,min/max)}
  $\phi(\D_{s,\text{\rm min/max}})\equiv\phi(\D'_{x,s,\text{\rm min/max}})$ on $\widetilde{\fH}_x$ for all $x\in\Crit_{\text{\rm rel}}(f)$.
\end{lem}

For each $x\in\Crit_{\text{\rm rel}}(f)$, let $\widetilde{\Pi}_x:L^2\Omega(M)\to\widetilde{\fH}_x$ and $\widetilde{\Pi}'_x:L^2\Omega(M'_x)\to\widetilde{\fH}_x$ denote the orthogonal projections. Since the subspaces $\widetilde{\fH}_x$ are orthogonal to each other, $\widetilde{\Pi}:=\sum_x\widetilde{\Pi}_x:L^2\Omega(M)\to\widetilde{\fH}:=\sum_x\widetilde{\fH}_x$ is the orthogonal projection.

\begin{lem}\label{l:Tr(Pi phi(Delta_s,min/max,r) Pi)=nu_min/max^r}
  $\Tr(\widetilde{\Pi}\,\phi(\D_{s,\text{\rm min/max},r})\,\widetilde{\Pi})\to\nu_{\text{\rm min/max}}^r$ as $s\to\infty$.
\end{lem}

\begin{proof}
  By Corollary~\ref{c:lim_s to infty Tr(h(rho) phi(Delta_x,s,min/max,r))} and Lemma~\ref{l:phi(Delta_s,min/max)}, and since $\Pi'_x$ is the multiplication operator by the characteristic function of $U_{x,3\rho_1}$ in $M'_x$ for all $x\in\Crit_{\text{\rm rel}}(f)$, we get
    \begin{multline*}
      \lim_{s\to\infty}\Tr(\widetilde{\Pi}\,\phi(\D_{s,\text{\rm min/max},r})\,\widetilde{\Pi})
      =\lim_{s\to\infty}\sum_x\Tr(\widetilde{\Pi}_x\,\phi(\D_{s,\text{\rm min/max},r})\,\widetilde{\Pi}_x)\\
      =\lim_{s\to\infty}\sum_x\Tr(\widetilde{\Pi}'_x\,\phi(\D'_{x,s,\text{\rm min/max},r})\,\widetilde{\Pi}'_x)
      =\sum_x\nu_{x,\text{\rm min/max}}^r=\nu_{\text{\rm min/max}}^r\;.\qed
  \end{multline*}
\renewcommand{\qed}{}
\end{proof}

By Lemmas~\ref{l:Tr ... to0} and~\ref{l:Tr(Pi phi(Delta_s,min/max,r) Pi)=nu_min/max^r}, and because $\Pi+\widetilde{\Pi}=1$, we have
  \[
    \lim_{s\to\infty}\Tr(\phi(\D_{s,\text{\rm min/max},r}))=\nu_{\text{\rm min/max}}^r\;,
  \]
showing Theorem~\ref{t:Morse inequalities} by Proposition~\ref{p:analytic inequalities}.

\section{The spaces $W^m(d_{\text{\rm min/max}})$ depend on the metric}\label{s: Sobolev}

Let $M$ be a stratum of an arbitrary compact stratification equipped with an adapted metric $g$. Since the operator $P$ of Section~\ref{s:P} has a version of the Sobolev embedding theorem \cite{AlvCalaza2014}, if the spaces $W^m(d_{\text{\rm min/max}})$ were independent of $g$, we could prove a version of the Sobolev embedding theorem for these spaces. This would allow to adapt the nice arguments of \cite[Lemma~14.6]{Roe1998} to show a stronger version of Lemma~\ref{l:Tr ... to0}: the Schwartz kernel of $\psi(\D_{s,\text{\rm min/max}})$ would converge uniformly to zero on $(M\sm U_{2\rho_1})\times(M\sm U_{2\rho_1})$. However the spaces $W^m(d_{\text{\rm min/max}})$ may depend on the choice of $g$. By taking local charts and arguing like in Section~\ref{s:proof of Theorem spectrum of D_min/max}, it is enough to check this assertion for the perturbed ``rel-local'' models $d_{s,\text{\rm min/max}}^\pm$, which can be done as follows. 

With the notation of Section~\ref{ss:1st type}, consider the case where $n$ is odd, $r=\frac{n-1}{2}$ and $a=0$; thus $\sigma=0$. We have $ \chi_0\,\gamma\in W^\infty(d_{s,\text{\rm min/max}}^\pm)$ with the metric $g$. Let $\tilde g'$ be another adapted metric on $N$, and consider the corresponding adapted metric $g'=\rho^2\tilde g'+d\rho^2$ on $M$. Let $\widetilde{\D}'$ and $\D'$ be the Laplacians on $\Omega(N)$ and $\Omega(M)$ defined by $\tilde g'$ and $g'$, respectively, and let $\D^{\prime\,\pm}_s$ be the Witten's perturbation of $\D'$ induced by the function $\pm\frac{1}{2}\rho^2$. Let $\langle\ ,\ \rangle_{\tilde g'}$ and $\langle\ ,\ \rangle'$ denote the scalar products of $L^2\Omega(N,\tilde g')$ and $L^2\Omega(M,g')$, respectively, and let $\|\ \|_{\tilde g'}$ denote the norm defined by $\langle\ ,\ \rangle_{\tilde g'}$. Suppose that  $\widetilde{\D}'\gamma\ne0$. By Corollary~\ref{c:D_s^pm}, we have $\D^{\prime\,\pm}_s=\rho^{-2}\widetilde{\D}'+H\mp s$ on $\Cinf(\R_+)\,\gamma$. Then
  \[
    \langle\D^{\prime\,\pm}_s(\chi_0\,\gamma),\chi_0\gamma\rangle'
    =\langle\widetilde{\D}'\gamma,\gamma\rangle_{\tilde g'}\int_0^\infty\rho^{-2}\chi_0^2\,d\rho
    +\|\gamma\|_{\tilde g'}^2(1\mp1)s=\infty
  \]
according to~\eqref{L^2Omega^r(M)} and Section~\ref{ss:1st type}, and because $\chi_0(\rho)=\sqrt{2}p_0e^{-s\rho^2/2}$ is bounded away from zero for $0<\rho\le1$. So $\chi_0\,\gamma\not\in W^1(d_{s,\text{\rm min/max}}^\pm)$ with the metric $g'$, obtaining different spaces $W^1(d_{s,\text{\rm min/max}}^\pm)$ by using $g$ and $g'$.

\appendix

\section{Proofs about stratifications}\label{a: proofs stratifications}

This appendix contains the proofs of the new results stated about stratifications, as well as their adapted metrics and rel-Morse functions (Sections~\ref{s:preliminaries on stratifications} and~\ref{s:rel-Morse}).

\begin{proof}[Proof of Lemma~\ref{l:a product of two cones is a cone}]
With the notation of Section~\ref{sss:cones}, let $\rho:c(L)\to[0,\infty)$ and $\rho':c(L')\to[0,\infty)$ be the radial functions, and let $\rho''=h(\rho\times\rho'):c(L)\times c(L')\to[0,\infty)$ for $h$ like in Section~\ref{sss:products}. Since the restrictions $\rho:L\times\R_+\to\R_+$ and $\rho':L'\times\R_+\to\R_+$ are submersive weak morphisms, and $h:\R_+^2\to\R_+$ is non-singular, it follows that $\rho'':c(L)\times c(L')\setminus\{(*,*')\}\to\R_+$ is a submersive weak morphism. Hence $L''={\rho''}^{-1}(1)$ is saturated in $c(L)\times c(L')$ \cite[Lemma~2.9, p.~17]{Verona1984}. Let $*''$ denote the vertex of $c(L'')$. Since $h$ is homogeneous of degree one, the mapping
  \[
    [([x,r],[x',r']),s]\mapsto([x,rs],[x',r's])
  \]
defines an isomorphism $c(L'')\to c(L)\times c(L')$. Its inverse is given by $(*,*')\mapsto*''$ and, for $(r,r')\ne(0,0)$,
  \[
    ([x,r],[x',r'])\mapsto\left[\left(\left[x,\frac{r}{h(r,r')}\right],\left[x',\frac{r'}{h(r,r')}\right]\right),h(r,r')\right]\;.\qed
  \]
\renewcommand{\qed}{}
\end{proof}

\begin{proof}[Proof of Lemma~\ref{l:uniqueness}]
  Let $(\SS,\tau)$ be a Thom-Mather stratification on $A$ satisfying the conditions of the statement. Then the elements of $\SS$ are the connected components $X$ of the sets $f^{-1}(X')$ for $X'\in\SS$, equipped with the unique differential structure so that $f:X\to X'$ is a local diffeomorphism. Thus $\SS$ is determined by $f$ and $\SS'$.
  
  Let $X\in\SS$ and $X'\in\SS'$ with $f(X)\subset X'$, and let $(T,\pi,\rho)\in\tau_X$ and $(T',\pi',\rho')\in\tau'_{X'}$ with $f(T)\subset T'$, $\pi'\, f=f\,\pi$ and $\rho'\, f=\rho$; in particular, $\rho$ is determined by $f$ and $\rho'$. Let $x\in T$ and $x'=f(x)\in T'$. Then $f\,\pi(x)=\pi'(x')$, obtaining that $\pi(x)$ is the unique point of $X\cap f^{-1}(\pi'(x'))$ that is contained in the connected component of $x$ in $f^{-1}{\pi'}^{-1}(\pi'(x'))$. It follows that $\pi$ is also determined by $f$ and $\pi'$, and therefore $\tau_X$ is determined by $f$ and $\tau'_{X'}$.
\end{proof}

\begin{proof}[Proof of Proposition~\ref{p:widehat M}]
  This is proved by induction on $\depth M$. If $\depth M=0$, then $\widehat{M}\equiv\ol{M}=M$, and there is nothing to prove.
  
  Suppose that $\depth M>0$ and the statement holds for strata of lower depth. We can assume that the strata of $\ol{M}$ is connected. For each stratum $X$ of $\ol{M}$, let $(T_X,\pi_X,\rho_X)$ be a representative of the tube around $X$ in $\ol{M}$ satisfying the conditions of Section~\ref{sss:conic bundles} with a compact Thom-Mather stratification $L_X$ and a family $\{(U_i,\phi_i)\}$ of local trivializations of $\pi_X$. The corresponding cocycle with values in $c(\Aut(L_X))$ consists of the maps $h_{ij}:U_i\cap U_j\to c(\Aut(L_X))$ defined by $h_{ij}(x)=(\phi_j\,\phi_i^{-1})(x,\cdot)$. We have $h_{ij}(x)=c(g_{ij}(x))$ for a cocycle consisting of maps $g_{ij}:U_i\cap U_j\to\Aut(L_x)$. 
  
  By the density of $M$ in $\ol{M}$ and Remark~\ref{r:by using charts}-(i), there is a dense stratum $N$ of $L_X$ so that $\phi_i(M\cap\pi_X^{-1}(U_i))=U_i\times N\times\R_+$ for all $i$. Consider triples $(x,i,P)$ such that $x\in U_i$ and $P\in\pi_0(N)$. Two triples of this type, $(x,i,P)$ and $(y,j,Q)$, are declared to be equivalent if $x=y$ and $g_{ij}(x)(P)=Q$. The equivalence class of each triple $(x,i,P)$ is denoted by $[x,i,P]$, and let $X'$ denote the corresponding quotient set. There is a canonical map $f_X:X'\to X$, defined by $f_X([x,i,P])=x$. Consider the topology on $X'$ determined by requiring that the sets $U'_{i,P}=\{\,[x,i,P]\mid x\in U_i\,\}$ are open, and the restrictions $f_X:U'_{i,P}\to U_i$ are homeomorphisms. Notice that  $f_X$ is a finite fold covering map; in particular, in the case $X=M$, $f_M$ is a homeomorphism. Consider the differential structure on each $X'$ so that $f_X$ is a local diffeomorphism.
  
  By the induction hypothesis, for each $P\in\pi_0(N)$, $\widehat{P}$ satisfies the statement of the proposition with some Thom-Mather stratification. Consider quadruples $(x,i,P,u)$ such that $x\in U_i$, $P\in\pi_0(N)$ and $u\in c(\widehat{P})$. Two such quadruples, $(x,i,P,u)$ and $(y,j,Q,v)$, are said to be equivalent if $x=y$, $g_{ij}(x)(P)=Q$ and $c(\widehat{g_{ij}(x)})(u)=v$. The equivalence class of each quadruple $(x,i,P,u)$ is denoted by $[x,i,P,u]$, and let $T'_X$ denote the corresponding quotient set. There are canonical maps, $\pi'_X:T'_X\to X'$, $\lim'_X:T'_X\to T_X$, $\rho'_X:T'_X\to[0,\infty)$ and $\iota'_X:M\cap T_X\to T'_X$ defined by $\pi'_X([x,i,P,u])=[x,i,P]$, $\lim'_X([x,i,P,u])=\phi_i^{-1}(x,c(\lim_P)(u))$, $\rho'_X([x,i,P,u])=\rho(u)$, and $\iota'_X(z)=[x,i,P,(\iota_P(v),r)]$ if $z\in M\cap\pi_X^{-1}(U_i)$ and $\phi_i(z)=(x,v,r)\in U_i\times P\times\R_+$. Notice that $f_X\,\pi'_X=\pi_X\,\lim'_X$ and $\rho_X\,\pi'_X=\rho'_X$. 
    
  Let $G\subset\Aut(L_X)$ be the subgroup generated by the above elements $g_{ij}(x)$. Since the canonical action of $G$ on $L_X$ preserves $N$, we get an induced action of $G$ on $\pi_0(N)$. Since $X$ is connected, there is a bijection between $G\backslash\pi_0(N)$ and $\pi_0(X')$, where any orbit $\OO\in G\backslash\pi_0(N)$ corresponds to the connected component $X'_\OO\in\pi_0(X')$ consisting of the points $[x,i,P]\in X'$ with $P\in\OO$. Also, let $T'_{X,\OO}=(\pi'_X)^{-1}(X'_\OO)\subset T'_X$. 
  
  Given any $\OO\in G\backslash\pi_0(N)$, fix some $P_0\in\OO$. For any other $P\in\OO$, there is some $g_P\in G$ such that $g_P(P)=P_0$. Thus the restriction $g_P:P\to P_0$ induces a map $\widehat{g_P}:\widehat{P}\to\widehat{P_0}$, and let $\phi'_{i,P}:(\pi'_X)^{-1}(U'_{i,P})\to U'_{i,P}\times c(\widehat{P_0})$ be the bijection defined by $\phi'_{i,P}([x,i,P,u])=([x,i,P],c(\widehat{g_P})(u))$. Consider the topology on $T'_{X,\OO}$ determined by requiring that the sets $(\pi'_X)^{-1}(U'_{i,P})$ are open, and the maps $\phi'_{i,P}$ are homeomorphisms. Then the maps $\phi'_{i,P}$ are local trivializations of the restriction $\pi'_{X,\OO}:T'_{X,\OO}\to X'_\OO$ of $\pi'_X$, obtaining that $\pi'_{X,\OO}$ is a fiber bundle with typical fiber $c(\widehat{P_0})$. The associated cocycle has values in $c(\Aut(\widehat{P_0}))$; in fact, it consists of the functions $h'_{i,P;j,Q}:U'_{i,P}\cap U'_{j,Q}\to c(\Aut(\widehat{P_0}))$ defined by
    \[
      h'_{i,P;j,Q}([x,i,P])(u)=c(g'_{i,P;j,Q}([x,i,P]))(u)\;,
    \]
  where $g'_{i,P;j,Q}:U'_{i,P}\cap U'_{j,Q}\to\Aut(\widehat{P_0})$ is the cocycle given by
    \[
      g'_{i,P;j,Q}([x,i,P])=\widehat{g_Q}\,\widehat{g_{ij}(x)}\,\widehat{g_P}^{-1}\;.
    \]
  The conditions of Section~\ref{sss:conic bundles} are satisfied, obtaining that $\pi'_{X,\OO}$ is a conic bundle, which induces a Thom-Mather stratification on $T'_{X,\OO}$. 
  
  Since $N_{X,\OO}:=\bigcup_{P\in\OO} P$ is $G$-invariant, the set $N_{X,\OO}\times\R_+$ is invariant by all transformations $h_{ij}(x)$ for $x\in U_{ij}$, and therefore it defines an open subspace $M_{X,\OO}\subset M\cap T_X$. Let $\lim'_{X,\OO}:T'_{X,\OO}\to T_X$, $\rho'_{X,\OO}:T'_{X,\OO}\to[0,\infty)$ and $\iota'_{X,\OO}:M_{X,\OO}\to T'_{X,\OO}$ be defined by restricting $\lim'_X$, $\rho'_X$ and $\iota'_X$. Then $(T'_{X,\OO},\pi'_{X,\OO},\rho'_{X,\OO})$ is the canonical representative of the tube of $X'$ in $T'_{X,\OO}$, $\iota'_{X,\OO}$ is a dense open embedding, $\lim'_{X,\OO}\,\iota'_{X,\OO}=\id$, and $\lim'_{X,\OO}$ is the conic bundle morphism over $f_X:X'_\OO\to X$ induced by the maps $\kappa_{i,P}:U'_{i,P}\to\Mor(\widehat{P_0},L_X)$ given by $\kappa_{i,P}([x,i,P])=\lim_P\,\widehat{g_P}^{-1}$ (Section~\ref{sss:conic bundles}). By the induction hypothesis, $\kappa_{i,P}([x,i,P])$ restricts to local diffeomorphisms between corresponding strata, and therefore $\lim'_{X,\OO}$ restricts to local diffeomorphisms between corresponding strata.
  
  On $T'_X\equiv\bigsqcup_{\OO\in G\backslash\pi_0(N)}T'_{X,\OO}$, consider the sum of the topologies and Thom-Mather stratifications of the spaces $T'_{X,\OO}$ (Remark~\ref{r:bigsqcup_iA_i}). By Lemma~\ref{l:morphism is local}-(i), $\lim'_X:T'_X\to T_X$ is a morphism that restricts to local diffeomorphisms between corresponding strata. Observe that the strata of $T'_X$ are connected.
  
  By using the local trivializations of $\pi_X$ and each $\pi'_{X,\OO}$, and Example~\ref{ex:rel-local completion of the strata of cones}, it follows that $\iota'_{X,\OO}:M_{X,\OO}\to T'_{X,\OO}$ extends to an isomorphism $\widehat{M_{X,\OO}}\to T'_{X,\OO}$ such that $\lim'_{X,\OO}$ corresponds to $\lim_{M_{X,\OO}}$. Hence $\iota'_X:M\cap T_X\to T'_X$ extends to an isomorphism $\widehat{M\cap T_X}\to T'_X$ such that $\lim'_X$ corresponds to $\lim_{M\cap T_X}$. Then, according to Remark~\ref{r:widehat M}-(ii), we can consider the spaces $T'_X$ as open subspaces of $\widehat{M}$, obtaining an open covering of $\widehat{M}$ as $X$ runs in the family of strata of $\ol{M}$. Moreover each restriction $\lim_M:T'_X\to\ol{M}\cap T_X$ restricts to local diffeomorphisms between the corresponding strata. Hence, by Lemma~\ref{l:uniqueness}, for strata $X$ and $Y$ of $\ol{M}$,  the restrictions of the Thom-Mather stratifications of $T'_X$ and $T'_Y$ to  $T'_X\cap T'_Y$ induce the same Thom-Mather stratification with connected strata. By Lemma~\ref{l:stratification is local}-(ii), it follows that there is a unique Thom-Mather stratification with connected strata on $\widehat{M}$ whose restriction to each $T'_X$ induces the above conic bundle Thom-Mather stratification. By Lemma~\ref{l:morphism is local}-(ii), $\lim_M$ is a morphism because its restriction to each $T'_X$ is a morphism. This completes the proof of~(i).
  
  In the above construction, consider every $U'_{i,P}\times P_0$ as a stratum of each $U'_{i,P}\times c(\widehat{P_0})$ via $\id\times\iota_{P_0}$. Let $g'_{i,P}$ be any Riemannian metric on $U'_{i,P}$, and let $\tilde g_0$ be an adapted metric on $P_0$ with respect to $\ol{P_0}\subset L_X$. Thus $g'_{i,P}+\tilde g_0$ is an adapted metric on $U'_{i,P}\times P_0$, and therefore, by the induction hypothesis, it is also adapted with respect to $U'_{i,P}\times c(\widehat{P_0})$. Hence, considering each $M_{X,\OO}$ as a stratum of $T'_{X,\OO}$ via $\iota'_{X,\OO}$, the restriction of $g$ to each $M_{X,\OO}$ is adapted with respect to $T'_{X,\OO}$, and~(ii) follows.
  
  Part~(iii) follows from~(i),~(ii) and Remark~\ref{r:adapted metric}-(iii).
\end{proof}

\begin{proof}[Proof of Lemma~\ref{l:|d lambda_a| is rel-locally bounded}]
   If $\depth M=0$, then the statement is obvious. Thus suppose that $\depth M>0$. For $0\le k\le\depth M$, let $\fF_k$ denote the union of all strata $X<M$ with $\depth X\le k$. The result follows from the following assertion.
  
    \begin{claim}\label{cl:|d lambda_a| is rel-locally bounded}
      For $0\le k\le\depth M$, there is a family of smooth functions $\{\lambda_{a,k}\}$ on $M$ such that:
        \begin{itemize}
        
          \item[(i)] $0\le\sum_a\lambda_{a,k}\le1$ for all $k$;
          
          \item[(ii)] $\lambda_{a,k}$ is supported in $M\cap O_a$ for all $a\in\AA$;
          
          \item[(iii)] there is some open neighborhood $U_k$ of $\fF_k$ in $A$ so that $\sum_a\lambda_{a,k}=1$ on $U_k\cap M$; and, 
          
          \item[(iv)] for any adapted metric on $M$, each function $|d\lambda_{a,k}|$ is rel-locally bounded.
          
        \end{itemize}
    \end{claim}
    
  This claim is proved by induction on $k$. To simplify its proof, observe that it is also satisfied for $k=-1$ with $\fF_{-1}=U_{-1}=\emptyset$, and $\lambda_{a,-1}=0$ for all $a\in\AA$.
  
   Now, assume that Claim~\ref{cl:|d lambda_a| is rel-locally bounded} holds for some $k\in\{-1,0,\dots,\depth M-1\}$. Let $V_k$ be another open neighborhood of $\fF_k$ in $A$ such that $\ol{V_k}\subset U_k$. We can assume that the strata of $A$ are connected by Remark~\ref{r:stratification}-(v).
   
   $\fF_{k+1}\sm \fF_k$ is the union of the strata $X$ that satisfy $\ol{X}\sm X\subset \fF_k$, and therefore the sets $X\sm\ol{V_k}$ are closed in $A\sm\ol{V_k}$ and disjoint from each other. For the strata $X\subset \fF_{k+1}\sm \fF_k$, choose representatives $(T_X,\pi_X,\rho_X)\in\tau_X$ satisfying the properties of Definition~\ref{d:stratification}-(iv)--(vi), Proposition~\ref{p:pi is a conic bundle} and Remark~\ref{r:pi_X is a conic bundle}-(ii). Let $\Phi_X$ denote the conic bundle structure of $\pi_X$. Moreover, like in Remark~\ref{r:stratification}-(ii), we can assume that the sets $T_X\sm\ol{V_k}$ are disjoint one another.
   
   By refining $\{O_a\}$ if necessary, we can suppose that, for each stratum $X\subset \fF_{k+1}\sm \fF_k$, any point in $X\sm\ol{V_k}$ is in some set $O_a$ such that there is a chart of $A$ of the form $(O_a,\xi_a)$, obtained from a local trivialization in $\Phi_X$ according to Definition~\ref{d:chart}; in this case, let $\xi_a(O_a)=B_a\times c_{\epsilon_a}(L_X)$ for some open $B_a\subset\R^{m_X}$ and $\epsilon_a>0$, where $m_X=\dim X$; let $\AA_X$ be the family the indices $a\in\AA$ that satisfy this condition. For each $a\in\AA_X$, take a smooth function $h_a:[0,\infty)\to[0,1]$ supported in $[0,\epsilon_a)$ and such that $h_a=1$ around $0$. Let $\{\,\mu_a\mid a\in\AA_X\,\}$ be a smooth partition of unity on $\fF_{k+1}\sm\ol{V_k}$ subordinated to the open covering $\{\,O_a\sm\ol{V_k}\mid a\in\AA_X\,\}$. Set $\lambda_k=\sum_a\lambda_{a,k}$. Then define
     \[
       \lambda_{a,k+1}=\lambda_{a,k}+(1-\lambda_k)\cdot\rho_X^*h_a\cdot\pi_X^*\mu_a
     \]
   if $a\in\AA_X$ for some stratum $X\subset \fF_{k+1}\sm \fF_k$, and $\lambda_{a,k+1}=\lambda_{a,k}$ otherwise. These functions are smooth on $M$ because $\lambda_k$ is smooth and equals $1$ on $U_k$. It is easy to check that they also satisfy Claim~\ref{cl:|d lambda_a| is rel-locally bounded}-(i)--(iv).
\end{proof}

To prove Proposition~\ref{p:existence of g so that f is rel-admissibility is rel-local}, we use the following lemma whose proof is elementary.

\begin{lem}\label{l:partial_i partial_j f(p)=0}
  Let $X$ be a Riemannian manifold of dimension $n$, and let $f\in\Cinf(X)$ and $p\in X$. If $df(p)\ne0$, then there is a system of coordinates $(x^1,\dots,x^n)$ of $X$ around $p$ such that $(\partial_1(p),\dots,\partial_n(p))$ is an orthonormal reference and $\partial_i\partial_j f=0$ for all $i,j\in\{1,\dots,n\}$, where $\partial_i=\partial/\partial x^i$.
\end{lem}

\begin{proof}[Proof of Proposition~\ref{p:existence of g so that f is rel-admissibility is rel-local}]
  Let $|\ |_a$ and $\nabla^a$ denote the norm and Levi-Civita connection of each $g_a$, and let $|\ |$ and $\nabla$ denote the norm and Levi-Civita connection of $g$. On every $M\cap O_a$, the functions $|df|_a$ and $|\nabla^adf|_a$ are rel-locally bounded. Since $g$ and $g_a$ are rel-locally quasi-isometric on $M\cap O_a$, we get that $|df|$ and $|\nabla^adf|$ are rel-locally bounded on $M\cap O_a$. By shrinking $\{O_a\}$ if necessary, we can assume that there are constants $K_a\ge0$ and $C_a\ge1$ such that
    \begin{gather}
      |df|,|\nabla^adf|,|d\lambda_a|\le K_a\quad\text{on}\quad M\cap O_a\;,\label{K_a}\\
      \frac{1}{C_a}\,|X|_a\le|X|\le C_a\,|X|_a\quad\forall X\in T(M\cap O_a)\;.\label{C_a}
    \end{gather}
  
  For any fixed $a_0\in\AA$, it is enough to prove that $|\nabla df|$ is bounded on $M\cap O_{a_0}$. For each $p\in M\cap O_{a_0}$, take any system of coordinates $(x^1,\dots,x^n)$ on some open neighborhood $U$ of $p$ in $M$ such that $(\partial_1(p),\dots,\partial_n(p))$ is an orthonormal reference with respect to $g$. Let $g_{a,ij}$ and $g_{ij}$ be the corresponding metric coefficients of $g_a$ and $g$ on $O_a\cap U$ and $U$, respectively; thus $g_{ij}(p)=\delta_{ij}$, and we can write $g_{ij}=\sum_a\lambda_a\,g_{a,ij}$ on $U$. As usual, the inverses of the matrices $(g_{a,ij})$ and $(g_{ij})$ are denoted by $(g_a^{ij})$ and $(g^{ij})$. By~\eqref{C_a} and since $g_{ij}(p)=\delta_{ij}$, we have
    \[
      \frac{1}{C_a^2}\,g_{a,ii}(p)\le1\le C_a^2\,g_{a,ii}(p)
    \]
  for all $i\in\{1,\dots,n\}$ if $p\in O_a$, obtaining
    \begin{align*}
      |g_{a,ij}(p)|&=\frac{1}{2}\,|\,|\partial_i(p)+\partial_j(p)|_a^2-g_{a,ii}(p)-g_{a,jj}(p)\,|\\
      &\le\frac{1}{2}\,(|\partial_i(p)+\partial_j(p)|_a^2+g_{a,ii}(p)+g_{a,jj}(p))\\
      &\le\frac{C_a^2}{2}\,(|\partial_i(p)+\partial_j(p)|^2+2)=2C_a^2
    \end{align*}
  for all $i,j\in\{1,\dots,n\}$. Since $O_{a_0}$ meets a finite number of sets $O_a$, it follows that $|g_{a,ij}(p)|$ and $|g_a^{ij}(p)|$ are bounded by some $C\ge1$, independent of the point $p\in O_{a_0}$. Similarly, by~\eqref{K_a}, we get that $|df(p)|$, $|\nabla^adf(p)|$ and $|d\lambda_a(p)|$ are bounded by some $K\ge0$ independent of the point $p\in O_{a_0}$.
  
  Let $\Gamma_{a,ij}^k$ and $\Gamma_{ij}^k$ be the Christoffel symbols of $g_a$ and $g$ on $O_a\cap U$ and $U$, respectively, corresponding to $(x^1,\dots,x^n)$. Since $g_{ij}(p)=\delta_{ij}(p)$, we have\footnote{Einstein convention is used for the sums involving local coefficients.}
    \begin{align}
      \Gamma_{ij}^k(p)&=\frac{1}{2}(\partial_ig_{jk}+\partial_jg_{ik}-\partial_kg_{ij})(p)\notag\\
      &=\frac{1}{2}\sum_a(g_{a,jk}\,\partial_i\lambda_a+\lambda_a\,\partial_ig_{a,jk}+g_{a,ik}\,\partial_j\lambda_a+\lambda_a\,\partial_jg_{a,ik}\notag\\
      &\phantom{=\frac{1}{2}\sum_a(}\text{}-g_{a,ij}\,\partial_k\lambda_a-\lambda_a\,\partial_kg_{a,ij})(p)\notag\\
      &\left.
        \begin{aligned}
          &=\frac{1}{2}\sum_a(g_{a,jk}\,\partial_i\lambda_a+g_{a,ik}\,\partial_j\lambda_a-g_{a,ij}\,\partial_k\lambda_a)(p)\\
          &\phantom{=\text{}}\text{}+\sum_a\lambda_a(p)\,\Gamma_{a,ij}^\ell(p)\,g_{a,\ell k}(p)\;.
        \end{aligned}\right\}\label{Gamma_{ij}^ell(p)}
    \end{align}
  On the other hand,
    \begin{multline}
      \nabla df=dx^i\otimes\nabla_i(\partial_kf\,dx^k)
      =\partial_i\partial_kf\,dx^i\otimes dx^k-\partial_kf\,\Gamma_{ij}^k\,dx^i\otimes dx^j\\
      =(\partial_i\partial_jf-\partial_kf\,\Gamma_{ij}^k)\,dx^i\otimes dx^j\;.\label{nabla df}
    \end{multline}
  Similarly,
    \begin{equation}\label{nabla^a df}
      \nabla^adf=(\partial_i\partial_jf-\partial_kf\,\Gamma_{a,ij}^k)\,dx^i\otimes dx^j\;.
    \end{equation}
    
  If $df(p)=0$, then
    \[
      \nabla df(p)=(\partial_i\partial_jf\,dx^i\otimes dx^j)(p)=\nabla^adf(p)
    \]
  by~\eqref{nabla df} and~\eqref{nabla^a df}, and therefore $|\nabla df(p)|\le K$.
  
  If $df(p)\ne0$, by Lemma~\ref{l:partial_i partial_j f(p)=0}, we can assume that the coordinates $(x^1,\dots,x^n)$ also satisfy $\partial_i\partial_jf(p)=0$ for all $i,j\in\{1,\dots,n\}$. So, by~\eqref{nabla df} and~\eqref{nabla^a df},
    	$$
    		\nabla df(p)=-(\partial_kf\,\Gamma_{ij}^k\,dx^i\otimes dx^j)(p)\;,\quad
      		\nabla^adf(p)=-(\partial_kf\,\Gamma_{a,ij}^k\,dx^i\otimes dx^j)(p)\;.
    	$$
  Since $g^{ij}(p)=\delta_{ij}$, it follows that $|(\partial_kf\,\Gamma_{a,ij}^k)(p)|\le K$ for all $i,j\in\{1,\dots,n\}$, and it is enough to find a similar bound for each $|(\partial_kf\,\Gamma_{ij}^k)(p)|$. But, by~\eqref{Gamma_{ij}^ell(p)},
    \begin{align*}
      |(\partial_kf\,\Gamma_{ij}^k)(p)|&\le\frac{1}{2}|df(p)|\sum_a|d\lambda_a(p)|\,(|g_{a,jk}(p)|+|g_{a,ik}(p)|+|g_{a,ij}(p)|)\\
      &\phantom{=\text{}}\text{}+\sum_a\lambda_a(p)\,|(\partial_kf\,\Gamma_{a,ij}^\ell)(p)|\,|g_{a,\ell k}(p)|\\
      &\le\left(\frac{3}{2}K^2C+KC\right)\cdot\#\{\,a\in\AA\mid O_a\cap O_{a_0}\ne\emptyset\,\}\;.\qed
    \end{align*}
\renewcommand{\qed}{}
\end{proof}

\begin{proof}[Proof of Proposition~\ref{p:existence of rel-Morse functions}]
  If $\depth M=0$, then the statement holds by the density of the Morse functions in $C^\infty(M)$ with the strong $C^\infty$ topology \cite[Theorem~6.1.2]{Hirsh1976}. Thus suppose that $\depth M>0$. Let the sets $\fF_k$ be defined like in the proof of Lemma~\ref{l:|d lambda_a| is rel-locally bounded}.
  
    \begin{claim}\label{cl:f_k}
      For $0\le k\le\depth M$, there is an open neighborhood $U_k$ of $\fF_k$ in $A$ and some $f_k\in C(U_k\cap\ol{M})$ such that, for each stratum $X\le M$,
        \begin{itemize}
        
          \item[(i)] $f_k$ restricts to a rel-Morse function on $U_k\cap X$; and,
          
          \item[(ii)] if $\depth X>k$, then:
            \begin{itemize}
            
              \item[(a)] the restriction of $f_k$ to $U_k\cap X$ has no critical points, and
          
              \item[(b)] there is some $(T_X,\pi_X,\rho_X)\in\tau_X$ such that $f_k$ is constant on the fibers of $\pi_X:U_k\cap\ol{M}\cap T_X\to X$.
        
            \end{itemize}
        \end{itemize}
    \end{claim}
    
  This assertion is proved by induction on $k$.  To simplify its proof, observe that it is also satisfied for $k=-1$ with $\fF_{-1}=U_{-1}=\emptyset$ and $f_{-1}=\emptyset$.
  
  Now, assume that Claim~\ref{cl:f_k} holds for some $k\in\{-1,0,\dots,\depth M-1\}$. Let $V_k$ be another open neighborhood of $\fF_k$ in $A$ so that $\ol{V_k}\subset U_k$. We can assume that the strata of $A$ are connected by Remark~\ref{r:stratification}-(v). For the strata $X\subset \fF_{k+1}\sm \fF_k$, we can choose representatives $(T_X,\pi_X,\rho_X)\in\tau_X$ satisfying Definition~\ref{d:stratification}-(iv)--(vi), Proposition~\ref{p:pi is a conic bundle}, Remark~\ref{r:pi_X is a conic bundle}-(ii), and Claim~\ref{cl:f_k}-(ii)-(b) with $f_k$. We can also suppose that $\pi_X^{-1}(V_k\cap X)= V_k\cap T_X$. Fix an adapted metric $g$ on $M$.
  
  Let $X$ be a stratum contained in $\fF_{k+1}\sm \fF_k$. By the density of the Morse functions in $C^\infty(X)$ with the strong $C^\infty$ topology, and since the restriction of $f_k$ to $U_k\cap X$ has no critical points by Claim~\ref{cl:f_k}-(iii), it is easy to construct a Morse function $h_X$ on $X$ such that $h_X=f_k$ on $V_k\cap X$. Since $U_k$ and $f_k$ satisfies Claim~\ref{cl:f_k}-(ii)-(b) with $(T_X,\pi_X,\rho_X)$, we get $\pi_X^*h_X=f_k$ on $V_k \cap \ol M \cap T_X$. Furthermore $h_X$ has no critical points on $X\cap V_k$ because $U_k$ and $f_k$ satisfy Claim~\ref{cl:f_k}-(ii)-(a).
  
  If $\depth M=k+1$, then $M$ is the only $X$ as above, and $f_{k+1}=h_M$ satisfies the conditions of Claim~\ref{cl:f_k}. Thus suppose that $\depth M>k+1$. Let $W_k$ be another open neighborhood of $\fF_k$ in $A$ so that $\ol{W_k}\subset V_k$. Let $\lambda_X$ be a $C^\infty$ function on $X$ such that $0\le\lambda_X\le1$, $\lambda_X=0$ on $X\cap W_k$, and $\lambda_X=1$ on $X\sm V_k$. Let $\widetilde U_{k+1}$ is the open neighborhood of $\fF_{k+1}$ given as the union of $W_k$ and the sets $T_X$ for strata $X\subset \fF_{k+1}\sm \fF_k$. The function $f_k$ on $W_k\cap\ol{M}$ and the functions $\pi_X^*h_X+\pi_X^*\lambda_X\cdot\rho_X^2$ on the sets $T_X\cap\ol{M}$ can be combined to define a function $\tilde f_{k+1}\in C(\widetilde U_{k+1}\cap\ol{M})$. For all strata $X,Y\subset\ol M$ with $\depth X=k+1$ and $\depth Y\ge k+1$, we have 
  	\begin{multline}\label{d tilde f_k+1}
		d\tilde f_{k+1}=\\
			\begin{cases}
				df_k & \text{on $Y\cap W_k$} \\
				\pi_X^*dh_X+2\rho_X\,d\rho_X & \text{on $Y\cap(T_X\sm V_k)$} \\
				\pi_X^*dh_X+\pi_X^*d\lambda_X\cdot\rho_X^2+\pi_X^*\lambda_X\cdot2\rho_X\,d\rho_X
				& \text{on $Y\cap T_X\cap(V_k\sm W_k)$}\;.
			\end{cases}
	\end{multline}
Thus $d\tilde f_{k+1}\ne0$ at every point of $Y\cap W_k$ because $f_k$ satisfies Claim~\ref{cl:f_k}-(ii)-(a). If $\depth Y>k+1$, then $d\tilde f_{k+1}\ne0$ also at every point of $Y\cap(T_X\sm V_k)$ because $d\rho_X\ne0$ modulo $\pi_X$-basic forms. Since $dh_X=df_k\ne0$ at every point in the compact subset $X\cap(\ol{V_k}\sm W_k)$ of $X\cap U_k$, it also follows from~\eqref{d tilde f_k+1} that $d\tilde f_{k+1}\ne0$ at the points of $Y\cap T_X\cap(V_k\sm W_k)$ with $\rho_X$ small enough. So the restriction $f_{k+1}$ of $\tilde f_{k+1}$ to some open neighborhood $U_{k+1}$ of $\fF_{k+1}$ in $\widetilde U_{k+1}$ satisfies Claim~\ref{cl:f_k}-(ii)-(a).

  We already know that $f_{k+1}$ restricts a rel-Morse function on $Y\cap W_k$ because it is the restriction of $f_k$. The above argument also shows that the rel-critical points of the restriction of $f_{k+1}$ to $Y\cap(U_{k+1}\sm\ol{W_k})$ must be over critical points of $h_X$ in $X\sm\ol{W_k}$, which are in $X\sm\ol{V_k}$. Since $f_{k+1}=\pi_X^*h_X+\rho_X^2$ on $Y\cap(T_X\sm V_k)$, we easily get that the rel-critical points of the restriction of $f_{k+1}$ to $Y\cap(U_{k+1}\sm\ol{W_k})$ satisfy the condition of Definition~\ref{d:rel-Morse function}. Thus $U_{k+1}$ and $f_{k+1}$ satisfy Claim~\ref{cl:f_k}-(i). On the other hand, $U_{k+1}$ and $f_{k+1}$ satisfy Claim~\ref{cl:f_k}-(ii)-(b) by Definition~\ref{d:stratification}-(vi), completing the proof of the claim.

  Finally, let us complete the proof of Proposition~\ref{p:existence of rel-Morse functions}. A basic neighborhood $\NN$ of any $h\in C^\infty(M)$ with respect to the weak $C^\infty$ topology can be determined by a finite family of charts $(U_i,\phi_i)$ of $M$, compact subsets $K_i\subset U_i$, some $k\in\N$ and some $\epsilon>0$. Precisely, $\NN$ consists of the functions $h'\in C^\infty(M)$ such that $|D^\ell((h'-h)\,\phi_i^{-1})|<\epsilon$ on $\phi_i(K_i)$ for all $i$ and $0\le\ell\le k$. By Claim~\ref{cl:f_k}, there is some open neighborhood $U$ of $\ol{M}\sm M$ in $A$ and some $f\in C(U\cap\ol{M})$ that restricts to rel-Morse functions on $U\cap X$ for all strata $X\le M$, and whose restriction to $U\cap M$ has no critical points. By shrinking $U$ if necessary, we can assume that $\ol{U}\cap K_i=\emptyset$ for all $i$. Let $V$ be another open neighborhood of $\ol{M}\sm M$ in $A$ so that $\ol{V}\subset U$. By the density of the Morse functions in $C^\infty(M)$ with the strong $C^\infty$ topology, it is easy to check that there is a Morse function $h'\in\NN$ such that $h'=f$ on $V\cap M$. Therefore $h'\in\FF\cap\NN$.
\end{proof}

\section{Proofs about Hilbert complexes}\label{a: proofs Hilbert}

This appendix contains the proofs of the new auxiliary results stated about Hilbert complexes, specially for i.b.c.\ of elliptic complexes (Sections~\ref{s:preliminaries Hilbert} and~\ref{s:Sobolev sp. defined by an i.b.c.}).

\begin{proof}[Proof of Lemma~\ref{l:minimal/maximal Hilbert complex extension}]
  Property~(i) follows because $d$ is dense in $\bd$ if each $d^a$ is dense in $\bd^a$.
  
  Now, assume the conditions of~(ii) and let $\delta=\bigoplus_a\delta^a$. Then each ${\bd^a}$ is the adjoint of the minimum Hilbert complex extension of $(\EE^a,\delta^a)$. So, by~\eqref{bd^*} and~(i), $(\DD,\bd)$ is the adjoint of the minimum Hilbert complex extension of $(\EE,\delta)$, and therefore it is the maximum Hilbert complex extension of $(\EE,d)$.
\end{proof}

\begin{proof}[Proof of Lemma~\ref{l:W^m}]
  The part ``$\text{(i)}\Rightarrow\text{(iii)}$'' follows with the arguments of the proof of the Rellich's theorem on a torus (see e.g.\ \cite[Theorem~5.8]{Roe1998}). The part ``$\text{(ii)}\Rightarrow\text{(i)}$'' follows with the arguments to prove that any Dirac operator on a closed manifold has a discrete spectrum (see e.g.\ \cite[pp.~81--82]{Roe1998}).
\end{proof}

\begin{proof}[Proof of Lemma~\ref{l:f DD(d_min/max) subset DD(d_min/max)}]
  For each $u\in\DD(d_{\text{\rm min}})$, there is a sequence $(u_n)$ in $\Cinf_0(E)$ such that $u_n\to u$ and $(du_n)$ is convergent in $L^2(E)$; in fact, $d_{\text{\rm min}}u=\lim_ndu_n$. Then $fu_n\to fu$ and 
    \[
      d(fu_n)=f\,du_n+[d,f]u_n\to f\,d_{\text{\rm min}}u+[d,f]u
    \]
   in $L^2(E)$ because $f$ and $|[d,f]|$ are bounded. So $fu\in\DD(d_{\text{\rm min}})$ and $d_{\text{\rm min}}(fu)=f\,d_{\text{\rm min}}u+[d,f]u$.
   
   Now, suppose that $u\in\DD(d_{\text{\rm max}})$. Thus there is some $v\in L^2(E)$ such that $\langle u,\delta w\rangle=\langle v,w\rangle$ for all $w\in\Cinf_0(E)$; indeed, $v=d_{\text{\rm max}}u$. Then
     \begin{multline*}
       \langle fu,\delta w\rangle=\langle u,f\delta w\rangle=\langle u,\delta(fw)-[\delta,f]w\rangle\\
       =\langle v,fw\rangle-\langle u,[\delta,f]w\rangle=\langle fv+[d,f]u,w\rangle
     \end{multline*}
   for all $w\in\Cinf_0(E)$. So $fu\in\DD(d_{\text{\rm max}})$ and $d_{\text{\rm max}}(fu)=f\,d_{\text{\rm max}}u+[d,f]u$. This completes the proof of~(i).
   
   Property~(ii) follows from~\eqref{W^1} by applying~(i) to $d$ and $\delta$.
\end{proof}

\begin{proof}[Proof of Lemma~\ref{l:zeta(f DD(d min/max)) subset DD(d'_min/max)}]
  Let $u\in f\,\DD(d_{\text{\rm min}})$. Then $u\in\DD(d_{\text{\rm min}})$ by Lemma~\ref{l:f DD(d_min/max) subset DD(d_min/max)}-(i); in fact, according to its proof, there is a sequence $(u_n)$ in $\Cinf_0(E)$ such that $u_n\to u$ and $du_n\to d_{\text{\rm min}}u$ in $L^2(E)$, and with $\supp u_n\subset\supp f$ for all $n$. Then $\zeta u_n\in\Cinf_0(E')$, $\zeta u_n\to\zeta u$ and $d'\zeta u_n=\zeta du_n\to\zeta d_{\text{\rm min}}u$ in $L^2(E')$. Hence $\zeta u\in\DD(d'_{\text{\rm min}})$ and $d'_{\text{\rm min}}\zeta u=\zeta d_{\text{\rm min}}u$.
  
  To prove the case of $d_{\text{\rm max}}$, since $\DD(d'_{\text{\rm max}})$ is invariant by quasi-isometric changes of the metrics of $M'$ and $E'$, after shrinking $U$ and $U'$ if necessary, we can assume that $\zeta:(E|_U,d)\to(E'|_{U'},d')$ is an isometric isomorphism of elliptic complexes. Such a change of metrics can be achieved by taking an open subset $V'\subset M'$ so that $\xi(\supp f)\subset V'$ and $\ol{V'}\subset U'$, and using a smooth partition of unity of $M'$ subordinated to $\{V',M'\sm\xi(\supp f)\}$ to combine metrics. Let $u\in f\,\DD(d_{\text{\rm max}})$. Then $u\in\DD(d_{\text{\rm max}})$ by Lemma~\ref{l:f DD(d_min/max) subset DD(d_min/max)}-(i); indeed, according to its proof, the support of $v:=d_{\text{\rm max}}u$ is contained in $\supp f$. Thus
    \[
      \langle\zeta u,\delta'\zeta w\rangle'=\langle\zeta u,\zeta\delta w\rangle'
      =\langle u,\delta w\rangle=\langle v,w\rangle=\langle\zeta v,\zeta w\rangle'
    \]
  for each $u\in f\,\DD(d_{\text{\rm max}})$ and all $w\in\Cinf_0(E|_U)$. So $\langle\zeta u,\delta'w'\rangle'=\langle\zeta v,w'\rangle'$ for all $w'\in\Cinf_0(E')$, obtaining $\zeta u\in\DD(d'_{\text{\rm max}})$ and $d_{\text{\rm max}}(\zeta u)=\zeta d_{\text{\rm max}}u$. This completes the proof of~(i).
  
  If $\zeta$ is isometric, then it is also an isometric isomorphism $(E|_U,\delta)\to(E'|_{U'},\delta')$. So~(ii) follows from~\eqref{W^1} by applying~(i) to $d$ and $\delta$.
\end{proof}



\providecommand{\bysame}{\leavevmode\hbox to3em{\hrulefill}\thinspace}
\providecommand{\MR}{\relax\ifhmode\unskip\space\fi MR }
\providecommand{\MRhref}[2]{%
  \href{http://www.ams.org/mathscinet-getitem?mr=#1}{#2}
}
\providecommand{\href}[2]{#2}

\end{document}